\RequirePackage{fix-cm}
\documentclass{amsart} 
\usepackage{mathptmx}
\usepackage{latexsym}
\usepackage{amsmath}
\usepackage{amssymb}
\usepackage{amsfonts} 
\usepackage{eucal} 
\usepackage{amsbsy}
\usepackage{amsthm}
\usepackage{graphicx}
\usepackage[all]{xy}
\xyoption{knot}
\usepackage{hyperref}

\newtheorem{thm}{Theorem}[section]
\newtheorem*{thm*}{Theorem}
\newtheorem{lemma}[thm]{Lemma}
\newtheorem{prop}[thm]{Proposition}
\newtheorem{cor}[thm]{Corollary}
\newtheorem*{cor*}{Corollary}

\theoremstyle{definition}
\newtheorem{defn}[thm]{Definition}
\newtheorem{example}[thm]{Example}
\newtheorem{examples}[thm]{Examples}

\theoremstyle{remark}
\newtheorem{remark}[thm]{Remark}

\newcommand {\Ba}    {\ensuremath{\mbox{$\mathcal{B}$}}}

\newcommand {\Fa}    {\ensuremath{\mbox{$\mathcal{F}$}}}

\newcommand {\real}  {\ensuremath{\mathbb{R}}}
\newcommand {\mbs}  {\ensuremath{\mathbb{S}}}
\newcommand {\intg}  {\ensuremath{\mathbb{Z}}}

\newcommand {\cplx}  {\ensuremath{\mathbb{C}}}

\newcommand {\Hom}   {\ensuremath{\operatorname{Hom}}}

\newcommand {\Ima}    {\operatorname{Im}}

\newcommand {\interi}{\operatorname{int}}

\newcommand {\smlhf} {\ensuremath{\mbox{$\frac{1}{2}$}}}

\newcommand {\supp}  {\ensuremath{\operatorname{supp}}}

\newcommand {\id}    {\ensuremath{\operatorname{id}}}

\newcommand {\hotimes}   {\ensuremath{\widehat{\otimes}}}

\newcommand {\catc}   {\ensuremath{\mathbf{C}}}
\newcommand {\catd}   {\ensuremath{\mathbf{D}}}

\newcommand {\vect}   {\ensuremath{\mathbf{Vect}}}

\newcommand {\fm} {\ensuremath{\operatorname{FM}}}
\newcommand {\nn}   {\ensuremath{\mathbb{N}[\mathbb{N}]}}
\newcommand {\nat}   {\ensuremath{\mathbb{N}}}
\newcommand {\bool}   {\ensuremath{\mathbb{B}}}

\newcommand {\closed} {\ensuremath{\operatorname{closed}}}
\newcommand {\dom} {\ensuremath{\operatorname{dom}}}
\newcommand {\cod} {\ensuremath{\operatorname{cod}}}

\newcommand {\tr} {\ensuremath{\operatorname{Tr}}}
\newcommand {\reg} {\ensuremath{\operatorname{Reg}}}
\newcommand {\genim} {\ensuremath{\operatorname{GenIm}}}

\newcommand {\Ob} {\ensuremath{\operatorname{Ob}}}
\newcommand {\Mor} {\ensuremath{\operatorname{Mor}}}

\newcommand {\card} {\ensuremath{\operatorname{card}}}

\newcommand {\wlh} {\ensuremath{W_{\leq 1/2}}}
\newcommand {\wgh} {\ensuremath{W_{\geq 1/2}}}
\newcommand {\fim} {\ensuremath{f|_{I\times M}}}
\newcommand {\fimk} {\ensuremath{f(k)|_{I\times M(k)}}}
\newcommand {\fun} {\ensuremath{\operatorname{Fun}}}

\newcommand {\lh}   {\ensuremath{\widehat{\lambda}}}
\newcommand {\OP} {\ensuremath{\operatorname{OP}}}
\newcommand {\tpi}   {\ensuremath{\widetilde{\pi}}}

\newcommand {\ophi} {\ensuremath{\overline{\phi}}}
\newcommand {\Br} {\ensuremath{\mathbf{Br}}}
\newcommand {\bz} {\ensuremath{\mathbf{0}}}

\begin{document}

\UseComputerModernTips


\title{High-Dimensional Topological Field Theory, Positivity, and Exotic Smooth Spheres}

\author{Markus Banagl}

\address{Mathematisches Institut, Universit\"at Heidelberg,
  Im Neuenheimer Feld 288, 69120 Heidelberg, Germany}

\email{banagl@mathi.uni-heidelberg.de}

\date{August, 2015}

\subjclass[2010]{57R56, 81T45, 57R45, 57R60, 57R55, 16Y60.}

\keywords{Topological quantum field theories, smooth manifolds, homotopy spheres, singularities of smooth maps, fold maps, complete semirings.}


\begin{abstract}
In previous work, we proposed a general framework of positive topological field
theories (TFTs) based on Eilenberg's notion of summation completeness for semirings.
In the present paper, we apply this framework in constructing explicitly a concrete positive
TFT defined on smooth manifolds of any dimension greater than $1$.
We prove that this positive TFT detects exotic smooth spheres.
We show further that polynomial invariants (subject to boundary conditions) can be extracted from the
state sum if the dimension of the cobordisms is at least $3$.
\end{abstract}

\maketitle


\tableofcontents


\section{Introduction}

In \cite{banagl-postft}, we introduced a general method which assigns
to any system of fields on cobordisms and system of action functionals on these
fields a positive topological field theory (TFT). Here, positivity refers to the fact that the state sum of such
a theory takes values in a semiring, and is not definable over rings.
Thus positivity here does not refer to positivity of the associated Hermitian
pairing on the state modules in a unitary theory, as e.g. in \cite{kreckteichnerpostft}.
The reason why semirings are used is a fundamental discovery of Eilenberg (\cite{eilenbergalm}):
Certain semirings allow for an additive completeness that enables summation
over families with large index sets. The so-called ``Eilenberg-swindle'' shows
that such a completeness can never hold in nontrivial rings.
In \cite{banagl-postft}, we also announced an application of the abstract
framework to explicitly constructing a positive TFT for all compact smooth manifolds
in every dimension at least $2$, which is capable of detecting exotic smooth structures. 
In the present paper, we carry out this construction in full detail. From now on,
all manifolds are understood to be smooth.\\

Based on certain algebraic data, which are independent of any manifold and 
to be described more precisely below, we first construct a commutative,
Eilenberg-complete monoid $Q$, which comes naturally equipped with two
multiplications, yielding semirings $Q^c$ and $Q^m$. The elements of
$Q$ are formal power series in one variable $q$ with coefficients in finite-dimensional
matrices. The variable $q$ has the interpretation as a bookkeeping device
for certain loops on cobordisms.
Given any integer $n\geq 2,$ we construct a theory $Z$ with the following properties.
To any closed $(n-1)$-manifold $M$ (not necessarily orientable), we assign state-semimodules
$Z(M)$ over $(Q^c, Q^m)$. These are infinitely generated unless $M$ is empty.
Therefore, as in Hilbert space theory, the algebraic tensor product is inadequate for
describing the states of a disjoint union $M\sqcup N$, where $N$ is another closed
$(n-1)$-manifold. Instead, using a suitably completed tensor product
$\hotimes$, we construct in Proposition \ref{prop.zmonoidalonmodules} 
an isomorphism 
\[ Z(M\sqcup N) \cong Z(M) \hotimes Z(N). \]
Diffeomorphisms $\phi: M\to N$ induce isomorphisms
$\phi_*: Z(M)\to Z(N)$.
If $\psi: N\to P$ is another diffeomorphism, then the functoriality relation
\[ \psi_\ast \circ \phi_\ast = (\psi \circ \phi)_\ast \]
holds, as well as
$(\id_M)_\ast = \id_{Z(M)}: Z(M)\to Z(M).$
There is a contraction product
\[ \langle \cdot, \cdot \rangle: (Z(M)\hotimes Z(N))\times (Z(N)\hotimes Z(P)) 
 \longrightarrow Z(M)\hotimes Z(P), \]
which serves as a means of propagating states along cobordisms and uses the
multiplication of $Q^c$.
To any $n$-dimensional cobordism $W^n$ (not necessarily orientable) from $M$ to $N$, 
we assign an element
\[ Z_W \in Z(M)\hotimes Z(N), \]
called the \emph{state sum} of $W$.
Technically, we work with embedded cobordisms, which is no loss of generality,
but it needs to be investigated to what extent the state sum is independent of the embedding.
Let $W$ be a cobordism from $M$ to $N$ and
$W'$ a cobordism from $M'$ to $N'$. If $W$ and $W'$ are disjoint, then
\[ Z_{W\sqcup W'} = Z_W \hotimes Z_{W'} \]
under the identification
\[ Z(M\sqcup M')\hotimes Z(N\sqcup N') \cong 
  Z(M)\hotimes Z(M')\hotimes Z(N)\hotimes Z(N') \]
(Theorem \ref{thm.statesumdisjunion}).
The tensor product on the right hand side of the above equation of state sums
uses the multiplication in $Q^m$.
The most important characteristic of any TFT is the gluing property.
Let $W'$ be a cobordism from $M$ to $N$ and $W''$ a cobordism from
$N$ to $P$.
Let $W=W' \cup_N W''$ be the cobordism from $M$ to $P$
obtained by gluing $W'$ and $W''$ along $N$.
Then the Gluing Theorem \ref{thm.gluingmain} asserts that
the state sums
\[ Z_{W'} \in Z(M)\hotimes Z(N),~ Z_{W''} \in Z(N)\hotimes Z(P),~
  Z_W \in Z(M)\hotimes Z(P) \]
are related by the gluing law
\[ Z_W = \langle Z_{W'}, Z_{W''} \rangle, \]
where $\langle \cdot, \cdot \rangle$ is the above contraction product.
The embedding of a cobordism gives in particular rise to a ``time'' function.
We say that time on $W$ is \emph{progressive}, if at every time $t$, there is a 
point on $W$ at which the time function is submersive.
Let $W$ be a cobordism with progressive time from $M$ to $N$ and 
$W'$ a cobordism with progressive time from $M'$ to $N'$.
Let $\phi: \partial W \to \partial W'$ be a diffeomorphism that maps $M$ to $M'$ (and thus $N$ to $N'$).
We show in Theorem \ref{thm.tcdiffeoinvariance} that
if $\phi$ extends to a time consistent diffeomorphism (cf. Definition \ref{def.timeconsistent})
$\Phi: W\to W'$, then
the state sums of $W$ and $W'$ are related by
\[ \phi_\ast (Z_W) = Z_{W'} \in Z(M')\hotimes Z(N') \]
under the isomorphism
\[ \phi_\ast: Z(M)\hotimes Z(N)\longrightarrow Z(M')\hotimes Z(N'). \]
In particular, the state sum can at most depend on the time function of the embedding,
if at all.\\

The value of the state sum $Z_W$ on a given boundary condition lies in $Q$ and hence is a
power series in $q$. In Section \ref{sec.rationality},
we prove that for cobordisms $W$ of dimension $n\geq 3$, this value is
actually a rational function in $q$ (Theorem \ref{thm.rationalinq}). Moreover, the denominator is universal
(independent of $W$), and hence all the information is contained in a polynomial in $q$
(which does depend on the boundary conditions). \\

It is not a priori clear that TFTs are capable of recognizing exotic smooth structures.
Could it not be that the gluing law required of a TFT trivializes the theory to such an
extent, that it cannot recognize the subtle global phenomenon of exotic smooth structures?
The present paper shows that this is not the case. Recall that an \emph{exotic sphere}
is a closed smooth $n$-dimensional manifold $\Sigma^n$ which is homeomorphic but
not diffeomorphic to the smooth standard sphere $S^n$. For any smooth closed manifold 
$M^n,$ $n\geq 5,$ homeomorphic to $S^n$, we define in Section \ref{sec.exoticspheres}
an invariant
$\mathfrak{A}(M)\in Q$ and call it the \emph{aggregate invariant} of $M$.
In fact we already showed in \cite{banagl-postft} that, given any positive TFT, this kind of invariant is
always defined. It, too, uses Eilenberg-completeness in an essential way and
has no counterpart in classical TFTs over rings.
We prove that if $\Sigma^n,$ $n\geq 5,$ is exotic, then $\mathfrak{A}(\Sigma^n)\not=
\mathfrak{A}(S^n)$ in $Q$ (Corollary \ref{cor.asigmanotasn}).
The proof relies on the work \cite{saekihtpyspheres} of O. Saeki on cobordism of so-called
special generic functions.
In seeking a TFT that can detect exotic spheres, one is initially encouraged by Milnor's
classical $\lambda$-invariant (\cite{milnor7sphere}) in dimension $7$. 
It is given on a closed oriented $7$-manifold $M$
by first choosing a compact, smooth, oriented
manifold $W^8$ with $\partial W=M$ (always possible) and then setting
$\lambda (M) = 2p_1^2 [W] - \sigma (W)$ modulo $7$,
where $p^2_1 [W]$ is a Pontrjagin number and $\sigma (W)$ is the signature of the 
intersection form. This residue class modulo $7$ is independent of the choice of $W$ by the 
Hirzebruch signature theorem for closed manifolds.
By Novikov additivity, the signature always behaves like a TFT.
It turns out that the Pontrjagin number is also additive, and thus behaves like a TFT,
\emph{provided} $H^3 (M)= 0 = H^4 (M)$. In fact, only under these assumptions
is the Pontrjagin number $p^2_1 [W]$ topologically well-defined.
It follows that $\lambda (W) := 2p_1^2 [W] - \sigma (W) \in \intg$ transforms like a TFT
if $H^3 (\partial W)=0=H^4 (\partial W)$. 
The problem is that Pontrjagin numbers do not in general satisfy a gluing law:
Easy examples such as $\cplx \mathbb{P}^2$ show that
it may well be that $W=W' \cup_N W''$ has a nontrivial Pontrjagin number, while
the Pontrjagin classes of $W'$ and $W''$ are trivial. This can also be understood from the
analytic viewpoint: A choice of Riemannian metric on $W$ determines an associated
curvature two-form $R$ and, expressed in terms of $R$, 
top-degree differential form representatives $p$ of polynomials in Pontrjagin classes.
By the additivity of integrals, $\int_W p$ does of course satisfy a gluing law, but this
integral is not topological. Indeed, according to Atiyah-Patodi-Singer \cite{APSI}, if we use
the Hirzebruch $L$-polynomial, then 
$\int_W p$ is the sum of $\sigma (W)$ and the $\eta$-invariant of the boundary,
at least when $W$ is isometric to a product near $\partial W$. The $\eta$ invariant,
however, is known to depend on the metric of the boundary. Thus, at best, one would
obtain a Riemannian field theory, not a topological field theory.
Kreck, Stolz and Teichner informed us that using the work of
Galatius, Madsen, Tillmann and Weiss (\cite{galmadtilwei}), they obtained results on expressing
Pontrjagin numbers as partition functions of invertible TFTs.
Witten's paper \cite{wittengravanoms} shows how exotic spheres arise in Physics in the context of
global gravitational anomalies.\\

This brings us to our method of constructing $Z$, which does not use Pontrjagin numbers
at all. For our overall blueprint, we take inspiration from Physics:
First, specify a system of fields on cobordisms and closed manifolds.
Then construct an action functional on these fields. This action should assign
a ``quantized'' algebraic/combinatorial structure to a given smooth field.
Finally define the state sum by summing the actional functional of all fields on the
cobordism. The fundamental idea of positive TFTs, as explained in \cite{banagl-postft},
is to use Eilenberg-completeness to implement this summation, thereby
circumventing measure-theoretic problems that frequently plague the rigorous development
of path integrals.
As the fields, we use in principle smooth maps from the cobordism into the plane which have
only fold singularities. Such maps are not dense in the space of all smooth maps
and are known to carry a lot of information on the topology of their domain.
If the boundary of a connected cobordism $W$ is empty, then a smooth map $W\to \real^2$
is homotopic to a fold map if and only if the Euler characteristic of $W$ is even.
Thus our TFT will have nothing to say about closed cobordisms of odd Euler characteristic.
If the boundary of $W$ is not empty, then the parity of the Euler characteristic is no 
obstruction to the existence of fold maps: For example if $W$ is the cylinder $[0,1]\times M$
on a closed manifold $M$ of odd Euler characteristic, then suspending a Morse function
on $M$ will give a fold map on $W$, despite the Euler characteristic of $W$ being odd.
We summarize the required basic properties of fold maps in Section \ref{sec.foldmaps}.
The problem is that using all such fold maps will not yield a TFT for two reasons:
it is not clear how to define the action functional, and the gluing law cannot be established.
Therefore, one must find suitable further conditions to place on a fold map.
This is somewhat delicate since one must simultaneously enable the gluing law
while not losing too many fold maps in order to still detect exotic spheres.
This problem is solved in Definition \ref{def.foldfield}, where we introduce the concept
of a fold \emph{field}, which is central to this work.
We do not prove any new results on fold maps. The goal was rather to understand
whether, and how, they can be used to obtain TFTs.
Now given a fold field on a cobordism, we must define the action functional on it.
This is done in two steps: The first one uses the singular set of the fold map to assign
a combinatorial object to the map. The second step uses representation theory to map 
the combinatorial object to linear algebra.
For step one, we define the Brauer category $\Br$, a categorification of the
Brauer algebras $D_m$. The latter arose in the representation theory of the orthogonal
group $O(n)$, see \cite{brauer}, \cite{wenzl}, and have since played an important role in
knot theory. The Brauer category is a
compact (i.e. every object is dualizable), symmetric,
strict monoidal category.
Loosely speaking, the morphisms are $1$-dimensional unoriented tangles in a 
high-dimensional Euclidean space. As those can always be disentangled,
$\Br$ is very close, but not equal, to the category of $1$-dimensional
cobordisms. One difference is that the
objects of the latter, being $0$-manifolds, are unordered (finite) sets, whereas the
objects of $\Br$ are \emph{ordered} tuples of points.
Another difference is that the cobordism category has a huge number of
objects (though few isomorphism types), whereas the Brauer category has very
few objects to begin with and has the property that two objects are isomorphic
if and only if they are equal. The important point is that the morphisms of $\Br$
are determined entirely by the combinatorics of their endpoint connections and the
number of loops they contain.
In Section \ref{ssec.actionfoldfieldsbrauer}, 
we construct a map $\mbs$, which
assigns to every fold field a morphism in the Brauer category.
For the second step, we construct in Theorem \ref{thm.functorytangvect}
linear representations of the Brauer category,
that is, symmetric strict monoidal functors $Y:\Br \to \vect,$ where $\vect$ is the
category of finite dimensional real vector spaces and linear maps.
It is technically important here that one endows $\vect$ (as such, i.e. without
changing its objects or morphisms) with a tensor product
so that it becomes a \emph{strict} monoidal category. That this can in fact be
done was shown by Schauenburg in \cite{schauenburg}.
To construct representations $Y$, we begin in Section \ref{ssec.linearduality} by introducing
the algebraic concept of a \emph{duality structure} (Definition \ref{def.dualitystructure}).
Using techniques from knot theory introduced by Turaev in
\cite{turaevtanglesrmatrices} (see also \cite{kassel}), we find in
Proposition \ref{prop.tanggensrels} a presentation of the Brauer category by generating 
morphisms and relations.
This result then allows us to prove that every duality structure on a finite dimensional
vector space gives rise to a symmetric strict monoidal functor $Y$.
In order to detect exotic smooth structures, it is important that some of these representations $Y$
are faithful on loops. This is shown in Proposition \ref{prop.expleloopfaithful}.
Given such a $Y$, we form the composition $Y\circ \mbs$. Given a cobordism $W$,
we get the state sum $Z_W$ in Section \ref{ssec.statesums} by
evaluating $Y\circ \mbs$ on all fold fields 
(obeying boundary conditions)
on $W$ and summing the resulting values
using the Eilenberg-completeness of $Q$ (see in particular Equation (\ref{equ.statesumdef})). \\

The commutative monoid $Q$ depends on a choice of duality structure and uses
the associated representation $Y$ as described above. It is constructed in Section
\ref{sec.proidemcompletion} by applying a general process called profinite idempotent
completion. Roughly, this process completes a subset of a vector space which is invariant
under rescaling by powers of a given scalar (in the application the trace of the loop) to
an idempotent complete semimodule. The rest of that section describes the purely algebraic,
as well as the important completeness and continuity properties, of $Q$.
It also introduces the semirings $Q^c$ and $Q^m$, whose common underlying
additive monoid is $Q$. The product on $Q^c$ comes from the composition of morphisms,
while the product on $Q^m$ comes from the monoidal product of morphisms.\\

The informational content of our state sum invariant seems to be a mixture of
index constraints such as
Morse inequalities, i.e. roughly homological content, and certain
characteristic classes. However, our TFT packages this mixture in such
a way that a Gluing Theorem holds, which does not involve contributions
from the common boundary. Recall that homology itself does not glue in this
sense, since the Mayer-Vietoris sequence \emph{does} involve the
homology of the boundary. Similarly, as discussed above, characteristic classes and numbers
do not in general glue like TFT-type invariants.\\

Atiyah's original TQFT axioms imply that
state modules are always finite dimensional, but on p. 181 of
\cite{atiyahtqft}, he indicates that allowing state modules to be 
infinite dimensional may be necessary in interesting examples of TFTs.
The state modules of our positive TFT are indeed infinitely generated, as mentioned before.
Atiyah's classical axioms also demand that 
the state sum $Z_{I\times M}$ of a cylinder $W=I\times M$ be the identity
when viewed as an endomorphism.
The map $Z_{I\times M}$ on $Z(M)$ should be the
``imaginary time''  evolution operator $e^{-tH}$ (where $t$ is the length
of the interval $I$), so Atiyah's axiom means that the Hamiltonian
$H=0$ and there is no dynamics along a cylinder.
We do not require this for positive TFTs and will allow interesting
propagation along the cylinder. In particular, one can not phrase positive
TFTs as monoidal functors on cobordism categories, as this would imply that
the cylinder, which is an identity morphism in the cobordism category, would
have to be mapped to an identity morphism. 
Since $Z_{I\times M}$ need not be the identity, it can also
generally not be deduced in a positive TFT that $Z_{S^1 \times M} =
\dim Z(M)$, an identity that would not make sense in the first place,
as $Z(M)$ need not have finite dimension.
The present paper does not make substantial use of the cobordism category. 
Nor do we consider $n$-categories or
manifolds with corners here. \\

We have tried to make the exposition as self-contained as possible.
In particular, the reader is not assumed to be familiar with \cite{banagl-postft}.
Likewise, no prior exposure to semirings and semimodules over them is necessary.
Section \ref{sec.monssemirings} is devoted to a careful review of the required
elements of semimodule theory. Particular emphasis is placed on explaining the
notion of Eilenberg-completeness, as well as additive idempotence and continuity.
Proposition \ref{prop.eilensummaslovint} relates Eilenberg summation to
idempotent integration theory as developed in idempotent analysis 
by the Russian school around V. P. Maslov, see e.g. \cite{litmasshpizidemfunctana}.
It follows that the entire paper, in particular our state sum, could be recast
in terms of Maslov's idempotent integral. A few words about the completed
tensor product $\hotimes,$ which is the topic of Section \ref{sec.funcompltensorprod}, are in order.
In the complete idempotent \emph{commutative} setting, Litvinov, Maslov and Shpiz have
constructed in \cite{litmasshpiztensor} such a tensor product. 
This is an idempotent analog of topological tensor products in the sense of Grothendieck.
Now, the present paper needs such an analog even over noncommutative semirings,
since the semirings $Q^c, Q^m$ are generally not commutative.
We construct such a completed tensor product $\hotimes$ in Section 
\ref{sec.funcompltensorprod}. Our main result is Theorem \ref{thm.funiso},
which provides a tensor decomposition of function semimodules.
This is essential for establishing some of the above properties of $Z$.\\

\textbf{General Notation.} 
The letter $I$ will denote both the unit interval $[0,1]$ and unit objects in
monoidal categories.
The symbol $1$ will denote the (real) number as well as identity morphisms
of various categories. In addition, we often find it convenient to denote certain
identity morphisms by the symbol $\id$. We do this in the interest of readability:
We usually reserve $\id_X$ for the identity map on some topological space $X$,
whereas $1_X$ is used for any identity morphism in categories different from the
topological one.
The imaginary part of a complex number $z$ will be denoted by $\Ima (z)$.
This is not to be confused with the image of a map. Thus $\Ima F$ is the
imaginary part of a complex valued function $F$, not the image of $F$. 
Categories will be denoted by boldface letters such as $\catc$. 
We write $\Ob \catc$ for the class of objects and $\Mor \catc$ for the class
of morphisms.
If $\phi$ is a morphism in some category, then $\dom \phi$ and $\cod \phi$
denote its domain and codomain, respectively.
Closed smooth 
$(n-1)$-dimensional manifolds will be named $M,N,P$, while compact smooth
$n$-manifolds with possibly nonempty boundary (cobordisms) will be named
$W,W',$ etc.
For the present paper, we wish the natural numbers to be an additive monoid.
Thus we include zero, $\nat = \{ 0,1,2,\ldots \}.$\\

\textbf{Acknowledgments.}
We are grateful to Vladimir Turaev for helpful comments and to Osamu Saeki
for providing us with some initial orientation concerning fold maps.
We thank Dominik Wrazidlo for his remarks on an early version of the manuscript,
leading to several improvements. \\

\section{The Brauer Algebra and its Categorification}
\label{sec.brauercat}

In the course of the construction
of our TFT, two particular monoidal categories will play an important role:
the Brauer category and the category of real vector spaces $\vect$.
The Brauer algebras $D_m$ arose in the representation theory of the orthogonal
group $O(n)$, see \cite{brauer}, \cite{wenzl}, and have since played an important role in
knot theory. We shall require a categorification, which we denote by $\Br$, of Brauer's algebras.
The present section gives a detailed construction of this categorification, together with
derivations of its main properties. We also construct linear representations of $\Br$ 
(Theorem \ref{thm.functorytangvect}) and prove
that there exist linear representations of $\Br$ which are faithful on loops
(Proposition \ref{prop.expleloopfaithful}).

Let $\vect$ denote the category of finite dimensional real vector spaces and
linear maps. We shall endow $\vect$ with the structure of a symmetric 
strict monoidal category.

\subsection{The Schauenburg Tensor Product}

The ordinary tensor product of vector spaces is well-known
not to be associative, though it is associative up to natural isomorphism. Thus,
if we endowed $\vect$ with the ordinary tensor product
and took the unit object $I$ to be the one-dimensional
vector space $\real$, then, using obvious associators and unitors,
$\vect$ would become a monoidal category, but not a strict one.
There is an abstract process of turning a monoidal category $\catc$ into 
a monoidally equivalent strict monoidal category $\catc^{\operatorname{str}}$. 
However, this process
changes the category considerably and is thus not always practical.
Instead, we base our monoidal structure on the Schauenburg tensor product
$\odot$ introduced in \cite{schauenburg}, which does not change the
category $\vect$ at all. 
The product $\odot$ satisfies the strict associativity
\[ (U\odot V)\odot W = U\odot (V\odot W). \]
We shall thus simply write $U\odot V \odot W$ for this vector space.
The unit object $I$ remains the same as in the usual nonstrict monoidal structure, $I=\real$,
and one has
$V\odot I = V,~ I\odot V =V.$
The strict
monoidal category $(\vect, \odot, I)$ thus obtained is monoidally equivalent to the usual nonstrict
monoidal category of vector spaces. The underlying functor of this monoidal
equivalence is the identity.
In particular, there is a natural isomorphism $\xi: \otimes \to \odot,$
$\xi_{VW}: V\otimes W \to V\odot W$, where $\otimes$ denotes the standard
tensor product of vector spaces.
Note that via $\xi$ we are able to speak of elements
$v\odot w\in V\odot W,$ $v\odot w := \xi_{VW} (v\otimes w),$ $v\in V,$ $w\in W$.
The identity
$(u\odot v)\odot w = u\odot (v\odot w)$
holds for elements $u\in U, v\in V$ and $w\in W$. 
The basic idea behind the construction of $\odot$ is to set up a specific new
equivalence of categories 
$L: \vect \rightleftarrows \vect^{\operatorname{str}}: R$ such that $LR$ is the
identity and then setting $V\odot W = R(LV*LW),$ where $*$ is the strictly associative 
tensor product in $\vect^{\operatorname{str}}$. Then
\begin{eqnarray*}
(U\odot V)\odot W & = & R(L(U\odot V)*LW) =
  R(LR (LU*LV)*LW) = R((LU*LV)*LW) \\
& = & R(LU*(LV*LW)) = R(LU*LR(LV*LW)) =
   R(LU*L(V\odot W)) \\
& = & U\odot (V\odot W).
\end{eqnarray*}
Let $\beta$ be the standard braiding $\beta (v\otimes w)=w\otimes v$ on $V\otimes W$.
Then $\beta$ is symmetric and satisfies the two hexagon equations with respect to the
standard associator $\alpha: (V\otimes W)\otimes U \to V\otimes (W\otimes U)$.
We define isomorphisms
$b_{VW}: V\odot W \to W\odot V$
by $b_{VW} = \xi_{WV} \circ \beta_{VW} \circ \xi^{-1}_{VW}$, i.e.
$b_{VW} (v\odot w)=w\odot v$.
Note that $b$ is natural and one verifies easily:
\begin{prop}
The natural isomorphism $b$ is a symmetric braiding on the strict monoidal
category $(\vect, \odot, I)$.
\end{prop}
Thus $(\vect, \odot, I, b)$ is a symmetric strict monoidal category.
For the rest of this paper we will always use the Schauenburg tensor product
on $\vect$ and thus will from now on write $\otimes$ for $\odot$.

\subsection{Linear Duality}
\label{ssec.linearduality}

Let us recall the general notion of duality in a monoidal context.
\begin{defn} \label{def.dual}
An object $X$ of a symmetric monoidal category $\catc$ is called \emph{(right) dualizable},
if there is an object $X^\ast \in \catc$, called a
\emph{dual} of $X$, and morphisms
\[ i_X: I\longrightarrow X^\ast \otimes X,~ e_X: X\otimes X^\ast \longrightarrow I, \]
called \emph{unit} and \emph{counit}, respectively, satisfying the two zig-zag equations,
that is,
\begin{equation} \label{equ.zigzagx}
X \cong X\otimes I \stackrel{1_X \otimes i_X}{\longrightarrow}
X \otimes (X^\ast \otimes X)\cong (X\otimes X^\ast)\otimes X
\stackrel{e_X \otimes 1_X}{\longrightarrow} I\otimes X \cong X
\end{equation}
is the identity and
\begin{equation} \label{equ.zigzagxstar}
X^\ast \cong I\otimes X^\ast \stackrel{i_X \otimes 1_{X^\ast}}{\longrightarrow}
(X^\ast \otimes X) \otimes X^\ast \cong X^\ast \otimes (X \otimes X^\ast)
\stackrel{1_{X^\ast} \otimes e_X}{\longrightarrow} X^\ast \otimes I \cong X^\ast
\end{equation}
is the identity. (The natural isomorphisms appearing in the above compositions are
the unitors and associators which are part of the monoidal structure on $\catc$.)
\end{defn}
\begin{defn}
A symmetric monoidal category $\catc$ is called \emph{compact}, if every
object of $\catc$ is (right) dualizable.
\end{defn}
The dual of an object is unique up to canonical isomorphism.
Let $V$ be a finite dimensional real vector space.
A \emph{symmetric pairing} on $V$ is a linear map
$e: V\otimes V \to \real$ such that
\[ \xymatrix@R=10pt{ 
V\otimes V \ar[rd]^e  \ar[dd]_b^{\cong} & \\
& \real \\
V\otimes V \ar[ru]_e &
} \]
commutes.
The dual notion of a symmetric copairing can be obtained by reversing all arrows.
Explicitly, a \emph{symmetric copairing} on $V$ is a linear map
$i: \real \to V\otimes V$ such that
\[ \xymatrix@R=10pt{ 
& V\otimes V \ar[dd]_b^{\cong} \\
\real \ar[ru]^i \ar[rd]_i & \\
& V\otimes V 
} \]
commutes.
In order to prepare the notion of a duality structure on a vector space, we observe
the following, where $V^{\otimes 3} = V\otimes V\otimes V$:
\begin{lemma} \label{lem.bothzigokbysymm}
If $e:V\otimes V\to \real$ is a symmetric pairing and
$i:\real \to V\otimes V$ a symmetric copairing, then the diagram
\[ \xymatrix{
V \ar[r]^{i\otimes 1_V} \ar[d]_{1_V \otimes i} & V^{\otimes 3}
  \ar[d]^{1_V \otimes e} \\
V^{\otimes 3} \ar[r]^{e\otimes 1_V} & V
} \]
commutes.
\end{lemma}
\begin{defn} \label{def.dualitystructure}
A \emph{duality structure} on $V$ is a pair $(i,e)$ whose components are
a symmetric copairing $i: \real \to V\otimes V$ and a symmetric pairing
$e: V\otimes V \to \real$,
called \emph{unit} and \emph{counit}, respectively, satisfying the zig-zag equation, i.e.
the composition
\[
V=  V\otimes I \stackrel{1_V \otimes i}{\longrightarrow}
V^{\otimes 3}
\stackrel{e \otimes 1_V}{\longrightarrow} I\otimes V =V
\]
is the identity.
\end{defn}
Lemma \ref{lem.bothzigokbysymm} ensures that a duality structure $(i,e)$ on $V$ satisfies both
zig-zag equations of Definition \ref{def.dual}, taking $V^\ast =V$.
Consequently, $V$ is then dualizable with dual $V$,
that is, $V$ is \emph{self-dual}. 
Fix a basis $\{ e_1,\ldots, e_n \}$ of the vector space $V$.
A symmetric copairing $i$ on $V$ determines, and is determined by, a
symmetric $(n\times n)$-matrix $\operatorname{Mat}(i)=(i_{jk})_{j,k}$ such that
$i(1) = \sum_{j,k} i_{jk} e_j \otimes e_k$.
Similarly, a symmetric pairing $e$ on $V$ determines, and is determined by, a
symmetric $(n\times n)$-matrix $\operatorname{Mat}(e)=(e(e_j \otimes e_k))_{j,k}$.
\begin{prop} \label{prop.ieinverse}
(1) Given a symmetric copairing $i$ on $V$, there exists a symmetric pairing $e$ on $V$
such that $(i,e)$ is a duality structure on $V$ if and only if $\operatorname{Mat}(i)$ is
invertible. In this case, $\operatorname{Mat}(e)=\operatorname{Mat}(i)^{-1}$ and $e$
is uniquely determined by $i$. \\
(2) Given a symmetric pairing $e$ on $V$, there exists a symmetric copairing $i$ on $V$
such that $(i,e)$ is a duality structure on $V$ if and only if $\operatorname{Mat}(e)$ is
invertible. In this case, $\operatorname{Mat}(i)=\operatorname{Mat}(e)^{-1}$ and $i$
is uniquely determined by $e$.
\end{prop}
\begin{proof}
Suppose that $(i,e)$ is a duality structure on $V$.
Evaluating the zig-zag equation on a basis vector yields
\[ e_l = (e\otimes 1_V)(1_V \otimes i)(e_l)
 = (e\otimes 1_V)\sum_{j,k} i_{jk} e_l \otimes e_j \otimes e_k =
 \sum_{j,k} i_{jk} e(e_l \otimes e_j) e_k. \]
Comparing coefficients and using symmetry, it follows that
$\sum_{j} i_{kj} e(e_j \otimes e_l) = \delta_{kl},$ that is, $\operatorname{Mat}(i)
\operatorname{Mat}(e)=1,$ the identity matrix.
Conversely, symmetric matrices $U$ and $C$ which are inverse to each other
define a unit $i$ and a counit $e$ with
$\operatorname{Mat}(i)=U$ and $\operatorname{Mat}(e)=C,$
such that the zig-zag equation holds.
\end{proof}
The proposition shows that the algebraic variety of duality structures on $V$ can be
identified with $\operatorname{Sym}(n,\real) \subset GL (n,\real),$ the symmetric invertible
real matrices of size $n$.
It also follows that the pairing $e:V\otimes V\to \real$ is  
nondegenerate if it is the component of a duality structure on $V$.
The reason why we required $V$ to be finite dimensional from the outset is given
in the next proposition, which is an easy exercise.
\begin{prop}
If $V$ is any real vector space which possesses a duality structure, then
$V$ is finite dimensional.
\end{prop}
A general notion of \emph{trace} on a braided monoidal category (with a so-called twist)
has been introduced in \cite{jsvtracedmoncat}. The present paper does not require the
full level of generality, only the following quantity will play an important role:
\begin{defn}
The \emph{trace} of a duality
structure $(i,e)$ on $V$ is $\tr (i,e) = e\circ i$. Being a linear endomorphism of $\real$,
this trace is uniquely determined by the real number $\tr (i,e)(1),$ which we 
also call the trace of $(i,e)$.
\end{defn}
\begin{prop} \label{prop.traceisdimension}
For any duality structure $(i,e)$ on a finite dimensional real vector space $V,$
the trace formula
$\tr (i,e) = \dim V$
holds.
\end{prop}
\begin{proof}
Using a basis $e_1, \ldots, e_n$ of $V,$
\[ \tr (i,e) = e(\sum_{j,k} i_{jk} e_j \otimes e_k) =
  \sum_{j,k} i_{jk} e(e_j \otimes e_k) = \sum_{k=1}^n (\operatorname{Mat}(i)
  \operatorname{Mat}(e))_{k,k} = \sum_{k=1}^n 1_{k,k} =n. \]
\end{proof}
If an endomorphism $a\in \operatorname{End}(V)$ factors
through a vector space $W$ with $\dim W < \dim V$, then
the determinant of $a$ vanishes.
Since $ie$ factors through the $1$-dimensional space $\real$, we have:
\begin{prop} \label{prop.detieiszero}
If $\dim V\geq 2$, then the determinant of the endomorphism 
$i\circ e: V\otimes V \to V\otimes V$ vanishes.
\end{prop}
In Section \ref{ssec.reptangle}, we shall use duality structures on
vector spaces to construct linear monoidal representations of
a natural categorification of the Brauer algebras. But first let us introduce this
categorification.

\subsection{The Brauer Category}
\label{ssec.tangles}

We shall next construct the Brauer category $\Br$, which will ultimately serve as an
intermediary structure between the fields of our TFT on the one hand and real
matrices on the other.
In some sense, it may thus be construed as a device for linearization.
Loosely speaking, the morphisms will be $1$-dimensional unoriented tangles in a 
high-dimensional Euclidean space. As those can always be disentangled,
$\Br$ is very close, but not equal, to the category of $1$-dimensional
cobordisms. One difference is that the
objects of the latter, being $0$-manifolds, are unordered (finite) sets, whereas the
objects of $\Br$ will be \emph{ordered} tuples of points.
Another difference is that the cobordism category has a huge number of
objects (though few isomorphism types), whereas the Brauer category has very
few objects to begin with and has the property that two objects are isomorphic
if and only if they are equal. \\

Let us detail the formal definition of $\Br$. Given $n=1,2,\ldots,$
we write $[n]$ for the set $\{ 1,\ldots, n \}$. We write $[0]$ for the
empty set. The objects of $\Br$ are $[0], [1], [2],\ldots$.
Each object $[n]$
determines a $0$-submanifold $M[n]$ of $\real^1$ by taking
$M[n]=\{ 1,\ldots, n \} \subset \real^1$.
Morphisms $[m] \to [n]$ in $\Br$ are represented by compact
smooth $1$-manifolds $W$, smoothly embedded in 
$[0,1]\times \real^3,$ such that $\partial W = W\cap (\{ 0,1 \} \times \real^3)$
with
\[ \partial W \cap \{ 0 \} \times \real^3 = 0\times M[m] \times 0 \times 0,~
 \partial W \cap \{ 1 \} \times \real^3 = 1\times M[n] \times 0 \times 0. \]
We require that near the boundary, the embedding of $W$ is the product
embedding
\[ [0,\epsilon]\times M[m]\times 0 \times 0 \sqcup
  [1-\epsilon,1]\times M[n]\times 0 \times 0, \]
for some small $\epsilon >0$. Two such $W$ for fixed $[m],[n]$
define the same morphism
in $\Br$, if they are smoothly isotopic in $[0,1]\times \real^3$ by an isotopy
that is the identity near $\{ 0,1 \} \times \real^3$.

\begin{example}
The diagram 
\[
\xygraph{ !{0;/r1pc/:}
 !{\xcaph@(0)} !{\xcaph@(0)} !{\xcaph@(0)} !{\xcaph@(0)}  !{\hcap} [rr] !{\hcap-} [ll]
[lllld]  !{\xcaph@(0)} !{\xcaph@(0)} !{\hcap} [rr] !{\hcap-} [l]
[llld]  !{\hcap} [rr] !{\hcap-} [d] !{\xcaph@(0)} !{\xcaph@(0)} 
[u] !{\hcross}  !{\xcaph@(0)} [dl]  !{\xcaph@(0)}
[dlll] !{\hcap-} !{\hcap}
}
\] \\
\noindent determines a morphism $[4] \to [4]$ in $\Br$. (The overpass/underpass
information is merely an artifact of the pictorial representation and is irrelevant for the morphism.)
\end{example}

The composition of two morphisms $\phi: [m] \to [n],$
$\psi: [n] \to [p]$ is defined in the most natural manner:
If $\phi$ is represented by the cobordism $V$ and $\psi$ by $W$,
then we translate $W$ from $[0,1]\times \real^3$ to
$[1,2]\times \real^3$ and define a cobordism $U$ as the union
along $\{ 1 \}\times M[n]\times 0 \times 0$ of $V$ and the
translated copy of $W$. Then we reparametrize the embedding
of $U$ from $[0,2]\times \real^3$ to $[0,1]\times \real^3$.
The resulting cobordism represents $\psi \circ \phi$;
its isotopy class depends clearly only on $\phi$ and $\psi$,
not on the particular choice of representatives $V$ and $W$.
The identity $1_{[0]}: [0] \to [0]$ is represented
by the empty cobordism $W=\varnothing$.
For $n>0$, the identity $1_{[n]}: [n] \to [n]$
is represented by the product $[0,1]\times M[n]\times 0\times 0.$ 
Then $\Br$ is indeed a category. Note that
$\Hom_{\Br} ([m],[n])$ is empty for $m+n$ odd (or equivalently
$m-n$ odd) and nonempty for $m+n$ (or equivalently $m-n$) even. \\

We make $\Br$ into a strict monoidal category by defining a
tensor product $\otimes: \Br \times \Br \to \Br$ on objects by
$[m] \otimes [n] = [m+n]$. Let the unit object $I$
be $[0]$. Then 
$\otimes$ on objects of $\Br$ is strictly associative, strictly commutative, and
has a strict unit.
The tensor product $\phi \otimes \phi'$ of two morphisms
$\phi:[m] \to [n]$ and $\phi': [m'] \to [n']$ is defined by
``stacking'' a representative $W'$ of $\phi'$ on top of a representative
$W$ of $\phi$. More precisely, apply the translation $(x,y,z,t)\mapsto (x,y+m,z,t)$
to $W'$. In a product region near the $[n']$-endpoints of the translated
copy of $W'$, redirect those parallel strands so that they connect to the
correct points in $M([m']\otimes [n'])$. Finally, use a small isotopy to ensure that
the new embedding of $W'$ is disjoint from $W$. Then the union of this
embedding of $W'$ together with $W$ represents $\phi \otimes \phi'$.
Since $1_{[m]} \otimes 1_{[m']} = 1_{[m] \otimes [m']}$
and $(\psi \phi)\otimes (\psi' \phi') = (\psi \otimes \psi')\circ
(\phi \otimes \phi'),$ the assignment $\otimes:\Br \times \Br \to \Br$
is a functor. As $\phi \otimes 1_{[0]} = \phi = 1_{[0]} \otimes
\phi$, we may take the unitors in $\Br$ to be the identity morphisms.
As $(\phi \otimes \phi')\otimes \phi'' = \phi \otimes (\phi' \otimes \phi'')$
for morphisms $\phi, \phi', \phi'',$ we may take the associators in
$\Br$ to be the identity morphisms. Thus $(\Br, \otimes, I)$ is a
strict monoidal category. However, although the tensor product is commutative
on objects, it is noncommutative on morphisms. For instance, the two
morphisms $[3]\to [1]$ given by
\[ 
\xygraph{ !{0;/r1pc/:}
!{\hcap} [dd] !{\xcaph@(0)} !{\xcaph@(0)}
} \hspace{1cm} \text{ and } \hspace{1cm}
\xygraph{ !{0;/r1pc/:}
 !{\xcaph@(0)} !{\xcaph@(0)} !{\hcap} [rr]  [l]
[llld]  !{\hcap} [rr] !{\hcap-} [d] !{\xcaph@(0)}  
} \]
are not equal. 

There is precisely one endomorphism $\lambda: [0] \to [0],$
$\lambda \not= 1_{[0]},$ such that $\lambda$ is represented by a
connected, nonempty manifold.
This morphism is represented by a smooth embedding of a
circle in $(0,1)\times \real^3$. Any two embeddings of a circle in
$(0,1)\times \real^3$ are isotopic and therefore represent the same morphism.
We will call this endomorphism $\lambda$ the \emph{loop}.
The endomorphisms in $\Br$ of the identity object $I$ are then given by
$\operatorname{End}_{\Br} (I)= \{ \lambda^{\otimes n} ~|~
   n\geq 0 \},$
with $\lambda^{\otimes n} \not= \lambda^{\otimes m}$
for $n\not= m$.
Here we wrote $\lambda^{\otimes n}$ for the $n$-fold tensor product
$\lambda \otimes \cdots \otimes \lambda$ and $\lambda^{\otimes 0} =
1_{[0]}$. 

Given two objects $[m]$ and
$[n]$ in $\Br,$ we define the braiding
$b_{m,n}: [m] \otimes [n] \to [n] \otimes [m]$ to be the
isomorphism represented by a Brauer diagram which is loop-free and
connects $i\in M([m] \otimes [n]) = \{ 1,\ldots, m, m+1,\ldots, m+n \},$
$1\leq i \leq m,$ to $n+i \in M([n] \otimes [m]) = \{ 1,\ldots, n, n+1,
\ldots, n+m \}$ and $m+i \in M([m] \otimes [n]),$ $1\leq i \leq n,$ to
$i\in M([n] \otimes [m]).$ Since we are in codimension $3$, $b$ is
symmetric, $b_{m, n} = b^{-1}_{n,m}$. Consequently, we only need
to verify one of the two hexagon equations. By strictness, the hexagonal
shape deflates to the triangular shape
\[ \xymatrix@C=6pt{
[m]\otimes [n]\otimes [p] \ar[rd]_{b_{m,n} \otimes 1_{[p]}}
\ar[rr]^{b_{m,n+p}} & &
 [n]\otimes [p]\otimes [m], \\
& [n]\otimes [m]\otimes [p] \ar[ru]_{1_{[n]} \otimes b_{m,p}}
} \]
the commutativity of which is best checked by drawing a picture.
We conclude that the structure
$(\Br, \otimes, I, b)$ is a symmetric strict monoidal category. 
The elementary braiding 
\[ b_{1,1} : \hspace{.6cm}
\xygraph{ !{0;/r1pc/:}
  [r] !{\xcaph@(0)} !{\hcross}  !{\xcaph@(0)} [r] 
[dlllll]  [r] !{\xcaph@(0)} [r]  !{\xcaph@(0)} [r] 
} \] 
is of fundamental importance. 

There are natural unit and counit morphisms in $\Br$ so that
each object of $\Br$ is self-dual in the sense of
Definition \ref{def.dual}. The unit $i_n: I\to [n] \otimes [n]$
is given by
\begin{center}
\includegraphics[height=3.5cm]{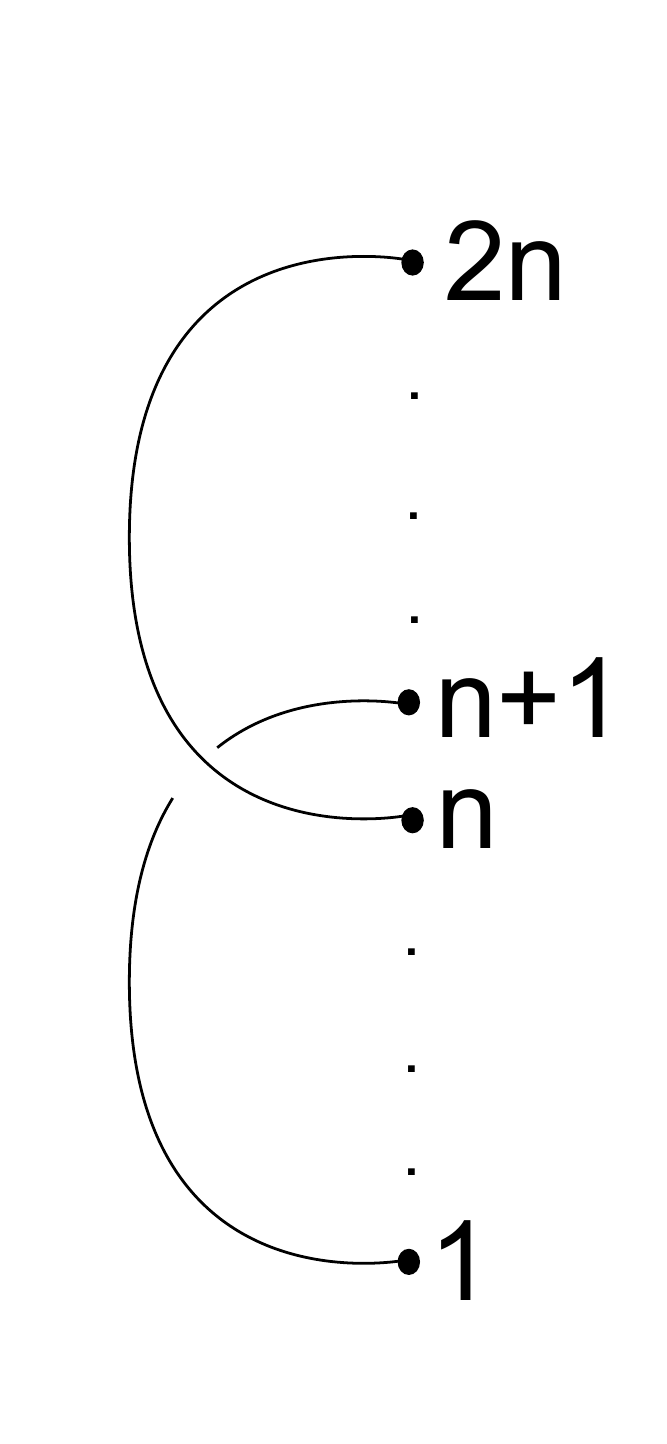}
\end{center}
That is, interpret the cylinder $[0,1]\times M[n]\times 0\times 0$
as a morphism $I\to [n] \otimes [n]$.
In particular, we have
the elementary unit
\[ 
\xygraph{ !{0;/r1pc/:}
 !{\hcap-|{i_1}}  !{\xcaph@(0)} [r]  [dll]  !{\xcaph@(0)} [r].  
} \] 
Note the symmetry identity $i_1 = b_{1,1} \circ i_1$;
$i_1$ is isotopic rel endpoints to
\[ 
\xygraph{ !{0;/r1pc/:}
 !{\hcap-} 
 !{\xcaph@(0)} !{\hcross}  !{\xcaph@(0)} [r] 
[dllll]  !{\xcaph@(0)} [r]  !{\xcaph@(0)} [r]  [r].
} \]
The counit $e_n: [n] \otimes [n] \to I$ is given by
\begin{center}
\includegraphics[height=3.5cm]{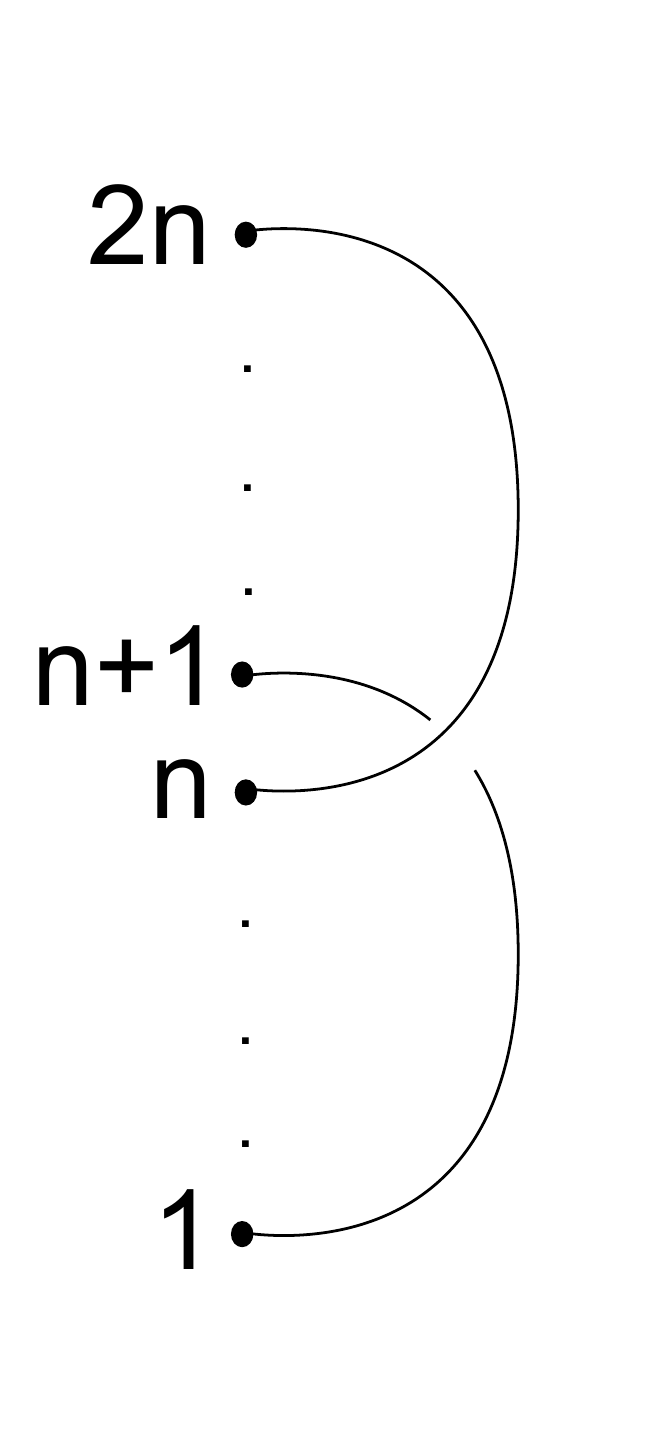}
\end{center}
i.e. this time interpret the
cylinder $[0,1]\times M[n]\times 0\times 0$ as a morphism
$[n] \otimes [n] \to I$. In particular, we have the elementary counit
\[ e_1: \hspace{.6cm}
\xygraph{ !{0;/r1pc/:}
 [r] !{\xcaph@(0)} !{\hcap} [dll]  [r] !{\xcaph@(0)} [r].
}
\] 
Note the symmetry identity $e_1 = e_1 \circ b_{1,1}$;
$e_1$ is isotopic rel endpoints to
\[ 
\xygraph{ !{0;/r1pc/:}
  [r] !{\xcaph@(0)} !{\hcross}  !{\xcaph@(0)} !{\hcap}
[dllll]  [r] !{\xcaph@(0)} [r]  !{\xcaph@(0)}  [r].
} \]
The zig-zag equation (\ref{equ.zigzagx}) is satisfied, as
\begin{center}
\includegraphics[height=6.5cm]{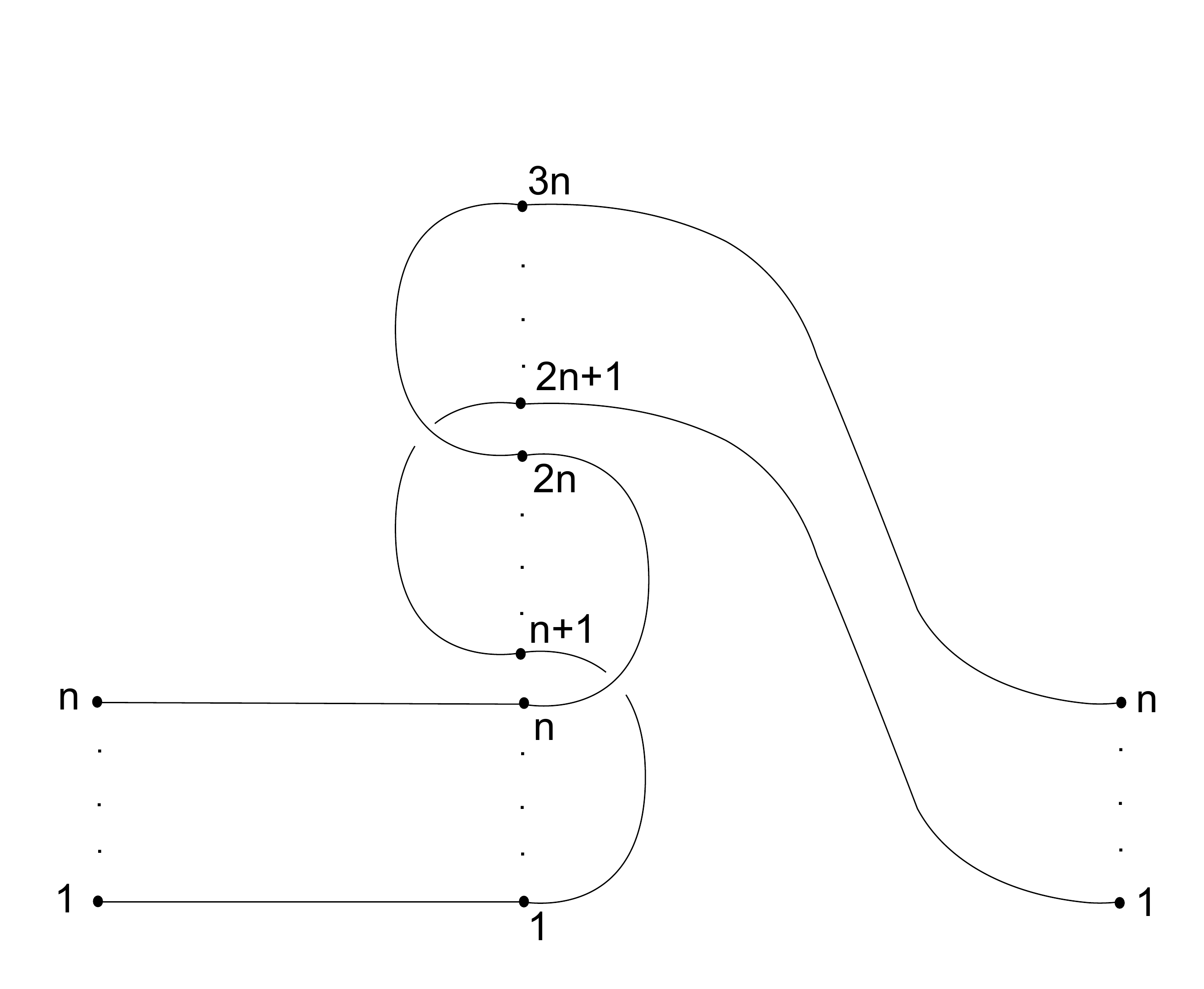}
\end{center}
is isotopic to the cylinder, i.e. the identity $[n] \to [n]$.
Similarly, the zig-zag equation (\ref{equ.zigzagxstar}) is satisfied, as
\begin{center}
\includegraphics[height=5.5cm]{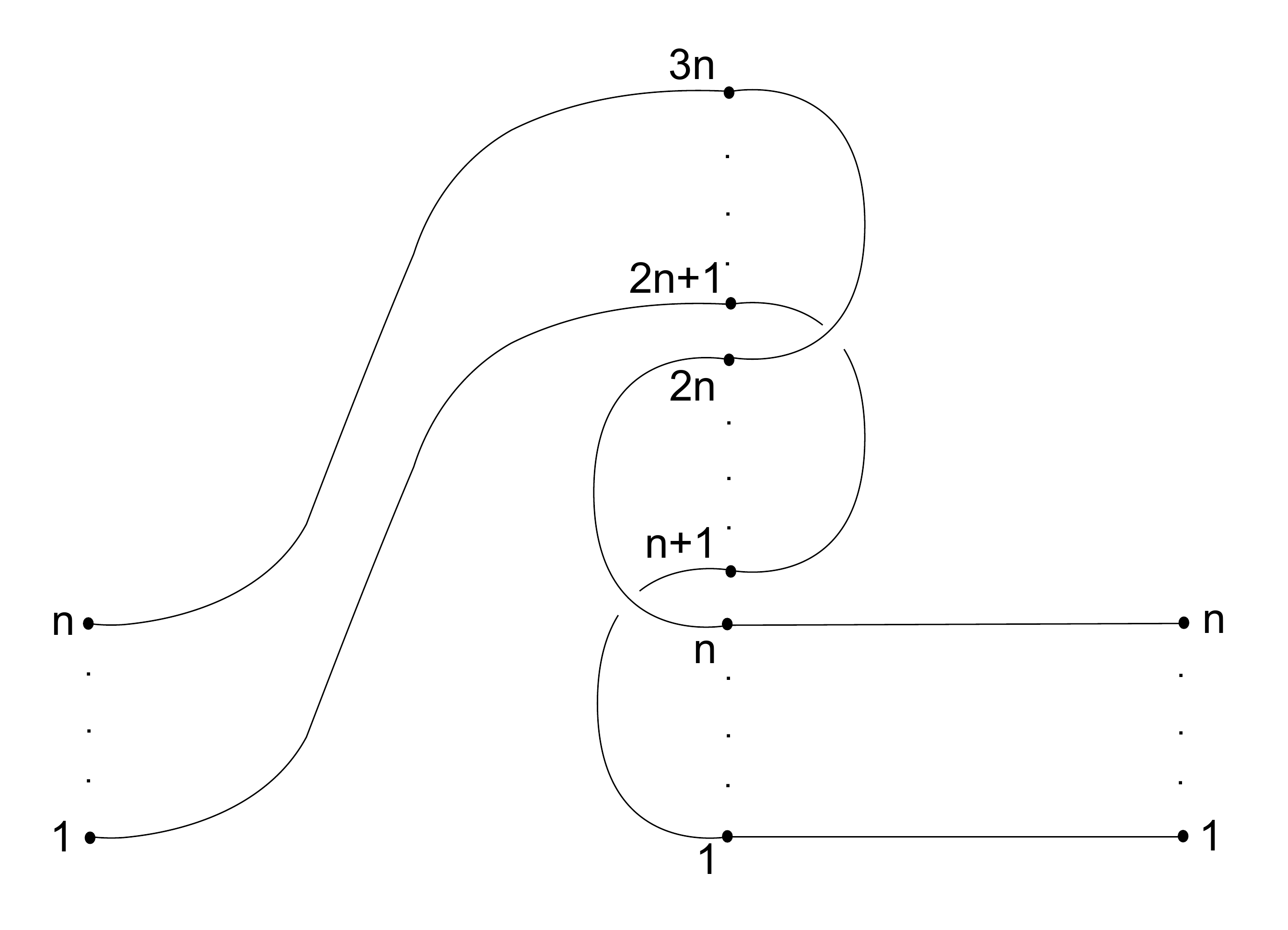}
\end{center}
is isotopic to the cylinder, i.e. the identity $[n] \to [n]$.
This shows that the category $(\Br, \otimes, I, b)$ is compact.
We summarize:
\begin{prop}
The structure $(\Br, \otimes, I, b, i,e)$ is a compact, symmetric,
strict monoidal category.
\end{prop}
The loop endomorphism $\lambda: I\to I$ can be factored as
\begin{equation} \label{equ.lambdafact}
\lambda = e_1 \circ i_1. 
\end{equation}
It commutes with every morphism, $\lambda \otimes \phi =
\phi \otimes \lambda$ for any morphism $\phi$.
Loops are cancellative: if $\lambda \otimes \phi = \lambda \otimes \psi$, then
$\phi = \psi$.
Loops are persistent: if $\phi: [m]\to [n]$ has $k$ loops and
$\psi: [n]\to [p]$ has $l$ loops, then $\psi \phi$ has at least
$k+l$ loops. In particular, an isomorphism can never contain a loop.
Note that if $\phi: [m]\to [n]$ is an isomorphism, then
$m=n$ and every connected component of a representative of
$\phi$ has precisely one endpoint on the hyperplane $0\times \real^3$
and the other endpoint on the hyperplane $1\times \real^3$.
Hence $\phi$ determines a bijection $[m] \to [m]$.
We will therefore write $\phi (i)=j$ when the point $(0,i,0,0)\in
0\times M[m] \times 0\times 0$ is connected by $\phi$ to the point
 $(1,j,0,0)\in 1\times M[m] \times 0\times 0$.
Conversely, every bijection $[m]\to [m]$ determines an
isomorphism $\phi \in \Hom_\Br ([m], [m]).$ With respect to
the tensor product, an identity morphism cannot be factored nontrivially:
\begin{lemma} \label{lem.idfactorstriv}
Let $\phi$ and $\psi$ be morphisms of $\Br$ such that
$\phi \otimes \psi = 1$, an identity morphism. Then each of the
factors is an identity: $\phi=1$ and $\psi =1$.
\end{lemma}
This is evident and requires no proof. \\

Let $(\catc, \otimes, I)$ be a strict monoidal category and $G$ a collection
of morphisms in $\catc$. Interpreting $G$ as an alphabet of formal symbols
$\{ [g] ~|~ g\in G \},$
we may form \emph{words} as follows: $[g]$ is a word for all $g\in G$ and $[1_X]$ is
a word for all $X\in \Ob \catc$. If $w_1$ and $w_2$ are words, then the string
$(w_1 \otimes w_2)$ is a word and the string $(w_2 \circ w_1)$ is a word
if $\cod w_1 = \dom w_2$. Every word $w$ determines a morphism $|w|$ of $\catc$
by the rules
\[ |[g]|=g,~ |[1_X]|=1_X,~ |(w_1 \otimes w_2)| = |w_1| \otimes |w_2|,~
 |(w_2 \circ w_1)| = |w_2| \circ |w_1|. \]
Two words are called \emph{freely equivalent} (we write $\sim$), if they can be
obtained from each other by a finite sequence of subword substitutions
implementing associativity for $\circ$ and $\otimes$, identity cancellation
for $\circ$ and $\otimes$ and compatibility between $\circ$ and $\otimes$.
Note that if $w_1$ and $w_2$ are freely equivalent, then $|w_1|=|w_2|$.
\begin{lemma} \label{lem.freenormalform}
Any word in $G$ is freely equivalent to either $[1_X]$ for some object $X$
or to a word of the form
\[ ([1_{X_1}] \otimes [g_1] \otimes [1_{Y_1}]) \circ
  ([1_{X_2}] \otimes [g_2] \otimes [1_{Y_2}]) \circ \cdots \circ
  ([1_{X_k}] \otimes [g_k] \otimes [1_{Y_k}]).
   \]
with $g_1, \ldots, g_k \in G$.
\end{lemma}
See \cite[Lemma XII.1.2.]{kassel} for the (easy) proof.
Let $\Fa (G)$ be the class of free equivalence classes of words in $G$.
As noted above, the realization $|\cdot |$ is still well-defined on $\Fa (G)$.
Let $R$ be a collection of pairs $(w_1, w_2)$ of words in $G$ such that
$|w_1| = |w_2|$. For two elements
$x,y\in \Fa (G)$ we define $x\sim_R y,$ and say that
$x,y$ are $R$-equivalent,
if and only if one can obtain some representative of $y$ from
some representative of $x$ by a finite sequence of subword substitutions, where an allowable 
substitution consists of replacing a subword $w_1$ by $w_2$ for
$(w_1, w_2)\in R$. We say that $(\catc,\otimes, I)$ is
\emph{generated by the generators $G$ and the relations $R$}, if
\begin{itemize}
\item any morphism in $\catc$ can be obtained as $|w|$ for some word $w$ in $G$, and

\item for any $x,y\in \Fa (G),$ we have $x\sim_R y$ if and only if $|x|=|y|$
 in $\catc$. 
\end{itemize}
The structure of morphisms in $\Br$ is then elucidated by the 
following result; 
we simply write $1$ for the identity morphism on $[1]$ and omit square brackets in words.
\begin{prop} \label{prop.tanggensrels}
The compact, symmetric, strict monoidal category $\Br$ is generated
by the three morphisms $i_1, e_1, b_{1,1}$
and the following relations: \\

\noindent (B1) Zig-Zag:
\[ (1\otimes e_1)\circ (i_1 \otimes 1) = 1 = (e_1 \otimes 1)\circ (1\otimes i_1), \]

\noindent (B2) Twisted Zig-Zag:
\[ (e_1 \otimes 1 \otimes 1)\circ (1\otimes b_{1,1} \otimes 1)\circ
  (1\otimes 1\otimes i_1)= b_{1,1} = (1\otimes 1\otimes e_1)\circ
 (1\otimes b_{1,1} \otimes 1)\circ (i_1 \otimes 1 \otimes 1), \]

\noindent (B3) Reidemeister I:
\[ (1\otimes e_1)\circ (b_{1,1} \otimes 1)\circ (1\otimes i_1)=1=
  (e_1 \otimes 1)\circ (1\otimes b_{1,1})\circ (i_1 \otimes 1), \]

\noindent (B4) Reidemeister II:
\[ b_{1,1} \circ b_{1,1} = 1\otimes 1, \]

\noindent (B5) Reidemeister III (a.k.a. the Yang-Baxter equation):
\[ (b_{1,1} \otimes 1)\circ (1\otimes b_{1,1})\circ (b_{1,1} \otimes 1) =
  (1\otimes b_{1,1})\circ (b_{1,1} \otimes 1)\circ (1\otimes b_{1,1}). \]
\end{prop}
\begin{proof}
This can be proved following Turaev's methods of
\cite{turaevtanglesrmatrices}, see also \cite{kassel}; thus there is no need to go into
detail here.
As in knot theory, these methods, adapted to the present case, involve the introduction of planar polygonal
Brauer diagrams contained in $[0,1]\times \real$ and then reasoning with these diagrams using appropriate moves.
Ultimately, one is reduced to Reidemeister's classical arguments
in \cite{reidemeisterknoten}. 
It is sufficient to limit oneself to generic (see below) Brauer diagrams, since
every diagram can be transformed to such a generic one by a planar isotopy.
Polygonal representatives of $i_1, e_1$ and 
$b_{1,1}$ are shown in Figure \ref{fig.ieb};
we will simply write $i,e,b$ for these
representatives.
\begin{figure} 
\frame{\includegraphics[height=3cm]{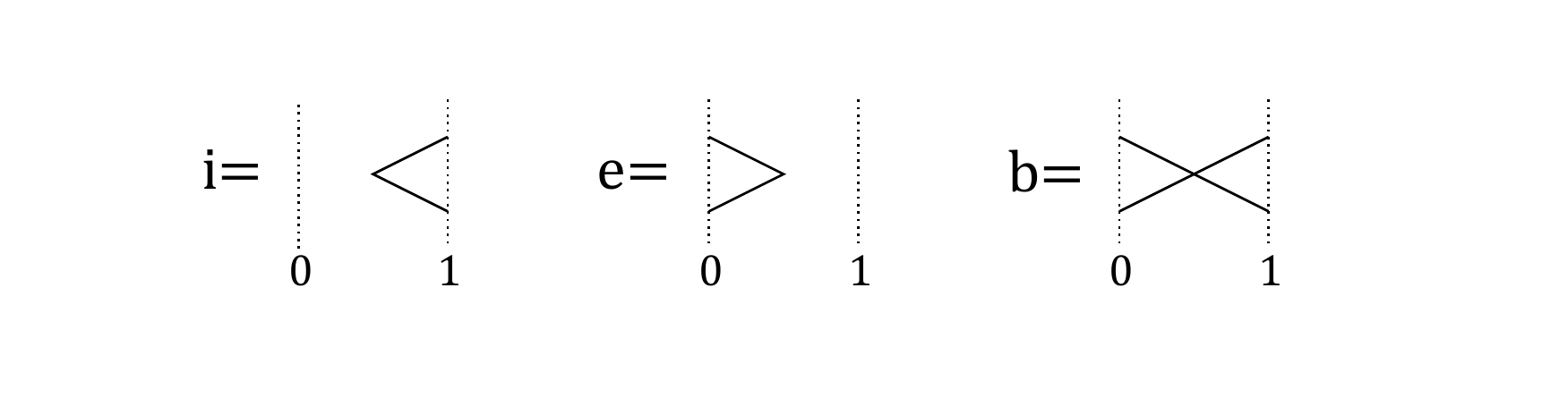}}
\caption{\label{fig.ieb} The polygonal generators $i,e$ and $b$.}
\end{figure}
A point in a Brauer diagram $B$ is called \emph{singular}, if it 
lies in $(0,1)\times \real$ and is either
a vertex which is a local minimum or maximum for the ``horizontal'' coordinate, 
or a crossing point. 
Each of $i,e,b$ has precisely one singular point: $i$ a local minimum,
$e$ a local maximum and $b$ a crossing point.
A Brauer diagram $B$ is called \emph{generic} if any two distinct singular points
have different horizontal coordinates.
The Brauer diagrams $i,e$ and $b$ are generic.
As mentioned earlier, every Brauer diagram is related by a planar isotopy to a generic Brauer diagram.
In particular, every morphism of $\Br$ can be represented by a generic Brauer diagram.
Two generic Brauer diagrams $B$ and $B'$ are related by a \emph{generic} planar isotopy,
if there is a planar isotopy $H$ with $H(B,1)=B'$ such that $H(B,t)$ is a generic
Brauer diagram for every $t\in [0,1]$. The proof then rests on two facts, (A) and (B):\\

Fact (A):\\
Two generic Brauer diagrams are related by a planar isotopy if and only if
one is obtained from the other by a finite number of applications of the 
moves (GI), (EX), (ZZ), (TZ) below:\\
(GI) A generic planar isotopy,\\
(EX) an isotopy exchanging the order of two singular points with respect to
their horizontal coordinate, see Figure \ref{fig.exmove},\\
(ZZ) a zig-zag move as shown in the top row of Figure \ref{fig.zzmove}, and\\
(TZ) a twisted zig-zag move in the neighborhood of a crossing point, as
 shown in the bottom of Figure \ref{fig.zzmove}. \\
\begin{figure} 
\frame{\includegraphics[height=7cm]{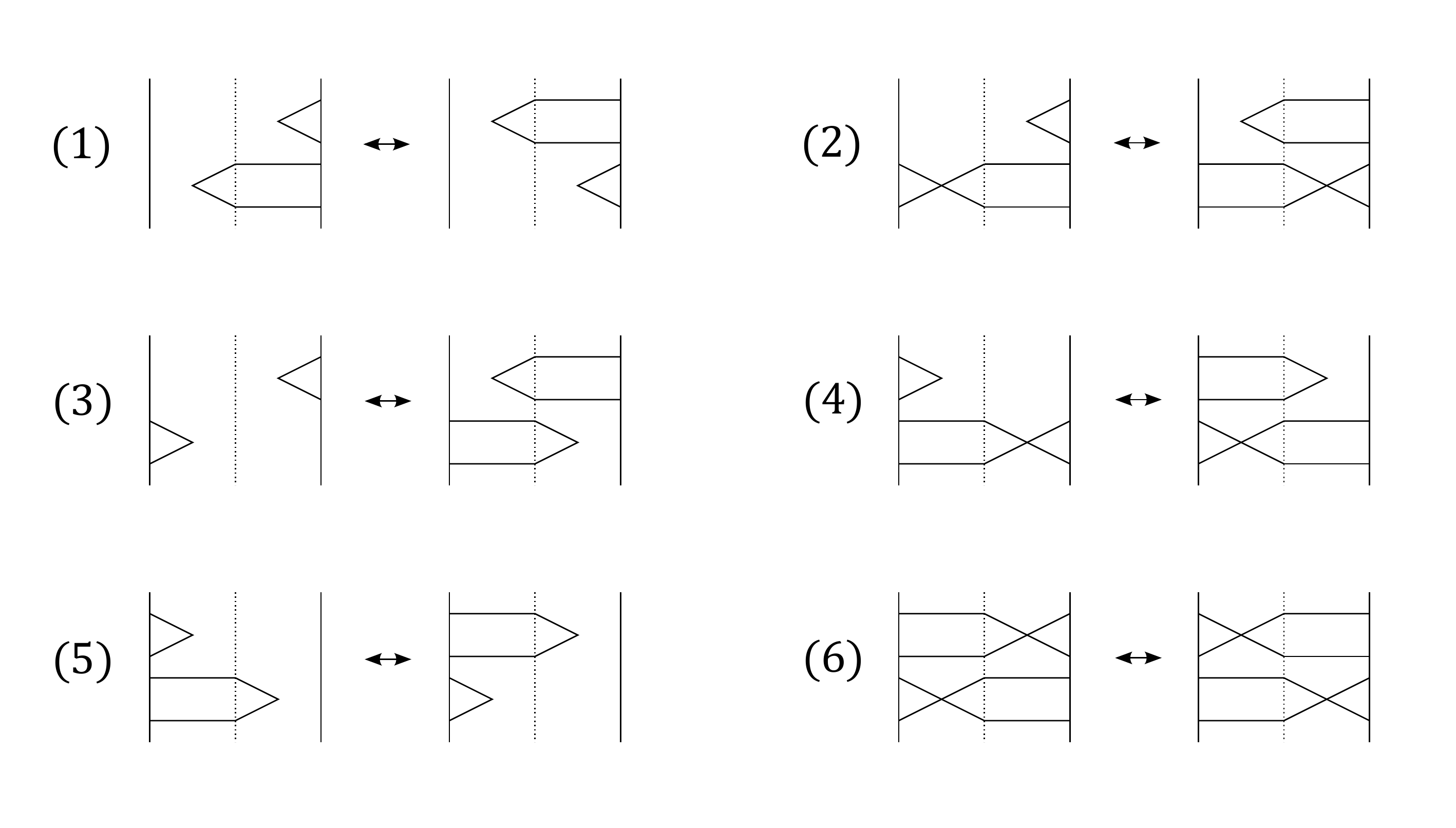}}
\caption{\label{fig.exmove} Exchange moves (EX) on (generic) Brauer diagrams.}
\end{figure}
\begin{figure} 
\frame{\includegraphics[height=6cm]{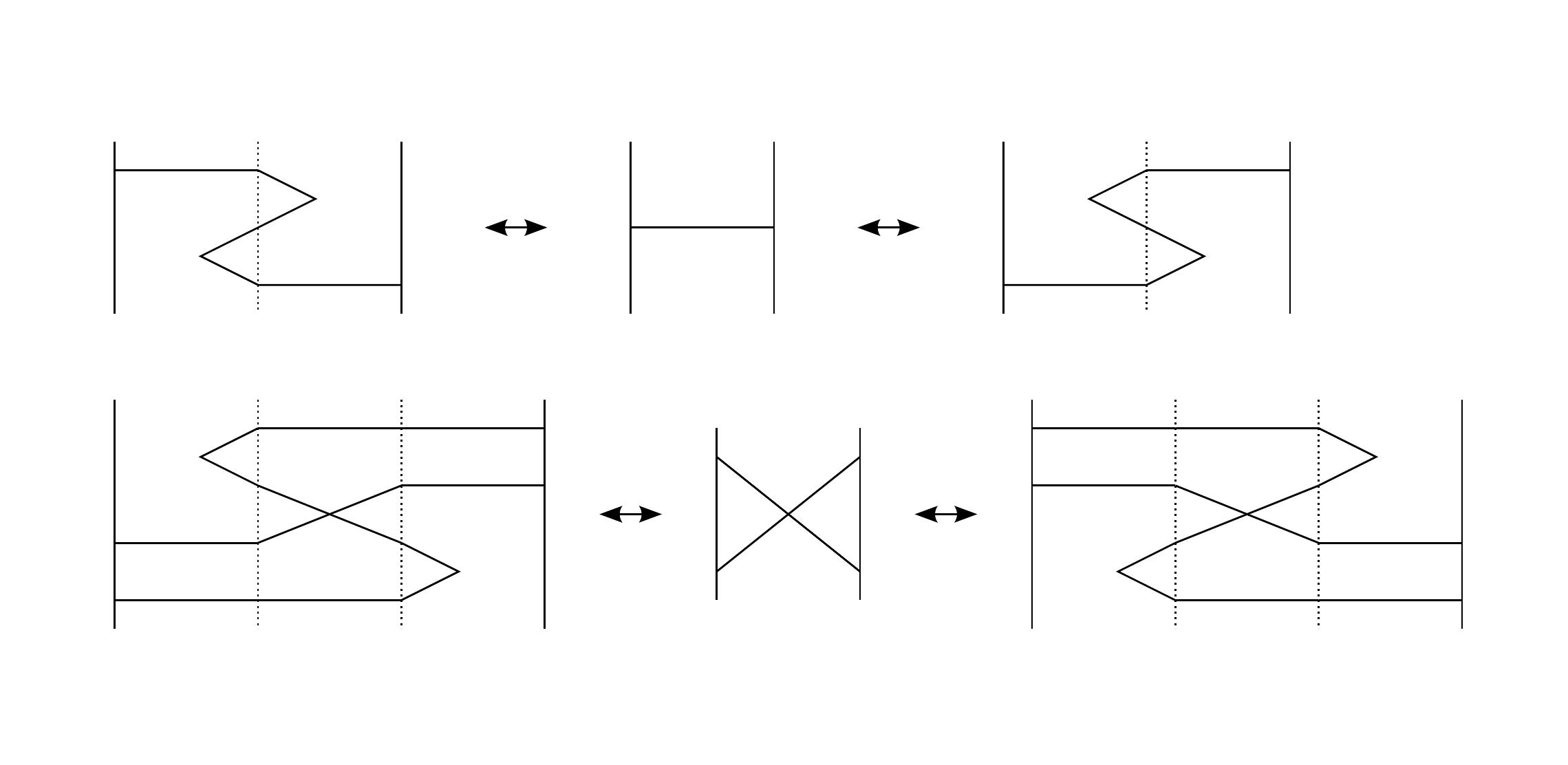}}
\caption{\label{fig.zzmove} The zig-zag move (top) and twisted zig-zag move (bottom) on (generic) Brauer diagrams.}
\end{figure}

Fact (B):\\
Two generic Brauer diagrams represent the same morphism in $\Br$
if and only if one is obtained from the other by a finite sequence of
planar isotopies of Brauer diagrams and the Reidemeister moves (RI),
(RII), (RIII) as shown in Figure \ref{fig.reidemoves}.\\
\begin{figure} 
\frame{\includegraphics[height=5.5cm]{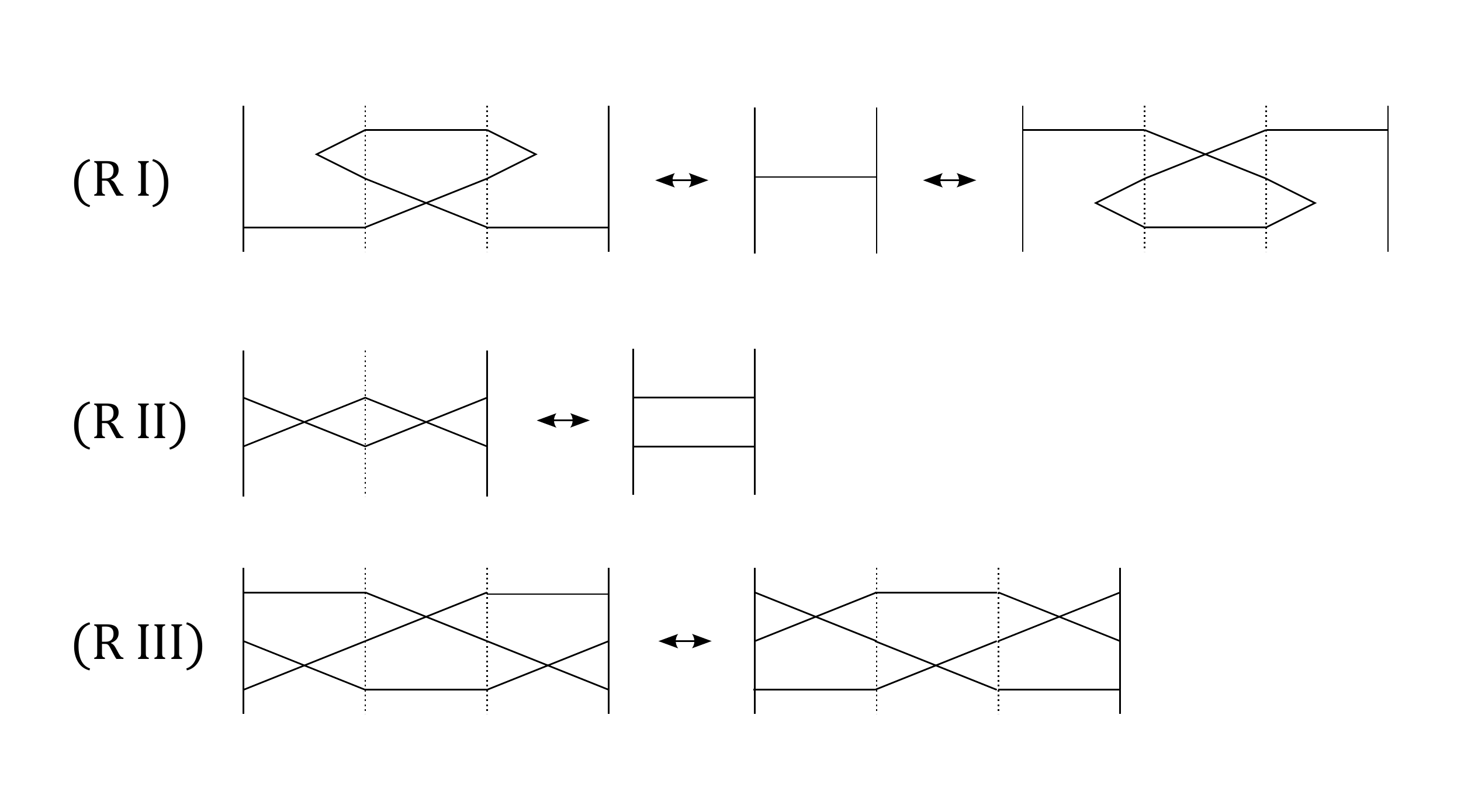}}
\caption{\label{fig.reidemoves} The Reidemeister moves on (generic) Brauer diagrams.}
\end{figure}

The proof of these two facts is essentially contained in \cite{reidemeisterknoten}.
(For tangles, the analogs of these two facts are
\cite[Lemma X.5.7, Theorem X.5.9]{kassel}, see also 
Lemma X.3.5 and Theorem X.3.7 in \emph{loc. cit.}) 
Note that the exchange moves of Figure \ref{fig.exmove} can be derived using only
free equivalences.
For example in case (4) of the figure, the word
$[b]\circ ([1]\otimes [1]\otimes [e])$ is freely equivalent (we write $\sim$) to
$(([1]\otimes [1]) \otimes [e]) \circ ([b]\otimes ([1]\otimes [1]))$ via the following
chain of free equivalences:
\begin{eqnarray*}
[b]\circ ([1]\otimes [1]\otimes [e]) & \sim &
([b] \otimes [\id_{[0]}]) \circ ([1]\otimes [1]\otimes [e]) \\
& \sim & ([b]\circ ([1]\otimes [1]))\otimes ([\id_{[0]}] \circ [e]) \\
& \sim & [b]\otimes [e] \\
& \sim & (([1]\otimes [1])\circ [b])\otimes ([e]\circ ([1]\otimes [1])) \\
& \sim & (([1]\otimes [1])\otimes [e])\circ ([b]\otimes ([1]\otimes [1])).
\end{eqnarray*} 
\end{proof}

\begin{remark}
Turaev \cite{turaevtanglesrmatrices}
considers oriented tangles in $\real^3$. These
can of course not generally be disentangled and consequently have
a more involved structure theory, which is reflected in a more
complicated representation theory, involving so-called enhanced
R-matrices. Turaev's Theorem 3.2 lists 6 generators
($\cap$ with the 2 possible orientations, $\cup$ with the 2 possible
orientations, and
$X_+, X_-$ corresponding to overpass/underpass), and 7 relations (10)--(16).
Now, our Brauer category is obtained from the tangle category by forgetting
orientations and allowing strands to pass through each other.
Forgetting orientations leads to one generator
$\cap = e_1$ and one generator $\cup = i_1$. Forgetting the overpass/underpass
information (and the orientation) leads to one generator $X=b_{1,1}$.
Rewriting Turaev's relations (10)--(16) accordingly in terms of $e_1, i_1$ and $b_{1,1}$,
one arrives at relations ($10'$)--($16'$), valid in $\Br$, which can indeed be derived from our relations
(B1)--(B5).
In fact, relations ($10'$) and ($11'$) are equal and agree with (B1).
Relation ($13'$) is (B4), ($14'$) is (B5), and ($15'$) is the first half of (B3).
Using four instances of (B2), we derive ($12'$):\\

\noindent $(e_1 \otimes 1\otimes 1)\circ (1\otimes e_1 \otimes 1\otimes 1\otimes 1)
  \circ (1\otimes 1\otimes b_{1,1} \otimes 1 \otimes 1) 
\circ (1\otimes 1\otimes 1 \otimes i_1 \otimes 1)\circ
     (1\otimes 1 \otimes i_1)$ \\
$=(e_1 \otimes 1\otimes 1)\circ (1\otimes ((e_1 \otimes 1\otimes 1)
  \circ (1\otimes b_{1,1} \otimes 1)
 \circ (1\otimes 1 \otimes i_1))\otimes 1)\circ (1\otimes 1 \otimes i_1)$ \\
$=(e_1 \otimes 1\otimes 1)\circ (1\otimes b_{1,1} \otimes 1)\circ (1\otimes 1 \otimes i_1)$ \\
$= b_{1,1}$ \\
$=(1 \otimes 1\otimes e_1)\circ (1\otimes b_{1,1} \otimes 1)\circ (i_1\otimes 1 \otimes 1)$ \\
$=(1 \otimes 1\otimes e_1)\circ (1\otimes ((1 \otimes 1\otimes e_1)
  \circ (1\otimes b_{1,1} \otimes 1)
 \circ (i_1\otimes 1 \otimes 1))\otimes 1)\circ (i_1\otimes 1 \otimes 1)$ \\
$=(1 \otimes 1\otimes e_1)\circ (1\otimes 1 \otimes 1\otimes e_1\otimes 1)
  \circ (1\otimes 1\otimes b_{1,1} \otimes 1 \otimes 1) 
\circ (1\otimes i_1\otimes 1 \otimes 1 \otimes 1)\circ
     (i_1\otimes 1 \otimes 1).$ \\

\noindent By (B2), the left hand side of ($16'$) is $b_{1,1}^2,$ which is $1\otimes 1$
by (B4).
\end{remark}
\begin{remark}
We do not claim that the relations provided in Proposition \ref{prop.tanggensrels} are 
minimal in any sense.
\end{remark}
To illustrate Proposition \ref{prop.tanggensrels}, we show how
(B1)--(B5) imply the symmetry equations
$e_1 \circ b_{1,1} = e_1,~ b_{1,1}\circ i_1 = i_1$:
\begin{eqnarray*}
e_1 \circ b_{1,1} & = &
e_1 \circ (1\otimes 1\otimes e_1)\circ
 (1\otimes b_{1,1} \otimes 1)\circ (i_1 \otimes 1 \otimes 1) \\
& = & (e_1 \otimes e_1)\circ
 (1\otimes b_{1,1} \otimes 1)\circ (i_1 \otimes 1 \otimes 1) \\
& = & e_1 \circ (e_1 \otimes 1\otimes 1)\circ
 (1\otimes b_{1,1} \otimes 1)\circ (i_1 \otimes 1 \otimes 1) \\
& = & e_1 \circ (((e_1 \otimes 1)\circ
 (1\otimes b_{1,1})\circ (i_1 \otimes 1)) \otimes 1) \\
& = & e_1 \circ (1\otimes 1)  =  e_1,
\end{eqnarray*}
similarly for the other symmetry equation. 
The loop-commutativity 
$1\otimes \lambda = \lambda \otimes 1$ can be derived using (B1)--(B5) and the above
symmetry equations
(and free equivalence, of course) as follows:
\[
\begin{array}{rcll}
1\otimes \lambda 
&=& 1\otimes (e \circ i) & (\ref{equ.lambdafact}) \\
&=& (1\otimes e)\circ (1\otimes i) & (\text{free}) \\
&=& (1\otimes e)\circ (1\otimes 1\otimes 1) \circ (1\otimes i) & (\text{free}) \\
&=& (1\otimes e)\circ (b^2 \otimes 1) \circ (1\otimes i) & (\text{B4}) \\
&=& (1\otimes e)\circ (b\otimes 1) \circ (b\otimes 1)\circ (1\otimes i) & (\text{free}) \\
&=& (1\otimes e)\circ (b\otimes 1) \circ (1\otimes 1\otimes 1) \circ (b\otimes 1)\circ (1\otimes i) & (\text{free}) \\
&=& (1\otimes e)\circ (b\otimes 1) \circ (1\otimes b^2) \circ (b\otimes 1)\circ (1\otimes i) & (\text{B4}) \\
&=& (1\otimes e)\circ (b\otimes 1) \circ (1\otimes b) \circ 
             (1\otimes b)\circ (b\otimes 1) \circ (1\otimes i) & (\text{free}) \\
&=& (1\otimes e)\circ (b\otimes 1) \circ [((e\otimes 1)\circ (1\otimes i))\otimes b] \\
 & &\circ [((1\otimes e)\circ (i\otimes 1))\otimes b]\circ (b\otimes 1) \circ (1\otimes i) & (\text{B1}) \\
&=& (e\otimes 1)\circ (1\otimes 1\otimes 1\otimes e)\circ (1\otimes 1\otimes b\otimes 1)\circ
         (1\otimes i\otimes 1\otimes 1)\circ (1\otimes b) \\
 & & \circ (1\otimes b) \circ (1\otimes e\otimes 1\otimes 1)\circ (1\otimes 1\otimes b\otimes 1)\circ
       (1\otimes 1\otimes 1\otimes i)\circ (i\otimes 1) & (\text{Lem. \ref{lem.freenormalform}}) \\
&=& (e\otimes 1)\circ [1\otimes ((1\otimes 1\otimes e)\circ (1\otimes b\otimes 1)\circ
         (i\otimes 1\otimes 1))] \circ (1\otimes b^2) \\
 & & \circ [1\otimes ((e\otimes 1\otimes 1)\circ (1\otimes b\otimes 1)\circ
       (1\otimes 1\otimes i))] \circ (i\otimes 1) & (\text{free}) \\
&=& (e\otimes 1)\circ [1\otimes b] \circ (1\otimes b^2) \circ [1\otimes b] \circ (i\otimes 1) & (\text{B2}) \\
&=& (e\otimes 1) \circ (i\otimes 1) & (\text{B4}) \\
&=& (ei)\otimes 1 & (\text{free}) \\
&=& \lambda \otimes 1.
\end{array}
\]
(In the above derivation, we simply wrote $b,e,i$ for $b_{1,1}, e_1, i_1$.)
As a final example, the equation $(1\otimes b)\circ (i\otimes 1) = (b\otimes 1)\circ (1\otimes i)$
(valid in $\Br$) can be derived using
\[
\begin{array}{rcll}
(1\otimes b)\circ (i\otimes 1)
& = & (1\otimes 1\otimes 1)\circ (1\otimes b)\circ (i\otimes 1) & (\text{free}) \\
& = & [1\otimes 1\otimes ((e\otimes 1)\circ (1\otimes i))]\circ (1\otimes b)\circ (i\otimes 1) & (\text{B1}) \\
& = & [((1\otimes 1\otimes e)\circ (1\otimes b \otimes 1)\circ (i\otimes 1\otimes 1))\otimes 1]
      \circ (1\otimes i) & (\text{free}) \\
& = & (b\otimes 1)\circ (1\otimes i) & (\text{B2}).
\end{array}
\]

\subsection{Representations of the Brauer Category}
\label{ssec.reptangle}

We shall use duality structures on vector spaces to construct linear representations
of $\Br$, i.e. symmetric strict monoidal functors $Y:\Br \to \vect$ which preserve
duality. Given a presentation of a strict monoidal category $\catc$ by generators
and relations, the following proposition constructs strict monoidal functors on
$\catc$.

\begin{prop} (cf. \cite[Prop. XII.1.4]{kassel}.) \label{prop.existfunctor}
Let $\catc$ and $\catd$ be strict monoidal categories.
Suppose that $\catc$ is generated by the morphisms $G$ and the relations $R$.
Let $F_0: \Ob \catc \to \Ob \catd$ be a map such that $F_0 (I)=I$ and
$F_0 (X\otimes Y)=F_0 (X)\otimes F_0 (Y)$ for all $X,Y \in \Ob \catc$.
Let $F_1: G\to \Mor (\catd)$ be a map such that
$\dom (F_1 (g)) = F_0 (\dom (g))$ and $\cod (F_1 (g)) = F_0 (\cod (g))$.
Suppose that any pair $(w_1, w_2)\in R$ yields equal morphisms in $\catd$
after replacing any symbol $g\in G$ of $w_1$ and $w_2$ by $F_1 (g)$ and any
symbol $1_X$ by $1_{F_0 (X)}$. 
Then there exists a unique strict monoidal functor $F: \catc \to \catd$ such that
$F(X)=F_0 (X)$ for all $X\in \Ob \catc$ and $F(g)=F_1 (g)$ for all
$g\in G$.
\end{prop}

We apply this Proposition to construct functors on $\Br$:
\begin{thm} \label{thm.functorytangvect}
Let $V$ be a finite dimensional real vector space and $(i,e)$ a duality structure
on $V$. Then there exists a unique symmetric strict monoidal functor
$Y: \Br \to \vect$ which satisfies $Y([1])=V$ and preserves duality, that is,
$Y(i_1)=i,$ $Y(e_1)=e.$
\end{thm}
\begin{proof}
We define a map $Y_0: \Ob \Br \to \Ob \vect$ by setting
$Y_0 ([n])=V^{\otimes n},$ with the understanding that
$V^{\otimes 0} = I=\real,$ the one-dimensional vector space which is the
unit object of $\vect$. Then $Y_0 (I)=Y_0 ([0])=I$ and
\[ Y_0 ([m]\otimes [n]) = Y_0 ([m+n])=V^{\otimes (m+n)} =
  V^{\otimes m} \otimes V^{\otimes n} = Y_0 ([m])\otimes Y_0 ([n]). \]
Note that it is crucial here to use the strictly associative Schauenburg tensor
product, which also has strict units. According to Proposition \ref{prop.tanggensrels},
$\Br$ is generated by the morphisms $G=\{ i_1, e_1, b_{1,1} \}$ and the
relations (B1)--(B5). A function $Y_1: G\to \Mor (\vect)$ is given by
\[ Y_1 (i_1)=i,~ Y_1 (e_1)=e,~ \text{ and }~ Y_1 (b_{1,1})=b, \]
where $b:V\otimes V\to V\otimes V$ is the braiding automorphism
induced by $v\otimes w \mapsto w\otimes v$ which makes $\vect$ symmetric.
This definition of $Y_1$ is forced by the requirement that $Y$ be symmetric
and preserve duality. Domains and codomains are transformed compatibly under $Y_0$ and $Y_1$.
Thus by Proposition \ref{prop.existfunctor}, there exists a
unique strict monoidal functor $Y:\Br \to \vect$ such that
$Y([n])=Y_0 ([n])=V^{\otimes n}$ for all $[n]\in \Ob \Br$ and
$Y (g)=Y_1 (g)$ for all $g\in G,$ i.e. $Y(i_1)=i,$ $Y(e_1)=e,$ $Y(b_{1,1})=b$,
provided the following identities hold in $\vect$ for the duality structure $(i,e)$ and $b$:\\

\noindent Zig-Zag:
\[ (1\otimes e)\circ (i \otimes 1) = 1 = (e \otimes 1)\circ (1\otimes i), \]

\noindent Twisted Zig-Zag:
\[ (e \otimes 1 \otimes 1)\circ (1\otimes b \otimes 1)\circ
  (1\otimes 1\otimes i)= b = (1\otimes 1\otimes e)\circ
 (1\otimes b \otimes 1)\circ (i \otimes 1 \otimes 1), \]

\noindent Reidemeister I:
\[ (1\otimes e)\circ (b \otimes 1)\circ (1\otimes i)=1=
  (e \otimes 1)\circ (1\otimes b)\circ (i \otimes 1), \]

\noindent Reidemeister II:
\[ b \circ b = 1\otimes 1, \]

\noindent Reidemeister III (a.k.a. the Yang-Baxter equation):
\[ (b \otimes 1)\circ (1\otimes b)\circ (b \otimes 1) =
  (1\otimes b)\circ (b \otimes 1)\circ (1\otimes b). \]

The zig-zag
equations are satisfied by definition of a duality structure and by Lemma \ref{lem.bothzigokbysymm}.
The transposition
$b$ clearly satisfies Reidemeister II. Furthermore, $b$ is a well-known
solution of the Yang-Baxter equation so that Reidemeister III is satisfied as well.
Let us verify Reidemeister I: Let $e_1, \ldots, e_n$ be a basis of $V,$ 
$i(1) = \sum_{j,k} i_{jk} e_j \otimes e_k$ 
and $e_{jk} = e(e_j \otimes e_k)$.
For $v\in V$,
\[ (1\otimes e)(b\otimes 1)(1\otimes i)(v)=
  (1\otimes e)(b\otimes 1)\sum_{j,k} i_{jk} v\otimes e_j \otimes e_k \]
\[ =(1\otimes e)\sum i_{jk} e_j \otimes v\otimes e_k =
 \sum i_{jk} e(v\otimes e_k) e_j. \]
Thus the left Reidemeister I relation holds for $v=e_l$ if and only if
$\sum_{j,k} i_{jk} e_{lk} e_j = e_l,$
that is, if and only if $\sum_k i_{jk} e_{lk} = \delta_{jl}$
for all $j,l$. But this holds by Proposition \ref{prop.ieinverse}.
Similar routine calculations using this proposition will readily verify the 
right Reidemeister I equation and the twisted zig-zag
identities.
\end{proof}
\begin{remark}
The author has developed computer software that enumerates (generic) polygonal
Brauer diagrams of any size and computes the corresponding representation matrices,
given by a duality structure.
\end{remark}

In light of Proposition \ref{prop.tanggensrels} and Theorem 
\ref{thm.functorytangvect}, it is clear that the zig-zag identity must be
required in the definition of a duality structure $(i,e)$, if such a structure
is to yield representations of $\Br$. But the symmetry requirements on
$(i,e)$ are also forced by
\[ e=Y(e_1)=Y(e_1 \circ b_{1,1})=Y(e_1)\circ Y(b_{1,1})=e\circ b, \]
and
\[ i=Y(i_1)=Y(b_{1,1}\circ i_1)=Y(b_{1,1})\circ Y(i_1)=b\circ i. \]

The following observation on the nontriviality of images under $Y$ will
be required in Section \ref{sec.proidemcompletion} on profinite idempotent completions.
\begin{lemma} \label{lem.yhomtangnozeros}
If $V$ is a real vector space of finite positive dimension and $(i,e)$ any duality
structure on $V$, then $Y(\Hom_\Br ([m], [n]))$ does not contain the zero map
$0: V^{\otimes m} \to V^{\otimes n}.$
\end{lemma}
\begin{proof}
We begin by observing that every loop-free morphism 
$\phi \in \Hom_\Br ([m], [n])$ of $\Br$ can be brought
into the following normal form: There are isomorphisms $\alpha$ and
$\beta$ such that $\phi = \beta \circ \phi_0 \circ \alpha$, where
$\phi_0$ has the form
\[ \phi_0 = 1_{[m-2p]} \otimes e_1^{\otimes p} \otimes i_1^{\otimes q}, \]
with $p,q\in \nat$ such that $2p\leq m,$ $2q\leq n$ and $m-2p = n-2q$.
Note that in $\vect$, the tensor product of any linear map and the zero map is again zero.
Also, the tensor product of two nonzero linear maps is again nonzero, which
can be seen by representing the maps by matrices and noting that the tensor
product corresponds to the Kronecker product of matrices.
Now, neither $Y(i_1)=i$ nor $Y(e_1)=e$ can be the zero map, for otherwise
at least one of $1_V \otimes i,$ $e\otimes 1_V$ would be zero and thus
$(e\otimes 1_V)\circ (1_V \otimes i)$ would be zero, a contradiction to the
zig-zag equation of the duality structure $(i,e)$ (and the positive dimensionality of $V$). 
Therefore,
\[ Y(\phi_0) = 1_{V^{\otimes (m-2p)}} \otimes e^{\otimes p} \otimes i^{\otimes q} \]
is not zero. As $\alpha$ and $\beta$ are isomorphisms, their images
$Y(\alpha)$ and $Y(\beta)$ are isomorphisms of positive-dimensional vector spaces, which implies that
$Y(\phi) = Y(\beta)\circ Y(\phi_0)\circ Y(\alpha)$ cannot be the zero map.
Finally, if $\psi: [m]\to [n]$ is any morphism, possibly containing loops, then write
it as $\psi = \lambda^{\otimes p} \otimes \phi$, with $\phi$ loop-free.
By Proposition \ref{prop.traceisdimension},
\[  Y(\lambda) = Y(e_1 \circ i_1) = ei =\tr (i,e)=\dim V>0, \]
from which we deduce that 
\[ Y(\psi) = Y(\lambda)^{\otimes p} \otimes Y(\phi) =
  (\dim V)^p Y(\phi) \]
is not zero.
\end{proof}
A particular duality structure on the 
two-dimensional vector space $V=\real^2$ is
\begin{equation} \label{equ.exie}
e(e_{11})=0,~ e(e_{12})=1,~ e(e_{21})=1,~ e(e_{22})=-1,~
i(1)=e_{11} + e_{12} + e_{21},
\end{equation}
where $e_{ij} = e_i \otimes e_j,$ and $e_1, e_2$ is the standard basis of
$\real^2$.
\begin{prop} \label{prop.expleloopfaithful}
The symmetric monoidal representation $Y:\Br \to \vect$ determined
by the duality structure (\ref{equ.exie}) is faithful on loops, that is,
if $\phi$ and $\psi$ are any two morphisms in $\Br$ such that
$Y(\phi)=Y(\psi),$ then $\phi$ and $\psi$ have the same number of loops.
\end{prop}
\begin{proof}
Again, we employ the normal form used in the proof of Lemma \ref{lem.yhomtangnozeros}:
Given any loop-free morphism $\phi$ of $\Br,$
there are isomorphisms $\alpha$ and
$\beta$ such that $\phi = \beta \circ \phi_0 \circ \alpha$, where
$\phi_0$ has the form
$\phi_0 = 1 \otimes e_1^{\otimes p} \otimes i_1^{\otimes q},$
with $1$ an identity morphism and $p,q\in \nat$.
For every $n=1,2,\ldots,$ we equip $V^{\otimes n}$ with the basis
$\{ e_{i_1} \otimes \cdots \otimes e_{i_n} ~|~ i_1, \ldots, i_n \in \{ 1,2 \} \}$.
With respect to this basis, linear maps $Y(\psi)$ can be written as matrices.
Recall that the matrix representative of a tensor product of linear maps is
the \emph{Kronecker product} of matrices. If $A=(a_{ij})$ is an $a\times a'$ matrix
and $B$ is a $b\times b'$ matrix, then their Kronecker product $A\otimes B$ is the
$ab \times a'b'$ block matrix with $(i,j)$-th block $a_{ij} B$.
Therefore, if all entries of $A$ and $B$ are from the set $\{0, 1, -1 \},$ then 
the entries of their Kronecker product $A\otimes B$ are also in the set
$\{ 0,1,-1 \}$. The matrix of $Y(e_1)$ is
\[ \begin{pmatrix} 0 & 1 & 1 & -1 \end{pmatrix}, \]
the matrix of $Y(i_1)$ is the transpose of
\[ \begin{pmatrix} 1 & 1 & 1 & 0 \end{pmatrix}, \]
and the matrix of $Y(b_{1,1})$ is the permutation matrix
\[ \begin{pmatrix} 1 & 0 & 0 & 1 \\ 0 & 0 & 1 & 0 \\
  0 & 1 & 0 & 0 \\ 0 & 0 & 0 & 1 \end{pmatrix}. \]
All three of these matrices have entries from the set $\{ 0,1,-1 \}$.
Hence the matrix of $Y(e_1^{\otimes p}) = Y(e_1)^{\otimes p}$, which
is the $p$-th Kronecker power of $Y(e_1)$, has entries in the set $\{ 0,1,-1\}$.
Furthermore, the matrix of $Y(i_1^{\otimes q}) = Y(i_1)^{\otimes q}$, which
is the $q$-th Kronecker power of $Y(i_1)$, has entries in the set $\{ 0,1,-1\}$.
Consequently, the matrix of 
\[ Y(\phi_0) =Y(1\otimes e_1^{\otimes p} \otimes i_1^{\otimes q}) = 1\otimes
Y(e_1^{\otimes p}) \otimes Y(i_1^{\otimes q}) \] 
has entries in the set $\{ 0,1,-1 \}$. 
Any permutation can be written as a product of \emph{adjacent}
transpositions. Thus the isomorphism $\alpha$ can be written as a composition
of morphisms of the form
\begin{equation} \label{equ.intermbraid}
1_V\otimes \cdots \otimes 1_V \otimes b_{1,1} \otimes 1_V \otimes \cdots \otimes 1_V. 
\end{equation}
As $Y(b_{1,1})$ is a permutation matrix, the matrix of the $Y$-image of a morphism
of the form (\ref{equ.intermbraid}) is a permutation matrix as well. Since permutation matrices
form a group under composition, $Y(\alpha)$ is a permutation matrix. In particular,
the entries of $Y(\alpha)$ lie in the set $\{ 0,1,-1 \}$. 
By the same token, $Y(\beta)$ is also a permutation matrix.
Multiplication of a matrix $A$ by a permutation matrix permutes the rows or columns of $A$.
Thus, the matrix of
\[ Y(\phi) = Y(\beta)\circ Y(\phi_0) \circ Y(\alpha) \]
has entries in the set $\{ 0,1,-1 \}$.
We summarize what we have shown so far: The matrix of $Y(\phi)$ for every loop-free
morphism $\phi$ of $\Br$ has entries in the set $\{ 0,1,-1 \}$.

Let us now prove that if $Y(\phi)=Y(\psi)$ and $\phi$ has a loop, then
$\psi$ also has a loop. As $\phi$ has a loop, we can write it as $\phi = \lambda \otimes \phi'$
for some $\phi'$. Set $\lh = Y(\lambda)$. In view of
\[ Y(\lambda) = Y(e_1 i_1) = Y(e_1)Y(i_1)=ei= \operatorname{Tr}(i,e), \]
the scalar $\lh$ is the trace of the duality structure $(i,e)$, $\lh =2$.
Hence, $Y(\phi) = \lh Y(\phi') = 2Y(\phi')$. If $\psi$ were loop-free, then the matrix
of $Y(\psi)$ would have to have all entries in $\{ 0,1,-1 \}$, contradicting
$Y(\psi)=2Y(\phi')$, since $Y(\phi')$ is a nonzero matrix with integer entries.
We conclude that $\psi$ has a loop as well.

Finally, suppose that $Y(\phi)=Y(\psi),$ but $\phi = \lambda^{\otimes (n+m)} \phi_0,$
$\psi = \lambda^{\otimes n} \psi_0$, $m>0,$ and $\phi_0, \psi_0$ are
loop-free. Then
\[ 2^n Y(\lambda^{\otimes m} \phi_0) = Y(\lambda^{\otimes n})
 Y(\lambda^{\otimes m} \phi_0) =
 Y(\phi) = Y(\psi) = Y(\lambda^{\otimes n})
 Y(\psi_0) = 2^n Y(\psi_0) \]
and thus $Y(\lambda^{\otimes m} \phi_0) = Y(\psi_0)$.
But as $\lambda^{\otimes m} \phi_0$ has at least one loop, this implies
that $\psi_0$ must have a loop, a contradiction. Therefore, $\phi$ and $\psi$
must have an equal number of loops.
\end{proof}

\section{Fold Maps}
\label{sec.foldmaps}

We recall here a number of concepts from differential topology which are relevant for
the present paper. Of central importance is the notion of a \emph{fold map}. 
For smooth manifolds $M$ and $N$, let $C^\infty (M,N)$ denote the
space of smooth maps from $M$ to $N$ endowed with the Whitney
$C^\infty$ topology.
Let $m$ denote the dimension of $M$ and $n$ the dimension of $N$.
Let $J^k (M,N)$ be the space of $k$-jets of smooth maps $M\to N$ and let
$\alpha: J^k (M,N)\to M,$ $\beta: J^k (M,N)\to N$ denote the source- and
target-map, respectively. The jet space is a smooth manifold of dimension
$\dim J^k (M,N) = m + n(1+ \dim P^k_m),$ where $P^k_m$ is the vector space
of polynomials in $m$ variables of degree at most $k$, which have zero
constant term. We have $J^0 (M,N)=M\times N$. 
Only the $1$-jet space will play a role in the present paper.
Elements in $J^1 (M,N)$ are represented by pairs $(x,f)$, where
$x\in M,$ $f\in C^\infty (M,N)$. Two pairs $(x,f)$ and $(x',f')$ represent
the same element of $J^1 (M,N)$ if and only if $x=x'$, 
$f(x)=f'(x)$ and $f, f'$ have first order contact at $x$, that is,
$D_x f = D_x f': T_x M \to T_{f(x)} N$. If $\sigma$ is the $1$-jet
represented by $(x,f)$, then $\alpha (\sigma)=x$ is the source of
$\sigma$ and $\beta (\sigma)= f(x)$ is the target of $\sigma$.
The projection
$\alpha \times \beta: J^1 (M,N)\to M\times N$ makes $J^1 (M,N)$ into the
total space of a vector bundle over $M\times N$ with fiber
$\Hom (\real^m, \real^n).$ Let $j^k f: M\to J^k (M,N),$ 
$(j^k f)(x) = [(x,f)],$
denote the $k$-jet extension of a smooth map $f:M\to N$.

Let $\sigma$ be a jet in $J^1 (M,N)$ and $f\in C^\infty (M,N)$ a smooth
representative of $\sigma$. Let $x$ be the source of $\sigma$ and $y=f(x)$ the
target. Then, by definition of a $1$-jet, 
the induced linear map $D_x f: T_x M \to T_y N$ depends
only on $\sigma,$ not on the choice of representative $f$. Put
$\operatorname{rank} \sigma = \operatorname{rank} (D_x f)$ and
$\operatorname{corank} \sigma = \min (m,n) - \operatorname{rank} \sigma$.
The subsets
\[ S_r = \{ \sigma \in J^1 (M,N) ~|~ \operatorname{corank} \sigma = r \},~
  r=0,1,2,\ldots, \]
are submanifolds of $J^1 (M,N)$. Moreover, the restriction
$(\alpha \times \beta)|: S_r \to M\times N$ is a subfiber-bundle of
$\alpha \times \beta: J^1 (M,N)\to M\times N$ with fiber
\[ L^r (\real^m, \real^n) = \{ A\in \Hom (\real^m, \real^n) ~|~
 \operatorname{corank} A=r \}. \]
Thus, with $q=\min (m,n),$ the codimension of $S_r$ in $J^1 (M,N)$
is $(m-q+r)(n-q+r)$. The full-rank set $S_0$ is open in $J^1 (M,N)$.

Morse theory is concerned with the case $N=\real$. As $\dim \real =1,$
the jet manifold is given by $J^1 (M,\real) = S_0 \cup S_1,$
$S_r = \varnothing$ for $r\geq 2$.
Let $S(f)$ denote the singular set (set of critical points) of a 
smooth function $f: M\to \real$. A point $x\in M$ is in $S(f)$ precisely
when $(j^1 f) (x)\in S_1$.
Such a critical point $x$ is nondegenerate (i.e. the Hessian of $f$ is nonsingular
at $x$) if and only if
$j^1 f:M\to J^1 (M,\real)$ is transverse to $S_1$ at $x$.
Consequently, the singular set of a Morse function $f$ can be expressed
as the transverse preimage $S(f) = (j^1 f)^{-1} (S_1)$ of a universal
singularity locus $S_1 \subset J^1 (M,\real)$.
The following terminology is due to R. Thom.
\begin{defn} \label{def.excellent}
A Morse function $f: M\to \real$ is called \emph{excellent}
if no two critical points lie on the same level, that is,
$f(x)\not= f(y)$ for any two distinct critical points $x,y$ of $f$.
\end{defn}
A standard tool from Morse theory allows us to approximate a given Morse function, 
without changing the location of the singular points, by an excellent Morse function, 
see \cite[Lemma 2.8, p. 17]{milnorsiebensond}.
The well-known ``Morse-Lemma'' asserts that near a singular point
$x\in S(f)$ of a Morse function $f:M\to \real,$ one can choose local
coordinates $x_1, \ldots, x_m$ centered at $x$ such that $f$ is given by
\begin{equation} \label{equ.morselemma}
(x_1,\ldots, x_m)\mapsto f(x) - (x^2_1 + \ldots + x^2_i) + x^2_{i+1} +
 \ldots + x^2_m
\end{equation}
in these coordinates. The number $i$ is invariantly associated with the
critical point $x$ and is called the \emph{index} of $x$. 
\begin{defn} \label{def.specialgenfn}
Following \cite{burletderham} and established use in the literature, a Morse
function $f:M^m\to \real$ is called \emph{special generic}, if at every critical point, $f$
attains a local minimum or maximum, that is, the index of every critical point is
$0$ or $m$.
\end{defn}
According to Reeb \cite{reeb}, a closed connected $m$-dimensional manifold $M$
that admits a special generic function is homeomorphic (but not necessarily diffeomorphic)
to the $m$-sphere $S^m$.

We shall arrive at the notion of fields to be used on cobordisms in our
field theory by ``suspending'' Morse functions: Here, the
suspension of $M$ is the cylinder $W = \real \times M$ (or $[0,1] \times M$)
and the suspension of (\ref{equ.morselemma}) is
\begin{equation} \label{equ.foldlocal}
(t, x_1,\ldots, x_m)\mapsto (t,-(x^2_1 + \ldots + x^2_i) + x^2_{i+1} +
 \ldots + x^2_m), 
\end{equation}
a map $F: W\to \real^2$ into the plane (assuming $f(x)=0$).
What is an invariant definition of maps that have this form locally?
Let $W$ be an $n$-manifold, $n\geq 2$. Let $S(F)\subset W$ denote the
singular set of a smooth map $F:W\to \real^2$.
The $1$-jet manifold $J^1 (W,\real^2)$ has dimension
\[ \dim J^1 (W,\real^2) = n+2(1+\dim P^1_n)=3n+2. \]
As $\dim \real^2 =2,$
the jet manifold is now given by $J^1 (W,\real^2) = S_0 \cup S_1 \cup S_2,$
$S_r = \varnothing$ for $r\geq 3$.
The codimension of $S_1$ in $J^1 (W,\real^2)$ is
$(n-2+1)(2-2+1)=n-1$. Thus, if $j^1 F:W\to J^1 (W,\real^2)$ is
transverse to $S_1$, then the inverse image
$S_1 (F) := (j^1 F)^{-1} (S_1)$ is a codimension $(n-1)$ submanifold
of $W$, that is, $\dim S_1 (F)=1.$ Note that by definition $\dim \ker D_x F =n-1$
for all $x\in S_1 (F)$ (regardless of whether $j^1 F$ is transverse to $S_1$ or not).
The codimension of $S_2$ in $J^1 (W,\real^2)$ is
$2n$. Thus, if $j^1 F:W\to J^1 (W,\real^2)$ is
transverse to $S_2$, then the inverse image
$S_2 (F) := (j^1 F)^{-1} (S_2)$ is a codimension $2n$ submanifold
of $W$, that is, $S_2 (F)=\varnothing$ and $S(F)=S_1 (F)$. This means
that there are no singular points $x$ with $D_x F =0$.
\begin{defn}
A smooth map $F:W\to \real^2$ is called a \emph{fold map}, if
$j^1 F$ is transverse to $S_1$, $S(F)=S_1 (F),$ and for all 
$x\in S(F),$ $T_x S(F) + \ker D_x F = T_x W.$
\end{defn}
The transversality condition ensures that $S_1 (F)$ is a smoothly
embedded submanifold of $W$ and that the normal Hessian is nondegenerate. 
In particular, in the context of this definition, the tangent space
$T_x S(F)$ is well-defined, since $S(F)=S_1 (F)$ is a smoothly embedded
submanifold of $W$. The third condition involving the kernel of the Jacobian
excludes cusps:
Let $F$ be a fold map. Since for any $x\in S(F),$ $\ker D_x F$ has
dimension $n-1,$ we actually have a direct sum decomposition
$T_x S(F) \oplus \ker D_x F = T_x W;$ a nonzero vector
tangent to $S(F)$ cannot be in the kernel of the Jacobian map
$D_x F$. We conclude that $F|_{S(F)}$ is an immersion.
\begin{examples}
The stable Whitney cusp $F:W=\real^2 \to \real^2,$ $F(t,x)=(t,xt+x^3),$ satisfies
$S(F)=S_1 (F)$ and $j^1 F$ is transverse to $S_1$ but not
$T_{(t,x)} S(F) + \ker D_{(t,x)} F = T_{(t,x)} W$:
At $(t,x)=(0,0),$ $F|_{S(F)}$ has a singularity.
The map $F: \real^2 \to \real^2,$ $F(t,x)=(t,0),$
has singular set $S(F)=\real^2$ but $S_2 (F)=\varnothing$.
Thus $S(F)=S_1 (F)$ and $T_{(t,x)} S(F) + \ker D_{(t,x)} F = T_{(t,x)} \real^2$
but $j^1 F$ is not transverse to $S_1$.
Now let $F$ be the map $F:\real^2 \to \real^2,$ $F(t,x)=(t^2, x^2).$
Its singular set $S(F) = S_1 (F)\sqcup S_2 (F)$ is the union of the two coordinate
axes, with $S_2 (F) = \{ (0,0) \},$ the origin. A calculation using the
intrinsic derivative of Porteous shows that $j^1 F$ is transverse to $S_1$
(but obviously not to $S_2$). The condition
$T_{(t,x)} S_1 (F) + \ker D_{(t,x)} F = T_{(t,x)} \real^2,$ $(t,x)\in S_1 (F),$ is satisfied
and there are no cusp singularities in $S_1 (F)$.
All three conditions are satisfied for (\ref{equ.foldlocal}).
\end{examples}

If $F$ is a fold map, then we refer to its singular set $S(F)$
also as its \emph{fold locus} or \emph{fold lines}.
Points $x$ in $S(F)$ are called \emph{fold points}. Let $x$ be
such a point. As $F|_{S(F)}$ is an immersion, its image is locally
a $1$-manifold, so we can choose coordinates $y_1, y_2$ centered
at $F(x)$ so that $F(S(F))$ is locally given by $\{ y_2 =0 \}.$
Since $F|: S(F)\to \{ y_2 =0 \}$ is a local diffeomorphism near $x$,
there are local coordinates $t, x_1,\ldots, x_{n-1}$ centered at $x$
such that $t=y_1 \circ f$ and $S(F)$ is given near $x$ by
$\{ x_1 = \cdots = x_{n-1} =0 \}$. In these coordinates, $F$ takes
the form
\[ (t, x_1, \ldots, x_{n-1})\mapsto (t,
 \sum_{i,j} h_{ij} x_i x_j), \]
where the $h_{ij}$ are smooth functions. The transversality condition
$j^1 F \pitchfork S_1$ at $x$ implies that the $(n-1)\times (n-1)$ matrix
$(h_{ij} (0))$ is nonsingular. Thus one achieves a normal form for fold maps
analogous to the statement of the Morse Lemma for Morse functions, see
\cite[Thm. 4.5, p. 88]{golubguillstable}:
\begin{prop} \label{prop.foldmapnormalform}
Let $W$ be an $n$-dimensional manifold, $n\geq 2,$ $F:W\to \real^2$
a fold map, and $x$ a singular point of $F$. Then there exist local
coordinates $t,x_1,\ldots, x_{n-1}$ centered at $x$ and coordinates
$y_1, y_2$ centered at $F(x)$ so that $F$ is given by
\[ (t, x_1,\ldots, x_{n-1})\mapsto (t,-(x^2_1 + \ldots + x^2_i) + x^2_{i+1} +
 \ldots + x^2_{n-1}) \]
in these coordinates, for some $0\leq i \leq n-1$.
\end{prop}
The number $\tau (x) = \max (i, n-1-i)$ is called the
\emph{absolute index} of the fold point $x$. The absolute index
is invariantly defined and is constant along every connected component
of $S(F)$, which will be an essential feature in the proof of
Theorem \ref{thm.iexoticsphere} on exotic spheres. 
If $\tau (x)=n-1,$ then $x$ is called a \emph{definite} fold point.
\begin{defn}
Again following \cite{burletderham}, a fold
map $F:W\to \real^2$ is called \emph{special generic}, if $S(F)$ consists
entirely of definite fold points.
\end{defn}
Cobordism groups of special generic maps have been introduced by
O. Saeki in \cite{saekitopspecgen}. In \cite{saekihtpyspheres}, it is shown that for $n\geq 6,$
the oriented cobordism group of special generic functions on $n$-manifolds
is isomorphic to the group $\Theta_n$ of homotopy $n$-spheres introduced in \cite{kervairemilnor}.
Saeki's methods, based on the notion of Stein factorization, play a central role
when we establish in Section \ref{sec.exoticspheres} that our TFT distinguishes
exotic smooth spheres from standard spheres.

\section{Monoids, Semirings, and Semimodules}
\label{sec.monssemirings}

We recall some foundational material on monoids, semirings and semimodules
over semirings. Such structures seem to have appeared first in Dedekind's study of ideals in a
commutative ring: one can add and multiply two ideals, but one cannot subtract them.
The theory has been further developed by H. S. Vandiver, S. Eilenberg, A. Salomaa,
J. H. Conway and others.
A monoid $(M,\cdot)$ is \emph{idempotent} if $m\cdot m =m$ for
all elements $m\in M$. The Boolean monoid $(\bool,+)$, $\bool=\{ 0,1 \},$ is idempotent:
$0$ is the neutral element and $1+1=1$ in $\bool$. 

Roughly, a semiring is a ring without additive inverses. More precisely,
a \emph{semiring} is a set $S$ together with two binary operations $+$ and $\cdot$
and two elements $0,1\in S$ such that $(S,+,0)$ is a commutative monoid,
$(S,\cdot, 1)$ is a monoid, the multiplication $\cdot$ distributes over the addition
from either side, and $0$ is absorbing, i.e. $0\cdot s = 0 = s \cdot 0$ for every $s\in S$.
For instance, the natural numbers $S=\nat = \{ 0,1,2,\ldots \}$ form a semiring under
standard addition and multiplication.
If the monoid $(S,\cdot,1)$ is commutative, the semiring $S$ is called commutative.
The semiring is called \emph{idempotent}, if $(S,+,0)$ is an idempotent monoid.
By distributivity, this holds if and only if $1+1=1$. The Boolean monoid $\bool$ becomes
an idempotent commutative semiring by defining $1\cdot 1 =1$. A morphism of
semirings sends $0$ to $0$, $1$ to $1$ and respects addition and multiplication.

Let $S$ be a semiring. A \emph{(formal) power series} with coefficients in $S$ is
a function $a:\nat \to S$. It is customary, and convenient, to write such a function
as
\[ \sum_{i=0}^\infty a(i)q^i, \]
where $q$ is a formal variable acting as a bookkeeping device. The element $a(i)$
is referred to as the \emph{coefficient} of $q^i$. Let $S[[q]]$ be the set of all
power series with coefficients in $S$. Write $0$ for the power series $a$ with
$a(i)=0$ for all $i$. Write $1$ for the power series $a$ with $a(0)=1$ and 
$a(i)=0$ for all $i>0$. Define an addition on power series by $a+b=c$, where
$c(i)=a(i)+b(i)$ for all $i$. Define a multiplication on power series by the
Cauchy product, that is, $a\cdot b=c$ where $c(i) = \sum_{n+m=i} a(n)b(m)$.
Then $(S[[q]], +, \cdot, 0, 1)$ is a semiring, the \emph{semiring of power series
over $S$.} If $S$ is idempotent, then so is $S[[q]]$. We shall particularly make use
of the idempotent commutative semiring $\bool[[q]]$. In a similar manner, using
finite sums, one defines the \emph{polynomial semiring} $S[q]$.

Given a ring $R$ and a group $G$, one may form the group ring $R[G]$.
Similarly, given a semiring $S$ and a monoid $M$, one can form the
\emph{monoid semiring} $S[M]$. Its elements are functions $a:M \to S$ such that
$a(m)\not= 0$ only for a finite number of $m\in M$. Again, such a function
is customarily written as a finite sum
\[ a(m_1)m_1 + \cdots + a(m_k)m_k. \]
(In order for this notation to work, the monoid operation on $M$ should be
written multiplicatively, even when $M$ is commutative.)
The sum of $a,b\in S[M]$ is $c:M\to S$ with $c(m)=a(m)+b(m)$ for all $m$.
The product $a\cdot b$ in $S[M]$ is $c:M\to S$ with
$c(m) = \sum_{pq=m} a(p)b(q).$ In other words,
\[ \left( \sum_{i=1}^k a(m_i)m_i \right) \left(\sum_{j=1}^l b(n_j)n_j \right) = 
 \sum_{i=1}^k \sum_{j=1}^l (a(m_i) b(n_j))(m_i n_j). \]
Equipped with these operations, $S[M]$ is a semiring whose zero is the
function $a(m)=0$ for all $m$, and whose one is the function $a(1_M)=1_S$
and $a(m)=0$ for all $m\not= 1_M$. (Here, $1_M$ denotes the neutral element
of $M$ and $1_S$ the one-element of the semiring $S$.) For example, if
$S=\nat$ is the semiring of natural numbers 
and $M=\nat$ is the additive monoid of natural numbers, we obtain
the monoid semiring $\nn$. Here, one should write $M=\nat$ multiplicatively,
that is, we write $M=\{ \tau^i ~|~ i\in \nat \}$. It is then clear that
$\nn$ is isomorphic to the polynomial semiring $\nat [\tau]$.

Let $S$ be a semiring. A \emph{left $S$-semimodule} is a commutative monoid
$(M,+,0_M)$ together with a scalar multiplication $S\times M \to M,$
$(s,m)\mapsto sm,$ such that for all $r,s\in S,$ $m,n\in M,$ we have
$(rs)m=r(sm),$ $r(m+n)=rm+rn,$ $(r+s)m=rm+sm,$ $1m=m,$ and
$r0_M = 0_M = 0m.$ Right semimodules are defined similarly using
scalar multiplications $M\times S \to M,$ $(m,s)\mapsto ms$.
Given semirings $R$ and $S$, an \emph{$R$-$S$-bisemimodule} is a commutative
monoid $(M,+,0)$, which is both a left $R$-semimodule and a right $S$-semimodule
such that $(rm)s=r(ms)$ for all $r\in R,$ $s\in S,$ $m\in M$. (Thus the notation
$rms$ is unambiguous.)
An \emph{$R$-$S$-bisemimodule homomorphism} is a homomorphism
$f:M\to N$ of the underlying additive monoids such that $f(rms) =rf(m)s$ for all $r,m,s$. 
If $R=S$, we shall also speak of an $S$-bisemimodule.
Every semimodule $M$ over a commutative semiring $S$ can and will be
assumed to be both a left and right semimodule with $sm=ms$.
In fact, $M$ is then a bisemimodule, as for all $r,s\in S,$ $m\in M,$
\[ (rm)s = s(rm) = (sr)m = (rs)m = r(sm) = r(ms). \]
A semimodule $M$ is called \emph{idempotent} if the underlying
additive monoid $(M,+,0_M)$ is idempotent.
If $\phi: S\to T$ is a morphism of semirings, then
$T$ becomes a left $S$-semimodule by setting $st = \phi(s)\cdot t$, using the
multiplication in $T$. Let $\phi: \nat \to \bool$ be the morphism of semirings
uniquely specified by $\phi(0)=0$ and $\phi(1)=1$.
Via $\phi,$ $\bool$ is a $\nat$-semimodule. 
By setting $\phi(\tau)=q,$ the map $\phi$ extends to a semiring morphism
$\phi: \nat [\tau] \to \bool[[q]]$. Explicitly,
\[ \phi (n_0 + n_1 \tau + \cdots +n_k \tau^k)=
   \phi(n_0) + \phi(n_1) q + \cdots + \phi(n_k) q^k \in \bool[[q]]. \]
Via this morphism, $\bool[[q]]$ is a semimodule over $\nat [\tau]$.

If $A$ is a set, let $\fm (A)$ denote the free commutative monoid generated
by $A$. Its elements are finite formal linear combinations
$\sum_{i=1}^k \alpha_i a_i,$ $\alpha_i\in \nat,$ $a_i \in A$. The addition of two such
linear combinations is defined in the obvious way.

\begin{lemma} \label{lem.antausemimod}
Suppose that $A$ is equipped with an action $\nat \times A\to A$ of the 
monoid $(\nat,+)$, $(\tau^i, a)\mapsto \tau^i a.$ Then 
\[ \sum_i m_i \tau^i \cdot \sum_j \alpha_j a_j = \sum_{i,j} (m_i \alpha_j)
  (\tau^i a_j),~ m_i, \alpha_j \in \nat, a_j \in A, \]
makes $\fm (A)$ into an $\nat [\tau]$-semimodule.
\end{lemma}

Let $S$ be any semiring, not necessarily commutative.
Regarding the tensor product of a right $S$-semimodule $M$ and 
a left $S$-semimodule $N$, one
has to exercise caution because even when $S$ is commutative,
two nonisomorphic tensor products,
both called \emph{the} tensor product of $M$ and $N$ and both written
$M\otimes_S N$, exist in the literature. A map $\phi: M\times N\to A$ into
a commutative additive monoid $A$ is called \emph{middle $S$-linear}, if
it is biadditive, $\phi (ms,n)=\phi (m,sn)$ for all $m,s,n$, and $\phi (0,0)=0$.
For us, an (algebraic) \emph{tensor product of
$M$ and $N$} is a commutative monoid $M\otimes_S N$ (written additively) satisfying the
following (standard) universal property: $M \otimes_S N$ comes
equipped with a middle $S$-linear map
$M\times N \to M\otimes_S N$ such that
given any commutative monoid $A$ and 
middle $S$-linear map $\phi: M\times N \to A$, there exists a unique monoid homomorphism
$\psi: M\otimes _S N \to A$ such that
\begin{equation} \label{equ.uniproptens}
\xymatrix{
M\times N \ar[r]^{\phi} \ar[d] & A \\
M\otimes_S N \ar@{..>}[ru]_{\psi}
} \end{equation}
commutes. The existence of such a tensor product is shown for example in
\cite{katsovtensor}, \cite{katsov}. To construct it, take $M\otimes_S N$ to be the quotient 
monoid $\fm (M\times N) /\sim$, where $\sim$ is the congruence relation on $\fm (M\times N)$
generated by all pairs of the form
\[ ((m+m', n), (m,n)+(m',n)),~ ((m,n+n'), (m,n)+(m,n')),~
  ((ms,n), (m,sn)), \]
$m,m'\in M,$ $n,n' \in N,$ $s\in S$. 
If one of the two 
semimodules, say $N$, is idempotent, then $M\otimes_S N$ is idempotent as well,
since $m\otimes n + m\otimes n = m\otimes (n+n) = m\otimes n$.
If $M$ is an $R$-$S$-bisemimodule and $N$ an $S$-$T$-bisemimodule, then
the monoid $M\otimes_S N$ as constructed above becomes an $R$-$T$-bisemimodule
by declaring
\[ r\cdot (m\otimes n) = (rm)\otimes n,~ (m\otimes n)\cdot t = m\otimes (nt). \]
If in diagram (\ref{equ.uniproptens}), the monoid $A$ is an $R$-$T$-semimodule
and $M\times N\to A$ satisfies
\[ \phi (rm,n) = r\phi(m,n),~ \phi (m,nt)=\phi(m,n)t \]
(in addition to being middle $S$-linear;
let us call such a map \emph{$R_S T$-linear}), then the uniquely determined monoid map 
$\psi: M\otimes_S N\to A$ is an $R$-$T$-bisemimodule homomorphism, for
\[ \psi (r(m\otimes n)) = \psi ((rm)\otimes n) = \phi (rm,n)=r\phi(m,n)=r\psi(m\otimes n) \]
and similarly for the right action of $T$.
If $R=S=T$ and $S$ is commutative, the above means that
the commutative monoid
$M\otimes_S N$ is an
$S$-semimodule with $s(m\otimes n) = (sm) \otimes n = m\otimes (sn)$ and
the diagram (\ref{equ.uniproptens}) takes place in the category of $S$-semimodules.

The tensor product of \cite{takahashi} and \cite{golan} --- let us here write it
as $\otimes'_S$ --- satisfies a different
universal property. A semimodule $C$
is called \emph{cancellative} if $a+c=b+c$ implies $a=b$ for all $a,b,c\in C$.
A nontrivial idempotent semimodule is never cancellative.
Given an arbitrary right $S$-semimodule $M$ and an arbitrary left $S$-semimodule $N$, 
the product $M\otimes'_S N$ is always
cancellative.
Thus if one of $M, N$ is idempotent, then $M\otimes'_S N$ is
trivial, being both idempotent and cancellative. Since in our applications, we desire
nontrivial tensor products of idempotent semimodules, the product $\otimes'_S$
will not be used in this paper. \\

We continue to assume that the set $A$ is equipped with an $\nat$-action.
Applying $\otimes_{\nat [\tau]}$ to
the semimodules $\fm (A)$ and $\bool[[q]]$ over the commutative semiring
$\nat [\tau]$, we obtain the $\nat [\tau]$-semimodule
\[ Q(A) = \fm (A)\otimes_{\nat [\tau]} \bool[[q]]. \]
Since $\bool[[q]]$ is idempotent, so is $Q(A)$. Any element of $Q(A)$ can be
written in the form
\[ \sum_{i=1}^k a_i \otimes b_i,~ a_i \in A,~ b_i \in \bool[[q]]. \]
To see this, we note first that elements of the tensor product can be written
as
\[ \sum_{i=1}^k m_i \otimes b_i,~ m_i \in \fm (A),~ m_i \not= 0,~ b_i \in \bool[[q]]. \]
Every $m_i$ has the form $m_i = \sum_{j=1}^{k_i} \alpha_{ij} a_{ij},$
$\alpha_{ij} \in \nat - \{ 0 \},$ $a_{ij}\in A,$ $k_i \geq 1$. Hence
\[ \sum_i m_i \otimes b_i = \sum_i (\sum_j \alpha_{ij} a_{ij})\otimes b_i
 = \sum_{i,j} a_{ij} \otimes (\alpha_{ij} b_i) = \sum_{i,j} a_{ij}\otimes b_i, \]
using the idempotency of $\bool[[q]]$.

The key feature of the state (semi)modules (to be introduced in Section \ref{quantization})
that will allow us to form well-defined state sums is 
their completeness.
Thus let us recall the notion of a complete monoid, semiring, etc.
as introduced by S. Eilenberg on p. 125 of \cite{eilenbergalm}; see also
\cite{karner}, \cite{krob88}.
A \emph{complete monoid} is a commutative monoid $(M,+,0)$ together with
an assignment $\sum$, called a \emph{summation law}, which assigns to every
family $(m_i )_{i\in I}$, indexed by an arbitrary set $I$, an element
$\sum_{i\in I} m_i$ of $M$ (called the \emph{sum} of the $m_i$), such that
\[ \sum_{i\in \varnothing} m_i =0,~
  \sum_{i\in \{ 1 \}} m_i = m_1,~
 \sum_{i\in \{ 1,2 \}} m_i = m_1 + m_2, \]
and for every partition $I = \bigcup_{j\in J} I_j,$
\[ \sum_{j\in J} \left( \sum_{i\in I_j} m_i \right) = \sum_{i\in I} m_i. \]
Note that these axioms imply that if $\sigma: J\to I$ is a bijection, then
\[ \sum_{i\in I} m_i = \sum_{j\in J} \sum_{i\in \{ \sigma (j) \}} m_i =
 \sum_{j\in J} m_{\sigma (j)}. \]
Also, since a cartesian product $I\times J$ comes with two canonical
partitions, namely $I\times J = \bigcup_{i\in I} \{ i \}\times J =
  \bigcup_{j\in J} I\times \{ j \},$ one has
\begin{equation} \label{equ.sumijmonoid}
\sum_{(i,j)\in I\times J} m_{ij} = \sum_{i\in I} \sum_{j\in J} m_{ij} =
 \sum_{j\in J} \sum_{i\in I} m_{ij} 
\end{equation}
for any family $(m_{ij} )_{(i,j)\in I\times J}$. This is the analog of
Fubini's theorem in the theory of integration.
Given $(M,+,0)$, the summation law $\sum$, if it exists, is not in general
uniquely determined by the addition, as examples in \cite{goldstern85} show.
For a semiring to be complete one requires that $(S,+,0,\sum)$ be a
complete monoid and adds the infinite distributivity requirements
\[ \sum_{i\in I} ss_i = s(\sum_{i\in I} s_i),~
  \sum_{i\in I} s_i s = (\sum_{i\in I} s_i)s. \]
Note that in a complete semiring, the sum over any family of zeros must be
zero. If $(s_{ij} )_{(i,j)\in I\times J}$ is a family in a complete semiring
of the form $s_{ij} = s_i t_j,$ then, using (\ref{equ.sumijmonoid}) together
with the infinite distributivity requirement,
\begin{equation} \label{equ.sumoverproductfamily}
\sum_{(i,j)\in I\times J} s_i t_j = \sum_{i\in I} \sum_{j\in J} s_i t_j =
 \big( \sum_{i\in I} s_i \big) \big( \sum_{j\in J} t_j \big). 
\end{equation}
A semiring is \emph{zerosumfree}, if $s+t=0$ implies $s=t=0$ for all
$s,t$ in the semiring. The so-called ``Eilenberg-swindle'' shows that
every complete semiring is zerosumfree.
In a ring we have additive inverses, which means that a nontrivial ring
is never zerosumfree and so cannot be endowed with an infinite summation
law that makes it complete. This shows that giving up additive inverses,
thereby passing to semirings that are not rings, is an essential prerequisite
for completeness and in turn essential for the construction of our
topological field theories.
Complete semimodules are defined analogously: A semimodule $M$ over a
commutative semiring $S$ is called complete if its underlying additive
monoid is equipped with a summation law that makes it complete as a
commutative monoid and the infinite distribution requirement
\[ \sum_{i\in I} s m_i = s(\sum_{i\in I} m_i) \]
holds for every $s\in S$ and every family $(m_i )_{i\in I}$ in $M$.
The Boolean semiring $\bool$ is complete (using the obvious summation law). 
If $S$ is a complete semiring, then the
semiring $S[[q]]$ of formal power series becomes a complete semiring
by transferring the summation law on $S$ pointwise to $S[[q]]$, see
\cite{karner}. Hence, $\bool[[q]]$ is a complete semiring.
If $\phi: S\to T$ is a morphism of semirings and $T$ is complete
as a semiring, then $T$ is automatically complete as an
$S$-semimodule because
\[ s\sum_{i\in I} t_i = \phi (s)\sum_{i} t_i =
  \sum_i \phi (s)t_i = \sum_i st_i \]
for all elements $s\in S$ and families $(t_i )_{i\in I}$ in $T$.
Applying this to the morphism of semirings 
$\nat [\tau] \to \bool [[q]]$, we see that $\bool [[q]]$ is complete
as an $\nat [\tau]$-semimodule. For numerous other examples of complete
semirings, see \cite{banagl-postft}.

Let $R,S$ be any semirings, not necessarily commutative.
An (associative, unital) \emph{$R$-$S$-semialgebra} is a semiring $A$ which is in addition
an $R$-$S$-bisemimodule such that for all $a,b\in A,$ $r\in R,$ $s\in S$, one has
$r(ab)=(ra)b,~ (ab)s=a(bs).$
(Note that we refrain from using the term ``$R$-$S$-bisemialgebra'' for such a
structure, since ``bialgebra'' refers to something completely different,
namely a structure with both multiplication and comultiplication.)
If $R=S$, we shall also use the term \emph{two-sided $S$-semialgebra}
for an $S$-$S$-semialgebra.
If $S$ is commutative, then a two-sided $S$-semialgebra $A$ with $sa=as$ is simply
a semialgebra over $S$ in the usual sense, as
\[ (sa)b=s(ab)=(ab)s=a(bs)=a(sb). \]
A morphism of $R$-$S$-semialgebras is a morphism of semirings which is in addition
a $R$-$S$-bisemimodule homomorphism.

A \emph{partially ordered semiring} is a quadruple
$(S,+,\cdot, \leq)$, where $(S,+,\cdot)$ is a semiring and
$(S,\leq)$ a partially ordered set such that for all $s,t,u\in S,$
\[ 0\leq s \text{ and} \]
\[ s\leq t \text{ implies } s+u\leq t+u,~ su\leq tu,~ us\leq ut. \]
Every partially ordered semiring is zerosumfree, since
$s+t=0$ implies $0\leq s \leq s+t=0.$ An idempotent semiring
always possesses a unique partial order, namely $s\leq t$ if and only if
$s+t=t$. (Note that $s\leq t$ implies
$t\leq s+t\leq t+t=t$; cf. \cite[Thm. 5.2 and Exercise 5.4]{kuichsalomaa}.)
An important formal observation is that in an idempotent semiring,
the addition can always be expressed as a supremum,
\[ s+t = \sup \{ s,t \}. \]
To see this, we note that $s,t\geq 0$ and the monotonicity of addition
imply $s,t\leq s+t$, which shows the inequality $\sup \{ s,t \} \leq s+t$.
On the other hand, $s,t\leq \sup \{ s,t \}$ and monotonicity together with idempotence yield
\[ s+t \leq s + \sup \{ s,t \} \leq \sup \{ s,t \} + \sup \{ s,t \} = \sup \{ s,t \}. \]
An idempotent complete monoid $(M,+,0, \sum)$ is
\emph{continuous} (cf. \cite{krob88}, \cite{sakarovitch}, \cite{goldstern85},
\cite{karner}) if for all families $(m_i )_{i\in I},$ $m_i \in M$, and
for all $c\in M$, $\sum_{i\in F} m_i \leq c$ for all finite $F\subset I$
implies $\sum_{i\in I} m_i \leq c$.
In other words: The supremum of
\[ \left\{ \sum_{i\in F} m_i ~|~ F\subset I,~ F \text{ finite} \right\} \subset M \]
exists and equals $\sum_{i\in I} m_i$. Requiring this of the underlying additive monoid
of an idempotent complete semiring, we arrive at the concept of a continuous semiring.
This concept can be extended to non-idempotent semirings by
requiring the existence of a certain natural partial order (of which the above idempotent 
partial order is a special case) and will not be further discussed here.
Almost all complete semirings that are important in applications are
continuous. The idempotent complete semirings $\bool$ and $\bool [[q]]$
are both continuous (\cite[p. 157]{karner}). The (idempotent, complete) value semiring $Q$ used
in our quantization will be shown to be continuous in
Proposition \ref{prop.qcontinuous}.

\begin{prop} \label{prop.contidemsemiring}
Let $(M,+,0,\sum)$ be a continuous, idempotent, complete monoid. Then the value of a sum
over an arbitrary family of elements in $M$ depends only on the underlying
set of the family. That is, if $(m_i)_{i\in I}$ and $(n_j)_{j\in J}$ are
families in $M$, then
\[ \sum_{i\in I} m_i = \sum_{j\in J} n_j \]
when $\{ m_i ~|~ i\in I \} = \{ n_j ~|~ j\in J \}$ as subsets of $M$.
\end{prop}
\begin{proof}
Let $m\in M$ be any element and $K$ an arbitrary index set. If $F\subset K$
is a finite subset, then $\sum_F m=m,$ since $M$ is idempotent.
Thus, as $M$ is continuous,
\begin{equation} \label{equ.sumconstel}
\sum_K m = \sup \left\{ \sum_F m ~|~ F\subset K,~ F \text{ finite} \right\}
  = \sup \{ m \} = m. 
\end{equation}
Given two families  $(m_i)_{i\in I}$ and $(n_j)_{j\in J}$ in $M$, we define
subsets $[I], [J]\subset M$ by
\[ [I] = \{ m_i ~|~ i\in I \},~  [J] = \{ n_j ~|~ j\in J \} \]
and consider the partitions
\[ I = \bigcup_{m\in [I]} I_m,~ J = \bigcup_{n\in [J]} J_n, \]
where
\[ I_m = \{ i\in I ~|~ m_i =m \},~ J_n = \{ j\in J ~|~ n_j = n \}. \]
Then, assuming $[I] = [J]$, the partition axiom for the infinite summation
law together with (\ref{equ.sumconstel}) imply
\begin{eqnarray*}
\sum_{i\in I} m_i & = & 
  \sum_{m\in [I]} (\sum_{i\in I_m} m_i ) =
  \sum_{m\in [I]} (\sum_{i\in I_m} m) =
  \sum_{m\in [I]} m \\
& = &  \sum_{n\in [J]} n =
    \sum_{n\in [J]} (\sum_{j\in J_n} n) =
   \sum_{n\in [J]} (\sum_{j\in J_n} n_j) =  \sum_{n\in J} n_j. 
\end{eqnarray*}
\end{proof}

In the case of a continuous, idempotent, complete semiring $S$, the summation law
agrees with the notion of idempotent integral in idempotent analysis, as developed
by the Russian school around V. P. Maslov, see e.g. \cite{litmasshpizidemfunctana}.
Maslov and his collaborators define the idempotent integral of a function
$f:X\to S,$ where $X$ is any set, to be $\int_X f(x)dx := \sup f(X)$.
The analogy to the ordinary Riemann integral is revealed by replacing ordinary
addition in a Riemann sum by
$x \oplus_h y = h\log (e^{x/h} + e^{y/h}),$ replacing ordinary multiplication by
$x \odot y = x+y,$ and letting the parameter $h$ (playing the role of Planck's constant)
tend to $0$.
\begin{prop} \label{prop.eilensummaslovint}
Let $S$ be a continuous, idempotent, complete semiring with summation law $\sum$ and
let $X$ be any set. Then the image $f(X)$ of any function $f:X\to S$ possesses a 
supremum and the corresponding Eilenberg sum and Maslov idempotent integral are
equal,
\[ \sum_{x\in X} f(x) = \int_X f(x)dx. \]
\end{prop}
\begin{proof}
Set $E = \sum_{x\in X} f(x)$. For any $x_0 \in X,$
\[ f(x_0) + E = f(x_0) + f(x_0) + \sum_{x\in X- \{ x_0 \}} f(x) =
   f(x_0) + \sum_{x\in X- \{ x_0 \}} f(x) = E, \]
by idempotence and using the partition axiom of the summation law.
Thus $E$ is an upper bound for the image $f(X)$. We claim that it is in fact the
least upper bound: Let $B$ be any upper bound for $f(X)$. Recall that the supremum
satisfies the associativity
\[ \sup \{ \sup \{ y_1, y_2 \}, y_3 \} = \sup \{ y_1, y_2, y_3 \} =
  \sup \{ y_1, \sup \{ y_2, y_3 \} \}. \]
Hence inductively, and using $y_1 + y_2 = \sup \{ y_1, y_2 \}$ in $S$, one finds
for any finite subset $\{ x_1, \ldots, x_k \} \subset X$,
\[ f(x_1)+\cdots +f(x_k) = \sup \{ f(x_1),\ldots, f(x_k) \} \leq B. \]
Therefore, by the continuity of $S$,
\[ E = \sup \{ \sum_{x\in F} f(x) ~|~ F\subset X,~ F \text{ finite} \} \leq B. \] 
Thus $\sup f(X) = \int_X f$ exists and is equal to $E$.
\end{proof}
Though we did provide a direct proof of Proposition \ref{prop.contidemsemiring}, that
proposition is also an immediate consequence of Proposition \ref{prop.eilensummaslovint}.

\section{Function Semimodules and Their Tensor Products}
\label{sec.funcompltensorprod}

Let $S$ be a semiring, not necessarily commutative. Given a set $A$, let
$\fun_S (A) = \{ f:A\to S \}$
be the set of all $S$-valued functions on $A$. If $S$ is understood,
we will also write $\fun (A)$ for $\fun_S (A)$.
\begin{prop} \label{prop.funsemimod}
Using pointwise addition and multiplication, $\fun_S (A)$ inherits the
structure of a two-sided $S$-semialgebra from the operations of $S$. 
If $S$ is complete (as a semiring), then $\fun_S (A)$ is complete as a semiring and as
an $S$-bisemimodule. If $S$ is commutative, then $\fun_S (A)$ is a commutative
$S$-semialgebra.
\end{prop}
\begin{proof}
We only discuss the completeness.
If $S$ is a complete semiring, then an infinite summation law in $\fun_S (A)$
can be introduced by
$\big( \sum_{i\in I} f_i \big) (a) = \sum_{i\in I} f_i (a),$
$f_i \in \fun_S (A)$. With this law, $(\fun_S (A),+,0,\sum)$ is a complete
monoid, $\fun_S (A)$ is complete as a semiring and
complete as an $S$-bisemimodule.
\end{proof}
It is clear that if $S$ is idempotent, then $\fun_S (A)$ is idempotent.
Furthermore,
if a complete, idempotent semiring $S$ is continuous, then $\fun_S (A)$ is continuous.
In Section \ref{sec.proidemcompletion}, we will introduce
the composition semiring $Q^c$ and the monoidal semiring $Q^m$ of the profinite idempotent
completion $Q$ of $Y(\Mor (\Br))$. These semirings will turn out to be complete, idempotent and continuous
(Propositions \ref{prop.qccomplidemp}, \ref{prop.qmcomplidemp}, \ref{prop.qcontinuous}).
Hence,
the state-(semi)modules $Z(M) = \fun_Q (\Fa (M))$ to be introduced in
Section \ref{quantization} are continuous.

Let $B$ be another set. Then, regarding $\fun_S (A)$ and $\fun_S (B)$ as
$S$-bisemimodules, the algebraic tensor product $\fun_S (A)\otimes_S \fun_S (B)$ is
defined. It is an $S$-bisemimodule such that given any $S$-bisemimodule $M$ and
$S_S S$-linear map $\phi: \fun (A)\times \fun (B)\to M,$ there exists a unique
$S$-$S$-bisemimodule homomorphism $\psi: \fun (A)\otimes_S \fun (B)\to M$ such that
\[ \xymatrix@R=15pt{
\fun (A)\times \fun (B) \ar[r]^>>>>>{\phi} \ar[d] & M \\
\fun (A)\otimes_S \fun (B) \ar[ru]_{\psi}
} \]
commutes. The $S$-bisemimodule $\fun (A\times B)$ comes naturally
equipped with an $S_S S$-linear map
\[ \beta: \fun (A)\times \fun (B) \longrightarrow \fun (A\times B),~
  \beta (f,g) = ((a,b)\mapsto f(a)\cdot g(b)). \]
(If $S$ is commutative, then $\beta$ is $S$-bilinear.)
Thus, taking $M = \fun_S (A\times B)$ in the above diagram,
there exists a unique
$S$-$S$-bisemimodule homomorphism $\mu: \fun (A)\otimes_S \fun (B)\to \fun (A\times B)$ such that
\begin{equation} \label{equ.funatimesb} 
\xymatrix@R=15pt{
\fun (A)\times \fun (B) \ar[r]^>>>>>{\beta} \ar[d] & \fun (A\times B) \\
\fun (A)\otimes_S \fun (B) \ar[ru]_{\mu}
} \end{equation}
commutes. In the commutative setting, this homomorphism was studied in \cite{banagl-tensor}, where we
showed that it is generally neither surjective nor injective when $S$ is not
a field. If $A$ and $B$ are finite, then $\mu$ is an isomorphism.
If $A,B$ are infinite but $S$ happens to be a field, then $\mu$ is still injective,
but not generally surjective. This is
the reason why the functional analyst completes the tensor product $\otimes$ using
various topologies available, arriving at products $\hotimes$.
For example, for compact Hausdorff spaces $A$ and $B$, let $C(A),C(B)$ 
denote the Banach spaces of all complex-valued continuous functions on
$A, B,$ respectively, endowed with the supremum-norm, yielding the topology of
uniform convergence. Then the image of 
$\mu: C(A)\otimes C(B) \to C(A\times B)$, while not all of $C(A\times B),$
is however dense in $C(A\times B)$ by the Stone-Weierstra{\ss} theorem.
After completion, $\mu$ induces an isomorphism
$C(A) \hotimes_\epsilon C(B) \cong C(A\times B)$
of Banach spaces, where $\hotimes_\epsilon$ denotes the so-called
$\epsilon$-tensor product or injective tensor product of two locally convex
topological vector spaces. For $n$-dimensional Euclidean space $\real^n$,
let $L^2 (\real^n)$ denote the Hilbert space of square integrable functions
on $\real^n$. Then $\mu$ induces an isomorphism
$L^2 (\real^n) \hotimes L^2 (\real^m) \cong L^2 (\real^{n+m}) =
L^2 (\real^n \times \real^m),$ where $\hotimes$ denotes the Hilbert space
tensor product, a completion of the algebraic tensor product $\otimes$ of two Hilbert spaces.
For more information on topological tensor products see 
\cite{schatten}, \cite{grothendiecktensor}, \cite{treves}.
In \cite{banagl-tensor} we show that even over the smallest complete (in particular
zerosumfree) and additively idempotent commutative semiring, namely the Boolean semiring $\bool$,
and for the smallest infinite cardinal $\aleph_0$, modeled by a countably infinite set $A$,
the map $\mu$ is not surjective, which means that in the context of the present paper,
some form of completion must be used as well.
However, there is an even more serious complication which arises over semirings:
In marked contrast to the
situation over fields, the canonical map $\mu$ ceases to be \emph{injective} in general
when one studies functions with values in a semiring $S$.
Given two infinite sets $A$ and $B$, we construct explicitly 
in \cite{banagl-tensor}
a commutative, additively idempotent
semiring $S= S(A,B)$ such that $\mu: \fun_S (A) \otimes \fun_S (B) \to 
\fun_S (A\times B)$ is
not injective. This has the immediate consequence that in 
performing functional analysis over a semiring which is not a field, one
cannot identify the function $(a,b)\mapsto f(a)g(b)$ on $A\times B$ with $f\otimes g$ for 
$f\in \fun_S (A)$, $g\in \fun_S (B)$. Thus the algebraic tensor product is not the correct device
to formulate positive topological field theories.

In the boundedly complete idempotent \emph{commutative} setting, Litvinov, Maslov and Shpiz have
constructed in \cite{litmasshpiztensor} a tensor product, let us here write it as $\hotimes$, which 
for bounded functions does not exhibit the above deficiencies
of the algebraic tensor product. 
This is an idempotent analog of topological tensor products in the sense of Grothendieck.
Now, the present paper needs such an analog even over noncommutative semirings,
since the aforementioned semirings $Q^c, Q^m$ are generally not commutative.
We shall introduce such a completed tensor product $\hotimes$ here, building on the
methods of \cite{litmasshpiztensor}, but using notation which is more in line with the one used
in the rest of our paper.

Let $S$ be any semiring.
Let $M$ be a right $S$-semimodule and $N$ a left $S$-semimodule.
We assume both of these semimodules to be complete, idempotent and continuous.
We shall define a complete tensor product $M \hotimes_S N$.
The \emph{support} of $f\in \fun_S (A)$ is
\[ \supp (f) = \{ a\in A ~|~ f(a)\not= 0 \}. \]
Let $\bool,$ as always, denote the Boolean semiring.
Given $f\in \fun_\bool (M\times N)$, we define a function $f_N:N\to M$ by
\[ f_N (n) = \sum \{ m\in M ~|~ f(m,n)=1 \}, \]
using the completeness of $M$. 
Similarly,  $f_M:M\to N$ is given by
\[ f_M (m) = \sum \{ n\in N ~|~ f(m,n)=1 \}, \]
using the completeness of $N$.
\begin{defn}
A \emph{tensor} is a function $f:M\times N \to \bool$ such that
for all $m\in M,$ $n\in N,$ $s\in S,$
\begin{enumerate}
\item $f(0,0)=1$,
\item $f(ms,n)=f(m,sn)$,
\item If $m\leq f_N (n)$ or $n\leq f_M (m),$ then $f(m,n)=1$.
\end{enumerate}
\end{defn}
If $f$ is a tensor, then 
$f(m,0)=f(m\cdot 0,0)=f(0,0)=1,~ \text{all } m\in M,$
and similarly $f(0,n)=1$ for all $n\in N$.
The function which is identically $1$ is a tensor.
As a set, we define $M\hotimes_S N$ to be the subset of $\fun_\bool (M\times N)$ consisting
of all tensors.
In general, $M\hotimes_S N$ is not closed under the standard addition $+$
on $\fun_\bool (M\times N)$. Thus, we need to give $M\hotimes_S N$ a different addition $\oplus$.
First, we define the product $\prod_{i\in I} f_i \in \fun_\bool (M\times N)$ 
over an arbitrary nonempty family $(f_i)_{i\in I}$ of functions
$f_i \in \fun_\bool (M\times N)$ to be
\[ (\prod_{i\in I} f_i)(m,n) =
 \begin{cases} 1,& \text{if } f_i (m,n)=1 \text{ for all i} \\
  0,& \text{otherwise.}
 \end{cases} \]
This product is clearly independent of the order of the factors.
If every $f_i$ is a tensor, then $\prod_{i\in I} f_i$ is a tensor.
Any given $f\in \fun_\bool (M\times N)$ generates a tensor $f^\otimes$ by
\[ f^\otimes = \prod \{ g\in M\hotimes_S N ~|~ \supp (g)\supset \supp (f) \}. \]
Note that the family over which the product
is taken is indeed nonempty, since it contains the function which is identically $1$.
By definition we have 
$\supp (f^\otimes)\supset \supp (f)$
and if $g$ is a tensor such that $\supp (g)\supset \supp (f)$, then $\supp (g)\supset \supp (f^\otimes)$.
If $f$ is already a tensor, then $f^\otimes =f$. For tensors $f,g$, we define 
\[ f\oplus g := (f+g)^\otimes,~ \bz := 0^\otimes. \]
It is not necessarily true that $\bz (m,n)=0$ for $m,n$ both nonzero.
In fact, it is not possible to give a universal formula for $\bz$; rather, $\bz$ does depend
on the semimodules and ground semiring.
If $f,g\in \fun_\bool (M\times N),$ then
$(f^\otimes + g)^\otimes = (f+g)^\otimes.$
More generally, if $(f_i)_{i\in I}$ is a family of functions $f_i\in \fun_\bool (M\times N),$ 
then
$\left( \sum_{i\in I} f_i^\otimes \right)^\otimes = \left( \sum_{i\in I} f_i \right)^\otimes.$
Using these formulae, one readily shows:
\begin{lemma}
The triple $(M\hotimes_S N,\oplus, \bz)$ is a commutative, idempotent monoid.
The summation law 
\[ \bigoplus_{i\in I} f_i := (\sum_{i\in I} f_i)^\otimes \]
on $M\hotimes_S N$ makes $(M\hotimes_S N, \oplus,\bz)$ into a complete monoid,
where the summation law $\sum$ on the right hand side is induced by the summation law
of $\bool$.
\end{lemma}
The idempotent addition $+$ on $\fun_\bool (M\times N)$ generates a partial order
$\leq$ given by $f\leq g$ if and only if $f+g=g$. Similarly, the idempotent addition
$\oplus$ on $M\hotimes_S N$ generates a partial order $\preceq$ given by
$f\preceq g$ if and only if $f\oplus g=g$.
We shall see that on the tensor product $M\hotimes_S N$ these two partial orders
agree. If
$f,g \in \fun_\bool (M\times N)$ are any functions, then $f\leq g$ if and only if
$\supp (f)\subset \supp (g)$.
In particular, $f\leq f^\otimes$
for any function $f\in \fun_\bool (M \times N)$. Using this one verifies:
\begin{lemma} \label{lem.leqispreceqontensors}
Let $f,g\in M\hotimes_S N$ be tensors. Then $f\leq g$ if and only if $f\preceq g$.
\end{lemma}
This lemma, together with the continuity of $\fun_\bool (M\times N)$, allows for a straightforward proof of continuity
of the tensor product:
\begin{lemma}
The idempotent complete monoid $(M\hotimes_S N, \oplus,\bz)$ is continuous.
\end{lemma}
Given a pair $(m,n)\in M\times N,$ we put
$m\hotimes n := (\chi_{(m,n)})^\otimes \in M\hotimes_S N,$
where $\chi_{(m,n)}:M\times N \to \bool$ is the characteristic function of $(m,n)$.
We call the elements $m\hotimes n$ \emph{elementary tensors}.
Note that $0_M \hotimes n = \bz = m \hotimes 0_N$.
Since
\[ f^\otimes = \bigoplus_{(m,n)\in \supp (f)} m\hotimes n, \]
every tensor $f$ can be written as a (possibly infinite) sum 
of elementary tensors. 
\begin{defn}
A homomorphism $\widehat{\phi}: M\to N$ of complete monoids is called \emph{continuous},
if $\widehat{\phi}(\sum_{i\in I} m_i) = \sum_{i\in I} \widehat{\phi}(m_i)$ for every family
$(m_i)_{i\in I}$ in $M$. Let $P$ be any complete commutative monoid.
A middle $S$-linear map $\phi: M\times N \to P$ is called \emph{bicontinuous}, if
$\phi (m,-)$ is continuous for all $m\in M$ and $\phi (-,n)$ is continuous for all $n\in N$.
\end{defn}
We continue to assume that $M$ and $N$ are complete, idempotent and continuous over $S$.
Elementary tensors satisfy
\[ (m+m')\hotimes n = m\hotimes n \oplus m' \hotimes n,~
  m\hotimes (n+n') = m\hotimes n \oplus m\hotimes n',~
 (ms)\hotimes n = m\hotimes (sn) \]
for all $m,m'\in M,$ $n,n' \in N,$ $s\in S$.
More generally, if $(m_i)_{i\in I}$ is any family of elements in $M$ and
$(n_i)_{i\in I}$ is any family of elements in $N$, then
\[ (\sum_{i\in I} m_i)\hotimes n = \bigoplus_{i\in I} m_i \hotimes n,~
   m\hotimes (\sum_{i\in I} n_i) = \bigoplus_{i\in I} m \hotimes n_i. \]
Hence, the map
$M\times N \to M\hotimes_S N,~ (m,n)\mapsto m\hotimes n,$
is middle $S$-linear and bicontinuous.
\begin{lemma} \label{lem.ttensor}
Let $P$ be any commutative, complete, idempotent, continuous monoid.
If $\widehat{\phi}: M\times N \to P$ is middle $S$-linear and bicontinuous, and $p\in P$ is any element,
then $t:M\times N\to \bool$, given by
\[ t(m,n)= \begin{cases} 1,& \widehat{\phi} (m,n)\leq p \\
  0,& \text{ otherwise,}
\end{cases} \]
is a tensor.
\end{lemma}
This technical result is used in proving:
\begin{lemma} \label{lem.summidslinffotimes}
Let $P$ be any commutative, complete, idempotent, continuous monoid.
If $\widehat{\phi}: M\times N \to P$ is middle $S$-linear and bicontinuous, and 
$f\in \fun_\bool (M\times N)$ is any function, then
\[ \sum_{(m,n)\in \supp (f)} \widehat{\phi}(m,n) =
  \sum_{(m,n)\in \supp (f^\otimes)} \widehat{\phi}(m,n). \]
\end{lemma}
One uses this formula to establish the universal property of $\hotimes_S$:
\begin{prop} \label{prop.universalprop} (Universal Property.)
Let $P$ be any continuous, idempotent, complete, commutative monoid.
For each middle $S$-linear, bicontinuous map $\widehat{\phi}: M\times N \to P$, there
exists a unique continuous monoid homomorphism $\phi: M\hotimes_S N \to P$ such that
\[ \xymatrix@R=15pt{
M\times N \ar[r]^{\widehat{\phi}} \ar[d] & P \\
M\hotimes_S N \ar@{..>}[ru]_{\phi} &
} \]
commutes. 
\end{prop}
\begin{proof}
Given $\widehat{\phi}$, we define $\phi$ on a tensor $f \in M\hotimes_S N$ as
\[ \phi (f) = \sum_{(m,n)\in \supp f} \widehat{\phi}(m,n) \]
and leave the required verifications, as well as uniqueness, to the reader.
\end{proof}

Let $R$ and $T$ be semirings and suppose that, in addition to our standing assumptions,
$M$ is an $R$-$S$-bisemimodule which is complete as a left $R$-semimodule, and
$N$ is an $S$-$T$-bisemimodule which is complete as a right $T$-semimodule.
Then the tensor product $M\hotimes_S N$ can be endowed with the structure of
an $R$-$T$-semimodule in a natural way, as we will show next.
Fix an element $r\in R$ and consider the middle $S$-linear, bicontinuous map
\[ \widehat{\alpha}:M\times N \longrightarrow M\hotimes_S N,~
   \widehat{\alpha}(m,n)=(rm)\hotimes n. \]
By the universal property, Proposition \ref{prop.universalprop}, there exists a unique
continuous homomorphism of additive monoids
$\alpha_r: M\hotimes_S N \rightarrow M\hotimes_S N$
such that $\alpha_r (m\hotimes n)=\widehat{\alpha}(m,n)=(rm)\hotimes n.$
We define a left scalar multiplication by
\[ R\times (M\hotimes_S N) \longrightarrow (M\hotimes_S N),~
  r\cdot f = \alpha_r (f). \]
Similarly, for $t\in T,$ there exists a unique
continuous homomorphism of additive monoids
$\beta_t: M\hotimes_S N \longrightarrow M\hotimes_S N$
such that $\beta_t (m\hotimes n)=\widehat{\beta}(m,n)=m\hotimes (nt).$
We define a right scalar multiplication by
\[ (M\hotimes_S N)\times T \longrightarrow (M\hotimes_S N),~
  f\cdot t = \beta_t (f). \]
The left and right scalar multiplications defined above make $M\hotimes_S N$
into an $R$-$T$-bisemimodule. Furthermore, $M\hotimes_S N$ is complete
as a left $R$-semimodule and complete as a right $T$-semimodule.
\begin{prop} (Bisemimodule Universality.)
Let $R,T$ be semirings and let $P$ be any idempotent $R$-$T$-bisemimodule equipped with
a summation law that makes it into a continuous complete left $R$-semimodule and into a continuous
complete right $T$-semimodule. For each $R_S T$-linear bicontinuous map
$\widehat{\phi}:M\times N\to P,$ there exists a unique $R$-$T$-bisemimodule homomorphism
$\phi: M\hotimes_S N\to P$ such that $\phi (m\hotimes n)=\widehat{\phi}(m,n)$ for all
$(m,n)\in M\times N$.
\end{prop}
The proof is straightforward, relying on the monoidal universal property already established.
We come to the main result concerning the complete tensor product $\hotimes_S$.
Let $A$ and $B$ be sets. Let $S$ be a continuous, complete, idempotent semiring. Then
$\fun_S (A)$ and $\fun_S (B)$ are continuous, complete, idempotent $S$-bisemimodules.
\begin{thm}  \label{thm.funiso}
There is an isomorphism
\[ \fun_S (A)\hotimes_S \fun_S (B) \cong \fun_S (A\times B) \]
of $S$-bisemimodules.
\end{thm}
\begin{proof}
A middle $S$-linear bicontinuous map
$\widehat{\alpha}: \fun_S (A)\times \fun_S (B) \rightarrow \fun_S (A\times B)$
is given by
$\widehat{\alpha} (F,G)(a,b) = F(a)G(b),~ (a,b)\in A\times B.$
By Proposition \ref{prop.universalprop}, $\widehat{\alpha}$ induces a unique continuous homomorphism
of additive monoids
\[ \alpha: \fun_S (A)\hotimes_S \fun_S (B) \longrightarrow \fun_S (A\times B) \]
such that $\alpha (F\hotimes G) = \widehat{\alpha} (F,G)$. In fact, $\alpha$ is an
$S$-bisemimodule homomorphism.

We denote the characteristic function of an element $a\in A$ by $\chi_a \in \fun_S (A)$.
Similarly we have $\chi_b \in \fun_S (B)$.
A continuous $S$-bisemimodule homomorphism
\[ \beta: \fun_S (A\times B) \longrightarrow \fun_S (A) \hotimes \fun_S (B) \]
is given by
\[ \beta (F) =  \bigoplus_{(a,b)\in A\times B} (\chi_a \cdot F(a,b))\hotimes \chi_b. \]
The maps $\alpha$ and $\beta$ are inverse to each other.
\end{proof}
As an application of Theorem \ref{thm.funiso},
we have for example
\[ \fun (A)\hotimes_S S \cong \fun (A) \hotimes_S \fun (\{ * \}) =
  \fun (A\times \{ * \}) \cong \fun (A). \]

For a fixed arbitrary semiring $S$, $\fun_S (-)$ is a contravariant functor $\fun_S: \mathbf{Sets}\to 
S\text{-}S\text{-}\mathbf{SAlgs}$ from the category of sets to the category of two-sided
$S$-semialgebras: A morphism $\phi:A\to B$ of sets induces a
morphism of two-sided $S$-semialgebras $\fun (\phi): \fun (B)\to \fun (A)$ by setting
$\fun (\phi)(f) = f\circ \phi$. 
Clearly, $\fun (\id_A) = \id_{\fun (A)}$ and
$\fun (\psi \circ \phi) = \fun (\phi)\circ \fun (\psi)$ for
$\psi: B\to C$.
The functor $\fun_S (-)$ is faithful. 

Let $S$ be a complete, idempotent and continuous semiring and let $A,B,C$ be sets.
By applying Theorem \ref{thm.funiso} twice, we may think of elements in
$\fun (A)\hotimes_S \fun (B)\hotimes_S \fun (C)$ as functions $A\times B \times C \to S$
on triples; similarly for an arbitrary finite number of factors.
A \emph{contraction}
\[ \gamma: \fun (A)\hotimes \fun (B)\hotimes \fun (B)\hotimes \fun (C)
 \longrightarrow \fun (A)\hotimes \fun (C) \]
can then be defined, using the summation law in $S$, by
\[ \gamma (f)(a,c) = \sum_{b\in B} f(a,b,b,c), \]
$f: A\times B\times B\times C \to S,$ $(a,c)\in A\times C$.
Given $f\in \fun (A)\hotimes \fun (B)$ and $g\in \fun (B)\hotimes \fun (C)$,
we shall also write
$\langle f,g \rangle = \gamma (f\hotimes g).$
This contraction appears in describing the behavior of our state sum
invariant under gluing of cobordisms. The proof of the following two statements is
straightforward.
\begin{prop} \label{prop.contractbilinear}
The contraction
\[ \langle -,- \rangle: (\fun (A)\hotimes \fun (B))\times
 (\fun (B)\hotimes \fun (C)) \longrightarrow
 \fun (A)\hotimes \fun (C) \]
is $S_S S$-linear.
\end{prop}
\begin{prop} \label{prop.contractassoc}
The contraction $\langle -,- \rangle$ is associative, that is, given
four sets $A,B,C,D$ and elements $f\in \fun (A)\hotimes \fun (B),$
$g\in \fun (B)\hotimes \fun (C)$ and $h\in \fun (C)\hotimes \fun (D)$,
the equation
\[ \langle \langle f,g \rangle, h \rangle = \langle f, \langle g,h \rangle \rangle \]
holds in $\fun (A)\hotimes \fun (D)$.
\end{prop}

\section{Profinite Idempotent Completion}
\label{sec.proidemcompletion}

We shall describe how a subset of a vector space which is invariant under rescaling
by powers of a given scalar can be completed in a natural way to an idempotent,
complete $\nat[\tau]$-semimodule. We are mostly interested in applying this construction
to subsets that are finitely generated with respect to rescaling, i.e. ``projectively finite''.
We call this process \emph{profinite idempotent completion}. When the given subset
consists of linear maps, so that composition and tensor product are defined, the profinite idempotent
completion in fact carries two internal multiplications that make it into a (complete, idempotent, continuous) semiring.
One comes from the composition of linear maps (we call this the \emph{composition semiring} $Q^c$
of the profinite idempotent completion $Q$), the other comes from the tensor product of linear
maps (and will be called the \emph{monoidal semiring} $Q^m$ of $Q$.

\begin{defn}
Let $\lambda \in \real - \{ 0,1,-1 \}$ be a scalar. A subset $H\subset V- \{ 0 \}$ of a real vector space $V$
is called \emph{$\lambda$-profinite} (projectively finite) if there exists a finite
subset $H_0 \subset H$ such that every $h\in H$ is of the form $h = \lambda^n h_0$
for suitable $n\in \nat,$ $h_0 \in H_0$. In this context, we say that $H_0$ 
\emph{generates} $H$.
\end{defn}

Given a $\lambda$-profinite set, the set of \emph{reducible elements} of a generating set $H_0$
is defined to be
\[ H^{\operatorname{red}}_0 = \{ h \in H_0 ~|~ h = \lambda^n h_0 \text{ for some }
  n\in \nat,~ n>0,~ h_0 \in H_0 \}. \]
The \emph{minimal shell} $S(H_0)$ of $H_0$ is defined as the set $S(H_0) = H_0 - 
H^{\operatorname{red}}_0$ of irreducible elements.
\begin{lemma} \label{lem.sh0generates}
The minimal shell $S(H_0)$ generates $H$. Moreover, given $h\in H$, there exist
unique $n\in \nat,$ $h_0 \in S(H_0)$ with $h= \lambda^n h_0$.
\end{lemma}
\begin{proof}
We prove first the claim of generation.
Let $c$ denote the cardinality of $H_0$, a finite set.
Given $h\in H,$ we can write $h=\lambda^{n_1} h_1,$ $n_1\in \nat,$ $h_1 \in H_0$.
If $h_1 \in S(H_0),$ we are done. Suppose that $h_1 \not\in S(H_0)$.
Then $h_1 \in H^{\operatorname{red}}_0$, whence there is an $n_2 \in \nat,$
$n_2 >0$, and an $h_2 \in H_0$ such that $h_1 = \lambda^{n_2} h_2$.
If $h_2 \in S(H_0),$ we are done, for then $h=\lambda^{n_1 + n_2} h_2.$
The vectors $h_1$ and $h_2$ cannot be equal, for if they were, then
$(1-\lambda^{n_2})h_2 =0$, which is impossible, as $h_2 \not=0$, $n_2 >0$ and
in the field of real numbers the only roots of unity are $\pm 1$.
So the set $\{ h_1, h_2 \}$ has cardinality $2$.
Suppose that $h_2 \not\in S(H_0)$, that is, $h_2 \in H^{\operatorname{red}}_0$.
Reducing $h_2,$ we obtain $h_3 \in H_0$ with $h_2 =\lambda^{n_3} h_3,$
$n_3 >0$. Either $h_3 \in S(H_0)$, in which case we are done, or $h_3$ is reducible
and $\{ h_1, h_2, h_3 \}$ has cardinality $3$. Continuing this reduction process,
we must arrive after at most $c$ reductions at an element $h_i \in S(H_0)$,
for otherwise we would have obtained a subset 
$\{ h_1, \ldots, h_{c+1} \} \subset H_0$ of cardinality $c+1$, which is impossible.
Since
\[ h= \lambda^{\sum_{j=1}^i n_j} h_i,~ h_i \in S(H_0), \]
the claim of generation is established. 

If $h= \lambda^n h_0 = \lambda^m h'_0,$ with $h_0, h'_0 \in S(H_0),$
then $m\geq n$ or $n\geq m$, say $m\geq n$. Writing $m=n+k,$
$k\in \nat,$ we have $h_0 = \lambda^k h'_0$. If $k$ were positive, then
$h_0 \in H^{\operatorname{red}}_0$, a contradiction. Thus $k=0$ and
consequently $m=n$. The equality $h_0 = h'_0$ is implied by
$\lambda^n h_0 = \lambda^m h'_0 = \lambda^n h'_0$ (and $\lambda \not= 0$).
\end{proof}
\begin{remark}
The above lemma becomes false if we allowed $\lambda =1$. We could then
take any finite set $H_0$ of nonzero vectors and $H=H_0$, which is $1$-profinite.
But $H^{\operatorname{red}}_0 =H_0$ and hence the minimal shell
$S(H_0)$ is empty and does not generate $H$.
\end{remark}
\begin{lemma}  \label{lem.uniqueminshell}
Any two generating sets $H_0, H'_0$ of $H$ have the same minimal shell
$S(H_0)=S(H'_0)$.
\end{lemma}
\begin{proof}
Let $h$ be an element of $S(H'_0)$. Since $h\in H$ and $S(H_0)$
generates $H$ by Lemma \ref{lem.sh0generates}, there are
$h_0 \in S(H_0)$ and $n\in \nat$ with $h=\lambda^n h_0$.
Since $h_0 \in H$ and $S(H'_0)$ also generates $H$, there are
$h'_0 \in S(H'_0)$ and $m\in \nat$ with $h_0 = \lambda^m h'_0$.
Therefore, $h = \lambda^{n+m} h'_0$. If $n+m$ were positive,
then $h \in H^{'\operatorname{red}}_0$, a contradiction to
$h\in S(H'_0)$. Thus $m+n=0$, so in particular $n=0$. Hence
$h=h_0$, from which we conclude that $h\in S(H_0)$. We have shown
that $S(H'_0)\subset S(H_0)$. The reverse inclusion follows from symmetry.
\end{proof}
In light of Lemma \ref{lem.uniqueminshell}, we may define the
\emph{minimal shell $S(H)$ of $H$} to be the minimal shell 
$S(H)=S(H_0)$ of any (finite) generating set $H_0$.

\begin{defn}
A subset $H\subset V$ of a real vector space $V$ is called \emph{$\lambda$-scalable}
if for every $h\in H,$ the product $\lambda h$ is again in $H$.
\end{defn}
A $\lambda$-scalable set $H$ is naturally equipped with an action
$\tau^i h = \lambda^i h\in H$ of the monoid $\nat$. Hence,
by Lemma \ref{lem.antausemimod}, $\fm (H)$ is an $\nat [\tau]$-semimodule.
Thus the $\nat [\tau]$-semimodule
\[ Q(H) = \fm (H) \otimes_{\nat [\tau]} \bool[[q]] \]
is defined. It was already pointed out in Section \ref{sec.monssemirings}
that $Q(H)$ is idempotent. We omit the easy proof of the following lemma.

\begin{lemma}  \label{lem.fmhnormalform}
Let $H \subset V$ be a $\lambda$-scalable $\lambda$-profinite set and
$S(H) = \{ s_1, \ldots, s_c \}$ its minimal shell.
Given $a\in \fm (H),$ there exist unique polynomials
$p_\sigma = p_\sigma (\tau)\in \nat [\tau],$ $\sigma =1,\ldots, c,$
such that $a = \sum_{\sigma =1}^c p_\sigma s_\sigma.$
\end{lemma}

Let $H \subset V$ be a $\lambda$-scalable $\lambda$-profinite set and
$S(H) = \{ s_1, \ldots, s_c \}$ its minimal shell. We shall use 
Lemma \ref{lem.fmhnormalform} to construct an $\nat [\tau]$-linear isomorphism
\[ \phi: Q(H) \longrightarrow \bigoplus_{\sigma =1}^c \bool[[q]]. \]
This will show in particular that $Q(H)$ is nonzero (for nonempty $H$)
and that a number of power series valued invariants can be extracted
from it. Let $\{ \beta_1, \ldots, \beta_c \}$ be the canonical basis of
$\bigoplus_{\sigma =1}^c \bool[[q]]$ regarded as a free $\bool[[q]]$-semimodule.
Given any $a\in \fm (H)$ and $b\in \bool[[q]]$, there are unique
polynomials $p_\sigma \in \nat [\tau]$ such that 
$a = \sum_\sigma p_\sigma s_\sigma$ by Lemma \ref{lem.fmhnormalform},
and we can define a $\nat [\tau]$-bilinear map
\[ \phi_0: \fm (H) \times \bool[[q]] \longrightarrow  \bigoplus_{\sigma =1}^c \bool[[q]],~
   \phi_0 (a,b) = \sum_\sigma (p_\sigma b)\beta_\sigma. \]
By the universal property of the algebraic tensor product of semimodules, there exists a
unique $\nat [\tau]$-linear map
\[ \phi:  Q(H) = \fm (H) \otimes_{\nat [\tau]} \bool[[q]] \longrightarrow  
   \bigoplus_{\sigma =1}^c \bool[[q]] \]
such that
\[ \xymatrix{
\fm (H) \times \bool[[q]] \ar[r]^{\phi_0} \ar[d] & \bigoplus_{\sigma =1}^c \bool[[q]] \\
Q(H) \ar@{..>}[ru]_{\phi}
} \]
commutes. The map $\phi$ is surjective: 
Any element $\sum_\sigma b_\sigma \beta_\sigma,$ $b_\sigma \in \bool[[q]],$ is in the image
of $\phi$, as
\[ \phi (\sum_\sigma s_\sigma \otimes b_\sigma) =
 \sum_\sigma \phi (s_\sigma \otimes b_\sigma) = 
 \sum_\sigma \phi_0 (s_\sigma, b_\sigma) = \sum_\sigma b_\sigma \beta_\sigma. \]
\begin{lemma} \label{lem.phiiso}
The map
\[ 
\psi: \bigoplus_\sigma \bool[[q]] \longrightarrow Q(H),~
\sum_\sigma b_\sigma \beta_\sigma \mapsto \sum_\sigma s_\sigma \otimes b_\sigma,
\]
is $\nat [\tau]$-linear and an inverse for $\phi$. In particular, $\phi$ and $\psi$ are
isomorphisms of $\nat [\tau]$-semimodules.
\end{lemma}
The verification is straightforward.
\begin{cor} \label{cor.qhmncompleteness}
The semimodule $Q(H)$ is complete with respect to the infinite
summation law induced by the summation law in $\bool[[q]]$.
\end{cor}
\begin{proof}
In Section \ref{sec.monssemirings}, it was pointed out that $\bool[[q]]$
is complete as a $\nat [\tau]$-semimodule. If $\{ M_i \}_{i\in I}$ is any family
of complete $S$-semimodules over a semiring $S$, then their product
$\prod_{i\in I} M_i$ is a complete $S$-semimodule. Hence
$\bigoplus_{\sigma =1}^c \bool [[q]]$ is a complete $\nat [\tau]$-semimodule. Via the isomorphism
$\psi$ of Lemma \ref{lem.phiiso}, the summation law of $\bigoplus_{\sigma}
\bool [[q]]$ is transferred to a summation law in $Q(H),$ with respect
to which $Q(H)$ is then complete as a $\nat [\tau]$-semimodule. Explicitly, if $\{ f_i \}_{i\in I}$ is any
family of elements in $Q(H),$ $f_i = \sum_{\sigma} s_\sigma \otimes
b_{\sigma i},$ $b_{\sigma i} \in \bool [[q]],$ then the summation law
is given by the formula
\begin{equation} \label{equ.sumlawonqh}
\sum_{i\in I} f_i = \sum_{\sigma =1}^c s_\sigma \otimes
 \big( \sum_{i\in I} b_{\sigma i} \big). 
\end{equation}
\end{proof}

Let $V$ be a real vector space of finite dimension $\dim V\geq 2$ and $(i,e)$ a duality
structure on $V$. Let $Y: \Br \to \vect$ be the symmetric monoidal functor
given by Theorem \ref{thm.functorytangvect}, determined by $V$ and the
duality structure (and the braiding $b:V^{\otimes 2} \to V^{\otimes 2}$).
Recall that $\lambda: [0]\to [0]$
denotes the loop endomorphism of $\Br$. Applying $Y$ to the loop, we
obtain a scalar
\[ \lh = Y(\lambda)\in \Hom_\vect (Y[0], Y[0]) = \Hom_\vect (\real, \real)=\real. \]
In view of
\[ Y(\lambda) = Y(e_1 i_1) = Y(e_1)Y(i_1)=ei= \operatorname{Tr}(i,e), \]
this scalar $\lh$ is the trace of the duality structure $(i,e)$.
Therefore, by Proposition \ref{prop.traceisdimension},
$\lh =\dim V\geq 2$. In particular, $\lh$ is an admissible scalar in $\real - \{ 0,1,-1 \}$.
Given objects $[m],[n]$ of $\Br$, we define a subset $H_{m,n}$ of the
vector space $\Hom_\vect (V^{\otimes m}, V^{\otimes n})$ by
\[ H_{m,n} = Y(\Hom_\Br ([m], [n])). \]
This set is nonempty only if $m+n$ is even.
By Lemma \ref{lem.yhomtangnozeros}, $H_{m,n}$ does indeed not
contain $0 \in \Hom_\vect (V^{\otimes m}, V^{\otimes n}).$
Let $\OP_{m,n}$ (the ``open'' morphisms) denote the set of all loop-free morphisms
$[m]\to [n]$ in $\Br$.
\begin{lemma}
For every $m,n\in \nat,$ $m+n$ even, the set $H_{m,n}$ is $\lh$-scalable and
$\lh$-profinite.
\end{lemma}
\begin{proof}
Given $h\in H_{m,n}$, the product $\lh h$ is again in $H_{m,n}$, since
there is a morphism $\phi: [m]\to [n]$ with $h=Y(\phi)$ and we have
\[ \lh h = Y(\lambda)\otimes Y(\phi) = Y(\lambda \otimes \phi). \]
This shows that $H_{m,n}$ is $\lh$-scalable. Let $H_0 \subset H_{m,n}$
be the set $H_0 = Y(\OP_{m,n})$. Any $h\in H_{m,n}$ can be written
as $h=Y(\phi)$ with $\phi = \lambda^{\otimes p} \otimes \phi'$ and
$\phi'$ loop-free. Then
\[ h = Y(\lambda^{\otimes p} \otimes \phi') =
 \lh^p Y(\phi') = \lh^p h_0, \]
where $h_0 = Y(\phi')$ is in $H_0$. Thus $H_0$ generates $H_{m,n}$.

It remains to be shown that $H_0$ is finite. This follows from the fact that
$\OP_{m,n}$ is a finite set of cardinality $(m+n-1)!!$, since there are
$m+n-1$ choices of joining a first point of $(0\times M[m] \times 0 \times 0)
\sqcup (1\times M[n]\times 0 \times 0)$ to any other point in that $0$-manifold,
then $m+n-3$ choices of joining the next unconnected point, etc.
\end{proof}
\noindent Thus the idempotent $\nat [\tau]$-semimodule
\[ Q(H_{m,n}) = \fm (H_{m,n}) \otimes_{\nat [\tau]} \bool[[q]] \]
is defined. The fact that the Boolean semiring, whose only nonzero value is $1$, appears here is
a reflection of the fact that the modulus of the integrand $e^{iS}$
appearing in the classical Feynman path integral is always
$|e^{iS}|=1$ and only the phase is relevant. Roughly, the terms in
$\fm (H_{m,n})$ play the role of the phase.
\begin{lemma} \label{lem.minshellshmn}
If the symmetric monoidal functor $Y$ induced by the duality structure $(i,e)$ 
is loop-faithful, then
the minimal shell $S(H_{m,n})$ consists of the $Y$-images of all
loop-free morphisms $[m]\to [n]$ in $\Br$, that is,
\[ S(H_{m,n})=Y(\OP_{m,n}). \]
\end{lemma}
\begin{proof}
Since $H_0 = Y(\OP_{m,n})$ was already shown to be a generating set,
it suffices to show that its subset $Y(\OP_{m,n})^{\operatorname{red}}$
of reducible elements is empty. Assume on the contrary that there
were a reducible element $y$, that is, there existed elements $y, y_0\in Y(\OP_{m,n})$ 
and an integer $k>0$ such that $y = \lh^k y_0$. Then $y=Y(\phi),$
$y_0 = Y(\phi_0)$ for loop-free morphisms $\phi, \phi_0$ in $\Br$ and thus
\[ Y(\phi) = \lh^k Y(\phi_0) = Y(\lambda)^k Y(\phi_0)
 = Y (\lambda^{\otimes k} \otimes \phi_0). \]
But $\phi$ is loop-free, while $\lambda^{\otimes k} \otimes \phi_0$ contains
at least one loop. This contradicts the assumption that $Y$ be  
faithful on loops.
\end{proof}
From now on, assume that $(i,e)$ is such that $Y$ is faithful on loops.
By Proposition \ref{prop.expleloopfaithful}, this is for instance the case
for the duality structure (\ref{equ.exie}).
It follows from Lemma \ref{lem.minshellshmn} and Lemma \ref{lem.phiiso} that an element
$f\in Q(H_{m,n})$ can be uniquely written as
\[ f = \sum_{y\in Y(\OP_{m,n})} y\otimes b_y,~ \hspace{.5cm} b_y\in \bool[[q]]. \]

The set $\OP_{m,m}$ contains the identity $1_{[m]}:[m]\to [m].$
Thus $1_{V^{\otimes m}} = Y(1_{[m]})$ is an element of $S(H_{m,m})$.
Tensoring with $1\in \bool[[q]],$ there is an element
$1_{V^{\otimes m}} \otimes 1 \in Q(H_{m,m})$.
For $m=0$, we have $1_I \otimes 1 \in Q(H_{0,0}),$ where
$I = V^{\otimes 0}$ is the unit object of $\vect$.
We shall next construct products $\cdot$, coming from the composition
$\circ$ of linear maps. 
Given $p\in \nat$, a $\nat [\tau]$-bilinear map
\begin{equation} \label{equ.prodqhmpqhpn}
Q(H_{m,p}) \times Q(H_{p,n}) \longrightarrow Q(H_{m,n}),~
(f,f') \mapsto f\cdot f',
\end{equation}
is given on
\[ f = \sum_{y\in Y(\OP_{m,p})} y\otimes b_y,~ 
  f' = \sum_{y'\in Y(\OP_{p,n})} y'\otimes b'_{y'}, \]
by
\begin{equation} \label{equ.prodmppn}
f\cdot f' = \sum_{y,y'} (y'\circ y)\otimes (b_y b'_{y'}). 
\end{equation}
Note that $y' \circ y$ is an element of $H_{m,n}$, since
$y=Y(h)$ for some $h\in \OP_{m,p}$ and $y' =Y(h')$ for some
$h'\in \OP_{p,n}$ and thus
$y'\circ y = Y(h')\circ Y(h) = Y(h'\circ h)$ with
$h'\circ h$ in $\Hom_\Br ([m],[n])$.
However, $y'\circ y$ will in general not lie in the minimal shell of $H_{m,n}$,
as $h'\circ h$ may have loops. For this reason, the formula
(\ref{equ.prodmppn}) cannot be used to compute triple products,
which is required in establishing associativity, for example.
The remedy is formula (\ref{equ.prodmppnrelaxed}) below.
\begin{lemma} \label{lem.prodmppnrelaxed}
The product of two elements $\sum_{i=1}^k h_i \otimes b_i \in
Q(H_{m,p})$ and $\sum_{j=1}^l h'_j \otimes b'_j \in
Q(H_{p,n})$ with $h_i \in H_{m,p},$ $h'_j \in H_{p,n}$
(not necessarily in the minimal shells)
is given by 
\begin{equation} \label{equ.prodmppnrelaxed}
\sum_{i,j} (h'_j \circ h_i)\otimes (b_i b'_j).
\end{equation}
\end{lemma}
Using Lemma \ref{lem.prodmppnrelaxed}, one proves readily:  
Given elements $f\in Q(H_{m,p}),$ $f' \in Q(H_{p,r}),$
$f''\in Q(H_{r,n}),$ we have the associativity relation
\begin{equation} \label{equ.mpprrnassoc}
(f\cdot f')\cdot f'' = f\cdot (f'\cdot f''). 
\end{equation}
By Corollary \ref{cor.qhmncompleteness}, the $Q(H_{m,n})$ are complete
$\nat [\tau]$-semimodules. With respect to the product
(\ref{equ.prodqhmpqhpn}), we have the distributive law
\begin{equation} \label{equ.hmppninfdistrright}
(\sum_{i\in I} f_i)\cdot f' = \sum_i (f_i \cdot f'),~
  f_i \in Q(H_{m,p}),~ f' \in Q(H_{p,n}), 
\end{equation}
as applications of formula (\ref{equ.sumlawonqh}) together with standard
properties of infinite summation laws show.
Similarly, distributivity from the left holds,
\begin{equation} \label{equ.hmppninfdistrleft}
f \cdot \sum_{i\in I} f'_i = \sum_i (f \cdot f'_i),~
  f \in Q(H_{m,p}),~ f'_i \in Q(H_{p,n}).
\end{equation}
For the product $Q(H_{m,n})\times Q(H_{n,n})\to Q(H_{m,n}),$
we have
\[ f\cdot (1_{V^{\otimes n}} \otimes 1) =
  \sum_{y\in Y(\OP_{m,n})} (1_{V^{\otimes n}} \circ y)\otimes (b_y \cdot 1) =
  \sum_{y\in Y(\OP_{m,n})} y\otimes b_y = f, \]
so that $1_{V^{\otimes n}} \otimes 1$ acts as a $1$-element with respect to
this product. Similarly, $1_{V^{\otimes m}} \otimes 1$ acts
as a $1$-element for the product
$Q(H_{m,m})\times Q(H_{m,n})\to Q(H_{m,n}),$
$(1_{V^{\otimes m}}\otimes 1)\cdot f' = f'.$

We set
\[ Q = Q(i,e) = \prod_{m,n\in \nat} Q(H_{m,n}) \]
in the category of $\nat [\tau]$-semimodules and call $Q$ the
\emph{profinite idempotent completion} of the set $Y(\Mor (\Br))$.
Using the natural embeddings $Q(H_{m,n}) \hookrightarrow Q$
given by sending $f$ to the family in $Q$ whose $(m,n)$-component is $f$
and all other components are $0$, we can and will think of  
$f\in Q(H_{m,n})$ as an element $f\in Q$. 
Since every $Q(H_{m,n})$ is additively idempotent and addition in $Q$
is done componentwise, $Q$ itself is additively idempotent.
We define the product of two elements $(f_{m,n}),(f'_{m,n})\in Q$
to be $(f_{m,n})\cdot (f'_{m,n}) = (f''_{m,n}),$ where the component
$f''_{m,n} \in Q(H_{m,n})$ is given by
\[ f''_{m,n} = \sum_{p\in \nat} f_{m,p} \cdot f'_{p,n}, \]
using the completeness of $Q(H_{m,n})$ (Corollary \ref{cor.qhmncompleteness}),
as well as the products declared above.
An element $0\in Q$ is given by the zero-family $0=(f_{m,n}),$ 
$f_{m,n}=0\in Q(H_{m,n})$ for all $m,n$. An element $1\in Q$ is given
by the family $1=(f_{m,n})$ with
\[ f_{m,n} = \begin{cases} 1_{V^{\otimes m}} \otimes 1,& 
  \text{ if } m=n,\\ 0,& \text{ if } m\not= n. \end{cases}
\]
Using (\ref{equ.mpprrnassoc}), (\ref{equ.hmppninfdistrright}) and (\ref{equ.hmppninfdistrleft}), one proves:
\begin{prop} \label{prop.qccomplidemp}
The tuple $Q^c =(Q,+,\cdot, 0,1)$ is a (generally noncommutative)
complete idempotent semiring, called the
\emph{composition semiring} of the
profinite idempotent completion of $Y(\Mor (\Br))$.
\end{prop}

We shall now construct a different product $\times$ on $Q$, coming from the
monoidal structure on $\vect$, i.e. the (Schauenburg) tensor product $\otimes$ of
linear maps.
Given $m,n,r,s\in \nat$, a $\nat [\tau]$-bilinear map
\begin{equation} \label{equ.prodqhmpqhpnmon}
Q(H_{m,n}) \times Q(H_{r,s}) \longrightarrow Q(H_{m+r,n+s}),~
(f,f') \mapsto f\times f',
\end{equation}
is given on
\[ f = \sum_{y\in Y(\OP_{m,n})} y\otimes b_y,~ 
  f' = \sum_{y'\in Y(\OP_{r,s})} y'\otimes b'_{y'}, \]
by
\begin{equation} \label{equ.prodmppnmon}
f\times f' = \sum_{y,y'} (y\otimes y')\otimes (b_y b'_{y'}),
\end{equation}
where $y\otimes y'$ is the Schauenburg tensor product of the linear
maps $y,y'$.
Note that $y \otimes y'$ is an element of $H_{m+r,n+s}$, since
$y=Y(h)$ for some $h\in \OP_{m,n}$ and $y' =Y(h')$ for some
$h'\in \OP_{r,s}$ and thus
$y\otimes y' = Y(h)\otimes Y(h') = Y(h\otimes h')$ with
$h\otimes h'$ in $\Hom_\Br ([m+r],[n+s])$.
Since $h$ and $h'$ have no loops, the product
$h\otimes h'$ has no loops. Thus $h\otimes h' \in \OP_{m+r,n+s}$
and $y\otimes y' \in S(H_{m+r,n+s})$.
For practical calculation, we record:
\begin{lemma} \label{lem.prodmppnrelaxedmon}
The $\times$-product of two elements $\sum_{i=1}^k h_i \otimes b_i \in
Q(H_{m,n})$ and $\sum_{j=1}^l h'_j \otimes b'_j \in
Q(H_{r,s})$ with $h_i \in H_{m,n},$ $h'_j \in H_{r,s}$
(not necessarily in the minimal shells)
is given by 
\begin{equation} \label{equ.prodmppnrelaxedmon}
\sum_{i,j} (h_i \otimes h'_j)\otimes (b_i b'_j)
\end{equation}
\end{lemma}
Given elements $f\in Q(H_{m,n}),$ $f' \in Q(H_{r,s}),$
$f''\in Q(H_{p,q}),$ we have the associativity relation
\begin{equation} \label{equ.mpprrnassocmon}
(f\times f')\times f'' = f\times (f'\times f''), 
\end{equation}
as can be seen from
\begin{eqnarray*}
(f\times f')\times f''
& = & (\sum_{y,y'} (y\otimes y')\otimes (b_y b'_{y'}))\times
  (\sum_{y''} y'' \otimes b''_{y''}) \\
& = & \sum_{y,y',y''} ((y\otimes y')\otimes y'')\otimes ((b_y b'_{y'})b''_{y''}) \\
& = & \sum_{y,y',y''} (y\otimes (y' \otimes y''))\otimes (b_y (b'_{y'} b''_{y''})) \\
& = & (\sum_y y\otimes b_y)\times (\sum_{y',y''} (y'\otimes y'')\otimes (b'_{y'}
  b''_{y''})) \\
& = & f\times (f'\times f'').
\end{eqnarray*}
It should be pointed out that the above calculation rests crucially on the
strict associativity $(y\otimes y')\otimes y''= y\otimes (y' \otimes y'')$ for linear
maps $y,y',y''$ in the strict monoidal category $(\vect, \otimes, I)$, with
$\otimes$ the Schauenburg tensor product. If we had used the ordinary
tensor product on $\vect$, the formula $(f\times f')\times f'' =
f\times (f' \times f'')$ would not hold. With respect to the product
(\ref{equ.prodqhmpqhpnmon}), we have the distributive law
\begin{equation} \label{equ.hmppninfdistrrightmon}
(\sum_{i\in I} f_i)\times f' = \sum_i (f_i \times f'),~
  f_i \in Q(H_{m,n}),~ f' \in Q(H_{r,s}), 
\end{equation}
as applications of formula (\ref{equ.sumlawonqh}) together with standard
properties of infinite summation laws show.
Similarly, distributivity from the left holds,
\begin{equation} \label{equ.hmppninfdistrleftmon}
f \times \sum_{i\in I} f'_i = \sum_i (f \times f'_i),~
  f \in Q(H_{m,n}),~ f'_i \in Q(H_{r,s}).
\end{equation}
For the product $Q(H_{m,n})\times Q(H_{0,0})\to Q(H_{m,n}),$
we have
\[ f\times (1_{I} \otimes 1) =
  \sum_{y\in Y(\OP_{m,n})} (y \otimes 1_{I})\otimes (b_y \cdot 1) =
  \sum_{y\in Y(\OP_{m,n})} y\otimes b_y = f, \]
so that $1_{I} \otimes 1$ acts as a $1$-element with respect to
this product. (Recall that $I$ denotes here the unit object of $\vect$.)
Similarly, $1_{I} \otimes 1$ acts
as a $1$-element for the product
$Q(H_{0,0})\times Q(H_{m,n})\to Q(H_{m,n}),$
$(1_{I}\otimes 1)\times f' = f'.$

We define the cross-product of two elements $(f_{m,n}),(f'_{m,n})\in Q$
to be $(f_{m,n})\times (f'_{m,n}) = (f''_{m,n}),$ where the component
$f''_{m,n} \in Q(H_{m,n})$ is given by
\[ f''_{m,n} = \sum_{\substack{p+r=m\\ q+s=n}} f_{p,q} \times f'_{r,s}, \]
using the products $\times$ declared above.
Since $p,q,r,s$ are nonnegative integers, this is a finite sum.
An element $1^\times \in Q$ is given by 
\[ 1^\times_{m,n} = \begin{cases} 1_{I} \otimes 1,& 
  \text{ if } m=n=0,\\ 0,& \text{ otherwise.} \end{cases}
\]
Using (\ref{equ.mpprrnassocmon}), (\ref{equ.hmppninfdistrrightmon}) and (\ref{equ.hmppninfdistrleftmon}), one verifies:
\begin{prop} \label{prop.qmcomplidemp}
The tuple $Q^m =(Q,+,\times, 0,1^\times)$ is a (generally noncommutative)
complete idempotent semiring, called the
\emph{monoidal semiring} of the profinite idempotent completion of $Y(\Mor (\Br))$.
\end{prop}

\begin{prop} \label{prop.qcontinuous}
The complete idempotent semirings $Q^c$ and $Q^m$ 
are both continuous.
\end{prop}
\begin{proof}
The argument to be given involves only the additive monoid of $Q$ and the
infinite summation law on it and thus establishes simultaneously that both semirings
are continuous.
Let $\leq$ denote that natural unique partial order relation on the idempotent
semiring $Q$: $f\leq g$ iff $f+g=g$.
Let $(f^i)_{i\in I}$ be a family of elements in $Q$ and $c\in Q$ an upper bound such that
$\sum_{i\in F} f^i  \leq c$
for all finite $F\subset I$. Thus for all $m,n$,
\begin{equation} \label{equ.fiupperbound}
c_{m,n} + \sum_{i\in F} f^i_{m,n} = c_{m,n} \in Q(H_{m,n}). 
\end{equation}
For every $c_{m,n}$ there exist unique power series
$c^y_{m,n}\in \bool [[q]]$, $y\in Y(\OP_{m,n}),$ such that
\[ c_{m,n} = \sum_{y\in Y(\OP_{m,n})} y\otimes c^y_{m,n} \]
and for every $f^i_{m,n}$ there exist unique power series
$f^{i,y}_{m,n}\in \bool [[q]]$ such that
\[ f^i_{m,n} = \sum_{y\in Y(\OP_{m,n})} y\otimes f^{i,y}_{m,n}. \]
By equation (\ref{equ.fiupperbound}),
$\sum_y y\otimes (c^y_{m,n} + \sum_{i\in F} f^{i,y}_{m,n}) =
 \sum_y y\otimes c^y_{m,n}.$
By uniqueness,
\[ c^y_{m,n} + \sum_{i\in F} f^{i,y}_{m,n} = c^y_{m,n} \]
in $\bool [[q]]$ for every finite $F\subset I$, that is,
$\sum_{i\in F} f^{i,y}_{m,n} \leq c^y_{m,n}$ in $\bool [[q]]$.
Since $\bool [[q]]$ is continuous,
$\sum_{i\in I} f^{i,y}_{m,n} \leq c^y_{m,n}$, i.e.
$c^y_{m,n} + \sum_{i\in I} f^{i,y}_{m,n} = c^y_{m,n}$
for all $y,m,n$. Thus
\[ c_{m,n} + \sum_{i\in I} f^i_{m,n} =
\sum_y y\otimes (c^y_{m,n} + \sum_{i\in I} f^{i,y}_{m,n}) =
 \sum_y y\otimes c^y_{m,n} = c_{m,n} \]
and hence $c+\sum_{i\in I} f^i = c$, i.e.
$\sum_{i\in I} f^i \leq c$ in $Q$.
\end{proof}

\section{Quantization}
\label{quantization}

We shall define our topological field theory $Z$ in this section. 
We will specify the state-module $Z(M)$ for a closed
smooth $(n-1)$-manifold $M$ (see (\ref{equ.defstatemodule})) as well as an element
$Z_W \in Z(\partial W)$, the \emph{Zustandssumme}, for a compact smooth $n$-manifold $W$ with
boundary $\partial W$ (see (\ref{equ.statesumdef})).
Neither $M$ nor $W$ have to be oriented; thus $Z$ will be a 
``nonunitary'' theory. \\

\subsection{Embeddings and Cobordisms}
\label{sec.embandcob}

Fix an integer $D\geq 2n+1$. A closed $(n-1)$-dimensional manifold can be embedded
in $\real^{D-1}$ and then, after having made a choice of $k\in \nat$, into a slice 
$\{ k \} \times \real^{D-1} \subset \real \times \real^{D-1} = \real^D$. In the present
paper, we assume that closed $(n-1)$-manifolds $M$ are always embedded in $\real^D$
in such a way that every connected component $M_0$ of $M$ lies entirely in some slice
$\{ k_0 \} \times \real^{D-1},$ $k_0 \in \nat$. Given such an embedding, we let
$M(k) = M\cap \{ k \} \times \real^{D-1}$ be the part of $M$ that lies in the $k$-slice.
Every $M(k)$ is a finite disjoint union of connected components of $M$. By compactness,
$M(k)$ is empty for $k$ large enough. Let $s(M)\subset \nat$ denote the set of all $k$
such that $M(k)$ is nonempty.

\begin{defn}
Let $M,N \subset \real^D$ be two disjoint submanifolds. Then their disjoint union
$M\sqcup N$ is called \emph{well-separated}, if $s(M)\cap s(N)=\varnothing$.
In this case, we shall also write $M\sqcup_s N$.
\end{defn}
Thus, if $M$ and $N$ are well-separated, then there is no component of $M$ and no
component of $N$ such that both these components are contained in the same slice. \\

A compact smooth $n$-dimensional manifold $W$ with boundary
can be smoothly embedded into a closed halfspace of $\real^{2n+1}$ in such
a way that the boundary of $W$ lies in the bounding hyperplane and the
interior lies in the interior of the halfspace, see e.g. 
\cite[Theorem 1.4.3]{hirsch}. This fact motivates the following definition.
Recall that an integer $D$ has been fixed with $D\geq 2n+1$.
\begin{defn}
Let $M,N \subset \real^{D}$ be closed, smoothly embedded, $(n-1)$-dimensional 
manifolds, not necessarily
orientable. A \emph{cobordism} from $M$ to $N$ is a compact, smoothly
embedded $n$-dimensional manifold $W\subset [0,1] \times \real^{D}$
with boundary $\partial W = M\sqcup N,$ such that
\begin{itemize}
\item $M\subset \{ 0 \} \times \real^{D}, N\subset \{ 1 \} \times \real^{D},$
\item $W-\partial W \subset (0,1)\times \real^{D},$ 
\item near the boundary of $[0,1]\times \real^{D}$, the embedding is 
 the product embedding, that is, there exists $0< \epsilon <\frac{1}{2}$ such that
 $W\cap [0,\epsilon]\times \real^{D} = [0,\epsilon] \times M$ and
 $W\cap [1-\epsilon,1] \times \real^{D} = [1-\epsilon,1]\times N$, and 
\item every connected component $W_0$ of $W$ lies entirely in some slice
$[0,1] \times \{ k_0 \} \times \real^{D-1},$ $k_0 \in \nat$.
\end{itemize}
\end{defn}
We will refer to any $\epsilon$  with the above properties as a \emph{cylinder scale} of $W$.
We shall also refer to $M$, i.e. the part of $\partial W$ that is contained
in the hyperplane $0\times \real^{D}$, as the \emph{incoming} boundary
and to $N$,  i.e. the part of $\partial W$ that is contained
in the hyperplane $1\times \real^{D}$, as the \emph{outgoing} boundary.
Let
$W(k) = W\cap [0,1] \times \{ k \} \times \real^{D-1}$ be the part of $W$ that lies in the $k$-slice.
As for closed manifolds, every $W(k)$ is a finite disjoint union of connected components of $W$ and
$W(k)$ is empty for $k$ large enough. Let $s(W)\subset \nat$ denote the set of all $k$
such that $W(k)$ is nonempty. The incoming boundary of $W(k)$ is $M(k)$ and the outgoing
boundary of $W(k)$ is $N(k)$.
\begin{defn}
Let $W,W' \subset [0,1]\times \real^D$ be two disjoint cobordisms. Then their disjoint union
$W\sqcup W'$ is called \emph{well-separated}, if $s(W)\cap s(W')=\varnothing$.
In this case, we shall also write $W\sqcup_s W'$.
\end{defn}
Note that if $W$ and $W'$ are well-separated, we are generally not able to deduce
that $M,N\subset \real^D$ (forgetting the first coordinate of $[0,1] \times \real^D$)
are well-separated, or even disjoint. However, the incoming
boundary of $W$ and the incoming boundary of $W'$ are well-separated, and similarly
for the outgoing boundaries.
The embedding also enables us to chop $W$ into the slices
$W_t = W \cap (\{ t \} \times \real^{D}).$
The first coordinate of $\real^{D+1}$ defines a smooth function
$\omega: W\to [0,1],$ i.e. $\omega$ is the composition
\[ W \hookrightarrow [0,1]\times \real^{D}  
 \stackrel{\operatorname{proj}_1}{\longrightarrow} [0,1]. \]
We think of the first coordinate, $t$, of $[0,1] \times \real^D$ as time.
Thus cobordisms $W$ come equipped with time functions $\omega$.
The time slice $W_t$ can alternatively be described as
the preimage $W_t = \omega^{-1}(t).$ The formula
$W_t (k) = W(k)_t$
holds.

Let $\reg (W)$ be the set of regular values of $\omega$.
A subset $R\subset X$ of a topological space $X$ is called
\emph{residual}, if it contains the intersection of a countable
family of dense open sets. The Baire category theorem asserts that a residual
subset of a complete metric space $X$ is dense. The following simple facts
will be needed later:

\begin{lemma} \label{lem.restrofresidual}
If $R\subset [a,b]$ is residual and $(c,d) \subset [a,b]$ is an open subinterval,
then $R\cap (c,d)$ is residual in $(c,d)$ and thus also in $[c,d]$.
\end{lemma}

\begin{lemma} \label{lem.unionofresidual}
Let $a,b$ be real numbers such that $0<a<b<1$. Suppose that
$R_0 \subset (0,a)$ is residual in $(0,a)$, $R_{ab} \subset (a,b)$
is residual in $(a,b)$, and $R_1 \subset (b,1)$ is residual in $(b,1)$.
Then the union $R_0 \cup R_{ab} \cup R_1$ is residual in $[0,1]$.
\end{lemma}
The intersection of countably many residual sets is residual. The
Morse-Sard theorem (\cite[Thm. 3.1.3]{hirsch}) implies that
$\reg (W)$ is residual, and hence dense. For $t\in \reg (W),$
the slice $W_t$ is a smooth codimension one submanifold of $W$.
Note that $0,1\in \reg (W)$ and $W_0 =M,$ $W_1 =N$.
Moreover, our definition of a cobordism implies that there is an
$\epsilon >0$ such that $W_t = \{ t \} \times M$ for
$t\in [0,\epsilon]$ and $W_t = \{ t \} \times N$ for
$t\in [1-\epsilon,1]$. \\

We proceed to discuss the operation of gluing two cobordisms.
Let $W'$ be a cobordism from $M$ to $N$ and $W''$ a cobordism
from $N$ to $P$. Let $e_1 = (1,0,\ldots,0)\in \real^{D+1}$ be the first
standard basis vector.
The translate $W'' + e_1$ is embedded in 
$[1,2] \times \real^{D}$. Since the embedding of $W'$ near
$1 \times \real^{D}$ looks like the product embedding
$[1-\epsilon,1]\times N$ and the embedding of $W'' +e_1$ near
$1\times \real^{D}$ looks like the product embedding
$[1,1+\epsilon'] \times N$, the set-theoretic union
\[ \widetilde{W} = W' \cup (W'' + e_1) \]
carries a unique smooth structure that restricts to the smooth structures
on $W'$ and $W''$. This smooth manifold $\widetilde{W}$ is embedded
in $[0,2] \times \real^{D}$. Reparametrizing this embedding linearly
from $[0,2]$ to $[0,1]$, we obtain a cobordism 
$W \subset [0,1] \times \real^{D}$ from $M$ to $P$. We call this cobordism
the result of \emph{gluing $W'$ and $W''$ along $N$} and will write
$W = W' \cup_N W''$. 
Given cylinder scales $\epsilon$ and $\epsilon'$ for $W'$ and $W''$,
respectively, the natural cylinder scale for $W$ is by definition
$\epsilon_W = \smlhf \min (\epsilon, \epsilon').$
The formula
$W(k) = W'(k) \cup_{N(k)} W''(k)$
holds for every $k=0,1,2,\ldots$.
It should be pointed out that one cannot take this gluing operation as the
composition law of a cobordism category because it does not, among other issues,
satisfy associativity. To get a category that way, one would have to consider
isotopy classes of embeddings, which would however render time functions
ill-defined.\\

More generally, we shall also speak of cobordisms in $[a,b]\times \real^{D}$
for any $a<b$. All of the previous definitions generalize to such cobordisms,
with $a$ playing the role of $0$ and $b$ playing the role of $1$. \\

\subsection{Fold Fields}

Let $W$ be a cobordism from $M$ to $N$.
Given a fold map $F:W\to \real^2 =\cplx$,
we set $F(k) = F|:W(k)\to \cplx$ and $SF_t = S(F)\cap W_t$. 
(Recall that $S(F) \subset W$ is the singular set of $F$.)
The imaginary part of a complex number $z\in \cplx$ will be denoted by $\Ima (z)$.
\begin{defn}
We say that $F$ has \emph{generic imaginary parts over $t\in [0,1]$}, if 
$\Ima \circ F|: SF_t \to \real$ is injective.
\end{defn}
\noindent We put
\[ \genim (F) = \{ t\in [0,1] ~|~ F \text{ has generic imaginary parts over } t \} \]
and
\[ \pitchfork (F) = \{ t\in \reg (W) ~|~ S(F)\pitchfork W_t \}. \]
Here, the symbol $A \pitchfork B$ means transverse intersection of two submanifolds $A,B$.
Note that for $t\in \pitchfork (F)$, $SF_t$ is a compact $0$-dimensional manifold
and thus a finite set of points.
The set $\pitchfork (F)$ can be expressed in terms of regular values:
\begin{lemma} \label{lem.transregu}
Let $W$ be a cobordism with time function $\omega: W\to [0,1]$ and
let $F:W\to \cplx$ be a fold map. Let $\reg (\omega_S)$ be the set of regular
values of the restriction $\omega_S = \omega|: S(F)\to [0,1].$ Then
\[ \pitchfork (F) = \reg (\omega_S)\cap \reg (W). \]
\end{lemma}
From this lemma and Brown's theorem, we deduce:
\begin{cor} \label{cor.pitchfresidual}
Let $F: W\to \cplx$ be a fold map on a cobordism $W$. Then $\pitchfork (F)$
is open and dense (and thus residual) in $[0,1]$.
\end{cor}
Our TFT will operate with the following notion of fields:
\begin{defn} \label{def.foldfield}
A \emph{fold field} on $W$ is a fold map $F:W\to \cplx$
such that for all $k\in s(W),$\\
(1) $0,1 \in~ \pitchfork (F(k))\cap \genim (F(k))$, and \\
(2) $\genim (F(k))$ is residual in $[0,1]$.
\end{defn}
\begin{remark} \label{rem.giequivtransplusgi}
It follows from Corollary \ref{cor.pitchfresidual} that condition (2) is equivalent to \\

(2') $\pitchfork (F(k))\cap \genim (F(k))$ is residual in $[0,1]$. 
\end{remark}

A fold map $F:W\to \cplx$ is a fold field on $W$ if and only if
$F(k)$ is a fold field on $W(k)$ for all $k=0,1,\ldots$.
Fold fields on a cobordism embedded in 
$[a,b] \times \real^{D},$ $a<b,$ are defined similarly: just replace
$0$ by $a$ and $1$ by $b$ in the above definition. 
For a nonempty cobordism $W$, let $\Fa (W)\subset C^\infty (W,\cplx)$ be the space
of all fold fields on $W$. For $W$ empty, we agree that $\Fa (W) = \{ * \},$ a set with one element.
 In the following very special case, fold maps are automatically
fold fields:

\begin{lemma} \label{lem.timelocconst}
Let $F:W\to \cplx$ be a fold map on a cobordism $W$.
If the time function $\omega$ on $W$ is locally constant, then $F$ is a fold field.
\end{lemma}
\begin{proof}
If $\omega$ is locally constant, then $\partial W$ is empty.
Hence $W_0 = \varnothing = W_1$ and so vacuously $S(F) \pitchfork W_0, W_1$.
Since $0,1$ are not in the image of $\omega$, $0,1\in \reg (W)$ and consequently
$0,1 \in~ \pitchfork (F)$. Moreover, as $SF_0 = \varnothing = SF_1,$
the restrictions $\Ima \circ F|_{SF_0},$ $\Ima \circ F|_{SF_1}$ are vacuously injective and we
conclude that $0,1\in \genim (F)$.

Let $\{ W[j] \}_{j\in J}$ be the distinct connected components of $W$.
As $W$ is compact, $J$ is finite.
Since $\omega|_{W[j]}$ is constant, there is a $t_j \in (0,1)$ such that $W[j] \subset \omega^{-1} (t_j)$.
If $t\in [0,1] - \{ t_j \}_{j\in J},$ then $W_t =\varnothing,$ $SF_t =\varnothing$ and thus
$\Ima \circ F|_{SF_t}$ is vacuously injective. This shows that $t\in \genim (F)$.
Hence $\genim (F)$ contains the open dense subset $[0,1] - \{ t_j \}_{j\in J}$ and so is residual in $[0,1]$.
\end{proof}

\subsection{The Action Functional: Fold Fields and the Brauer Category}
\label{ssec.actionfoldfieldsbrauer}

Recall that $\Br$ denotes the Brauer category as introduced in Section \ref{ssec.tangles}.
There is a natural function
\[ \mbs: \Fa (W) \longrightarrow \operatorname{Mor}(\Br) \]
that we shall describe next. 
We interpret this function as the (exponential of the) \emph{action functional} on our fields.
If $W$ is empty we define $\mbs (*) = 1_I,$ the identity on the
unit object $I=[0]$ of $\Br$. Next, suppose that $W$ is nonempty and entirely
contained in a slice $[0,1]\times \{ k \} \times \real^{D-1},$ i.e.
$W = W(k)$. Given $F\in \Fa (W),$ let
$m_S$ be the cardinality of $S(F)\cap M$ and let $n_S$ be the cardinality
of $S(F)\cap N$. The Brauer morphism $\mbs (F)$ will be a morphism
$\mbs (F):[m_S]\to [n_S].$ There is a canonical identification of
points of $S(F)\cap M$ with points of $M[m_S]$ given as follows:
By (1) of Definition \ref{def.foldfield}, $\Ima \circ F$ is injective on 
$S(F)\cap M = SF_0$ and therefore induces
a unique ordering $p_1, p_2, \ldots, p_{m_S}$ of the points of $S(F)\cap M$
such that 
\[ \Ima F(p_i) < \Ima F(p_j) \hspace{.3cm} \Longleftrightarrow \hspace{.3cm}
  i<j. \]
This gives a bijection $S(F)\cap M \cong M[m_S],$ $p_i \leftrightarrow i$.
Similarly for the outgoing boundary:
The function $\Ima \circ F$ is injective on $S(F)\cap N=SF_1$ and therefore induces
a unique ordering $q_1, q_2, \ldots, q_{n_S}$ of the points of $S(F)\cap N$
such that 
$\Ima F(q_i) < \Ima F(q_j) \Leftrightarrow i<j$.
This gives a bijection $S(F)\cap N \cong M[n_S],$ $q_i \leftrightarrow i$.
To construct the Brauer morphism $\mbs (F):[m_S]\to [n_S],$ connect the points of
$0\times M[m_S]\times 0\times 0$ and $1\times M[n_S]\times 0 \times 0$
by smooth arcs in $[0,1] \times \real^3$
in the following manner. Let $c$ be a connected component of the
compact $1$-manifold $S(F)$. We distinguish four cases.
If $\partial c = \{ p_i, p_j \}$, then connect $(0,i,0,0)$ to $(0,j,0,0)$ by an arc.
If $\partial c = \{ p_i, q_j \}$, then connect $(0,i,0,0)$ to $(1,j,0,0)$.
If $\partial c = \{ q_i, q_j \}$, then connect $(1,i,0,0)$ to $(1,j,0,0)$.
Finally, if $c$ is closed, i.e. $\partial c = \varnothing,$ then tensor with
the loop endomorphism $\lambda$.
Carrying this recipe out for every connected component $c$ of $S(F)$
completes the construction of $\mbs (F)$. Finally, if $W$ is nonempty
but otherwise arbitrary, we put
\[ \mbs (F) = \bigotimes_{k\in \nat} \mbs (F(k)). \]
(This tensor product is finite, as $W(k)$ is eventually empty.)

\begin{lemma} \label{lem.sfltsfgtissf}
Let $F$ be a fold field on $W$ and 
$t \in~ (0,1) \cap \bigcap_{a\in s(W)} (\pitchfork F(a)\cap \genim F(a)).$ 
Let $F_{\leq t}: W_{\leq t}\to \cplx$ and
$F_{\geq t}: W_{\geq t}\to \cplx$ denote the restrictions of $F$ to
$W_{\leq t} = W\cap [0,t]\times \real^{D},$
$W_{\geq t} = W\cap [t,1]\times \real^{D},$ respectively.
Then $F_{\leq t}$ is a fold field on $W_{\leq t}$, $F_{\geq t}$ is a fold
field on $W_{\geq t},$ and the Brauer morphism identity
\[ \mbs (F_{\geq t})\circ \mbs (F_{\leq t}) = \mbs (F) \]
holds.
\end{lemma}
\begin{proof}
Fix $a\in \nat$ with $W(a)\not= \varnothing$.
The set $\genim (F(a))$ is residual in $[0,1]$. 
Thus, by Lemma \ref{lem.restrofresidual},
\[ \genim (F(a)_{\leq t}) \cap (0,t) =
  \genim (F(a)) \cap (0,t) \]
is residual in $[0,t]$. Since 
$0,t \in \pitchfork (F(a)_{\leq t})\cap \genim (F(a)_{\leq t}),$
the map $F(a)_{\leq t}$ is a fold field on $W(a)_{\leq t}$.
Since this holds for any $a\in s(W),$ 
$F_{\leq t}$ is a fold field on $W_{\leq t}$.
A similar argument shows that 
$F_{\geq t}$ is a fold field on $W_{\geq t}.$ 

We turn to proving the Brauer morphism identity. Again, fix a 
natural number $a$.
The function $\Ima F(a)$ induces a unique ordering
\[ S(F(a)) \cap M(a) = \{ p_1,\ldots, p_m \} \]
such that $\Ima F(a)(p_i)<\Ima F(a)(p_j)$ if and only if $i<j$,
a unique ordering
\[ S(F(a)) \cap W(a)_t = \{ r_1,\ldots, r_k \} \]
such that $\Ima F(a)(r_i)<\Ima F(a)(r_j)$ if and only if $i<j$, and
a unique ordering
\[ S(F(a)) \cap N(a) = \{ q_1,\ldots, q_l \} \]
such that $\Ima F(a)(q_i)<\Ima F(a)(q_j)$ if and only if $i<j$.
Let $c$ be a connected component of $S(F(a))$ with 
nonempty boundary. There are three possible cases:
either $\partial c = \{ p_i, q_j \}$, or $\partial c = \{ p_i, p_j \}$
($i\not= j$), or $\partial c = \{ q_i, q_j \}$ ($i\not= j$).

Suppose that $\partial c = \{ p_i, q_j \}$.
We move along $c$, starting at the endpoint $p_i$.
As we move along $c$, let $r_{a(1)}$ be the first point on $c$
which lies on $W_t$. Note that such a point exists by the
intermediate value theorem. The segment of $c$ from $p_i$ to
$r_{a(1)}$ must be entirely contained in $W(a)_{\leq t},$ and is
a connected component $c_1$ of $S(F(a)_{\leq t})$.
We continue to move along $c$ in the same direction.
Passing $r_{a(1)},$ a sufficiently short segment of $c$ must
lie entirely in $W(a)_{\geq t},$ since $S(F(a))$ is transverse to $W(a)_t$.
If we encounter no point which lies on $W(a)_t$ again, then
the segment $c_2$ of $c$ from $r_{a(1)}$ to $q_j$ lies entirely
in $W(a)_{\geq t}$ and we stop. Otherwise, if there is a point
after $r_{a(1)}$ where $c$ and $W(a)_t$ intersect again, then let
$r_{a(2)}$ be the first such point, $a(2)\not= a(1)$. The segment
of $c$ from $r_{a(1)}$ to $r_{a(2)}$ is entirely contained in
$W(a)_{\geq t}$ and is a connected component $c_2$ of $S(F(a)_{\geq t})$.
Continuing in this fashion, we arrive at a finite list
$\{ c_1, c_2, \ldots, c_s \}$ such that
\begin{enumerate}
\item $s$ is even and $s-1\leq k$,
\item $c = \bigcup_{d=1}^s c_d,$
\item $c_{2d+1}$ is a connected component of $S(F(a)_{\leq t})$,
      $d=0,\ldots, \frac{s}{2}-1,$
\item $c_{2d}$ is a connected component of $S(F(a)_{\geq t})$,
      $d=1,\ldots, \frac{s}{2},$
\item $\partial c_1 = \{ p_i, r_{a(1)} \},$ 
        $\partial c_s = \{ r_{a(s-1)}, q_j \},$
\item $\partial c_d = \{ r_{a(d-1)}, r_{a(d)} \},$ 
       $d=2,\ldots, s-1,$
\item $a(d)\not= a(d'),$ whenever $d\not= d'$.
\end{enumerate}
Thus the Brauer morphism $\mbs (F(a)_{\leq t})$ connects
$(0,i,0,0)$ to $(1,a(1),0,0)$ (as $c_1$ is a connected component of
$\mbs (F(a)_{\leq t})$ with $\partial c_1 = \{ p_i, r_{a(1)} \}$),
$(1,a(2),0,0)$ to $(1,a(3),0,0)$ (as $c_3$ is a connected component of
$\mbs (F(a)_{\leq t})$ with $\partial c_3 = \{ r_{a(2)}, r_{a(3)} \}$),
$(1,a(4),0,0)$ to $(1,a(5),0,0)$, etc., and
$(1,a(s-2),0,0)$ to $(1,a(s-1),0,0)$ (as $c_{s-1}$ is a connected component of
$\mbs (F(a)_{\leq t})$ with $\partial c_{s-1} = \{ r_{a(s-2)}, r_{a(s-1)} \}$).
The Brauer morphism $\mbs (F(a)_{\geq t})$ connects
$(0,a(1),0,0)$ to $(0,a(2),0,0)$ (as $c_2$ is a connected component of
$\mbs (F(a)_{\geq t})$ with $\partial c_2 = \{ r_{a(1)}, r_{a(2)} \}$),
$(0,a(3),0,0)$ to $(0,a(4),0,0)$ (as $c_4$ is a connected component of
$\mbs (F(a)_{\geq t})$ with $\partial c_4 = \{ r_{a(3)}, r_{a(4)} \}$),
etc., and
$(0,a(s-1),0,0)$ to $(1,j,0,0)$ (as $c_s$ is a connected component of
$\mbs (F(a)_{\geq t})$ with $\partial c_s = \{ r_{a(s-1)}, q_j \}$).
The composition $\mbs (F(a)_{\geq t})\circ \mbs (F(a)_{\leq t})$ in $\Br$
is defined by translating (a representative of) $\mbs (F(a)_{\geq t})$ from
$[0,1]\times \real^3$ to $[1,2]\times \real^3,$ taking the union of
(a representative of) $\mbs (F(a)_{\leq t})$ with the translated copy of
$\mbs (F(a)_{\geq t})$, and reparametrizing $[0,2]\times \real^3$ back to
$[0,1]\times \real^3$. Thus $\mbs (F(a)_{\geq t})\circ \mbs (F(a)_{\leq t})$
connects $(0,i,0,0)$ to $(1,j,0,0)$. Since $c$ is a connected component
of $S(F(a))$ with $\partial c = \{ p_i, q_j \},$ $\mbs (F(a))$ also connects
$(0,i,0,0)$ to $(1,j,0,0)$. 

Suppose that $\partial c = \{ p_i, p_j \}$.
Moving along $c$ and proceeding as in the previous case, we obtain
a list
$\{ c_1, \ldots, c_s \}$ such that
\begin{enumerate}
\item $s$ is odd and $s-1\leq k$,
\item $c = \bigcup_{d=1}^s c_d,$
\item $c_{2d+1}$ is a connected component of $S(F(a)_{\leq t})$,
      $d=0,\ldots, \smlhf (s-1),$
\item $c_{2d}$ is a connected component of $S(F(a)_{\geq t})$,
      $d=1,\ldots, \smlhf (s-1),$
\item $\partial c_1 = \{ p_i, r_{a(1)} \},$ 
        $\partial c_s = \{ r_{a(s-1)}, p_j \},$
\item $\partial c_d = \{ r_{a(d-1)}, r_{a(d)} \},$ 
       $d=2,\ldots, s-1,$
\item $a(d)\not= a(d'),$ whenever $d\not= d'$.
\end{enumerate}
Thus the Brauer morphism $\mbs (F(a)_{\leq t})$ connects
$(0,i,0,0)$ to $(1,a(1),0,0)$,
$(1,a(2),0,0)$ to $(1,a(3),0,0)$, etc., and
$(0,j,0,0)$ to $(1,a(s-1),0,0)$.
The Brauer morphism $\mbs (F(a)_{\geq t})$ connects
$(0,a(1),0,0)$ to $(0,a(2),0,0)$,
$(0,a(3),0,0)$ to $(0,a(4),0,0)$, etc., and
$(0,a(s-2),0,0)$ to $(0,a(s-1),0,0)$.
Hence, $\mbs (F(a)_{\geq t})\circ \mbs (F(a)_{\leq t})$
connects $(0,i,0,0)$ to $(0,j,0,0)$. Since $c$ is a connected component
of $S(F(a))$ with $\partial c = \{ p_i, p_j \},$ $\mbs (F(a))$ also connects
$(0,i,0,0)$ to $(0,j,0,0)$. The third case $\partial c = \{ q_i, q_j \}$
is treated in a similar way. 

Now let $c$ be a component of $S(F(a))$ with empty boundary.
If $c$ does not intersect $W(a)_t$, then $c\subset W(a)_{\leq t}$ or
$c\subset W(a)_{\geq t}$, by the intermediate value theorem.
Thus $c$ contributes to precisely one of $\mbs (F(a)_{\leq t})$,
$\mbs (F(a)_{\geq t})$ by tensoring with the loop endomorphism 
$\lambda$. The closed component $c$ also contributes to
$\mbs (F(a))$ by tensoring with the loop endomorphism 
$\lambda$. Suppose that $c$ and $W(a)_t$ do intersect.
Let $r_{a(1)}$ be any point in the intersection.
We will move along $c$, starting at the point $r_{a(1)}$.
As $c$ is transverse to $W(a)_t$, we have a choice of either
moving into $W(a)_{\leq t}$ or into $W(a)_{\geq t}$. We choose to
move into $W(a)_{\leq t}$. Let $r_{a(2)}$ be the first point on $c$
after $r_{a(1)}$ that lies on $W(a)_t$, $a(1)\not= a(2)$.
The segment of $c$ from $r_{a(1)}$ to $r_{a(2)}$ is entirely
contained in $W(a)_{\leq t}$ and is a connected component
$c_1$ of $S(F(a)_{\leq t})$. We continue to move in the same
direction along $c$. Passing $r_{a(2)},$ a sufficiently short
segment of $c$ must lie entirely in $W(a)_{\geq t}$, since
$S(F(a))$ is transverse to $W(a)_t$. If we encounter no point in
$W(a)_t$ other than $r_{a(1)}$ anymore, then let
$c_2 \subset W(a)_{\geq t}$ be the segment of $c$ from
$r_{a(2)}$ to $r_{a(1)}$ and stop. Otherwise, let
$r_{a(3)}$, $a(3)\not\in \{ a(1), a(2) \},$ be the first point
on $c$ after $r_{a(2)}$ which lies on $W(a)_t$ again.
The segment of $c$ from $r_{a(2)}$ to $r_{a(3)}$ is
entirely contained in $W(a)_{\geq t}$ and is a connected
component of $S(F(a)_{\geq t})$. 
Continuing in this fashion, we arrive at a list
$\{ c_1, \ldots, c_s \}$ such that
\begin{enumerate}
\item $s$ is even and $s\leq k$,
\item $c = \bigcup_{d=1}^s c_d,$
\item $c_{2d+1}$ is a connected component of $S(F(a)_{\leq t})$,
      $d=0,\ldots, \frac{s}{2}-1,$
\item $c_{2d}$ is a connected component of $S(F(a)_{\geq t})$,
      $d=1,\ldots, \frac{s}{2},$
\item $\partial c_s = \{ r_{a(s)}, r_{a(1)} \},$
\item $\partial c_d = \{ r_{a(d)}, r_{a(d+1)} \},$ 
       $d=1,\ldots, s-1,$
\item $a(d)\not= a(d'),$ whenever $d\not= d'$.
\end{enumerate}
Thus the Brauer morphism $\mbs (F(a)_{\leq t})$ connects
$(1,a(1),0,0)$ to $(1,a(2),0,0)$,
$(1,a(3),0,0)$ to $(1,a(4),0,0)$, etc., and
$(1,a(s-1),0,0)$ to $(1,a(s),0,0)$.
The Brauer morphism $\mbs (F(a)_{\geq t})$ connects
$(0,a(2),0,0)$ to $(0,a(3),0,0)$,
$(0,a(4),0,0)$ to $(0,a(5),0,0)$, etc., and
$(0,a(s),0,0)$ to $(0,a(1),0,0)$.
Thus $\mbs (F(a)_{\geq t})\circ \mbs (F(a)_{\leq t})$
connects $(\tau,a(1),0,0)$ to $(\tau,a(1),0,0)$ (suitable $\tau$) and therefore,
$c_1,\ldots, c_s$ contribute to $\mbs (F(a)_{\geq t})\circ \mbs (F(a)_{\leq t})$
by tensoring with the loop endomorphism $\lambda$.
Since $c$ is a closed connected component
of $S(F(a))$, $\mbs (F(a))$ also contributes to $\mbs (F(a))$
by tensoring with the loop endomorphism $\lambda$. 

We have shown that
\[ \mbs (F(a)_{\geq t})\circ \mbs (F(a)_{\leq t}) = \mbs (F(a)) \]
for every $a=0,1,2,\ldots$. The desired Brauer morphism identity follows from
\begin{eqnarray*}
\mbs (F) & = & \bigotimes_a \mbs (F(a)) =
  \bigotimes_a \left( \mbs (F(a)_{\geq t}) \circ \mbs (F(a)_{\leq t}) \right) \\
& = & \Big\{ \bigotimes_a \mbs (F(a)_{\geq t}) \Big\} \circ
    \Big\{ \bigotimes_a \mbs (F(a)_{\leq t}) \Big\} = \mbs (F_{\geq t})\circ \mbs (F_{\leq t}).
\end{eqnarray*}
\end{proof}

Let $V$ be a real vector space of finite dimension $\dim V\geq 2$ and let $(i,e)$ be a duality
structure on $V$ such that the induced symmetric monoidal functor
$Y:\Br \to \vect$ is faithful on loops. (For instance, we may take $V=\real^2$ and
$(i,e)$ as in (\ref{equ.exie}) of Section \ref{ssec.reptangle}.)
Let $Q=Q(i,e)$ be the profinite idempotent completion
of $Y(\Mor (\Br))$. Via the composition
\[ \Fa (W) \stackrel{\mbs}{\longrightarrow} \Mor (\Br)
 \stackrel{Y}{\longrightarrow} \Mor(\vect), \]
every fold field $F\in \Fa (W)$ determines an element
$Y\mbs (F) \otimes 1$ in $Q(H_{\dom \mbs (F), \cod \mbs (F)}),$
and thus an element in $Q$, which we will denote simply by
$Y\mbs (F)\in Q$. 
 
For a nonempty, closed, smooth $(n-1)$-manifold $M\subset \real^{D}$ 
(not necessarily orientable), we have the trivial cobordism
from $M$ to $M$, i.e. the cylinder $[0,1]\times M \subset
[0,1]\times \real^{D},$ and we put
\[ \Fa (M) = \{ f\in \Fa([0,1]\times M) ~|~  
     \mbs (f)=1 \in \operatorname{Mor}(\Br) \}, \]
where $1$ denotes an identity morphism in $\Br$.
(Note that $[0,1] \times M$ indeed satisfies our standing embedding
conventions, since $M$ does; one has
$([0,1]\times M)(k) = [0,1]\times (M(k))$.)
For $M=\varnothing,$ we put $\Fa (\varnothing) = \{ * \},$
the one-element set. (Recall that the empty set is the unique initial object
in the category of sets, so there is a unique arrow $\varnothing \to X$
for every set $X$.)
These are the \emph{fields} associated to a closed $(n-1)$-manifold and will
act as boundary conditions. We shall write $I=[0,1]$ for the unit interval.

\begin{lemma} \label{lem.restrtomorse}
Let $M,N \subset \real^D$ be disjoint closed $(n-1)$-manifolds.
Then the map
\[ 
C^\infty (I\times (M\sqcup N),\cplx) \longrightarrow 
 C^\infty (I\times M,\cplx)\times C^\infty (I\times N,\cplx),~ 
f \mapsto  (f|_{I\times M}, f|_{I\times N})
\]
restricts to a map
$\Fa (M\sqcup N) \rightarrow \Fa (M)\times \Fa (N).$
\end{lemma}
\begin{proof}
Let $f$ be a map in $\Fa (M\sqcup N)$. The restriction $\fim$
is a fold map, using fold charts for $f$. Since $S(f)=S(\fim)\sqcup
S(f|_{I\times N})$ is transverse to $0\times (M\sqcup N)$ and to
$1\times (M\sqcup N),$ the singular set $S(\fim)$ of the restriction
is transverse to $0\times M$ and to $1\times M,$ whence
$0,1\in \pitchfork (\fim)$. 
Let $k$ be a natural number. We have $(M\sqcup N)(k) = M(k)\sqcup N(k)$.
Let $t\in \genim (f(k))$. The injectivity of 
$\Ima f(k)|: S(f(k))\cap t\times (M\sqcup N)(k)\to \real$
implies the injectivity of the restriction $\Ima f(k)|: S(\fimk)\cap t\times M(k)
\to \real$. Thus $t\in \genim (\fimk)$ and we have established
the inclusion
\[ \genim (f(k))\subset \genim (\fimk). \]
Since $\genim (f(k))$ is residual in $[0,1]$, the superset
$\genim (\fimk)$ is residual as well. Thus we have shown
that $\fim$ is in $\Fa (I\times M)$. 

It remains to be shown that $\mbs (\fim)=1$. Again, we fix a
natural number $a$.
The function $\Ima f(a)$ induces uniquely an ordering of
\[ S(f(a))\cap 0\times (M\sqcup N)(a) = \{ p_1,\ldots, p_k \} \]
such that $\Ima f(a)(p_i)<\Ima f(a)(p_j)$ if and only if $i<j$.
As 
$\bigotimes_{a'\in \nat} \mbs (f(a'))=\mbs (f)=1,$
Lemma \ref{lem.idfactorstriv} implies that the tensor factor $\mbs (f(a))$
is the identity morphism. Therefore,
$S(f(a))\cap 1\times (M\sqcup N)(a)$ has the
same number $k$ of points,
\[ S(f(a))\cap 1\times (M\sqcup N)(a) = \{ q_1,\ldots, q_k \}, \]
which are ordered such that $\Ima f(a)(q_i)<\Ima f(a)(q_j)$ if and only if $i<j$.

Since $\mbs (f(a))=1$, the singular set $S(f(a))$ has precisely $k$
connected components $c_1,\ldots, c_k$, and these components
have endpoints $\partial c_i = \{ p_i, q_i \}$ for all $i$.
Filtering out those points that lie in $0\times M(a)$ yields
\[ \{ p_1,\ldots, p_k \} \cap 0\times M(a) =
 \{ p_{m(1)},\ldots, p_{m(l)} \}, \]
where $l\leq k$ and $m: [l]\to [k]$ is a function such that
$m(i)<m(j)$ precisely when $i<j$. As $c_{m(i)}$ is connected and one
of its endpoints, namely $p_{m(i)}$, lies in $I\times M(a)$, the entire path
$c_{m(i)}$ lies in $I\times M(a)$. In particular, its other endpoint $q_{m(i)}$
lies in $1\times M(a)$ and thus
\[ \{ q_{m(1)},\ldots, q_{m(l)} \} \subset \{ q_1,\ldots, q_k \} \cap 1\times M(a). \]
Conversely, if $q_j$ lies in $1\times M(a)$, then $p_j$, being the other endpoint
of $c_j$, lies in $0\times M(a)$ and consequently $j=m(i)$ for some $i\in [l].$
This shows that
\[ \{ q_{m(1)},\ldots, q_{m(l)} \} = \{ q_1,\ldots, q_k \} \cap 1\times M(a). \]
The singular set of the restriction is
$S(\fim(a)) = c_{m(1)} \sqcup \cdots \sqcup c_{m(l)}.$
Setting $p'_i = p_{m(i)}$ and $q'_i = q_{m(i)},$ we have
$\Ima \fim(a) (p'_i) < \Ima \fim(a) (p'_j)$ if and only if
$\Ima \fim(a) (p_{m(i)}) < \Ima \fim(a) (p_{m(j)})$ iff $m(i)<m(j)$
iff $i<j$. Similarly, $\Ima \fim(a) (q'_i) < \Ima \fim(a) (q'_j)$ iff $i<j$.
Since $\partial c_{m(i)} = \{ p'_i, q'_i \},$ the Brauer morphism
$\mbs (\fim(a))$ is the identity. Consequently,
$\mbs (\fim) = \bigotimes_a \mbs (\fim (a)) = \bigotimes_a 1 = 1$
is the identity.
\end{proof}

Fold fields satisfy the additivity axiom of \cite[p. 6]{freedlecnotes}:
\begin{lemma}[Additivity Axiom] \label{lem.addaxiomfields}
Given two disjoint, well-separated, closed $(n-1)$-manifolds $M,N\subset \real^{D}$, there is
a bijection
\[ \Fa (M\sqcup N) \cong \Fa (M)\times \Fa (N). \]
\end{lemma}
\begin{proof}
We shall construct an inverse
\[ \rho': \Fa (M)\times \Fa (N) \longrightarrow \Fa (M\sqcup N) \]
to the restriction map
$\rho: \Fa (M\sqcup N) \longrightarrow \Fa (M)\times \Fa (N)$
of Lemma \ref{lem.restrtomorse}. Given $(g,h)\in \Fa (M)\times \Fa (N)$,
let $f=g\sqcup h:I\times (M\sqcup N) \to \cplx$. Then $f$ is a fold map. Fix a natural
number $k$. Since the disjoint union of $M$ and $N$ is well-separated,
we have $(M\sqcup N)(k) = M(k)$ or $(M\sqcup N)(k) = N(k)$, and correspondingly
$f(k)=g(k)$ or $f(k)=h(k)$. 
Therefore, $f(k)$ is a fold field on $I\times (M\sqcup N)$.
It remains to be shown that $\mbs (f)=1$. Since $\mbs (g)=1$ and $\mbs (h)=1$
and
\[ \mbs (g) =\bigotimes_k \mbs (g(k)),~ \mbs (h) =\bigotimes_k \mbs (h(k)), \]
we know that $\mbs (g(k))=1$ and $\mbs (f(k))=1$ for all $k$, 
invoking Lemma \ref{lem.idfactorstriv}.
For each $k$, $f(k)=g(k)$ or $f(k)=h(k)$ by well-separation.
In either of these cases, $\mbs (f(k))=1$. Hence
$\mbs (f) = \bigotimes_k \mbs (f(k)) = \bigotimes_k 1 = 1.$
Thus $f\in \Fa (M\sqcup N)$ and we can set $\rho' (g,h)=f$. It is clear that
$\rho$ and $\rho'$ are inverse to each other.
\end{proof}
If the disjoint union is not well-separated, then the above addivity axiom
need not hold, due to entanglement effects caused by the injectivity condition
required for fold fields.
\begin{remark} \label{rem.suspexcmorse}
If $g:M\to \real$ is an excellent Morse function on $M$, then
$\id \times g: [0,1]\times M \to [0,1]\times \real$ is a fold field
with $\mbs (\id \times g)=1 \in \Mor (\Br)$, hence an element
in $\Fa (M)$.
\end{remark}

\subsection{State Modules}
\label{ssec.statemodules}

Let $V$ be a real vector space of finite dimension $\dim V\geq 2$ and $(i,e)$ a duality
structure on $V$ such that the induced symmetric monoidal functor $Y$ is faithful on loops.
Let $Q = Q(i,e)$ be the corresponding profinite idempotent completion of
$Y(\Mor (\Br))$.
On closed $(n-1)$-manifolds $M$, the topological field theory $Z$ is defined to be
\begin{equation} \label{equ.defstatemodule}
Z(M) := \fun_Q (\Fa (M)). 
\end{equation}
This is the ``\emph{quantum Hilbert space}'', or \emph{state-module}, of our
theory. A \emph{state} is thus a map $z:\Fa (M)\to Q$.
Note that for the empty manifold,
\[ Z(\varnothing) = \fun_Q (\Fa (\varnothing)) =
\fun_Q (\{ \ast \}) = Q. \]
Recall from Section \ref{sec.proidemcompletion}
that the commutative monoid $Q$ comes with two multiplications yielding two
complete, idempotent, continuous semirings: the composition semiring
$Q^c$ and the monoidal semiring $Q^m$ (see Propositions
\ref{prop.qccomplidemp}, \ref{prop.qmcomplidemp}, \ref{prop.qcontinuous}).
By Proposition \ref{prop.funsemimod}, $Z(M)$ is thus a two-sided
$Q^c$-semialgebra and a two-sided $Q^m$-semialgebra, and $Z(M)$
is complete, idempotent and continuous. The complete tensor product
$\hotimes$ used below was introduced in Section \ref{sec.funcompltensorprod}
and we shall use the identification provided by Theorem \ref{thm.funiso}
without always explicitly mentioning it.

\begin{prop} \label{prop.zmonoidalonmodules}
The state-module of a disjoint and well-separated union of two closed $(n-1)$-manifolds 
$M$ and $N$ decomposes as a tensor product
\[ Z(M\sqcup N)  \cong Z(M) \hotimes Z(N). \]
\end{prop}
\begin{proof}
By the additivity axiom for fields on closed manifolds (Lemma \ref{lem.addaxiomfields}),
$\Fa (M\sqcup N) \cong \Fa (M)\times \Fa (N)$. Thus by Theorem \ref{thm.funiso},
\begin{align*}
Z(M\sqcup N) & =  \fun_Q (\Fa (M\sqcup N)) \cong  \fun_Q (\Fa (M)\times \Fa (N)) \\
& \cong  \fun_Q (\Fa (M))\hotimes \fun_Q (\Fa (N)) =  Z(M)\hotimes Z(N).
\end{align*}
\end{proof}
\begin{remark}
The assumption of well-separation is only needed to ensure
$\Fa (M\sqcup N)\cong \Fa (M)\times \Fa (N)$. For arbitrary disjoint $M$ and $N$
we always have the isomorphism
$Z(M)\hotimes Z(N) \cong \fun_Q (\Fa (M)\times \Fa (N)),$ regardless of whether
$M$ and $N$ are well-separated or not.
\end{remark}
Under the isomorphism of Theorem \ref{thm.funiso}, the canonical
middle $Q^c$-linear map
\[ \beta^c: Z(M)\times Z(N) \longrightarrow Z(M)\hotimes Z(N),~
  (z,z')\mapsto z\hotimes_{Q^c} z', \]
can be thought of as follows. Given $z:\Fa (M)\to Q$ and 
$z':\Fa (N)\to Q,$ define
$z'': \Fa (M)\times \Fa (N)\to Q$ on 
$(f_M, f_N)\in \Fa (M)\times \Fa (N)$ by
\begin{equation} \label{def.betacfirst}
 z''(f_M, f_N) = z(f_M)\cdot z'(f_N), 
\end{equation}
using the multiplication of $Q^c$.
Then $\beta^c (z,z') = z''$. We shall write $z\hotimes_c z' := \beta^c (z,z').$
There are various natural generalizations of this to more than two
manifolds, for example we shall also use
\[ \beta^c: (Z(M)\hotimes Z(N))\times (Z(M')\hotimes Z(N'))
 \longrightarrow Z(M)\hotimes Z(N)\hotimes Z(M')\hotimes Z(N') \]
for closed $(n-1)$-manifolds $M,N,M',N',$ given by
\begin{equation} \label{def.betacscnd}
\beta^c (\zeta, \zeta')(f_M, f_N, f_{M'}, f_{N'}) =
 \zeta (f_M, f_N) \cdot \zeta' (f_{M'}, f_{N'}), 
\end{equation}
where $\zeta \in Z(M)\hotimes Z(N),$ $\zeta' \in Z(M')\hotimes Z(N'),$
$f_M\in \Fa (M),$ $f_N \in \Fa (N),$ $f_{M'}\in \Fa (M')$
and $f_{N'} \in \Fa (N')$. We will of course also write
$\zeta \hotimes_c \zeta' = \beta^c (\zeta, \zeta')$.
If $M$ and $N$ are disjoint and well-separated and $M'$ and $N'$ are
disjoint and well-separated, then Proposition \ref{prop.zmonoidalonmodules}
provides bijections $Z(M\sqcup N)\cong Z(M)\hotimes Z(N)$ and
 $Z(M'\sqcup N')\cong Z(M')\hotimes Z(N')$. Definition \ref{def.betacfirst}
then yields a map
\[ \beta^c: Z(M\sqcup N)\times Z(M'\sqcup N') \longrightarrow
    Z(M\sqcup N)\hotimes Z(M'\sqcup N') \]
which is consistent with (\ref{def.betacscnd}): the diagram
\[ \xymatrix@R=15pt{
Z(M\sqcup N)\times Z(M'\sqcup N') \ar[d]_{\cong}
  \ar[r]^{\beta^c} & Z(M\sqcup N)\hotimes Z(M'\sqcup N')
 \ar[d]^{\cong} \\ 
(Z(M)\hotimes Z(N))\times (Z(M')\hotimes Z(N')) \ar[r]^{\beta^c} &
Z(M)\hotimes Z(N)\hotimes Z(M')\hotimes Z(N')
} \]
commutes.
Furthermore, there is a natural contraction map (considered abstractly in the last paragraph
of Section \ref{sec.funcompltensorprod})
\[ \gamma: Z(M)\hotimes Z(N)\hotimes Z(N)\hotimes Z(P)
 \longrightarrow Z(M)\hotimes Z(P), \]
which sends
\[ z:\Fa (M)\times \Fa (N) \times \Fa (N) \times \Fa (P)\to Q \] 
to the function $z':\Fa (M)\times \Fa (P)\to Q,$ given on
$(f_M, f_P)\in \Fa (M)\times \Fa (P)$ by
\[
 z'(f_M, f_P) = \sum_{u\in \Fa (N)} z(f_M, u, u, f_P), 
\]
using the completeness of $Q$.
Composing $\beta^c$ and the contraction $\gamma$, we obtain a
contraction product
\[ \langle \cdot, \cdot \rangle: (Z(M)\hotimes Z(N))\times (Z(N)\hotimes Z(P)) 
 \longrightarrow Z(M)\hotimes Z(P), \]
allowing us to multiply a state $z\in Z(M)\hotimes Z(N)$ and a state
$z'\in Z(N)\hotimes Z(P)$ to get a state
$\langle z,z' \rangle = \gamma (z\hotimes_c z') \in Z(M)\hotimes Z(P).$
This contraction product is a means of propagating states along cobordisms. \\  

Using the monoidal structure on $\Br$ and the corresponding product $\times$ on $Q^m$, 
the canonical middle $Q^m$-linear map
\[ \beta^m: Z(M)\times Z(N) \longrightarrow Z(M)\hotimes Z(N),~
  (z,z')\mapsto z \hotimes_{Q^m} z', \]
can be thought of as follows. Given $z:\Fa (M)\to Q$ and 
$z':\Fa (N)\to Q,$ define
$z'': \Fa (M)\times \Fa (N)\to Q$ on 
$(f_M, f_N)\in \Fa (M)\times \Fa (N)$ by
\begin{equation} \label{def.betamfirst}
 z''(f_M, f_N) = z(f_M) \times z'(f_N).  
\end{equation}
Then $\beta^m (z,z') = z''$ and we shall write $z\hotimes_m z' := \beta^m (z,z').$
Again, there are generalizations of this product to more than two
manifolds. In Theorem \ref{thm.statesumdisjunion}, we shall use
\[ \beta^m: (Z(M)\hotimes Z(N))\times (Z(M')\hotimes Z(N'))
 \longrightarrow Z(M)\hotimes Z(M')\hotimes Z(N)\hotimes Z(N') \]
given by
\begin{equation} \label{def.betamscnd}
\beta^m (\zeta, \zeta')(f_M, f_{M'}, f_N, f_{N'}) =
  \zeta (f_M, f_N) \times \zeta' (f_{M'}, f_{N'}), 
\end{equation}
where $\zeta \in Z(M)\hotimes Z(N),$ $\zeta' \in Z(M')\hotimes Z(N'),$
$f_M\in \Fa (M),$ $f_N \in \Fa (N),$ $f_{M'}\in \Fa (M')$
and $f_{N'} \in \Fa (N')$. We will also write
$\zeta \hotimes_m \zeta' = \beta^m (\zeta, \zeta')$. \\

\subsection{State Sums}
\label{ssec.statesums}

Subjecting fold fields on a cobordism to boundary conditions
requires a certain equivalence relation, which we shall describe now. These fields
with boundary conditions, together with our action functional, enter into the
definition of the state sum of a cobordism.

Let $X$ be a closed smooth manifold. We define an equivalence relation 
on the collection of smooth maps of the form $[a,b]\times X\to \cplx$
for some real numbers $a<b$.
\begin{defn} \label{def.equivalencesmmaps}
Two smooth maps $f:[a,b]\times X\to \cplx$ and
$f':[a',b']\times X\to \cplx$ are equivalent, written $f\approx f'$, if and only
if there exists a diffeomorphism $\xi: [a,b]\to [a',b']$ with $\xi (a)=a'$
such that $f(t,x) = f'(\xi (t),x)$ for all $(t,x)\in [a,b]\times X$. 
\end{defn}
It is readily verified that this relation is indeed reflexive, symmetric and transitive.
Note that $f\approx f'$ implies that $f(a,x)= f'(a',x)$ as functions
of $x$.
\begin{lemma} \label{lem.relapprox}
If $f\in \Fa ([a,b]\times X)$ and $f': [a',b']\times X\to \cplx$ is a smooth map
with $f'\approx f,$ then $f' \in \Fa ([a',b']\times X)$ and
$\mbs (f)=\mbs (f')$. 
\end{lemma}
This lemma is in fact a special case of:
\begin{lemma} \label{lem.singtanglereparam}
Suppose that two cobordisms $W\subset [0,1]\times \real^{D}$ and
$W' \subset [a,b]\times \real^{D}$ are related by a diffeomorphism
$\alpha: W' \to W$ of the form $\alpha (t,x) = (\tau (t),x),$
$(t,x)\in W',$ $t\in [a,b],$ $x\in \real^{D},$ where
$\tau: [a,b]\to [0,1]$ is a diffeomorphism with $\tau (a)=0.$
Let $F$ be a fold field on $W$. 
Then $F\circ \alpha$ is a fold field
on $W'$ and we have the Brauer invariance
$\mbs (F\circ \alpha)=\mbs (F).$
\end{lemma}
\begin{proof}
Given $F\in \Fa (W)$, set $F' = F\circ \alpha$. Precomposing a fold map
with a diffeomorphism yields a fold map again. Thus $F'$ is a fold map.
Let us prove that $a,b$ are in $\pitchfork (F')$. Since $W'$ is a cobordism,
$a,b\in \reg (W')$. As $\alpha$ is a diffeomorphism, the relation
$\alpha S(F') = \alpha S(F\alpha) = S(F)$ holds. More generally, for
any natural number $k$, we have
$\alpha S(F'(k)) = S(F(k))$.
From $\alpha (a,x)=(0,x)$ for all $x\in W'_a$ and $\alpha (t,x)\not\in
0\times \real^D$ for $t>a$ we conclude that
$\alpha (W'_a)=W_0$. More generally, for any natural number $k$
and any $t' \in [a,b],$ we have $\alpha W'(k)_{t'} = W(k)_{\tau (t')}$.
Since $F$ is a fold field on $W$, $0$ lies in $\pitchfork (F)$, that is,
$S(F) \pitchfork W_0.$ Substituting, we get
$\alpha S(F')\pitchfork \alpha (W'_a)$. This is equivalent to
$S(F') \pitchfork W'_a$, as $\alpha$ is a diffeomorphism.
Therefore, $a\in \pitchfork (F')$. Similarly, $S(F)\pitchfork W_1$
together with $W_1 = \alpha (W'_b)$ implies that $S(F')$ is
transverse to $W'_b$, i.e. $b\in \pitchfork (F')$. 

Let $k$ be a natural number in $s(W)$. The equality
\begin{equation} \label{equ.taugenimfprimeisgenimf}
\tau \genim F' (k) = \genim F(k)
\end{equation}
holds.
As $0=\tau (a)$ and $1=\tau (b)$ are in $\genim F(k),$ it follows from
(\ref{equ.taugenimfprimeisgenimf}) that $a,b\in \genim F'(k).$
Also, $\genim F(k)$ is residual in $[0,1]$, as $F$ is a fold field.
By (\ref{equ.taugenimfprimeisgenimf}), this implies that
$\genim F'(k)$ is residual in $[a,b]$, $\tau$ being a homeomorphism.
We conclude that $F'$ is a fold field on $W'$. 

The function $\Ima F'(k)$ induces a unique ordering
$S(F'(k)) \cap W'(k)_a = \{ p'_1,\ldots, p'_m \}$
such that $\Ima F'(p'_i)<\Ima F'(p'_j)$ if and only if $i<j$,
and a unique ordering
$S(F'(k)) \cap W'(k)_b = \{ q'_1,\ldots, q'_l \}$
such that $\Ima F'(q'_i)<\Ima F'(q'_j)$ if and only if $i<j$.
With $p_i = \alpha (p'_i),$ $q_i = \alpha (q'_i),$ we have
$S(F(k))\cap W(k)_0 =  \{ p_1, \ldots, p_m \}$
(a set of cardinality $m$, since $\alpha$ is injective) and
$S(F(k))\cap W(k)_1 = \{ q_1, \ldots, q_l \}$
(a set of cardinality $l$). Thus both $\mbs (F(k))$ and $\mbs (F'(k))$
are morphisms $[m]\to [l]$ in $\Br$.
The ordering principle is preserved because
\[ \Ima F'(p'_i)=\Ima (F\alpha)(p'_i)=\Ima F(p_i),~
  \Ima F'(q'_i)=\Ima (F\alpha)(q'_i)=\Ima F(q_i), \]
so  that $\Ima F(p_i) < \Ima F(p_j)$ if and only if $i<j$
if and only if $\Ima F(q_i) < \Ima F(q_j)$.
In the Brauer morphism $\mbs (F(k))$, the point $(0,i,0,0)$ is
connected to the point $(0,j,0,0)$ by an arc if and only if
there exists a connected component $c=\alpha (c')$ of
$S(F(k))$ with boundary $\partial c = \{ p_i, p_j \}$.
Such a $c$ exists if and only if there is a connected component
$c'$ of $S((F\alpha)(k))$ with $\partial c' = \{ p'_i, p'_j \}$,
which is equivalent to $(0,i,0,0)$ and $(0,j,0,0)$
being connected by an arc in $\mbs (F'(k))$.
Similarly for connections of the type
$(1,i,0,0)$ to $(1,j,0,0)$ and $(0,i,0,0)$ to $(1,j,0,0)$.
Finally, the number of loop tensor factors in $\mbs (F(k))$
equals the number of loop tensor factors in $\mbs (F'(k))$,
as the number of closed connected components of
$S(F(k))$ equals the number of closed connected components
of $S(F'(k))$. This shows that $\mbs (F(k))=\mbs (F'(k))$ for every $k$.
Therefore,
$\mbs (F) = \bigotimes_k \mbs (F(k)) = \bigotimes_k \mbs (F'(k))
  = \mbs (F').$
\end{proof}

Let $W^n$ be a cobordism from $M$ to $N$. We shall define an element,
called \emph{state sum}, 
\[ Z_W \in Z(M)\hotimes Z(N). \]
If $M$ and $N$ are well-separated, then we may regard the state sum
as an element
\[ Z_W \in Z(\partial W) = Z(M\sqcup N) \]
by Proposition \ref{prop.zmonoidalonmodules}.
Let $\epsilon_W >0$ be a cylinder scale of $W$.
Given a boundary condition
$(f_M, f_N)\in \Fa (M)\times \Fa (N),$ we put
\[ \Fa (W; f_M, f_N) = \{ F\in \Fa(W) ~|~ \exists
 \epsilon (k), \epsilon'(k) \in (0,\epsilon_W): \hspace{2.5cm} \]
\[ \hspace{2.7cm}   F|_{[0,\epsilon(k)] \times M(k)}\approx f_M (k),~ 
   F|_{[1-\epsilon'(k),1]\times N(k)}\approx f_N (k), \forall k \}. \]
On $(f_M, f_N)$ the state sum $Z_W$ is then given by
\begin{equation} \label{equ.statesumdef} 
Z_W (f_M, f_N) := \sum_{F\in \Fa (W;f_M, f_N)} Y\mbs (F). 
\end{equation}
This sum uses the infinite summation law of the complete monoid $Q$ and thus
yields a well-defined element of $Q$.
\begin{remark}
This sum replaces in our context the notional path integral
\[ \int_{F\in \Fa (W,f)} e^{iS_W (F)} d\mu_W \]
used in classical quantum field theory. As a mathematical object,
this path integral is problematic, since in many situations of interest, an
appropriate measure $\mu_W$ has not been defined or is known not to exist.
The present paper utilizes the notion of completeness in semirings to bypass
measure theoretic questions on spaces of fields.
\end{remark}

\subsection{The State Sum of a Disjoint Union}

Let $W\subset [0,1] \times 0 \times \real^{D-1}$ be a cobordism from $M$ to $N$ 
and $W' \subset [0,1] \times 1 \times \real^{D-1}$ a disjoint cobordism from
$M'$ to $N'$. The disjoint union $W\sqcup W'$ possesses a state sum
$Z_{W\sqcup W'} \in Z(M\sqcup M')\hotimes Z(N\sqcup N').$
Since $W$ and $W'$ are well-separated, $M$ and $M'$ are well-separated
and $N$ and $N'$ are well-separated. Thus, by Proposition \ref{prop.zmonoidalonmodules},
there are isomorphisms 
$Z(M\sqcup M')\cong Z(M)\hotimes Z(M')$ and
$Z(N\sqcup N')\cong Z(N)\hotimes Z(N'),$ which combine to an isomorphism
\[ \rho: Z(M\sqcup M')\hotimes Z(N\sqcup N') \stackrel{\cong}{\longrightarrow} 
  Z(M)\hotimes Z(M')\hotimes Z(N)\hotimes Z(N') \]
given by
\[ \rho (z)(f_M, f_{M'}, f_N, f_{N'}) = z(f_M \sqcup f_{M'}, f_N \sqcup f_{N'}) \]
on $z:\Fa (M\sqcup M')\times \Fa (N\sqcup N')\to Q,$
$f_M\in \Fa (M),$ $f_N \in \Fa (N),$ $f_{M'}\in \Fa (M')$
and $f_{N'} \in \Fa (N')$. 

\begin{thm} \label{thm.statesumdisjunion}
Let $W\subset [0,1] \times 0 \times \real^{D-1}$ be a cobordism from $M$ to $N$ 
and $W' \subset [0,1] \times 1 \times \real^{D-1}$ a disjoint cobordism from
$M'$ to $N'$. Assume without loss of generality that $W$ and $W'$ are
given equal cylinder scales $\epsilon_{W} = \epsilon_{W'}$.
(If they are not equal, replace the larger one by the smaller one.)
Equip the disjoint union $W\sqcup W'$ with the cylinder scale
$\epsilon_{W\sqcup W'} = \epsilon_W = \epsilon_{W'}$.
Then the state sum of the disjoint union can be
computed as
\[ \rho (Z_{W\sqcup W'}) = Z_W \hotimes_m Z_{W'} \in Z(M)\hotimes Z(M')
  \hotimes Z(N)\hotimes Z(N'). \]
\end{thm}
\begin{proof}
It is convenient to set $\Fa (M_1,\ldots,M_m) = \Fa (M_1)\times \cdots
   \times \Fa (M_m)$
for closed $(n-1)$-manifolds $M_1, \ldots, M_m$.
On a boundary condition 
\[ (f_M, f_{M'}, f_N, f_{N'}) \in \Fa (M,M',N,N'), \]
the monoidal tensor product of the state sums $Z_{W}$ and $Z_{W'}$ is given by
\begin{eqnarray*}
 (Z_{W} \hotimes_m Z_{W'})  (f_M, f_{M'}, f_N, f_{N'}) & = &
   Z_{W} (f_M, f_N) \times Z_{W'}(f_{M'}, f_{N'}) \\
& = & \Big\{ \sum_{F\in \Fa (W; f_M, f_N)} Y\mbs (F) \Big\} \times
     \Big\{ \sum_{F'\in \Fa (W'; f_{M'}, f_{N'})} Y\mbs (F') \Big\}. 
\end{eqnarray*}
By equation (\ref{equ.sumoverproductfamily}) in Section \ref{sec.monssemirings}, the product of the sums 
equals the sum of the products,
\[ (Z_{W} \hotimes_m Z_{W'}) (f_M, f_{M'}, f_N, f_{N'}) =
   \sum_{(F,F') \in \Fa (W; f_M, f_N)\times \Fa (W'; f_{M'}, f_{N'})} 
   Y\mbs (F) \times Y\mbs (F'). \]
With $[m]=\dom \mbs (F),$ $[n]=\cod \mbs (F),$ 
$[r] = \dom \mbs (F')$ and $[s] = \cod \mbs (F')$, the relevant product in 
evaluating $Y\mbs (F)\times Y\mbs (F')$ is of type (\ref{equ.prodqhmpqhpnmon}),
\[ Q(H_{m,n})\times Q(H_{r,s})\longrightarrow Q(H_{m+r,n+s}). \]
By Lemma \ref{lem.prodmppnrelaxedmon} and since $Y$ is a monoidal functor,
\[ (Y\mbs (F)\otimes 1)\times (Y\mbs (F')\otimes 1) =
  (Y\mbs (F)\otimes Y\mbs (F'))\otimes (1\cdot 1) =
  Y(\mbs (F)\otimes \mbs (F'))\otimes 1. \]
Let $\mathfrak{S}(W,W')$ be the set
\begin{eqnarray*}
\mathfrak{S}(W,W')
& = & \{ \mbs (F)\otimes \mbs (F') ~|~ F \in \Fa (W; f_M,f_N),~
  F' \in \Fa (W'; f_{M'},f_{N'}) \} \\
& = & \{ \mbs(F)\otimes \mbs(F') ~|~ F\in \Fa (W),~ F' \in \Fa (W'):\\
& &           \exists \epsilon (0),\epsilon' (0) \in (0,\epsilon_{W}), 
                        \epsilon (1),\epsilon' (1) \in (0,\epsilon_{W'}), \\
& &         F|_{[0,\epsilon (0)]\times M} \approx f_M,~ 
              F|_{[1-\epsilon' (0),1]\times N} \approx f_N, \\
& &        F'|_{[0,\epsilon (1)]\times M'} \approx f_{M'},~ 
             F'|_{[1-\epsilon' (1),1]\times N'} \approx f_{N'} \}.
\end{eqnarray*}
The ($\rho$-image of the) state sum of the disjoint union $W\sqcup W'$ is given on
$(f_M, f_{M'}, f_N, f_{N'})$ by
\[ \rho (Z_{W\sqcup W'})(f_M, f_{M'}, f_N, f_{N'}) =
   Z_{W\sqcup W'} (f_M \sqcup f_{M'}, f_N \sqcup f_{N'}) =
\sum_{H\in \Fa (W\sqcup W';f_M \sqcup f_{M'}, f_N \sqcup f_{N'})}
  Y\mbs (H). \]
Let $\mathfrak{S}(W\sqcup W')$ be the set
\begin{eqnarray*}
\mathfrak{S}(W\sqcup W')
& = & \{ \mbs (H) ~|~ H\in \Fa (W\sqcup W';f_M \sqcup f_{M'}, f_N \sqcup f_{N'}) \} \\
& = & \{ \mbs(H) ~|~ H\in \Fa (W\sqcup W'):\\
& &           \exists \epsilon_+ (0),\epsilon_+ (1), 
                        \epsilon'_+ (0),\epsilon'_+ (1) \in (0,\epsilon_{W\sqcup W'}): \\
& &         H|_{[0,\epsilon_+ (0)]\times M} \approx f_M,~ 
              H|_{[1-\epsilon'_+ (0),1]\times N} \approx f_N, \\
& &        H|_{[0,\epsilon_+ (1)]\times M'} \approx f_{M'},~ 
             H|_{[1-\epsilon'_+ (1),1]\times N'} \approx f_{N'} \}.
\end{eqnarray*}
Let us prove that $\mathfrak{S} (W\sqcup W')
\subset \mathfrak{S}(W,W')$.
Given $\mbs (H)$ in $\mathfrak{S} (W\sqcup W'),$
we put
\[ \begin{array}{lll}
F = H|_W, & \epsilon (0) = \epsilon_+ (0), & \epsilon' (0)=\epsilon'_+ (0), \\
F' = H|_{W'}, & \epsilon (1) = \epsilon_+ (1), & \epsilon' (1)=\epsilon'_+ (1). 
\end{array} \]
Then
\[ \begin{array}{ll}
\epsilon (0) = \epsilon_+ (0) < \epsilon_{W\sqcup W'} = \epsilon_W, &
\epsilon' (0) = \epsilon'_+ (0) < \epsilon_{W\sqcup W'} = \epsilon_W, \\
\epsilon (1) = \epsilon_+ (1) < \epsilon_{W\sqcup W'} = \epsilon_{W'}, &
\epsilon' (1) = \epsilon'_+ (1) < \epsilon_{W\sqcup W'} = \epsilon_{W'}. 
\end{array} \]
The map $F$ is a fold field on $W$, since $F=H(0)$ and
$H(0) \in \Fa ((W\sqcup W')(0))=\Fa (W)$.
Similarly, $F'$ is a fold field on $W'$, since $F'=H(1)$ and
$H(1) \in \Fa ((W\sqcup W')(1))=\Fa (W')$.
The correct boundary conditions are satisfied, since
\[ \begin{array}{ll}
F|_{[0,\epsilon (0)]\times M} = H|_{[0,\epsilon_+ (0)]\times M} \approx f_M, &
F|_{[1-\epsilon' (0),1]\times N} = H|_{[1-\epsilon'_+ (0),1]\times N} \approx f_N, \\
F'|_{[0,\epsilon (1)]\times M'} = H|_{[0,\epsilon_+ (1)]\times M'} \approx f_{M'}, &
F'|_{[1-\epsilon' (1),1]\times N'} = H|_{[1-\epsilon'_+ (1),1]\times N'} \approx f_{N'}. 
\end{array} \]
Finally, the equation
$\mbs (H) = \mbs (H(0))\otimes \mbs (H(1)) = \mbs (F)\otimes \mbs (F')$
shows that $\mbs (H)\in \mathfrak{S} (W,W')$. 

Conversely, we shall show that $\mathfrak{S}(W,W')
\subset \mathfrak{S} (W\sqcup W')$.
Given $\mbs (F)\otimes \mbs (F')$ in $\mathfrak{S} (W,W'),$
we put
\[ \begin{array}{l}
H = F\sqcup F': W\sqcup W' \longrightarrow \cplx, \\
\epsilon_+ (0) = \epsilon (0), \epsilon_+ (1)=\epsilon (1), 
 \epsilon'_+ (0) = \epsilon' (0), \epsilon'_+ (1)=\epsilon' (1). 
\end{array} \]
Then
\[ \begin{array}{ll}
\epsilon_+ (0) = \epsilon (0) < \epsilon_W = \epsilon_{W\sqcup W'}, &
\epsilon_+ (1) = \epsilon (1) < \epsilon_{W'} = \epsilon_{W\sqcup W'}, \\
\epsilon'_+ (0) = \epsilon' (0) < \epsilon_W = \epsilon_{W\sqcup W'}, &
\epsilon'_+ (1) = \epsilon' (1) < \epsilon_{W'} = \epsilon_{W\sqcup W'}. 
\end{array} \]
As $F$ and $F'$ are fold fields on $W= (W\sqcup W')(0)$ and
$W' = (W\sqcup W')(1),$ respectively, and
$H(0)=F,$ $H(1)=F',$ the map $H$ is a fold field on $W\sqcup W'$.
The correct boundary conditions are satisfied, since
\[ \begin{array}{ll}
H|_{[0,\epsilon_+ (0)]\times M} = F|_{[0,\epsilon (0)]\times M}  \approx f_M, &
H|_{[1-\epsilon'_+ (0),1]\times N} = F|_{[1-\epsilon' (0),1]\times N}  \approx f_N, \\
H|_{[0,\epsilon_+ (1)]\times M'} = F'|_{[0,\epsilon (1)]\times M'} \approx f_{M'}, &
H|_{[1-\epsilon'_+ (1),1]\times N'} = F'|_{[1-\epsilon' (1),1]\times N'}  \approx f_{N'}. 
\end{array} \]
Finally, the equation
$\mbs (H) = \mbs (H(0))\otimes \mbs (H(1)) = \mbs (F)\otimes \mbs (F')$
shows that $\mbs (F) \otimes \mbs (F') \in \mathfrak{S} (W \sqcup W')$. 

Consequently, $\mathfrak{S}(W\sqcup W') =
\mathfrak{S}(W,W')$ and, applying the functor $Y$,
\[ Y(\mathfrak{S}(W\sqcup W')) = Y(\mathfrak{S}(W,W')). \]
By Proposition \ref{prop.qcontinuous}, the idempotent semiring $Q^m =(Q,+,\times)$ is continuous.
Thus we can conclude from Proposition \ref{prop.contidemsemiring}
that
\[  \sum_{H\in \Fa (W\sqcup W'; f_M \sqcup f_{M'}, f_N \sqcup f_{N'})} Y\mbs (H) =
   \sum_{(F,F') \in \Fa (W; f_M, f_N)\times \Fa (W'; f_{M'}, f_{N'})} 
   Y(\mbs (F) \otimes \mbs (F')). \]
\end{proof}

\subsection{The Gluing Theorem}

This section proves the Gluing Theorem \ref{thm.gluingmain}.
The proof falls naturally into two parts:
Let $W$ be a cobordism obtained from gluing a cobordism $W'$ along its
outgoing boundary $N$ to a cobordism $W''$ with incoming boundary $N$.
Then one must show that any fold field $H$ on $W$ can be broken up into a fold field
$F$ on $W'$ and a fold field $G$ on $W''$ in such a way that the Brauer morphism
of $H$ is the composition of the Brauer morphisms of $F$ and $G$, i.e.
$\mbs (H)=\mbs (G)\circ \mbs (F)$. This is done in Proposition \ref{prop.splithintofandg}.
Conversely, one must show that fold fields $F$ on $W'$ and $G$ on $W''$
that agree in a suitable sense near $N$, can be ``glued'' to a fold field $H$ on $W$
such that $\mbs (H) = \mbs (G)\circ \mbs (F)$. This is the content of Proposition
\ref{prop.gluefandgtogeth}. Throughout these processes, the boundary conditions
on $\partial W$ must remain unchanged.

\begin{prop} \label{prop.splithintofandg}
Let $W'$ be a cobordism from $M$ to $N$ with cylinder scale $\epsilon_{W'}$
and $W''$ a cobordism from
$N$ to $P$ with cylinder scale $\epsilon_{W''}$.
Let $W=W' \cup_N W''$ be the cobordism from $M$ to $P$
obtained by gluing $W'$ and $W''$ along $N$ and let $\epsilon_W =
\smlhf \min (\epsilon_{W'}, \epsilon_{W''})$ be its natural cylinder scale. 
Given a fold field $H: W\to \cplx$ with $\epsilon (a), \epsilon'(a)\in
(0,\epsilon_{W})$ and $h_M \in \Fa (M)$ and $h_P \in \Fa (P)$
such that $H|_{[0,\epsilon (a)]\times M(a)} \approx h_M (a),$
$H|_{[1-\epsilon' (a),1]\times P(a)} \approx h_P (a)$ for all $a\in s(W)$,
there exist fold fields
$F: W' \to \cplx$ and $G: W'' \to \cplx$ 
and $\epsilon_1 (a), \epsilon'_1 (a) \in (0,\epsilon_{W'})$, 
$\epsilon_2 (a), \epsilon'_2 (a) \in (0,\epsilon_{W''}),$ 
$u\in \Fa (N)$
such that \\
(1) $F|_{[0,\epsilon_1 (a)]\times M(a)} \approx H|_{[0,\epsilon_1 (a)]\times M(a)},$ 
$G|_{[1-\epsilon'_2 (a),1]\times P(a)} \approx H|_{[1-\epsilon'_2 (a),1]\times P(a)}$, \\
(2) $F|_{[1-\epsilon'_1 (a),1]\times N(a)} \approx u(a) \approx 
   G|_{[0,\epsilon_2 (a)]\times N(a)}$, and \\
(3) $\mbs (G)\circ \mbs(F) = \mbs(H) \in \operatorname{Mor}(\Br)$, \\
for all $a\in s(W)$.
\end{prop}
\begin{proof}
Let $H$ be a fold field on $W$ and $\epsilon (a), \epsilon'(a) \in
(0,\epsilon_W)$ such that $H|_{[0,\epsilon (a)]\times M(a)} \approx 
h_M (a),$ $H|_{[1-\epsilon'(a),1]\times P(a)} \approx 
h_P (a)$ for all $a\in s(W)$.
As $W'$ and $W''$ are both cylindrical near $N$,
the glued cobordism $W$ is cylindrical over a neighborhood of $t=\smlhf$,
that is, there exists a small $\kappa >0$ such that 
$\max_a \epsilon (a) \leq \smlhf -\kappa,$ $\smlhf +\kappa \leq 1- \max_a \epsilon' (a),$
the preimage of the
interval $[\smlhf -\kappa, \smlhf +\kappa]$ under the first coordinate
function $\omega: W\to [0,1]$ is the band
\[ B= \omega^{-1} [\smlhf -\kappa, \smlhf +\kappa] =
  [\smlhf -\kappa, \smlhf +\kappa] \times N, \]
and the restriction $\omega|_B:B\to [\smlhf -\kappa, \smlhf +\kappa]$ is 
given by $\omega|_B (t,x)=t,$ $t\in [\smlhf -\kappa, \smlhf +\kappa],$
$x\in N$. In fact we can and will take $\kappa = \epsilon_W = \smlhf 
\min (\epsilon_{W'}, \epsilon_{W''}).$

Recall that $s(W)$ is the set of all $a\in \nat$ such that the slice $W(a)$
is nonempty. For every $a\in s(W),$ the set
$\pitchfork H(a)\cap \genim H(a)$ is residual in $[0,1]$ by Remark \ref{rem.giequivtransplusgi}.
Thus the finite intersection
\[ \bigcap_{a\in s(W)} (\pitchfork H(a)\cap \genim H(a)) \]
is residual, and hence dense, in $[0,1]$. Therefore,
there exists a value
\[ t_0 \in (\smlhf -\kappa, \smlhf +\kappa) \cap
    \bigcap_{a\in s(W)} (\pitchfork H(a)\cap \genim H(a)). \]
In particular, $t_0 \in \pitchfork (H)$.
By Corollary \ref{cor.pitchfresidual}, $\pitchfork (H)$ is open in $[0,1]$.
Thus there is a $\delta_0 >0$ such that
$[t_0 - \delta_0, t_0 + \delta_0] \subset \pitchfork (H)\cap
(\smlhf -\kappa, \smlhf +\kappa).$ Using Lemma \ref{lem.transregu},
$t_0 \in \pitchfork (H) = \reg (\omega|_{S(H)})\cap \reg (W),$ so
$t_0$ is a regular value of $\omega|_{S(H)}$. Thus at every point
$x\in S(H)\cap W_{t_0} = S(H)\cap t_0 \times N,$ the differential
$D_x (\omega|_{S(H)})$ is onto, hence an isomorphism.
By the inverse function theorem, $\omega|_{S(H)}$ is a local
diffeomorphism near every $x\in S(H)\cap t_0 \times N.$
As $S(H(a))$ and $W(a)_{t_0}$ are transverse and of complementary
dimension in $W$, and by compactness, their intersection is a finite set
\[ S(H(a))\cap t_0 \times N(a) = \{ p^a_1,\ldots, p^a_{k(a)} \},~
  p^a_i \not= p^a_j \text{ for } i\not= j. \]
Let us write $z^a_i = H(p^a_i),$ $i=1,\ldots, k(a),$ for the images.
Since $t_0 \in \genim H(a),$ the function
$\Ima \circ H(a)|: S(H(a))\cap t_0 \times N(a) \to \real$ is injective, that is,
$\Ima z^a_i \not= \Ima z^a_j$ whenever $i\not= j$. Thus the number
\[ \alpha (a)= \min \{ | \Ima (z^a_i) - \Ima (z^a_j)| ~:~ 1\leq i<j \leq k(a) \} \]
is positive. We may assume that the indexing of the $p^a_i$
has been chosen so that $\Ima z^a_i < \Ima z^a_j$ iff $i<j$.
As $H$ is continuous, there exists for every $i$ a small
open neighborhood $U^a_i$ of $p^a_i$ such that
$|H(a)(x)- z^a_i |<\alpha (a)/4$ for all $x\in U^a_i$. We may assume that
the $U^a_i$ have been arranged so that they do not overlap,
$U^a_i \cap U^a_j =\varnothing$ for $i\not= j$. Since $\omega|_{S(H)}$ is a
local diffeomorphism near every $p^a_i$, there is an open neighborhood
$V^a_i \subset S(H)$ of $p^a_i$ and a corresponding open neighborhood
$\hat{V}^a_i \subset (t_0 - \delta_0, t_0 + \delta_0)$ of $t_0 = \omega (p^a_i)$
such that $\omega|: V^a_i \to \hat{V}^a_i$ is a diffeomorphism and $V^a_i \subset U^a_i$.
The finite intersection 
$\bigcap_{a\in s(W)} (\hat{V}^a_1 \cap \cdots \cap \hat{V}^a_{k(a)})$ is open and contains
$t_0$. Consequently, there exists a $\delta >0$ with
\[ [t_0 - \delta, t_0 + \delta]\subset 
  \bigcap_{a\in s(W)} (\hat{V}^a_1 \cap \cdots \cap \hat{V}^a_{k(a)}). \]
Note that
\[ [t_0 -\delta, t_0 + \delta] \subset (t_0 -\delta_0, t_0 + \delta_0 )
  \subset~ \pitchfork (H)\cap (\smlhf - \kappa, \smlhf + \kappa). \]
We put
\[ \wlh = W\cap [0,\smlhf]\times \real^{D},~
   \wgh = W\cap [\smlhf, 1]\times \real^{D} \]
and
\[ \hat{W}' = W\cap [0,t_0 + \delta]\times \real^{D},~
   \check{W}'' = W\cap [t_0 - \delta, 1]\times \real^{D}. \]
The diffeomorphism $[0,1]\times \real^{D} \to [0,\smlhf]
\times \real^{D}$ given by $(t,x)\mapsto (t/2,x)$ restricts to a
diffeomorphism $\phi_1: W' \stackrel{\cong}{\longrightarrow}
\wlh.$ Since $\smlhf -\kappa < t_0 +\delta,$ there exists a
diffeomorphism $\lambda_1: [0,\smlhf] \stackrel{\cong}{\longrightarrow}
[0, t_0 +\delta]$ with $\lambda_1 (t)=t$ for $t\in [0, \smlhf - \kappa]$.
As $W\cap [\smlhf -\kappa, \smlhf +\kappa]\times \real^{D} =
[\smlhf -\kappa, \smlhf +\kappa] \times N,$ the diffeomorphism
$[0,\smlhf]\times \real^{D} \to [0, t_0 +\delta]
\times \real^{D}$ given by $(t,x)\mapsto (\lambda_1 (t),x)$ restricts to a
diffeomorphism $\overline{\lambda}_1: \wlh \stackrel{\cong}{\longrightarrow}
\hat{W}'.$ Let $\psi_1: W' \stackrel{\cong}{\longrightarrow} \hat{W}'$ be the
composition $\psi_1 = \overline{\lambda}_1 \phi_1$. Define
$F: W' \to \cplx$ to be the composition
\[ \xymatrix{
W' \ar[rd]_F \ar[r]^{\cong}_{\psi_1} & \hat{W}' \ar[d]^{H|} \\
& \cplx.
} \]

Given $a\in s(W),$
we verify that there is an $\epsilon_1 (a) \in (0,\epsilon_{W'})$ such that
$F|_{[0,\epsilon_1 (a)]\times M(a)} \approx
H|_{[0,\epsilon (a)]\times M(a)}$ for a suitable $\epsilon_1 (a)$.
Take $\epsilon_1 (a)= 2\epsilon (a)$ and consider the diffeomorphism
$\mu (a): [0,\epsilon_1 (a)]\to [0,\epsilon (a)],$ $\mu (a)(t)=t/2$.
Note that this $\epsilon_1 (a)$ is an admissible choice, since
\[ 0< \epsilon_1 (a) = 2\epsilon (a) < 2\epsilon_W = \min 
 (\epsilon_{W'}, \epsilon_{W''})\leq \epsilon_{W'}. \]
For $t\in [0,\epsilon_1 (a)]$, we have $\frac{t}{2} \leq \smlhf -\kappa$,
for $\epsilon (a) \leq \smlhf -\kappa$ by our choice of $\kappa$. Hence
for such $t$ and $x\in M(a),$
\[ F(t,x)=H\overline{\lambda}_1 \phi_1 (t,x) = H\overline{\lambda}_1 
  (\smlhf t, x) = H(\smlhf t,x)=H(\mu (a)(t),x). \]
This establishes the claim. \\

Next, we shall show that $F:W' \to \cplx$ is a fold field.
First, $F$ is a fold map: The restriction of $H$ to $\hat{W'}$ is a fold
map by restricting the fold charts for $H$ to $\hat{W}'$. Then the
composition of a diffeomorphism with a fold map is again a fold map
because one can use the diffeomorphism to transport the fold
charts from its codomain to its domain. 

Fix an $a\in s(W)$.
We check that $0,1 \in \pitchfork F(a)\cap \genim F(a).$
As $\psi_1$ is a diffeomorphism, we have
$\psi_1 S((H\psi_1)(a)) = S(H(a)|_{\hat{W}'}).$ Since $0\in \pitchfork H(a),$
the singular set $S(H(a))$ is transverse to $0\times M(a)$.
Hence $\psi_1 S((H\psi_1)(a)) \pitchfork \psi_1 (0\times M(a)),$ as
$\psi_1 (0\times M(a))=0\times M(a).$ This implies
$S((H\psi_1)(a))\pitchfork 0\times M(a)$, i.e. $S(F(a))\pitchfork 0\times M(a)$.
Therefore $0\in \pitchfork F(a)$.
Since $t_0 +\delta \in \pitchfork (H)$, the manifold $S(H(a))$ is
transverse to $W(a)_{t_0 +\delta} = (t_0 +\delta)\times N(a).$
Then $\psi_1 S((H\psi_1)(a))\pitchfork \psi_1 (1\times N(a)),$ as
$\psi_1 (1\times N(a))=(t_0 +\delta)\times N(a)$.
Hence $S(F(a))\pitchfork 1\times N(a)$ and so $1\in \pitchfork F(a)$.\\

The function $\Ima \circ H(a)|: S(H(a))\cap 0\times M(a)\to \real$ is injective, for
we know that $0\in \genim H(a)$. The diffeomorphism $\psi_1$
induces a bijection
\[ S(F(a))\cap 0\times M(a) \stackrel{\psi_1|}{\longrightarrow}
 \psi_1 S(F(a))\cap \psi_1 (0\times M(a)) = S(H(a))\cap 0\times M(a). \]
Thus $\Ima \circ F(a)|_{S(F(a))\cap 0\times M(a)} =
\Ima \circ (H \circ \psi_1)(a)|_{S(F(a))\cap 0\times M(a)}$ is injective,
which shows that $0\in \genim F(a)$. In order to prove that
$1\in \genim F(a)$, our immediate next objective is to demonstrate
that $t_0 +\delta \in \genim H(a)$. As $t_0 + \delta \in \hat{V}^a_i$ and
$\omega|: V^a_i \to \hat{V}^a_i$ is a diffeomorphism, there exists a
unique point $q_i \in V^a_i$ with $\omega (q_i)=t_0 +\delta.$
Since $V^a_i \cap V^a_j =\varnothing$ for $i\not= j$, we have
$q_i \not= q_j$ for $i\not= j$. We obtain a subset
\[ \{ q_1,\ldots, q_k \} \subset S(H(a))\cap (t_0 +\delta)\times N(a) \]
of cardinality $k=k(a)$. We claim that in fact these two sets are equal,
which can be seen as follows. Since
\[ [t_0 -\delta, t_0 +\delta]\subset~ \pitchfork H(a) = \reg (\omega|_{S(H(a))})
   \cap \reg W(a), \]
the restriction
\[ \omega|_{S(H(a))}: S(H(a))\cap [t_0 -\delta, t_0 +\delta]\times N(a) \longrightarrow
   [t_0 -\delta, t_0 +\delta] \]
has no critical points. By a standard result proven in treatments of
Morse theory (cf. \cite[Theorem 6.2.2]{hirsch}),
there exists a diffeomorphism
\[ d: (S(H(a))\cap t_0 \times N(a))\times [t_0 -\delta, t_0 +\delta]
  \stackrel{\cong}{\longrightarrow}
  S(H(a))\cap ([t_0 -\delta, t_0 +\delta]\times N(a)) \]
such that
\begin{equation} \label{equ.diffeod}
\xymatrix@C=3pt{
(S(H(a))\cap t_0 \times N(a))\times [t_0 -\delta, t_0 +\delta] \ar[rr]^d
\ar[rd]_{\operatorname{proj}_2} & &
S(H(a))\cap ([t_0 -\delta, t_0 +\delta]\times N(a))
\ar[ld]^{\omega|_{S(H(a))}} \\
& [t_0 -\delta, t_0 +\delta] &
} \end{equation}
commutes and $d$ is the identity on $(S(H(a))\cap t_0 \times N(a))\times \{ t_0 \}.$
In particular, the level sets of $\omega|_{S(H(a))}$
over $[t_0 -\delta, t_0 +\delta]$ are all diffeomorphic to each other.
So $S(H(a))\cap (t_0 +\delta)\times N(a)$ must have the same cardinality
as $S(H(a))\cap t_0 \times N(a),$ namely $k$. But since
$\{ q_1, \ldots, q_k \}$ already has cardinality $k$, we must have
\[ \{ q_1,\ldots, q_k \} = S(H(a))\cap (t_0 +\delta)\times N(a). \]
In order to establish $t_0 +\delta \in \genim H(a),$ we must
show that $\Ima \circ H(a)|: S(H(a))\cap (t_0 +\delta)\times N(a) \to \real$
is injective, that is, $\Ima H(a)(q_i)\not= \Ima H(a)(q_j)$ for $i\not= j$.
As $q_i \in V^a_i \subset U^a_i,$ we have the estimate
\[ | \Ima H(a)(q_i) - \Ima z^a_i | \leq |H(a)(q_i) - z^a_i | < \frac{\alpha (a)}{4}. \]
We bound $| \Ima H(a)(q_i) - \Ima H(a)(q_j)|$ away
from zero by arguing
\begin{equation*}
\begin{split}
| \Ima H(a)(q_i) - &\Ima H(a)(q_j)| \\
& = | \Ima H(a)(q_i) - \Ima z^a_i + \Ima z^a_i - \Ima z^a_j + \Ima z^a_j - \Ima H(a)(q_j)| \\
& \geq |\Ima z^a_i - \Ima z^a_j| - | \Ima H(a)(q_i) - \Ima z^a_i| - |\Ima H(a)(q_j) - \Ima z^a_j| \\
&  \geq \alpha (a) - \frac{\alpha (a)}{4} - \frac{\alpha (a)}{4} = \frac{\alpha (a)}{2} >0.
\end{split}
\end{equation*}
Thus $t_0 +\delta \in \genim H(a).$ The diffeomorphism $\psi_1$ induces
a bijection
\[ S(F(a))\cap 1\times N(a) \stackrel{\psi_1|}{\longrightarrow}
 \psi_1 S(F(a))\cap \psi_1 (1\times N(a)) = S(H(a))\cap (t_0 +\delta)\times N(a). \]
Hence $\Ima \circ F(a)|_{S(F(a))\cap 1\times N(a)} = \Ima \circ (H\circ
 \psi_1)(a)|_{S(F(a))\cap 1\times N(a)}$ is injective and we conclude that
$1\in \genim F(a)$. 

We move on to show that $\genim F(a)$ is
residual in $[0,1]$. Let
$h_1: [0,1]\to [0,\smlhf]$ denote the diffeomorphism
$h_1 (t)=\frac{t}{2}$. We shall use the commutative diagram
\begin{equation} \label{dia.wwwomomom} 
\xymatrix{
W' \ar@/^2pc/[rr]^{\psi_1}
\ar[r]^{\phi_1}_{\cong} \ar[d]_{\omega} &
  \wlh \ar[r]^{\overline{\lambda}_1}_{\cong} \ar[d]_{\omega} &
  \hat{W}' \ar[d]_{\omega} \\
[0,1] \ar[r]^{\cong}_{h_1} & [0,\smlhf] \ar[r]^{\cong}_{\lambda_1} &
  [0, t_0 +\delta]. 
} 
\end{equation}
We have $\lambda_1 h_1 \genim F(a) = \genim (H(a)|_{\hat{W}'}).$
As $H$ is a fold field on $W$, $\genim H(a)$ is residual in
$[0,1]$. By Lemma \ref{lem.restrofresidual},
$\genim H(a)\cap (0, t_0 +\delta)$ is residual in 
$(0,t_0 +\delta)$. Thus $\genim H(a)\cap [0, t_0 +\delta]$ is residual in 
$[0,t_0 +\delta]$. Using the above formula,
\[
\genim H(a) \cap [0,t_0 +\delta] = \genim (H(a)|_{\hat{W}'}) 
 = \lambda_1 h_1 \genim F(a).
\]
Since $\lambda_1 h_1$ is a homeomorphism, $\genim F(a)$
is residual in $[0,1]$. This finishes the proof that $F\in \Fa (W').$ \\

In a similar manner, we shall construct a fold field $G$ on $W''$:
The diffeomorphism $[0,1]\times \real^{D} \to [\smlhf,1]
\times \real^{D}$ given by $(t,x)\mapsto (\smlhf (t+1),x)$ restricts to a
diffeomorphism $\phi_2: W'' \stackrel{\cong}{\longrightarrow}
\wgh.$ Since $t_0 -\delta < \smlhf + \kappa,$ there exists a
diffeomorphism $\lambda_2: [\smlhf,1] \stackrel{\cong}{\longrightarrow}
[t_0 -\delta,1]$ with $\lambda_2 (t)=t$ for $t\in [\smlhf + \kappa,1]$.
As $W\cap [\smlhf -\kappa, \smlhf +\kappa]\times \real^{D} =
[\smlhf -\kappa, \smlhf +\kappa] \times N,$ the diffeomorphism
$[\smlhf,1]\times \real^{D} \to [t_0 -\delta,1]
\times \real^{D}$ given by $(t,x)\mapsto (\lambda_2 (t),x)$ restricts to a
diffeomorphism $\overline{\lambda}_2: \wgh \stackrel{\cong}{\longrightarrow}
\check{W}''.$ Let $\psi_2: W'' \stackrel{\cong}{\longrightarrow} \check{W}''$ be the
composition $\psi_2 = \overline{\lambda}_2 \phi_2$. Define
$G: W'' \to \cplx$ to be the composition
\[ \xymatrix{
W'' \ar[rd]_G \ar[r]^{\cong}_{\psi_2} & \check{W}'' \ar[d]^{H|} \\
& \cplx.
} \] 

Let $a\in s(W)$. Taking 
$\epsilon'_2 (a)= 2\epsilon' (a) \in (0,\epsilon_{W''})$, one verifies easily that
$G|_{[1-\epsilon'_2 (a),1]\times P(a)} \approx
H|_{[1-\epsilon'(a)]\times P(a)}$.
Arguments analogous to the ones used to show that $F$ is a fold field
also show that $G$ is a fold field on $W''$. The key points are that
$t_0 -\delta \in \pitchfork (H)$ (whence $0\in \pitchfork (G)$), and
$t_0 -\delta \in \genim H(a)$ for every $a\in s(W)$ 
(whence $0\in \genim G(a)$ for all such $a$). \\

Let us construct the required $u\in \Fa (N)$. 
Set $J =[t_0 -\delta, t_0 +\delta]$,
let $\mu: [0,1] \stackrel{\cong}{\longrightarrow} J$ be the
diffeomorphism $\mu (t) = 2\delta t + t_0 -\delta$ and
$\overline{\mu}: [0,1]\times N \stackrel{\cong}{\longrightarrow}
J\times N$ the diffeomorphism $\overline{\mu}(t,x)=(\mu (t),x)$.
Define $u$ to be the composition
\[ \xymatrix{
[0,1]\times N \ar[r]^{\cong}_{\overline{\mu}} \ar[rd]_u &
  J\times N \ar[d]^{H|} \\
 & \cplx. 
} \]
Then by construction $H|_{J\times N} \approx u.$ With
$\epsilon'_1 (a)= 1- h^{-1}_1 \lambda^{-1}_1 (t_0 -\delta)$
(actually $\epsilon'_1 = \epsilon'_1 (a)$ is independent of $a$),
the diffeomorphism $\lambda_1 h_1: [0,1]\to [0,t_0 +\delta]$
restricts to a diffeomorphism
$(\lambda_1 h_1)|: [1-\epsilon'_1,1] \stackrel{\cong}{\longrightarrow} J$.
Then $\epsilon'_1 < \epsilon_{W'}$. 
The equation $F(t,x)=H(\lambda_1 h_1 (t),x)$ for
$t\in [1-\epsilon'_1,1],$ $x\in N,$ shows that
$F|_{[1-\epsilon'_1,1]\times N} \approx H|_{J\times N}.$
Consequently, $F|_{[1-\epsilon'_1,1]\times N} \approx u$.
Let $h_2: [0,1]\stackrel{\cong}{\longrightarrow} [\smlhf,1]$ be given
by $h_2 (t) = \smlhf (t+1)$.
With $\epsilon_2 (a)= h^{-1}_2 \lambda^{-1}_2 (t_0 +\delta)$
(actually $\epsilon_2 = \epsilon_2 (a)$ is independent of $a$),
the diffeomorphism $\lambda_2 h_2: [0,1]\to [t_0 -\delta,1]$
restricts to a diffeomorphism
$(\lambda_2 h_2)|: [0,\epsilon_2] \stackrel{\cong}{\longrightarrow} J$.
We have $\epsilon_2 < \epsilon_{W''}$. 
The equation $G(t,x)=H(\lambda_2 h_2 (t),x)$ for
$t\in [0,\epsilon_2],$ $x\in N,$ shows that
$G|_{[0,\epsilon_2]\times N} \approx H|_{J\times N} \approx u$. 

Next, we argue that $u$ lies in $\Fa (N)$. The restriction of $H$
to $J\times N$ is a fold map. As $\overline{\mu}$ is a diffeomorphism,
the composition $u =H\overline{\mu}$ is a fold map as well.
The fact that $t_0 -\delta \in \pitchfork (H|_{J\times N})$ implies that
\[ 0 =\mu^{-1} (t_0 -\delta)\in \mu^{-1} \pitchfork (H|_{J\times N}) =
 \pitchfork (H\overline{\mu}) = \pitchfork (u), \]
while $t_0 +\delta \in \pitchfork (H)$ implies that $1\in \pitchfork (u)$.
Moreover, $t_0 -\delta \in \genim (H(a)|_{J\times N})$ implies 
\[ 0 =\mu^{-1} (t_0 -\delta)\in \mu^{-1} \genim (H(a)|_{J\times N}) =
 \genim (H\overline{\mu})(a) = \genim u(a), \]
and $t_0 +\delta \in \genim (H(a)|_{J\times N})$ implies $1\in \genim u(a)$.
By Lemma \ref{lem.restrofresidual},
\[
\mu \genim u(a) = \genim (H(a)|_{J\times N}) 
 = \genim H(a)\cap J
\]
is residual in $J = \mu [0,1]$. As $\mu$ is a homeomorphism,
$\genim u(a)$ is residual in $[0,1]$.
Thus $u$ is a fold field on $[0,1]\times N$.
Our next task is to show that the Brauer morphism $\mbs (u)$ is the identity. 
By Lemma \ref{lem.relapprox}, $\mbs (u)=\mbs (H|_{J\times N}).$
We claim that
\begin{equation} \label{equ.shintjncontuvi}
S(H(a))\cap J\times N(a) \subset \bigsqcup_i V^a_i.
\end{equation}
Suppose $x\in S(H(a))\cap J\times N(a)$ so that $t=\omega (x)$ lies in $J$.
Using the diffeomorphism $d$ of diagram (\ref{equ.diffeod}),
the set $S(H(a))\cap t\times N(a)$ is seen to have cardinality $k$.
For any $i,$ $t$ lies in $t\in J\subset \hat{V}^a_i$ and thus there exists a 
unique $x_i \in V^a_i$ such that $\omega (x_i)=t,$ as
$\omega|: V^a_i \to \hat{V}^a_i$ is a diffeomorphism. The subset
\[ \{ x_1, \ldots, x_k \} \subset S(H(a))\cap t\times N(a) \]
has cardinality $k$, as the $V^a_i$ are pairwise disjoint.
Since $S(H(a))\cap t\times N(a)$ has cardinality $k$, we actually have equality,
\[ \{ x_1, \ldots, x_k \} = S(H(a))\cap t\times N(a). \]
Consequently, there is an $i$ such that $x=x_i$. This shows that
$x\in V^a_i$ and proves the claim (\ref{equ.shintjncontuvi}).
Again using the diffeomorphism $d$, we set
\[ c^a_i = d(\{ p^a_i \} \times J). \]
Then $c^a_1, \ldots, c^a_k$ are the connected components of 
$S(H(a)|_{J\times N})=S(H(a))\cap J\times N(a),$ and since
$d$ is the identity on $(S(H(a))\cap t_0 \times N(a))\times \{ t_0 \},$
$c^a_i$ is the unique component that contains $p^a_i$. The boundary
$\partial c^a_i = \{ q_i, q'_i \}$ consists of the two endpoints
$q_i = d(p^a_i, t_0 +\delta)$ and $q'_i = d(p^a_i, t_0 -\delta)$.
As $c^a_i$ is contained in $S(H(a))\cap J\times N(a),$ it is by 
(\ref{equ.shintjncontuvi}) in particular contained in the disjoint
union $\bigcup_{j=1}^{k(a)} V^a_j$. Since $c^a_i$ is connected, it must
thus entirely lie in one $V^a_j$. As $p^a_i \in c^a_i \cap V^a_i,$ this
$V^a_j$ must equal $V^a_i$. We have shown that $c^a_i \subset V^a_i$
and in particular that $q_i, q'_i \in V^a_i \subset U^a_i$. We claim that
$\Ima H(a)(q_i) < \Ima H(a)(q_j)$ iff $i<j$. To see this, suppose that $i<j$.
Then $\Ima z^a_j - \Ima z^a_i >0$ and hence
$\Ima z^a_j - \Ima z^a_i = |\Ima z^a_j - \Ima z^a_i | \geq \alpha (a).$
Therefore,
\begin{eqnarray*}
\Ima H(a)(q_j) - \Ima H(a)(q_i) & = &
\Ima H(a)(q_j) - \Ima z^a_j + \Ima z^a_j - \Ima z^a_i + \Ima z^a_i - \Ima H(a)(q_i) \\
& \geq &
\alpha (a) - | \Ima H(a)(q_j) - \Ima z^a_j| - |\Ima H(a)(q_i) - \Ima z^a_i| \\
&  \geq & \alpha (a) - \frac{\alpha (a)}{4} - \frac{\alpha (a)}{4} = \frac{\alpha(a)}{2} >0,
\end{eqnarray*}
proving the claim. Using $q'_i \in U^a_i,$ one sees similarly that
$\Ima H(a)(q'_i) < \Ima H(a)(q'_j)$ iff $i<j$. Note that
$\{ q_1,\ldots, q_k \} = S(H(a)|_{J\times N})\cap (t_0 +\delta)\times N(a)$
and
$\{ q'_1,\ldots, q'_k \} = S(H(a)|_{J\times N})\cap (t_0 -\delta)\times N(a)$.
For every $i$, $c^a_i$ has boundary $\partial c^a_i = \{ q'_i, q_i \}$ and thus
$\mbs (H(a)|_{J\times N})$ connects $(0,i,0,0)$ to $(1,i,0,0)$ by an arc.
There are no loops, as $S(H(a)|_{J\times N})$ has no loops.
Therefore, $\mbs (u(a)) = \mbs (H(a)|_{J\times N}) = 1_{[k(a)]}$ and so
$\mbs (u) = \bigotimes_a \mbs (u(a)) = 1.$
This concludes the proof that $u$ lies in $\Fa (N)$. \\

It remains to verify the Brauer morphism equation $\mbs (G)\circ \mbs (F) = \mbs (H).$
By Lemma \ref{lem.singtanglereparam}, $\mbs (F) = \mbs (H|_{\hat{W}'})$
(using the diffeomorphism $\psi_1: W' \to \hat{W}'$) and
$\mbs (G) = \mbs (H|_{\check{W}''})$
(using the diffeomorphism $\psi_2: W'' \to \check{W}''$).
Set $\check{W}' = W\cap [0,t_0 -\delta]\times \real^{D}$.
As 
\[ t_0 - \delta \in (0,1)\cap \bigcap_{a\in s(W)} (\pitchfork H(a)\cap \genim H(a)), \]
Lemma \ref{lem.sfltsfgtissf} applies to yield 
$\mbs (H|_{\hat{W}'}) = \mbs (H|_{J\times N})\circ \mbs (H|_{\check{W}'})$
and $\mbs (H) = \mbs (H|_{\check{W}''})\circ \mbs (H|_{\check{W}'}).$ In summary,
\begin{eqnarray*}
\mbs (G)\circ \mbs (F) & = & \mbs (H|_{\check{W}''}) \circ \mbs (H|_{\hat{W}'}) 
 = \mbs (H|_{\check{W}''}) \circ  \mbs (H|_{J\times N})\circ \mbs (H|_{\check{W}'}) \\
& = & \mbs (H|_{\check{W}''}) \circ  1 \circ \mbs (H|_{\check{W}'}) 
  = \mbs (H),
\end{eqnarray*}
as required.
\end{proof}

\begin{lemma} \label{lem.techotherhalf}
Let $W'$ be a cobordism from $M$ to $N$ and $W''$ a cobordism from
$N$ to $P$. Let $W=W' \cup_N W''$ be the cobordism from $M$ to $P$
obtained by gluing $W'$ and $W''$ along $N$. Let $\epsilon_{W'},
\epsilon_{W''} >0$ be cylinder scales for $W'$ and $W''$, respectively. 
Let $F\in \Fa (W'),$ $G\in \Fa (W'')$ be fold fields 
such that there are $\epsilon'_1 (k) \in (0,\epsilon_{W'}),$
$\epsilon_2 (k) \in (0,\epsilon_{W''})$ and a map $u\in \Fa (N)$ with
\[ F|_{[1-\epsilon'_1 (k),1]\times N(k)} \approx u(k) \approx
  G|_{[0,\epsilon_2 (k)]\times N(k)} \]
for all $k\in s(W)$. Then
\[ \cod \mbs (F) = \dom \mbs (G). \]
\end{lemma}
\begin{proof}
Let $[d_k]=\dom \mbs (u(k))$ be the domain of $\mbs (u(k))$. Since $\mbs (u) = 1 \in \Mor (\Br),$ 
Lemma \ref{lem.idfactorstriv} implies that $\mbs (u(k))=1$ for all $k$ and thus
$\cod \mbs (u(k))=[d_k]$ as well.
By Lemma \ref{lem.relapprox},
$\mbs (F|_{[1-\epsilon'_1 (k),1]\times N(k)}) = \mbs (u(k)) = 
 \mbs (G|_{[0,\epsilon_2 (k)]\times N(k)}).$
Hence
\begin{eqnarray*}
\cod \mbs (F(k)) & = & [\card (S(F(k))\cap N)] =
 [\card (S(F|_{[1-\epsilon'_1 (k),1]\times N(k)})\cap 1\times N)] \\
& = & \cod \mbs (F|_{[1-\epsilon'_1 (k),1]\times N(k)}) = \cod \mbs (u(k)) 
 = [d_k] = \dom \mbs (u(k)) \\
& = & \dom \mbs (G|_{[0,\epsilon_2 (k)]\times N(k)}) =
 [\card (S(G|_{[0,\epsilon_2 (k)]\times N(k)})\cap 0\times N)] \\
& = & [\card (S(G(k))\cap N)] = \dom \mbs (G(k))
\end{eqnarray*}
and thus
\begin{eqnarray*}
\cod \mbs (F) & = & \cod \bigotimes_k \mbs (F(k)) =
   \bigotimes_k \cod \mbs (F(k)) = \bigotimes_k \dom \mbs (G(k)) \\
& = & \dom \bigotimes_k \mbs (G(k)) = \dom \mbs (G). 
\end{eqnarray*}
\end{proof}

\begin{prop} \label{prop.gluefandgtogeth}
Let $W'$ be a cobordism from $M$ to $N$ and $W''$ a cobordism from
$N$ to $P$. Assume without loss of generality that $W'$ and $W''$ are
given equal cylinder scales $\epsilon_{W'} = \epsilon_{W''}$.
(If they are not equal, replace the larger one by the smaller one.)
Let $W=W' \cup_N W''$ be the cobordism from $M$ to $P$
obtained by gluing $W'$ and $W''$ along $N$, equipped with
its natural cylinder scale
$\epsilon_W = \smlhf \min (\epsilon_{W'}, \epsilon_{W''}) =
  \smlhf \epsilon_{W'} = \smlhf \epsilon_{W''}.$
Let $F\in \Fa (W'),$ $G\in \Fa (W'')$ be fold fields such that there exist
$\epsilon_1 (k),\epsilon'_1 (k) \in (0,\epsilon_{W'}),$ $\epsilon_2 (k),\epsilon'_2 (k)\in (0,\epsilon_{W''}),$
and $h_M \in \Fa (M),$ $h_P \in \Fa (P),$ $u\in \Fa (N),$ with
\[ F|_{[0,\epsilon_1 (k)]\times M(k)} \approx h_M (k),~
 G|_{[1-\epsilon'_2 (k),1]\times P(k)} \approx h_P (k), \]
and
\[ F|_{[1-\epsilon'_1 (k),1]\times N(k)} \approx u(k) \approx G|_{[0,\epsilon_2 (k)]\times N(k)} \]
for all $k\in s(W)$.
Then there exists a fold field $H\in \Fa (W)$ and $\epsilon (k), \epsilon' (k)\in 
(0,\epsilon_W)$ with
\[ H|_{[0,\epsilon (k)]\times M(k)} \approx h_M (k),~
 H|_{[1-\epsilon' (k),1]\times P(k)} \approx h_P (k), \]
for all $k$ and
$\mbs (H) = \mbs (G)\circ \mbs (F).$
\end{prop}
\begin{proof}
Let $F\in \Fa (W'),$ $G\in \Fa (W'')$ be fold fields satisfying the hypotheses.
Set
\[ \wlh = W\cap [0,\smlhf]\times \real^{D},~
   \wgh = W\cap [\smlhf,1]\times \real^{D}. \]
The diffeomorphism
\[ 
[0,\smlhf]\times \real^{D} \stackrel{\cong}{\longrightarrow} 
[0,1]\times \real^{D},~
(t,x) \mapsto (2t,x),
\]
restricts to a diffeomorphism $\psi_1: \wlh \stackrel{\cong}{\longrightarrow}
W'$. Set $F^{1/2} = F\circ \psi_1: \wlh \to \cplx$.
The diffeomorphism
\[ 
[\smlhf,1]\times \real^{D} \stackrel{\cong}{\longrightarrow} 
[0,1]\times \real^{D},~
(t,x) \mapsto (2(t-\smlhf),x),
\]
restricts to a diffeomorphism $\psi_2: \wgh \stackrel{\cong}{\longrightarrow}
W''$. Set $G^{1/2} = G\circ \psi_2: \wgh \to \cplx$.
Since $F|_{[1-\epsilon'_1 (k),1]\times N(k)} \approx u(k),$ there exists a diffeomorphism
$\alpha (k): [1-\epsilon'_1 (k),1] \stackrel{\cong}{\longrightarrow} [0,1]$
such that $\alpha (k)(1-\epsilon'_1 (k))=0$ and
$F(t,x) = u(\alpha (k)(t),x)$ for $1-\epsilon'_1 (k) \leq t \leq 1$
and $x\in N(k).$
Define $\overline{\alpha}(k): [\smlhf(1-\epsilon'_1(k)),\smlhf]
\stackrel{\cong}{\longrightarrow} [0,1]$ by $\overline{\alpha}(k)(t)=
\alpha (k)(2t).$
As $G|_{[0,\epsilon_2 (k)]\times N(k)} \approx u(k),$ there exists a diffeomorphism
$\beta (k): [0,1] \stackrel{\cong}{\longrightarrow} [0,\epsilon_2 (k)]$ 
such that $\beta (k)(0)=0$ and
$u(t,x) = G(\beta (k)(t),x)$ for $0 \leq t \leq 1$ and $x\in N(k).$
It follows that
\[ F(k)(t,x) = G(k)(\beta(k)\circ \alpha (k)(t),x) \text{ for } 1-\epsilon'_1 (k) \leq t\leq 1
    \text{ and } x\in N(k). \]
Thus for $t\in [\smlhf (1-\epsilon'_1 (k)),\smlhf],$
\[ F^{1/2} (k)(t,x)=F(k)(2t,x)= G(k)(\beta (k) \alpha (k)(2t),x) =
  G(k)(\beta (k) \overline{\alpha}(k) (t),x). \]
Choose a number $\delta (k) \in (\epsilon_2 (k)/2, \epsilon_{W''}/2)$.
Let 
\[ \lambda (k): [\smlhf (1-\epsilon'_1 (k)), \smlhf (1+\epsilon_{W''})]
 \stackrel{\cong}{\longrightarrow} [0,\epsilon_{W''}] \]
be a diffeomorphism such that
\[ \lambda (k)(t) = \begin{cases} \beta (k) \overline{\alpha} (k)(t), &
 \smlhf (1-\epsilon'_1 (k))\leq t \leq \smlhf, \\
2(t-\smlhf), & \smlhf + \delta (k)\leq t \leq \smlhf (1+\epsilon_{W''}). 
\end{cases} \]
For $(t,x)\in W(k) \subset [0,1] \times \{ k \} \times \real^{D-1},$ set
\[ H(k)(t,x) = \begin{cases}
F^{1/2} (k)(t,x),& (t,x)\in \wlh (k) \\
G(k)(\lambda (k)(t),x), & (t,x)\in [\smlhf (1-\epsilon'_1 (k)), \smlhf (1+\epsilon_{W''})]\times
    \{ k \} \times \real^{D-1} \\
G^{1/2} (k)(t,x),& (t,x)\in \wgh (k)\cap [\smlhf + \delta (k),1]\times \real^{D}.
\end{cases} \]
and
$H = \bigsqcup_k H(k): W\longrightarrow \cplx.$
Then $H(k):W(k)\to \cplx$ is a smooth map since for
\[ (t,x)\in \wlh (k)\cap [\smlhf (1-\epsilon'_1 (k)), \smlhf (1+\epsilon_{W''})]\times
    \{ k \} \times \real^{D-1} = [\smlhf (1-\epsilon'_1 (k)),\smlhf]\times N(k), \]
we have
$G(k)(\lambda (k)(t),x) = G(k)(\beta (k) \overline{\alpha}(k)(t),x) =
  F^{1/2} (k)(t,x)$
and for
\[ (t,x)\in \wgh (k)\cap [\smlhf + \delta (k),1]\times \real^{D} \cap
 [\smlhf (1-\epsilon'_1 (k)), \smlhf (1+\epsilon_{W''})]\times
    \{ k \} \times \real^{D-1} \]
\[  = [\smlhf + \delta (k), \smlhf (1+\epsilon_{W''})]\times N(k) \]
we have
$G(k)(\lambda (k)(t),x)=G(k)(2(t-\smlhf),x)=G^{1/2} (k)(t,x).$
It follows that $H$ is smooth as well. 

We claim that $H$ is a fold field. We will write 
$G_\lambda (k)(t,x) = G(k)(\lambda (k)(t),x)$. If $A$ is a fold map and
$\Phi$ a diffeomorphism whose codomain is the domain of $A$, then
the composition $A\circ \Phi$ is again a fold map. Thus, $F^{1/2}$, 
$G_\lambda$ and $G^{1/2}$ are all fold maps.
Using local fold charts for these three maps, one sees that $H$ is
a fold map also. 

Let us prove that $S(H)$ is transverse to $W_0 =M$.
Note that $\psi_1 S(F^{1/2}) = S(F)$ and $\psi_1 (M)=M.$
As $F$ is a fold field on $W'$, we know that $S(F)$ is transverse to $M$,
which we can write as $\psi_1 S(F^{1/2})\pitchfork \psi_1 (M).$
Since $\psi_1$ is a diffeomorphism, this implies that
$S(F^{1/2})\pitchfork M$. Since on a neighborhood of $M=W_0$ in $W$,
$H$ is given by $F^{1/2},$ it follows that $S(H)\pitchfork M$.
In a similar manner, using $G^{1/2}$, we
see that $S(H)$ is transverse to $W_1 =P$. 

Let us proceed to verify that $0,1 \in \genim H(k)$ for every natural
number $k$.
As $\psi_1|_M$ is the identity on $M$, we have
$H|_M = F^{1/2}|_M = F|_M$ and $S(H)\cap M = S(F)\cap M$.
Since $F\in \Fa (W'),$ the function
$\Ima F(k)|: S(F(k))\cap M(k) \to \real$ is injective. But this function equals
$\Ima H(k)|: S(H(k))\cap M(k) \to \real$, which is thus also injective. This shows
that $0\in \genim H(k)$. 
Similarly, $1\in \genim H(k)$, using $\psi_2|_P =\id_P$ and $H|_P = G|_P$.

We prove that $\genim H(k)$ is residual in $[0,1]$ for $k\in s(W)$.
Setting
\[ \begin{array}{ccl}
R_0 & = & \genim F^{1/2}(k)\cap (0,\smlhf), \\
R_\delta & = & \genim G_\lambda (k)\cap (\smlhf,\smlhf + \delta (k)), \\
R_1 & = & \genim G^{1/2}(k)\cap (\smlhf + \delta (k),1), \\
\end{array} \]
we have
\begin{equation} \label{equ.r0rdelr1inphgenimh}
R_0 \cup R_\delta \cup R_1 \subset~ \genim H(k). 
\end{equation}
We claim that $R_0$ is residual in $(0,\smlhf)$. To see this, let $\sigma: [0,\smlhf] \to
[0,1]$ be the diffeomorphism $\sigma (t)=2t$. Then
$\sigma (\genim F^{1/2} (k)) = \genim F(k).$
Since the set on the right hand side of this equation is residual in $[0,1]$ and
$\sigma$ is a homeomorphism, 
$\genim F^{1/2} (k)$ is residual in $[0,\smlhf]$.
By Lemma \ref{lem.restrofresidual}, $R_0$ is residual in $(0,\smlhf)$. 

A similar argument shows that $R_\delta$ is residual in $(\smlhf,\smlhf + \delta (k))$:
The diffeomorphism $\lambda (k)$ restricts to a diffeomorphism
$\lambda'(k): [\smlhf, \smlhf+\delta (k)] \to [\epsilon_2 (k), 2\delta (k)],$ 
$\lambda'(k) (\smlhf)=\epsilon_2 (k),$ and we have
\[ \lambda' (k)(\genim (G_\lambda (k))\cap 
  (\smlhf, \smlhf + \delta(k))) =
  \genim G(k)\cap (\epsilon_2 (k), 2\delta (k)). \]
As the set $\genim G(k)$ is residual in $[0,1]$,
Lemma \ref{lem.restrofresidual} applies to ensure that 
$\genim G(k)\cap (\epsilon_2 (k), 2\delta (k))$ is residual
in $(\epsilon_2 (k), 2\delta (k))$. Since $\lambda'(k)$ is a homeomorphism,
the set $R_\delta = \genim G_\lambda (k)\cap 
(\smlhf, \smlhf + \delta (k))$ is residual in $(\smlhf, \smlhf+\delta (k))$.
Analogous reasoning proves that $R_1$ is residual in $(\smlhf +\delta (k),1)$.
According to Lemma \ref{lem.unionofresidual},
$R_0 \cup R_\delta \cup R_1$ is residual in $[0,1]$.
By (\ref{equ.r0rdelr1inphgenimh}), the superset $\genim H(k)$ 
is residual in $[0,1]$ as well. 
This completes the proof that $H$ is a fold field. \\

Next, we verify $H|_{[0,\epsilon (k)] \times M(k)}\approx h_M (k)$ 
and $H|_{[1-\epsilon' (k),1]\times P(k)}\approx h_P (k)$ for suitable
$\epsilon (k), \epsilon'(k)$.
The relation $F|_{[0,\epsilon_1 (k)]\times M(k)} \approx h_M (k)$ means
that there is a diffeomorphism
$\xi (k): [0,\epsilon_1 (k)]\to [0,1],$ $\xi (k)(0)=0,$ such that
$F(t,x) = h_M (\xi (k)(t),x)$
for all $0\leq t\leq \epsilon_1 (k),$ $x\in M(k)$. Set $\epsilon (k)= \smlhf \epsilon_1 (k)$
and let $\xi'(k): [0,\epsilon (k)]\to [0,1]$ be the diffeomorphism
given by $\xi' (k)(t) =\xi (k)(2t).$ 
Note that 
$\epsilon (k) = \smlhf \epsilon_1 (k) < \smlhf \epsilon_{W'} = \epsilon_W,$
as required. Then for all
$0\leq t\leq \epsilon (k),$ $x\in M(k)$,
\[
H(t,x) = F^{1/2} (t,x) = F\psi_1 (t,x) = F(2t,x) 
= h_M (\xi (k)(2t),x) = h_M (\xi' (k)(t),x).
\]
Hence $H|_{[0,\epsilon (k)] \times M(k)}\approx h_M (k)$.
Similarly, using $G|_{[1-\epsilon'_2 (k),1]\times P(k)} \approx h_P (k)$ 
and setting $\epsilon' (k)= \smlhf \epsilon'_2 (k)$,
one verifies that the relation $H|_{[1-\epsilon' (k),1]\times P(k)}\approx h_P (k)$ holds, too. \\

Let us prove the Brauer morphism identity $\mbs (G)\circ \mbs (F) = \mbs (H).$ 
Let $a_k= \smlhf (1-\epsilon'_1 (k))$ and $b_k =1$. The prescription
\[ \tau (k)(t) = \begin{cases}
 \lambda (k)(t),& t\in [a_k, \smlhf (1+\epsilon_{W''})], \\
 2(t-\smlhf),& t\in [\smlhf +\delta (k), b_k]
\end{cases} \]
yields a well-defined diffeomorphism $\tau (k): [a_k, b_k]\to [0,1],$
$\tau (k)(a_k)=0,$ since for $t$ in the overlap
$[\smlhf +\delta (k), \smlhf (1+\epsilon_{W''})]$, we have $\lambda (k)(t) =
2(t-\smlhf).$ 
We set $W_{\geq a_k} (k) = W(k) \cap [a_k, b_k]\times \real^{D}.$
Then the diffeomorphism $[a_k, b_k]\times \{ k \} \times \real^{D-1} 
  \to [0,1]\times \{ k \} \times \real^{D-1}$
given by $(t,x)\mapsto (\tau (k)(t),x)$ restricts to a diffeomorphism
$\phi (k): W_{\geq a_k} (k) \stackrel{\cong}{\longrightarrow} W''(k)$.
The restriction of $H(k)$ to $W_{\geq a_k} (k)$ is given by $G(k)\circ \phi (k)$.
By Lemma \ref{lem.singtanglereparam},
$G(k)\circ \phi (k)$ is a fold field on $W_{\geq a_k} (k)$ and
\[ \mbs (H(k)|_{W_{\geq a_k}}) = \mbs (G(k)\circ \phi (k))= \mbs (G(k)). \]
Let $W_{\leq a_k} (k) = W(k)\cap [0,a_k]\times \real^{D}$ and
$\widetilde{W}' (k) = W'(k) \cap [0,1-\epsilon'_1 (k)]\times \real^{D}.$
The value $1-\epsilon'_1 (k)$ lies in $\pitchfork F(k)\cap \genim F(k)$,
since $0\in \pitchfork u(k)\cap \genim u(k)$,
$F(t,x)=u(\alpha (k)(t),x)$ for $t\in [1-\epsilon'_1 (k),1],$ $x\in N(k),$ 
and $\alpha (k): [1-\epsilon'_1 (k),1] \stackrel{\cong}{\longrightarrow}
[0,1]$ maps $1-\epsilon'_1 (k)$ to $0$. It follows that $F|_{\widetilde{W}'(k)}$
is a fold field on $\widetilde{W}'(k)$. The diffeomorphism
$[0,a_k]\times \real^{D} \stackrel{\cong}{\longrightarrow}
[0,1-\epsilon'_1 (k)]\times \real^{D}$ given by
$(t,x)\mapsto (2t,x)$ restricts to a diffeomorphism
$\phi' (k): W_{\leq a_k} (k) \to \widetilde{W}'(k)$ and 
$H|_{W_{\leq a_k} (k)} = F|_{\widetilde{W}'(k)} \circ \phi'(k)$.
Thus by Lemma \ref{lem.singtanglereparam},
$H|_{W_{\leq a_k} (k)} \in \Fa (W_{\leq a_k} (k))$ and 
$\mbs (H|_{W_{\leq a_k} (k)}) = \mbs (F|_{\widetilde{W}'(k)})$.
The restriction of $F(k)$ to $W'_{\operatorname{out}} (k)= 
 W'(k) \cap [1-\epsilon'_1 (k),1]\times \real^{D}$
is a fold field on $W'_{\operatorname{out}} (k)$ as $1-\epsilon'_1 (k)\in \pitchfork F(k)\cap \genim F(k)$.
It follows from Lemma \ref{lem.sfltsfgtissf} that
\[ \mbs (F(k)) = \mbs (F|_{W'_{\operatorname{out}}(k)})\circ \mbs (F|_{\widetilde{W}'(k)}). \]
The diffeomorphism $\alpha (k)$ induces a diffeomorphism
$\nu (k): W'_{\operatorname{out}} (k) \stackrel{\cong}{\longrightarrow} [0,1]\times N(k),$
$\nu (k)(t,x) = (\alpha (k)(t),x).$ The diagram
\[ \xymatrix@R=15pt@C=10pt{
W'_{\operatorname{out}} (k) \ar[rd]_{F|_{W'_{\operatorname{out}}(k)}} \ar[rr]^{\nu (k)} & & [0,1]\times N(k)
  \ar[ld]^{u(k)} \\
& \cplx &
} \]
commutes. Since $\bigotimes_k \mbs (u(k)) = \mbs (u)=1,$ it follows from
Lemma \ref{lem.idfactorstriv} that $\mbs (u(k))=1$ for all $k$.
We deduce with the aid of Lemma \ref{lem.singtanglereparam}
that
$\mbs (F|_{W'_{\operatorname{out}}(k)}) = \mbs (u(k))=1.$
Consequently, $\mbs (F(k)) = \mbs (F|_{\widetilde{W}'(k)})$. 
From $1-\epsilon'_1 (k) \in \pitchfork F(k)\cap \genim F(k)$ it follows that
$a_k =\smlhf (1-\epsilon'_1 (k))$ lies in $\pitchfork F^{1/2}(k)\cap
\genim F^{1/2}(k)$.
Therefore, by Lemma \ref{lem.sfltsfgtissf},
\[ \mbs (H(k)) = \mbs (H|_{W_{\geq a_k} (k)})\circ \mbs (H|_{W_{\leq a_k}(k)}) =
  \mbs (G(k))\circ \mbs (F|_{\widetilde{W}'(k)}) = \mbs (G(k))\circ \mbs (F(k)). \]
Letting $k$ vary,
\begin{eqnarray*}
\mbs (H) & = & \bigotimes_k \mbs (H(k)) =
  \bigotimes_k \left( \mbs (G(k)) \circ \mbs (F(k)) \right) \\
& = & \Big\{ \bigotimes_k \mbs (G(k)) \Big\} \circ
    \Big\{ \bigotimes_k \mbs (F(k)) \Big\} =  \mbs (G)\circ \mbs (F).
\end{eqnarray*}
\end{proof}

\begin{thm}[Gluing]  \label{thm.gluingmain}
Let $W'$ be a cobordism from $M$ to $N$ and $W''$ a cobordism from
$N$ to $P$. Assume without loss of generality that $W'$ and $W''$ are
given equal cylinder scales $\epsilon_{W'} = \epsilon_{W''}$.
(If they are not equal, replace the larger one by the smaller one.)
Let $W=W' \cup_N W''$ be the cobordism from $M$ to $P$
obtained by gluing $W'$ and $W''$ along $N$, equipped with
its natural cylinder scale
$\epsilon_W = \smlhf \min (\epsilon_{W'}, \epsilon_{W''}) =
  \smlhf \epsilon_{W'} = \smlhf \epsilon_{W''}.$
Then the state sums
\[ Z_{W'} \in Z(M)\hotimes Z(N),~ Z_{W''} \in Z(N)\hotimes Z(P),~
  Z_W \in Z(M)\hotimes Z(P) \]
are related by the gluing law
\[ Z_W = \langle Z_{W'}, Z_{W''} \rangle. \]
\end{thm}
\begin{proof}
On a boundary condition $(h_M, h_P)\in \Fa (M)\times \Fa (P),$
the contraction product of the two state sums $Z_{W'}$ and $Z_{W''}$
is given by
\begin{eqnarray*}
\langle Z_{W'}, Z_{W''} \rangle (h_M, h_P) & = &
 \gamma (Z_{W'} \hotimes_c Z_{W''})(h_M, h_P) = \sum_{u\in \Fa (N)} (Z_{W'} \hotimes_c Z_{W''})(h_M, u,u, h_P) \\
& = & \sum_{u\in \Fa (N)} Z_{W'} (h_M, u)\cdot Z_{W''} (u,h_P) \\
& = & \sum_u \Big\{ \sum_{F\in \Fa (W'; h_M, u)} Y\mbs (F) \Big\} \cdot
     \Big\{ \sum_{G\in \Fa (W''; u, h_P)} Y\mbs (G) \Big\}. 
\end{eqnarray*}
By equation (\ref{equ.sumoverproductfamily}) on 
page \pageref{equ.sumoverproductfamily}, the product of the sums 
equals the sum of the products,
\[ \langle Z_{W'}, Z_{W''} \rangle (h_M, h_P) =
   \sum_{u\in \Fa (N)} \sum_{(F,G) \in \Fa (W'; h_M, u)\times \Fa (W''; u, h_P)} 
   Y\mbs (F) \cdot Y\mbs (G). \]
Note that if  $F\in \Fa (W'; h_M, u)$ and $G\in \Fa (W''; u, h_P)$, then
$\cod \mbs (F)=\dom \mbs (G)$ by Lemma \ref{lem.techotherhalf}.
Thus, with $[a]=\dom \mbs (F),$ $[b]=\cod \mbs (G),$ and
$[d] = \cod \mbs (F)=\dom \mbs (G)$, the relevant product in 
evaluating $Y\mbs (F)\cdot Y\mbs (G)$ is of type
$Q(H_{a,d})\times Q(H_{d,b})\longrightarrow Q(H_{a,b}).$
By Lemma \ref{lem.prodmppnrelaxed} and since $Y$ is a functor,
\[ (Y\mbs (F)\otimes 1)\cdot (Y\mbs (G)\otimes 1) =
  (Y\mbs (G)\circ Y\mbs (F))\otimes (1\cdot 1) =
  Y(\mbs (G)\circ \mbs (F))\otimes 1. \]
Let $\mathfrak{S}(W',W'')$ be the set
\begin{eqnarray*}
\mathfrak{S}(W',W'') & = & \{ \mbs (G)\circ \mbs (F) ~|~
 u\in \Fa (N),~ F\in \Fa (W'; h_M, u),~ G\in \Fa (W''; u, h_P) \} \\
& = & \{ \mbs (G)\circ \mbs (F) ~|~
F\in \Fa (W'),~ G\in \Fa (W''):\\
& &       \exists \epsilon_1 (k),\epsilon'_1 (k)\in (0,\epsilon_{W'}), 
             \epsilon_2 (k), \epsilon'_2 (k)\in (0,\epsilon_{W''}),
       u\in \Fa (N), \\
& &  F|_{[0,\epsilon_1 (k)]\times M(k)} \approx h_M (k),~
 G|_{[1-\epsilon'_2 (k),1]\times P(k)} \approx h_P (k),\\
& & F|_{[1-\epsilon'_1 (k),1]\times N(k)} \approx u(k) \approx
     G|_{[0,\epsilon_2 (k)]\times N(k)} \}.
\end{eqnarray*}
The state sum of $W$ is given on $(h_M, h_P)$ by
$Z_W (h_M, h_P) = \sum_{H\in \Fa (W; h_M, h_P)} Y\mbs (H).$
Let $\mathfrak{S}(W)$ be the set
\begin{eqnarray*}
\mathfrak{S}(W) & = &
 \{ \mbs (H) ~|~ H\in \Fa (W; h_M, h_P) \} \\
 & = & \{ \mbs (H) ~|~ H\in \Fa(W),~ \exists
 \epsilon (k),\epsilon'(k) \in (0,\epsilon_W): \\
 & &  H|_{[0,\epsilon (k)] \times M(k)}\approx h_M (k),~ 
   H|_{[1-\epsilon'(k),1]\times P(k)}\approx h_P (k) \}. 
\end{eqnarray*}
By Proposition \ref{prop.gluefandgtogeth} we have the inclusion
$\mathfrak{S}(W',W'') \subset \mathfrak{S}(W)$, and by
Proposition \ref{prop.splithintofandg}
the converse inclusion $\mathfrak{S}(W) \subset
\mathfrak{S}(W',W'')$. Consequently, $\mathfrak{S}(W) =
\mathfrak{S}(W',W'')$ and, applying the functor $Y$,
$Y(\mathfrak{S}(W)) = Y(\mathfrak{S}(W',W'')).$
By Proposition \ref{prop.qcontinuous}, the idempotent semiring $Q$ is continuous.
Thus we can conclude from Proposition \ref{prop.contidemsemiring}
that
\[  \sum_{H\in \Fa (W; h_M, h_P)} Y\mbs (H) =
   \sum_{u\in \Fa (N)} \sum_{(F,G) \in \Fa (W'; h_M, u)\times \Fa (W''; u, h_P)} 
   Y(\mbs (G) \circ \mbs (F)). \]
\end{proof}

\section{Rationality}
\label{sec.rationality}

We show that for cobordisms of dimension $n\geq 3$, the value of the state sum on
a given boundary condition is a rational function of the loop-variable $q$. In fact, the
denominator turns out to be universal (independent of the cobordism), whence all the
information is contained in the polynomial numerator.
The dimension restriction is probably not necessary, but not all elements of our
proof readily carry over to dimension $2$. We shall not discuss rationality for this special
dimension further in this paper.

\begin{lemma} \label{lem.addtwoloops}
Let $W$ be a cobordism of dimension $n\geq 3$
from $M$ to $N$ with cylinder scale $\epsilon_W$, 
let $F\in \Fa (W)$ be a fold field on $W$ and $\epsilon (k), \epsilon' (k) \in (0, \epsilon_W)$.
Then there exists a fold field $F^2$ on $W$ such that for all $k\in s(W),$
\[ F^2|_{[0,\epsilon (k)]\times M(k)} = F|_{[0,\epsilon (k)]\times M(k)},~
  F^2|_{[1- \epsilon' (k),1]\times N(k)} = F|_{[1- \epsilon' (k),1]\times N(k)}, \]
and
\[ \mbs (F^2) = \mbs (F)\otimes \lambda^2, \]
where $\lambda$ is the loop endomorphism in $\Br$.
\end{lemma}
\begin{proof}
First we shall create two new fold circles locally near a point off the given singular set.
This will produce a fold map which may not yet be a fold field. Thus we will then modify
the fold map, turning it into a fold field.
The latter step requires the dimensional restriction, whereas the former works for any $n\geq 2$.

Let $F$ be a fold field on $W$.
If the time function $\omega:W\to [0,1]$ is locally constant, then $\partial W=\varnothing$
and we choose $U\subset W -S(F)$ to be a small open ball off the singular set of $F$.
If $\omega$ is not locally constant on $W$, there exists a point
$x^* \in W-\partial W$ where the differential $D_{x^*} \omega: T_{x^*} W \to
T_{\omega (x^*)} \real$ is nonzero and thus onto. Hence the differential of $\omega$
is onto even on a small open neighborhood $U^* \subset W-\partial W$ of $x^*$,
i.e. $\omega|:U^* \to \real$ is a submersion.
Let $U\subset U^* - S(F)$ be a small open ball. 
If the incoming boundary $M$ is not empty, we can and will choose
$U\subset (\epsilon (k), \epsilon_W)\times M(k) \subset W(k),$ where $k\in s(W)$
is such that $M(k)\not= \varnothing$.
If $M=\varnothing,$ but the outgoing boundary $N$ is not empty,
we can and will choose
$U\subset (1-\epsilon_W, 1- \epsilon (k))\times N(k) \subset W(k),$ $k\in s(W),$
$N(k)\not= \varnothing$.
If $M$ and $N$ are both empty, let $k\in s(W)$ be such that $U\subset W(k)$.
There exists a diffeomorphism
$\phi: U\to \real^n,$ an open set
$U' \subset \cplx$, and a diffeomorphism $\psi: U' \to \cplx$
such that $\psi \circ F \circ \phi^{-1}$ is the standard projection
$\pi:\real^n \to \real^2 = \cplx,$ 
$\pi (x_1,\ldots, x_n) = (x_{n-1}, x_n)$.
We shall modify
$\pi$ in two stages on a compact subset $K$ of $\real^n$
so that the modified map $\tpi:\real^n \to \cplx$ is a fold
map and the singular set $S(\tpi)$ of $\tpi$ consists of
precisely two embedded disjoint circles in the interior of $K$.
The desired fold field $F^2$ will then be given by the gluing
\begin{equation} \label{equ.f2def}
F^2 (x) = \begin{cases} F(x),& x \in W-\phi^{-1} (K) \\
  \psi^{-1} \tpi \phi (x),& x\in \phi^{-1} (U_K),
 \end{cases}
\end{equation}
where $U_K \subset \real^n$ is an open neighborhood of $K$.
The map $F^2$ is smooth on $W$, since for
$x\in \phi^{-1} (U_K - K),$
$\psi^{-1} \tpi \phi (x) = \psi^{-1} \pi \phi (x) = F(x).$

We begin the first stage of modifying $\pi$ by embedding the circle $S^1$ as the unit circle
into the plane $\real^2$ and then embedding that plane into $\real^n$ via
$(x_{n-1}, x_n)\mapsto (0,\ldots, 0,x_{n-1}, x_n)$.
The normal bundle $\real^{n-1} \times S^1 \to S^1$ of this embedding 
$S^1 \hookrightarrow \real^n$ comes with a diffeomorphism
$\alpha: V\to \real^{n-1} \times S^1$ under which $S^1 \subset V$ is mapped to
$0\times S^1$ by the identity map, where $V$ is a small open tubular neighborhood
of $S^1$ in $\real^n$. The image of $S^1 \subset \real^n$ under
$\pi: \real^n \to \real^2$ is the unit circle $S^1 \subset \real^2$.
The diffeomorphism $\alpha$, an open tubular neighborhood $V_2 \subset \real^2$
of $S^1 \subset \real^2$ and a diffeomorphism $\alpha_2: V_2 \to \real \times S^1$
can be arranged in such a way that $\pi$ restricts to a surjection
$\pi|:V\to V_2$ and
\begin{equation} \label{equ.pialpha}
\xymatrix@R=15pt{
V \ar[r]^>>>>>\alpha_>>>>>\cong \ar[d]_{\pi|} & \real^{n-1} \times S^1 \ar[d]^{\pi_{n-1} \times \id_{S^1}} \\
V_2 \ar[r]^{\alpha_2}_\cong & \real \times S^1
} \end{equation}
commutes, where $\pi_{n-1} (\xi) = \xi_{n-1},$ $\xi \in \real^{n-1}$.
Let $\Phi: \real^{n-1} \times \real \to \real^{n-1} \times S^1$,
$\Phi (\xi, t)=(\xi, e^{2\pi it}),$ $\xi \in \real^{n-1},$ $t\in \real,$ be the universal
cover of $\real^{n-1} \times S^1$.
By \cite[Lemma 8.2, p. 101]{milnorsiebensond},
there is a Morse function $f: \real^{n-1} \to \real$ and a compact subset
$K_0 \subset \real^{n-1}$ such that $f(\xi)=\xi_{n-1}$ for $\xi \in \real^{n-1} - K_0,$
and $f$ has
precisely two (nondegenerate) critical points
$p_m, p_s \in \interi (K_0)$ of index $n-1$ and $n-2$, respectively, with
$f(p_m) > f(p_s)$.
Thus $f$ attains a local maximum at $p_m$ and
$p_s$ is a saddle point. Intuitively, $f$ is obtained
by smoothly ``pulling up'' the graph of $f$ over $p_m$ to create a local maximum.
Note that pulling up the graph automatically generates a second
critical point $p_s$.
The map $\widetilde{g}: \real^{n-1} \times \real \to \real \times \real = \real^2$ given by
$\widetilde{g}(\xi, t)=(f(\xi), t)$ is the suspension of a Morse function and hence a fold map.
Let $\Phi_2: \real \times \real \to \real \times S^1$,
$\Phi_2 (y, t)=(y, e^{2\pi it}),$ $y,t\in \real,$ be the universal
cover of $\real \times S^1$.
The fold map $\widetilde{g}$ descends to a smooth map
$g: \real^{n-1} \times S^1 \to \real \times S^1$ such that
\[ \xymatrix@R=15pt{
\real^{n-1} \times S^1 \ar[d]_g & \real^{n-1} \times \real \ar[d]^{\widetilde{g}} 
  \ar[l]_\Phi \\
\real \times S^1 & \real \times \real \ar[l]_{\Phi_2}
} \]
commutes. Since $\widetilde{g}$ is a fold map and the covering maps $\Phi$ and $\Phi_2$ are local diffeomorphisms,
$g$ is a fold map as well. Setting 
$\hat{\pi} (x) = \alpha_2^{-1} g \alpha (x),$ we obtain a fold map
$\hat{\pi}:V\to \real^2$. The interior of the compact annulus
$\hat{K}=\alpha^{-1} (K_0 \times S^1)\subset V$ 
contains two disjoint fold circles $C_m$ and $C_s$ corresponding
to $p_m$ and $p_s$, and these two circles constitute the singular set of $\hat{\pi}$ in $V$.
The absolute index of $C_m$ is
$n-1$ and the absolute
index of $C_s$ is $n-2$.
In the case $n=3,$ we claim that coordinates on $\real^{n-1} = \real^2$ can be chosen
in such a way that the two circles $C_m,$ $C_s \subset \real^3$ are unlinked.
They are also unknotted, but this is obvious and does not require proof.
To show that they are unlinked, choose coordinates $\xi = (\xi_1, \xi_2)$ on $\real^2$
so that $\xi^m_2 = \xi^s_2,$ where $p_m = (\xi^m_1, \xi^m_2)$ and
$p_s = (\xi^s_1, \xi^s_2)$. Since $p_m \not= p_s,$ it follows that $\xi^m_1 \not= \xi^s_1$.
The diffeomorphism $\alpha: V\to \real^2 \times S^1$ has components
\[ \alpha (x)=\alpha (x_1, x_2, x_3)= (a_1 (x), a_2 (x), a_3 (x)), \]
$a_1, a_2: V\to \real,$ $a_3: V\to S^1$.
The commutativity of (\ref{equ.pialpha}) means that
\[ (a_2 (x), a_3 (x)) = \alpha_2 (x_2, x_3) \]
for all $x\in V$. In particular, $a_2$ and $a_3$ depend only on $x_2, x_3$:
$a_2 (x)=a_2 (x_2, x_3),$ $a_3 (x)=a_3 (x_2, x_3).$
The two fold circles can be described as
\[ C_m = \{ x\in V ~|~ a_1 (x_1,x_2,x_3)=\xi^m_1,~ a_2 (x_2,x_3)=\xi^m_2 \}, \]
\[ C_s = \{ x\in V ~|~ a_1 (x_1,x_2,x_3)=\xi^s_1,~ a_2 (x_2,x_3)=\xi^s_2 \}. \]
Put
\[ E = \{ (x_2,x_3)\in V_2 ~|~ a_2 (x_2,x_3)=\xi^m_2 =\xi^s_2 \}. \]
The alternative description $E=\alpha_2^{-1} (\{ \xi^m_2 \} \times S^1)$
shows that $E$ is a smooth simply closed curve in $V_2$.
Let $(x_2,x_3)\in E$ be any point. Then
$(\xi^m_1, \xi^m_2, a_3 (x_2,x_3))\in \real^2 \times S^1$ and thus there
exists a unique $x^m = (x^m_1, x^m_2, x^m_3)\in V$ such that
$\alpha (x^m) = (\xi^m_1, \xi^m_2, a_3 (x_2,x_3))$. Since
\begin{align*}
\alpha_2 (x^m_2, x^m_3) &= (a_2 (x^m), a_3 (x^m)) =
   (\xi^m_2, a_3 (x_2, x_3)) \\
&= (a_2 (x_2, x_3), a_3 (x_2,x_3))= \alpha_2 (x_2,x_3), 
\end{align*}
the injectivity of $\alpha_2$ implies that $(x^m_2, x^m_3) = (x_2,x_3)$.
Thus, given $(x_2,x_3)\in E,$ there is a unique $x^m_1\in \real$
such that $(x^m_1,x_2,x_3) \in V$ and
$a_1 (x^m_1,x_2,x_3)=\xi^m_1$. This yields a function
$h_m:E\to \real,$ $h_m (x_2,x_3)=x^m_1$.
The graph of $h_m$ is precisely the curve $C_m$.
Similarly, given $(x_2,x_3)\in E,$ there is a unique $x^s_1\in \real$
such that $(x^s_1,x_2,x_3) \in V$ and
$a_1 (x^s_1,x_2,x_3)=\xi^s_1$. We obtain a function
$h_s:E\to \real,$ $h_s (x_2,x_3)=x^s_1$, whose
graph is the curve $C_s$.
Therefore, both $C_m$ and $C_s$ are contained as graphs over the
plane curve $E\subset \real^2 \times 0$ in the cylinder $E\times \real
\subset \real^3$ and hence cannot be linked.

We return to general $n\geq 3$.
For $x\in V-\hat{K},$ $\alpha (x)=(\xi, e^{2\pi it})$ with $\xi \in \real^{n+1} - K_0$.
Using Diagram (\ref{equ.pialpha}), we obtain for such $x$:
\begin{align*}
\hat{\pi}(x) &= \alpha_2^{-1} g \alpha (x) = \alpha_2^{-1} g(\xi, e^{2\pi it}) =
                  \alpha_2^{-1} (f(\xi), e^{2\pi it}) \\
&= \alpha_2^{-1} (\xi_{n-1}, e^{2\pi it}) = \alpha_2^{-1} (\pi_{n-1} \times \id_{S^1})(\xi, e^{2\pi it})
  = \alpha_2^{-1} (\pi_{n-1} \times \id_{S^1}) \alpha (x) \\
&= \pi (x).
\end{align*}
Therefore, using $\pi$ outside of $\hat{K}$, 
we can extend $\hat{\pi}$ to a fold map $\hat{\pi}: \real^n \to \real^2,$
which has precisely two fold circles $C_m,$ $C_s$ in the interior of $\hat{K}$ and no other singularities.
This completes stage $1$ of the modification of $\pi$. \\

If the time function $\omega$ is not locally constant, then
we begin the second stage by composing the restriction of 
$\omega: W\to [0,1]$ to $U$
with the diffeomorphism $\phi^{-1}: \real^n \to U$. 
On $U,$ $\omega$ is a submersion.
Thus $\widetilde{\omega} = \omega \phi^{-1}: \real^n \to
\real$ is a submersion, which defines a foliation
of $\real^n$ whose leaves are the connected components of the
level sets $\widetilde{\omega}^{-1} (t),$ 
$0 < t < 1$. The idea now is to
isotope the fold circles $C_m \sqcup C_s$ into a single leaf of this
foliation.
Fix $t_0 =\omega (x_0),$ where $x_0$ is any point in $U$, and let $\widetilde{S}$
be the $1$-manifold consisting of two disjoint smoothly embedded
circles in $\widetilde{\omega}^{-1} (t_0)$.
For $n\geq 4,$ there exists by general position a smooth
isotopy $\alpha: \widetilde{S} \times I \to \real^n$ with
$\alpha (y,0)=y$ and $\alpha (\widetilde{S} \times 1) =
S(\hat{\pi}).$ Such an isotopy exists also when $n=3,$ since
$C_m,$ $C_s$ are unknotted and unlinked (as is $\widetilde{S}$).
By the isotopy extension theorem, $\alpha$ extends to an
ambient isotopy $\widetilde{\alpha}: \real^n \times I \to
\real^n,$ $\widetilde{\alpha}_0 = \id,$ with compact support.
That is, there exists a compact $K\subset \real^n,$
$\hat{K} \cup \widetilde{S} \subset K,$ so that
$\widetilde{\alpha} (x,t)=x$ for all $x\in \real^n -K$
and all $t\in I$. Our final modification of $\pi$ is
\[ \tpi = \hat{\pi} \circ \widetilde{\alpha}_1: \real^n \to \cplx. \]
In the special case where $\omega$ is locally constant, we put $\tpi = \hat{\pi}$.
As $\tpi$ is the composition of a diffeomorphism and a fold map,
it is itself a fold map. The singular set is $S(\tpi)=\widetilde{S},$
two disjoint circles in $K$. For $x\in \real^n -K,$ we have
\[ \tpi (x) = \hat{\pi} \widetilde{\alpha}_1 (x) =
 \hat{\pi}(x) = \pi (x), \]
as $x\not\in \hat{K}$. \\

Now let $F^2: W\to \cplx$ be defined by (\ref{equ.f2def}).
Since $F(x)$ is a fold map for $x\in W- \phi^{-1}(K)$ and
$\psi^{-1} \tpi \phi (x)$ is a fold map for
$x\in \phi^{-1} (U_K),$ the glued map $F^2$ is a fold map
for all $x\in W$. If $\omega$ is locally constant, then
$F^2$ is automatically a fold field by Lemma \ref{lem.timelocconst}.
Suppose that $\omega$ is not locally constant.
If $x\in [0, \epsilon (a)]\times M(a)$ or
$x\in [1-\epsilon' (a),1]\times N(a)$ for some $a\in s(W),$
then $x\not\in U$ and in particular $x\not\in \phi^{-1}(K).$
Thus for such $x$, $F^2 (x) = F(x)$. The fold locus of $F^2$ is
\begin{equation} \label{equ.sf2}
S(F^2) = S(F) \sqcup \phi^{-1} (\widetilde{S}).
\end{equation}
Thus
\[ S(F^2) \cap [0,\epsilon (k)] \times M(k) =
  S(F) \cap [0,\epsilon (k)] \times M(k), \]
\[ S(F^2) \cap [1-\epsilon'(k),1] \times N(k) =
  S(F) \cap [1-\epsilon'(k),1] \times N(k) \]
and so $S(F^2)$ is transverse to $W_0$ and to $W_1$.
Since
\[ \Ima F^2 (a)| = \Ima F(a)|: S(F^2)\cap W_0 (a) =
  S(F)\cap W_0 (a) \longrightarrow \real \]
is injective, we have $0\in \genim (F^2 (a))$ for all $a$
and similarly $1\in \genim (F^2 (a))$ for all $a$.
We shall show next that $\genim (F^2 (a))$ is residual in 
$[0,1]$ for all $a$. If $a\not= k$, then $F^2 (a) = F(a)$
and thus $\genim (F^2 (a))=\genim (F(a))$. Since the latter
is residual, the former is so as well. Suppose that $a=k$
and let $t\in \genim (F(k)) - \{ t_0 \}.$ Then
$\phi^{-1}(\widetilde{S})\cap \omega^{-1}(t)$ is empty, since
\[
\phi (\phi^{-1}(\widetilde{S})\cap \omega^{-1}(t))
= \widetilde{S}\cap \phi \omega^{-1}(t) 
 =  \widetilde{S} \cap \widetilde{\omega}^{-1}(t)
 \subset \widetilde{\omega}^{-1} (t_0) \cap \widetilde{\omega}^{-1}(t)
  =\varnothing
\]
as $t\not= t_0$. Consequently,
\begin{eqnarray*}
SF^2 (k)_t & = & S(F^2 (k))\cap W_t = S(F^2 (k))\cap \omega^{-1}(t) \\
 & = & (S(F(k))\cup \phi^{-1} (\widetilde{S}))\cap \omega^{-1}(t) \\
 & = & (S(F(k))\cap \omega^{-1}(t))\cup (\phi^{-1}(\widetilde{S})\cap
   \omega^{-1}(t)) \\
& = & SF(k)_t \cup \varnothing = SF(k)_t.
\end{eqnarray*}
Since $U\cap S(F)=\varnothing$, we have $F^2|_{S(F)} = F|_{S(F)}.$
Therefore,
\[ \Ima F^2 (k)| = \Ima F(k)|: SF^2 (k)_t = SF(k)_t \longrightarrow \real \]
is injective and we conclude that $t\in \genim (F^2 (k))$. This shows
that
\[ \genim (F(k))\cap ([0,1] - \{ t_0 \}) \subset \genim (F^2 (k)). \]
The set $[0,1] - \{ t_0 \}$ is open and dense, in particular residual.
The set $\genim (F(k))$ is residual as $F$ is a fold field. Therefore,
the intersection $\genim (F(k))\cap ([0,1]- \{ t_0 \})$ is residual.
This implies that the superset $\genim (F^2 (k))$ is residual as well.
It follows that $F^2$ is a fold field on $W$. From equation
(\ref{equ.sf2}) we deduce that
\[ \mbs (F^2) = \mbs (F)\otimes \lambda^2, \]
which also holds when $\omega$ is locally constant.
\end{proof}
\begin{remark}
In the surface case $n=2$, the insertion of folds is discussed in 
\cite[Part D, Section 20]{whitneymapsofplane}.
Our construction of $\hat{\pi}$ in the above proof is a high-dimensional
generalization of Whitney's ``mushroom'' construction.
In particular, Whitney's fold insertion technique also produces two
circles simultaneously. 
\end{remark}

In the Boolean power series semiring $\bool [[q]],$ we set formally
\[ \frac{1}{1-q^2} := 1 + q^2 + q^4 + q^6 + \ldots. \]

\begin{thm} \label{thm.rationalinq}
Let $W$ be a cobordism of dimension $n\geq 3$ from $M$ to $N$. On any boundary condition
$(f_M, f_N)\in \Fa (M)\times \Fa (N),$ the state sum
$Z_W (f_M, f_N)$ is a rational function of $q$. In fact, there
exists a polynomial $p_{f_M, f_N} (q)$ in $q$ with linear coefficients
such that
\[ Z_W (f_M, f_N) = \frac{p_{f_M, f_N} (q)}{1-q^2}. \]
\end{thm}
\begin{proof}
As $f_M$ lies in $\Fa (M),$ we know that $\mbs (f_M)$ is
the identity morphism on some object $[\bar{m}]$ of $\Br$.
By Lemma \ref{lem.idfactorstriv}, $\mbs (f_M (k))$ is the identity
on some object $[m_k]$ of $\Br$.
Let $F\in \Fa (W; f_M, f_N)$ be a fold field satisfying the boundary conditions
$(f_M, f_N)$, that is, there are $\epsilon (k), \epsilon' (k) \in (0,\epsilon_W)$
such that
\[  F|_{[0,\epsilon (k)]\times M(k)} \approx f_M (k),~
  F|_{[1- \epsilon' (k),1]\times N(k)} \approx f_N (k),  \]
for all $k\in s(W)$.
By Lemma \ref{lem.relapprox},
$\mbs (F|_{[0,\epsilon (k)]\times M(k)})=\mbs (f_M (k)).$
Thus the domain of $\mbs (F)$ is
\[ \dom \mbs (F) = \dom \bigotimes_k \mbs(F(k)) = \dom
 \bigotimes_k \mbs (f_M (k)) = \bigotimes_k \dom \mbs (f_M (k)) =
  \left[ \sum_k m_k \right] = [\bar{m}] \]
for every $F\in \Fa (W; f_M, f_N)$. Using the boundary condition at $N$,
one sees similarly that there is one object $[\bar{n}]$ of $\Br$ which is the
codomain of every $\mbs (F),$ $F\in \Fa (W; f_M, f_N)$.
Hence, $Y\mbs (F)$ lies in $Q(H_{\bar{m}, \bar{n}}) \subset Q$ for all
fields $F\in \Fa (W; f_M, f_N)$. Therefore, the state sum
\[ Z_W (f_M, f_N) = \sum_{F\in \Fa (W;f_M, f_N)} Y\mbs (F) \]
lies in $Q(H_{\bar{m}, \bar{n}})$. Consequently, according to 
Lemma \ref{lem.phiiso}, there are uniquely determined power series
$b_y \in \bool [[q]]$ such that
\[ Z_W (f_M, f_N) = \sum_{y\in Y(\OP_{\bar{m}, \bar{n}})}
   y\otimes b_y. \]
This is a finite sum as $\OP_{\bar{m}, \bar{n}}$ is a finite set.
If $b_y$ is not zero, then it has the form
\[ b_y = q^{r_y} + \beta_1 (y) q^{r_y +1} + \beta_2 (y) q^{r_y +2}
  +\ldots,~ \beta_i (y)\in \bool. \]
Since $y\otimes q^{r_y}$ is then a summand of $Z_W (f_M, f_N),$
there exists a fold field $F\in \Fa (W; f_M, f_N)$ such that
$Y\mbs (F)=y\otimes q^{r_y}$. By Lemma \ref{lem.addtwoloops},
there is a field $F^2 \in \Fa (W; f_M, f_N)$ with
$\mbs (F^2) = \mbs (F)\otimes \lambda^2.$
Let $\lh = \operatorname{Tr}(i,e)$ be the trace of the duality
structure $(i,e)$. Then
\[ Y\mbs (F^2) = Y(\mbs (F)\otimes \lambda^2) =
  Y\mbs (F)\otimes Y(\lambda)^2 = \lh^2 (y\otimes q^{r_y}) =
  y\otimes q^{r_y +2}. \]
It follows that $\beta_2 (y)=1$ and by iteration (take
$(F^2)^2,$ etc.) that $\beta_{2s}(y)=1$ for all $s=1,2,3,\ldots$.
Thus
\[ b_y = q^{r_y}((1+q^2 +q^4 +\ldots)+ \beta_1 (y)q + \beta_3 (y)q^3 +
  \beta_5 (y)q^5 +\ldots). \]
If all $\beta_{\operatorname{odd}}(y)=0$, then $b_y$ is the rational function
\[ b_y = \frac{q^{r_y}}{1-q^2}. \]
Suppose that some $\beta_{2s_y +1} (y)\not= 0$ and choose $s_y$ to be the
first such index, that is, $\beta_{2s+1} (y)=0$ for all $s<s_y$.
The power series thus has the form
\[ b_y = q^{r_y} \left( \frac{1}{1-q^2} + q^{2s_y +1} +
  \beta_{2s_y +3} (y)q^{2s_y +3} +\ldots \right). \]
Since $y\otimes q^{r_y + 2s_y +1}$ is then a summand of $Z_W (f_M, f_N),$
there exists a fold field $F\in \Fa (W; f_M, f_N)$ such that
$Y\mbs (F)=y\otimes q^{r_y +2s_y +1}$. By Lemma \ref{lem.addtwoloops},
there is a field $F^2 \in \Fa (W; f_M, f_N)$ with
\[ \mbs (F^2) = \mbs (F)\otimes \lambda^2 \]
and we have
\[ Y\mbs (F^2) = \lh^2 (y\otimes q^{r_y +2s_y +1}) =
  y\otimes q^{r_y + 2s_y +3}. \]
It follows that $\beta_{2s_y +3} (y)=1$ and by iteration 
that $\beta_{2(s_y +s)+1}(y)=1$ for all $s=1,2,3,\ldots$.
So $b_y$ has the form
\[ b_y = q^{r_y} \left( \frac{1}{1-q^2} +
   q^{2s_y +1}(1+q^2+q^4+\ldots) \right)
  = \frac{q^{r_y} (1+ q^{2s_y +1})}{1-q^2}, \]
a rational function.
Let us summarize: For every $b_y,$ there are nonnegative integers $r_y$ and $s_y$
and $\alpha_y, \beta_y \in \bool$ such that
\[ b_y =  \frac{\alpha_y q^{r_y} (1+ \beta_y q^{2s_y +1})}{1-q^2} \]
(obviously $\alpha_y =0$ iff $b_y =0$ and if $\alpha_y =1$ then
$\beta_y =0$ iff all $\beta_{\operatorname{odd}} (y)=0$).
The state sum is the rational function
\[ Z_W (f_M, f_N) = \sum_y y\otimes 
\frac{\alpha_y q^{r_y} (1+ \beta_y q^{2s_y +1})}{1-q^2} \]
(finite sum). The polynomial $p_{f_M, f_N}(q)$ is
\[ p_{f_M, f_N}(q) = \sum_y y\otimes \alpha_y q^{r_y} (1+ \beta_y q^{2s_y +1}). \]
\end{proof}

\section{Time Consistent Diffeomorphism Invariance}

We turn to proving diffeomorphism invariance of the state sum.
We will need to assume that diffeomorphisms are compatible with the time functions
on the cobordisms, as explained below.
Let $W,W' \subset [0,1] \times \real^D$ be cobordisms with time functions
$\omega, \omega',$ respectively.
\begin{defn}
The time function $\omega$ on a cobordism $W$ is called \emph{progressive},
if for every $t \in [0,1]$, there is a point in $W_t$ at which $\omega$ is a submersion,
i.e. every slice $W_t$ contains an $\omega$-regular point.
\end{defn}
Thus for $W$ with progressive time, there exists at every point in time a place on
$W$, where time moves forward at a nonzero speed. In particular, such a cobordism
cannot have a slice $W_t$ which is a union of connected components of $W$.
\begin{example} \label{expl.cylprogr}
Cylinders $W = [0,1] \times M$ always have progressive time.
\end{example}
\begin{defn}
A diffeomorphism $\Phi: W\to W'$ is said
to be \emph{time preserving}, if $\omega' \Phi = \omega$.
\end{defn}
\begin{example}
Suppose that $\phi,\psi:M\to N$ are isotopic diffeomorphisms of
closed $(n-1)$-manifolds. Then an isotopy $\Xi:[0,1] \times M\to N$ between
$\phi$ and $\psi$ induces a time preserving diffeomorphism
$\Phi: [0,1]\times M \to [0,1]\times N,$ $\Phi (t,x)= (t,\Xi (t,x))$.
\end{example}
If $\Phi$ is time preserving, then
\begin{equation} \label{equ.phiwtwpt}
\Phi (W_t) = \Phi \omega^{-1} (t) = \Phi \Phi^{-1} (\omega')^{-1} (t) =
  (\omega')^{-1} (t) = W'_t.
\end{equation}
A weaker notion is given by diffeomorphisms that preserve \emph{simultaneity},
that is, if $p$ and $q$ occur at the same time in $W$, then $\Phi (p)$ and
$\Phi (q)$ are also simultaneous in $W'$:
\begin{defn} \label{def.timeconsistent}
A diffeomorphism $\Phi: W\to W'$ is said
to be \emph{time consistent}, if it sends slices $W_t$ to
slices $W'_{\tau}$. If the incoming boundary $W_0$ of $W$ is not empty, then such 
a diffeomorphism sends $W_0$ to either $W'_0$ or $W'_1$. 
A time consistent diffeomorphism $\Phi$ is 
called \emph{positive} if $\Phi (W_0)=W'_0$ and \emph{negative} if
$\Phi (W_0)=W'_1$.
\end{defn}
If $\Phi$ is time preserving, then (\ref{equ.phiwtwpt}) shows
that $\Phi$ is time consistent. We will see below that a time consistent
diffeomorphism $\Phi$ satisfying the initial condition $\Phi (W_0)=W'_0$
has the property that it preserves the time order of events, that is,
if event $p$ occurs before event $q$, then $\Phi (p)$ occurs before $\Phi (q)$.
\begin{lemma} \label{lem.invtimecons}
The inverse $\Phi^{-1}$ of a time consistent diffeomorphism $\Phi:W\to W'$
between cobordisms with progressive time is again time consistent.
\end{lemma}
\begin{proof}
Let $\tau \in [0,1]$. As $\omega'$ is progressive, the slice $W'_\tau$ is in 
particular nonempty. Choose a point $q\in W'_\tau$ and set
$t= \omega \Phi^{-1} (q)$ so that $\Phi^{-1} (q) \in W_t$. Since $\Phi$
is time consistent, there is a $\tau' \in [0,1]$ with $\Phi (W_t) = W'_{\tau'}.$
Thus $q=\Phi \Phi^{-1} (q)\in \Phi (W_t) = W'_{\tau'}$ and therefore
$\tau' = \omega'(q)=\tau$. It follows that
\[ \Phi^{-1} (W'_{\tau}) = \Phi^{-1} (W'_{\tau'}) =
  \Phi^{-1} \Phi (W_t) = W_t. \]
\end{proof}

\begin{lemma} \label{lem.alphadiffeo}
Let $W$ and $W'$ be cobordisms with progressive time. Let $\Phi:W\to W'$
be a time consistent diffeomorphism. Then there exists a unique diffeomorphism
$\alpha: [0,1]\to [0,1]$ such that $\Phi (W_t)=W'_{\alpha (t)}$.
\end{lemma}
\begin{proof}
Since $\Phi$ is time consistent, every $t\in [0,1]$ determines a unique $\tau \in [0,1]$
with $\Phi (W_t)=W'_{\tau}$. (Note that as time on $W$ is progressive, the slices
$W_t$ are all nonempty.) Thus, setting $\alpha (t)=\tau$, we obtain a function
$\alpha: [0,1]\to [0,1]$. We claim that $\alpha$ is a diffeomorphism.
Let us first establish that $\alpha$ is smooth. Let $t_0 \in [0,1]$ be any time point.
As $\omega$ is progressive, there exists a $p\in W$, $\omega (p)=t_0,$ such that
the differential $D_p \omega: T_p W\to T_{t_0} \real$ is onto. Thus there exists
a coordinate chart $g: U\stackrel{\cong}{\longrightarrow} U',$ 
$U\subset W$ open, $U' \subset \real^n$ open, $p\in U,$
$g(p)=0\in U'$ and a smooth chart $h: V\stackrel{\cong}{\longrightarrow} V',$ 
$V\subset \real$ open, $V' \subset \real$ open, $t_0 \in V,$
$h(t_0)=0\in V'$ such that 
$\pi = h\omega g^{-1}: U' \to V',$ where $\pi$ is the projection
$\pi (t,x_2,\ldots, x_n)=t.$ Let $\sigma: V' \to U'$ be the smooth map
$\sigma (t)=(t,0,\ldots, 0).$ Then $\pi \sigma =\id$ and $\sigma (0)=0$.
Using $\sigma$, we define a smooth local section $s$ for $\omega$ by
$s = g^{-1} \sigma h: V\to U.$ Then
\[ \omega s = (h^{-1} \pi g)(g^{-1} \sigma h) = h^{-1} (\pi \sigma) h=
 h^{-1} h=\id \]
and
\[ s(t_0) = g^{-1} \sigma h(t_0)= g^{-1} \sigma (0) = g^{-1} (0) =p. \]
For $t\in V$, we have $\omega s (t)=t,$ i.e. $s(t)\in W_t.$
Since $\Phi (W_t) = W'_{\alpha (t)},$ $\Phi s(t)$ is in $W'_{\alpha (t)}$,
which means $\omega' \Phi s(t)=\alpha (t)$. Therefore,
$\alpha|_V = \omega' \Phi s.$ This shows that $\alpha|_V$ is smooth, as
$\omega' \Phi s$ is smooth. Since $\alpha$ is thus smooth in a neighborhood
of every point $t_0 \in [0,1]$, $\alpha$ itself is smooth. \\

By Lemma \ref{lem.invtimecons},
the inverse diffeomorphism $\Phi^{-1}:W' \to W$ is also time consistent,
so a function $\beta: [0,1]\to [0,1]$ is given by
$\Phi^{-1} (W'_t)=W_{\beta (t)}$. By the preceding argument, $\beta$ is smooth.
Furthermore, $\alpha$ and $\beta$ are inverse to each other:
Given $t\in [0,1]$, it is possible to choose a point $p\in W_t$, since
time is progressive on $W$. By definition of $\alpha,$ 
$\Phi (p)\in W'_{\alpha (t)}$. As $\Phi^{-1} (W'_{\alpha (t)}) =
W_{\beta (\alpha (t))},$ we have
$p = \Phi^{-1} \Phi (p)\in W_{\beta (\alpha (t))},$ which means
$\omega (p) = \beta \alpha (t).$ But $\omega (p)=t$ as $p\in W_t$, whence
$\beta \alpha (t) = \omega (p) =t.$ The identity $\alpha \beta (t)=t$ follows
symmetrically. 
\end{proof}

\begin{defn}
The diffeomorphism $\alpha$ determined by $\Phi:W\to W'$ according to
Lemma \ref{lem.alphadiffeo} is called the \emph{time dilation} effected by $\Phi$.
\end{defn}
The time consistent diffeomorphism $\Phi$ in the situation of the lemma
is positive precisely when its time dilation is orientation preserving.
\begin{lemma} \label{lem.ggpagg}
Let $W$ and $W'$ be cobordisms with progressive time, $\Phi:W\to W'$ a time consistent
diffeomorphism and $G:W' \to \cplx$ a fold map. Then 
\[ \genim (G\Phi) = \alpha^{-1} \genim (G), \]
where $\alpha$ is the time dilation of $\Phi$.
\end{lemma}
\begin{proof}
Using the formulae $S(G\Phi) = \Phi^{-1} (S(G))$ and
$W_{\alpha^{-1} (t)} = \Phi^{-1} (W'_t)$, we see that
\[ S(G\Phi)\cap W_{\alpha^{-1}(t)} =
 \Phi^{-1} (S(G)) \cap \Phi^{-1} (W'_t) = \Phi^{-1} (S(G)\cap W'_t). \]
Let $t\in \genim (G)$ and $p,q\in S(G\Phi)\cap W_{\alpha^{-1}(t)}$ with
equality $\Ima G\Phi (p)=\Ima G\Phi (q)$ of imaginary parts.
Then $\Phi (p), \Phi (q) \in S(G)\cap W'_t$ and since
$\Ima G|: S(G)\cap W'_t \to \real$ is injective, we have $\Phi (p)=\Phi (q)$.
As $\Phi$ is one-to-one, $p=q$. This shows that $\Ima G\Phi$ is injective
on $S(G\Phi)\cap W_{\alpha^{-1}(t)},$ that is,
$\alpha^{-1} (t)\in \genim (G\Phi)$. The converse inclusion follows from
reversing the roles of $W$ and $W'$, that is, writing $F=G\Phi$ and applying
the previous reasoning to the time consistent $\Phi^{-1}: W' \to W$ and
$F\Phi^{-1} =G$.
\end{proof}

\begin{lemma}  \label{lem.phiresid}
Let $W$ and $W'$ be cobordisms with progressive time, $\Phi:W\to W'$ a time consistent
diffeomorphism and $G:W' \to \cplx$ a fold map. Then $\genim (G)$ is residual in $[0,1]$
if and only if $\genim (G\phi)$ is residual in $[0,1]$.
\end{lemma}
\begin{proof}
This follows readily from Lemma \ref{lem.ggpagg}:
Suppose that $\genim (G)$ is residual. Since the time dilation $\alpha:[0,1]\to [0,1]$ is
a diffeomorphism by Lemma \ref{lem.alphadiffeo}, $\alpha^{-1} \genim (G)$ is again 
residual. But the latter set equals $\genim (G\Phi)$. The converse direction follows in the 
same manner.
\end{proof}

\begin{lemma} \label{lem.zeroonegenimphi}
Let $W,W'$ be any cobordisms and $\Phi: W\to W'$ a time consistent diffeomorphism
such that for all $k\in \nat$ there is an
$l\in \nat$ with $\Phi (W(k))\subset W'(l)$.
Let $G:W' \to \cplx$ be a fold map.
If $0,1\in \genim G(k)$ for all $k\in \nat$, then
$0,1\in \genim (G\Phi)(k)$ for all $k\in \nat$.
\end{lemma}
\begin{proof}
Suppose that $0,1\in \genim G(l)$ for all $l\in \nat$. Thus
$\Ima G|: S(G)\cap W'_i (l)\to \real$ is injective for all $l$ and $i=0,1$.
Given $k\in \nat$, there is an $l_0$ such that $\Phi (W(k))\subset W'(l_0)$.
If $p,q$ are points in $S(G\Phi)\cap W_0 (k)$ with 
$\Ima G\Phi (p)=\Ima G\Phi (q),$ then $\Phi (p), \Phi (q)$ lie in
the singular set $S(G)$ of $G$ and in
$W'_j (l_0) = W'_j \cap W'(l_0)$, $j$ either $0$ (if $\Phi$ is positive) or $1$ (if $\Phi$ is negative).
Hence $\Phi (p)=\Phi (q)$. As $\Phi$ is one-to-one, we have $p=q$.
This shows that $\Ima G\Phi|: S(G\Phi)\cap W_0 (k) \to \real$ is injective, i.e.
$0\in \genim (G\Phi)(k)$ for every $k$. Similarly $1\in \genim (G\Phi)(k)$ for every $k$.
\end{proof}
A diffeomorphism $\Phi:W\to W'$ induces a homeomorphism
\[ 
\Phi_{\infty}: C^\infty (W',\cplx) \longrightarrow C^\infty (W,\cplx),~
G \mapsto G\circ \Phi
\]
with inverse $(\Phi^{-1})_\infty$.
\begin{lemma} \label{lem.phiresfawpfaw}
For $W$ and $W'$ with progressive time and $\Phi$ time consistent such that
for every $k$ there is an $l$ with $\Phi (W(k))=W'(l)$, 
the homeomorphism $\Phi_\infty$ restricts to a homeomorphism
$\Phi_{\Fa}$ of fold fields:
\[ \xymatrix@R=15pt{
C^\infty (W',\cplx) \ar[r]^{\Phi_\infty} & C^\infty (W,\cplx) \\
\Fa (W') \ar@{^{(}->}[u] \ar@{..>}[r]_{\Phi_{\Fa}} & \Fa (W). \ar@{^{(}->}[u]
} \]
Furthermore, if $\Phi (W(k))=W'(k)$ for all $k$ and $\Phi (W_0)=W'_0$, then
we have Brauer invariance:
\[ \mbs (\Phi_{\Fa} G)=\mbs (G) \in \Mor (\Br) \]
for fields $G\in \Fa (W')$.
\end{lemma}
\begin{proof}
Let $G\in \Fa (W')$ be a fold field. As $F=G\Phi$ is the composition
of a diffeomorphism and a fold map, $F$ is also a fold map.
The singular set $S(G)$ is transverse to the incoming boundary $W'_0$ and to 
the outgoing boundary $W'_1$. As $\Phi$ is a diffeomorphism, its differential
is an isomorphism everywhere. Thus $\Phi^{-1} S(G) = S(G\Phi)$
is transverse to $\Phi^{-1} (W'_0 \sqcup W'_1) = W_0 \sqcup W_1$.
This shows that $0,1\in~ \pitchfork (G\Phi)$. 

Given a natural number $k$, there is an $l$ with $\Phi(W(k))\subset W'(l)$. 
Since $0,1\in \genim (G(l))$ for every $l$,
we conclude with Lemma \ref{lem.zeroonegenimphi} that $0,1\in \genim (G\Phi)(k)$ for
every $k$.
Since $\genim (G)$ is residual in $[0,1]$, $\genim (G\Phi)$ is residual
in $[0,1]$ as well, by Lemma \ref{lem.phiresid}.
This shows that $\Phi_\infty (G) = G\Phi$ is a fold field on $W$ and
so defines the desired map $\Phi_{\Fa}$. The inverse of $\Phi_{\Fa}$ is given by the restriction
of the inverse of $\Phi_\infty$ to $\Fa (W)$.\\

In order to prove Brauer invariance, let us now assume
$\Phi (W(k))=W'(k)$ for all $k$ and $\Phi (W_0)=W'_0$.
Given a field $G\in \Fa (W')$, we need to show $\mbs (F)=\mbs (G)$, where
$F = \Phi_{\Fa} (G)=G\Phi \in \Fa (W)$. Fix a natural number $k$.
The function $\Ima F(k)$ induces a unique ordering
\[ S(F(k)) \cap W(k)_0 = \{ p_1,\ldots, p_m \} \]
such that $\Ima F(p_i)<\Ima F(p_j)$ if and only if $i<j$,
and a unique ordering
\[ S(F(k)) \cap W(k)_1 = \{ q_1,\ldots, q_l \} \]
such that $\Ima F(q_i)<\Ima F(q_j)$ if and only if $i<j$.
With $p'_i = \Phi (p_i),$ $q'_i = \Phi (q_i),$ we have
\[
S(G)\cap W'(k)_0  = 
 \Phi S(F)\cap \Phi W(k)_0 
 =  \Phi (S(F)\cap W(k)_0) 
 =  \{ p'_1, \ldots, p'_m \}
\]
(a set of cardinality $m$, since $\Phi$ is injective) and
\[
S(G)\cap W'(k)_1 =  \{ q'_1, \ldots, q'_l \}
\]
(a set of cardinality $l$). Thus both $\mbs (G(k))$ and $\mbs (F(k))$
are morphisms $[m]\to [l]$ in $\Br$.
The ordering principle is preserved because
\[ \Ima F(p_i)=\Ima (G\Phi)(p_i)=\Ima G(p'_i),~
  \Ima F(q_i)=\Ima (G\Phi)(q_i)=\Ima G(q'_i), \]
so  that $\Ima G(p'_i) < \Ima G(p'_j)$ if and only if $i<j$
if and only if $\Ima G(q'_i) < \Ima G(q'_j)$.
In the Brauer morphism $\mbs (G(k))$, the point $(0,i,0,0)$ is
connected to the point $(0,j,0,0)$ by an arc if and only if
there exists a connected component $c'=\Phi (c)$ of
$S(G(k))$ with boundary $\partial c' = \{ p'_i, p'_j \}$.
Such a $c'$ exists if and only if there is a connected component
$c$ of $S((G\Phi)(k))$ with $\partial c = \{ p_i, p_j \}$,
which is equivalent to $(0,i,0,0)$ and $(0,j,0,0)$
being connected by an arc in $\mbs (F(k))$.
Similarly for connections of type $(1,i,0,0)$ to $(1,j,0,0)$ and
$(0,i,0,0)$ to $(1,j,0,0)$.
The number of loop tensor factors in $\mbs (G(k))$
equals the number of loop tensor factors in $\mbs (F(k))$,
as the number of closed connected components of
$S(G(k))$ equals the number of closed connected components
of $S((G\Phi)(k))$. This shows that $\mbs (G(k))=\mbs (F(k))$ for every $k$.
Therefore,
$\mbs (G) = \bigotimes_k \mbs (G(k)) = \bigotimes_k \mbs (F(k))
  = \mbs (F).$
\end{proof}

Let $M$ and $N$ be closed, smooth $(n-1)$-dimensional manifolds and
$\phi: M\to N$ a diffeomorphism such that for all $k\in \nat$ there is an
$l\in \nat$ with $\phi (M(k)) = N(l)$.
(This implies that 
for every $l\in \nat$ there is a
$k\in \nat$ with $\phi^{-1} (N(l)) = M(k)$.)
We shall show in stages that $\phi$ induces
an isomorphism 
\[ \phi_\ast: Z(M)\to Z(N) \] 
of additive monoids. Let us denote the unit interval $[0,1]$ by $I$.
First, $\phi$ determines a diffeomorphism $\ophi: I\times M \to I\times N$
by $\overline{\phi} (t,x) = (t,\phi (x))$. For later reference, we observe that
if $\psi: N\to P$ is another diffeomorphism, then
\begin{equation} \label{equ.opsipgiisopsiophi}
(\overline{\psi \phi})(t,x) = (t,\psi \phi (x)) = \overline{\psi}(t,\phi (x)) =
  (\overline{\psi} \circ \ophi)(t,x). 
\end{equation}
Using this suspension of $\phi$, $\phi$
induces a homeomorphism
\[ 
\phi_{\infty}: C^\infty (I\times N,\cplx) \longrightarrow C^\infty (I\times M,\cplx),~ 
g \mapsto g\circ \overline{\phi}
\]
with inverse $(\phi^{-1})_\infty$.
The next lemma is a special case of Lemma \ref{lem.phiresfawpfaw}:
\begin{lemma}
The homeomorphism $\phi_\infty$ restricts to a homeomorphism
$\phi_{\Fa}$ of fold fields:
\[ \xymatrix{
C^\infty (I\times N,\cplx) \ar[r]^{\phi_\infty} & C^\infty (I\times M,\cplx) \\
\Fa (I\times N) \ar@{^{(}->}[u] \ar@{..>}[r]_{\phi_{\Fa}} & \Fa (I\times M). \ar@{^{(}->}[u]
} \]
\end{lemma}
\begin{proof}
By Example \ref{expl.cylprogr}, the cylindrical cobordisms $W=I\times M$ and $W'=I\times N$
have progressive time.
The diffeomorphism $\overline{\phi}$ is by definition time preserving, in particular
time consistent.
Given a natural number $k$, there is an $l$ with $\phi (M(k))= N(l)$. For this $l$,
$\overline{\phi}(W(k))= W'(l)$. The conclusion follows from Lemma \ref{lem.phiresfawpfaw}.
\end{proof}

\begin{lemma} \label{lem.phimorse}
The homeomorphism $\phi_{\Fa}$ restricts to a homeomorphism
$\phi_{\closed}$ on $\Fa (N)$:
\[ \xymatrix@R=15pt{
\Fa (I\times N) \ar[r]^{\phi_{\Fa}} & \Fa (I\times M) \\
\Fa (N) \ar@{^{(}->}[u] \ar@{..>}[r]_{\phi_{\closed}} & \Fa (M). \ar@{^{(}->}[u]
} \]
\end{lemma}
\begin{proof}
Let $g\in \Fa (N)$; that is, $g\in \Fa (I\times N)$ is a fold field with $\mbs (g)=1\in \Mor (\Br)$.
Since $\mbs (g) = \bigotimes_{l\in \nat} \mbs (g(l)),$ an application of 
Lemma \ref{lem.idfactorstriv} shows that $\mbs (g(l))=1$ for every $l$.
Let $k$ be a natural number. We shall show that $\mbs (f(k))=1$, where
$f =g\overline{\phi}$. Let $l$ be such that $\phi (M(k))= N(l)$.
Then $\ophi ([0,1]\times M(k))= [0,1]\times N(l).$ By the injectivity of $g(l)$
on $S(g(l))\cap 0\times N(l),$ there is a unique ordering $p_1,\ldots, p_c$ of the
points of $S(g(l))\cap 0\times N(l)$ such that $\Ima g(p_i) < \Ima g(p_j)$ if and only if
$i<j$. Furthermore, there is a unique ordering $q_1,\ldots, q_c$ of the
points of $S(g(l))\cap 1\times N(l)$ such that $\Ima g(q_i) < \Ima g(q_j)$ if and only if
$i<j$. Since $\mbs (g(l))$ is the identity, the sets $S(g(l))\cap 0\times N(l)$ and
$S(g(l))\cap 1\times N(l)$ indeed have the same cardinality $c$ and $\{ p_i, q_i \}$
is the boundary of a connected component $C_i$ of $S(g(l))$ for all $i$. 
From $\ophi S(g\ophi)=S(g)$ it follows that
\begin{align*}
\ophi S(f(k)) &= \ophi S(g \ophi|_{I\times M(k)}) = \ophi (S(g\ophi)\cap I\times M(k)) \\
&= \ophi S(g\ophi) \cap \ophi (I\times M(k)) = S(g)\cap I\times N(l) = S(g|_{I\times N(l)}) = S(g(l)).
\end{align*}
Hence
\[ \ophi (S(f(k))\cap r\times M(k)) = \ophi S(f(k))\cap \ophi (r\times M(k))
 = S(g(l))\cap r\times N(l),~ r=0,1, \]
and
\[ S(f(k))\cap 0\times M(k) = \{ \ophi^{-1} (p_1), \ldots, \ophi^{-1} (p_c) \},\]
\[ S(f(k))\cap 1\times M(k) = \{ \ophi^{-1} (q_1), \ldots, \ophi^{-1} (q_c) \}. 
\]
Set $P_i = \ophi^{-1} (p_{i})$ and $Q_i = \ophi^{-1} (q_{i}),$ $i=1,\ldots, c$.
Then $S(f(k))\cap 0\times M(k) = \{ P_1,\ldots, P_c \}$ with
\begin{align*}
\Ima f(P_i) < \Ima f(P_j) & \Leftrightarrow 
 \Ima (g\ophi)(\ophi^{-1} p_{i}) < \Ima (g\ophi)(\ophi^{-1} p_{j}) \\
& \Leftrightarrow \Ima g(p_{i}) < \Ima g(p_{j}) \\
& \Leftrightarrow i<j
\end{align*}
and
$S(f(k))\cap 1\times M(k) = \{ Q_1,\ldots, Q_c \}$ with
\[ \Ima f(Q_i) < \Ima f(Q_j)  \Leftrightarrow  i<j. \]
For every $i,$ $\ophi^{-1} (C_{i})$ is a connected component of
$S(f)\cap [0,1] \times M(k)$ with
$\partial \ophi^{-1} (C_{i}) = \{ P_i, Q_i \}.$ Thus
$\mbs (f(k))$ connects $(0,i,0,0)$ to $(1,i,0,0)$ for every $i=1,\ldots, c$.
Since there are no loops in $S(f(k))$ (because there are no loops in $S(g(l))$),
we conclude that $\mbs (f(k))=1_{[c]} \in \operatorname{End}_{\Br} ([c])$.
Letting $k$ vary,
$\mbs (f) = \bigotimes_{k\in \nat} \mbs (f(k)) = \bigotimes_k 1 =1.$
This shows that $f = \phi_{\Fa} (g)$ is an element of $\Fa (M) \subset
\Fa (I\times M)$ and we obtain a map $\phi_{\closed}: \Fa (N)\to \Fa (M)$.
The inverse of $\phi_{\closed}$ is given by restricting the inverse of 
$\phi_{\Fa}$ to $\Fa (M)$.
\end{proof}
We proceed to define $\phi_\ast: Z(M)\to Z(N)$. Let $z\in Z(M)$ be a state,
i.e. a function $z: \Fa (M)\to Q= Q(i,e),$ $(i,e)$ a duality structure on
a vector space $V$ as before. Composition with $\phi_{\closed}$ yields a state
$\phi_\ast (z)$ in $Z(N)$, that is, we put
\[ \phi_\ast (z) = z\circ \phi_{\closed} \in Z(N). \]
Given $g\in \Fa (N),$ we have the more explicit formula
\[ \phi_\ast (z)(g) = z(g\ophi) \in Q. \]
The function $\phi_\ast$ is additive: 
\[ \phi_\ast (z_1 + z_2) = \phi_\ast (z_1) + \phi_\ast (z_2). \]
Let $\psi: N\to P$ be another diffeomorphism such that for every $k$
there is an $l$ with $\psi (N(k))=P(l)$. Then, using
(\ref{equ.opsipgiisopsiophi}), we obtain for every $h\in \Fa (P),$
\[ \psi_\ast (\phi_\ast (z))(h) = \phi_\ast (z) (h\overline{\psi}) =
 z ((h\overline{\psi})\ophi) = z (h (\overline{\psi \phi}))=
 (\psi \phi)_\ast (z)(h), \]
proving the functoriality relation
\[ \psi_\ast \circ \phi_\ast = (\psi \circ \phi)_\ast. \]
Of course, $(\id_M)_\ast = \id_{Z(M)}: Z(M)\to Z(M).$
Together, these properties imply
\[ (\phi^{-1})_\ast \phi_\ast = (\phi^{-1} \phi)_\ast = (\id_M)_\ast = \id_{Z(M)} \]
and $\phi_\ast (\phi^{-1})_\ast = \id_{Z(N)}$. This shows that $\phi_\ast$ is
an isomorphism.\\

In a similar manner, diffeomorphisms induce maps on tensor products
of state modules. Let $\phi^1: M_1 \to N_1$ and $\phi^2: M_2 \to N_2$
be diffeomorphisms such that for every $k\in \nat$ and $i=1,2,$ there is an $l\in \nat$
with $\phi^i (M_i (k))=N_i (l).$ Let 
\[ \phi = \phi^1 \sqcup \phi^2: M_1 \sqcup M_2 \stackrel{\cong}{\longrightarrow}
  N_1 \sqcup N_2 \]
be the disjoint union of $\phi^1$ and $\phi^2$. Then $\phi$ induces an
isomorphism
\[ \phi_\ast: Z(M_1)\hotimes Z(M_2) \longrightarrow
  Z(N_1)\hotimes Z(N_2). \]
The image of a state $z: \Fa (M_1)\times \Fa (M_2)\to Q$
under $\phi_\ast$ on a pair 
$(g_1, g_2)\in \Fa (N_1)\times \Fa (N_2)$ is given by
\[ \phi_\ast (z)(g_1, g_2) = z(\phi^1_{\closed} (g_1), \phi^2_{\closed}(g_2)), \]
using the homeomorphisms
\[ \phi^i_{\closed}:\Fa (N_i) \to \Fa (M_i),~ i=1,2, \]
given by Lemma \ref{lem.phimorse}. The assumptions on $\phi^1$ and $\phi^2$
do not imply that for a given $k$, we can find an $l$ with
$\phi ((M_1 \sqcup M_2)(k))=(N_1 \sqcup N_2)(l)$, but if such an $l$ can be found
for every $k$, then $\phi$ also induces
an isomorphism $\phi_\ast: Z(M_1 \sqcup M_2)\to Z(N_1 \sqcup N_2)$.
If $M_1, M_2$ are well-separated and $N_1,N_2$ are well-separated, then
by Proposition \ref{prop.zmonoidalonmodules},
$Z(M_1\sqcup M_2)\cong Z(M_1)\hotimes Z(M_2)$ and
$Z(N_1\sqcup N_2)\cong Z(N_1)\hotimes Z(N_2)$. In this situation there
is then a commutative diagram
\[ \xymatrix{
Z(M_1 \sqcup M_2) \ar[r]^{\phi_\ast} & Z(N_1 \sqcup N_2) \\
Z(M_1)\hotimes Z(M_2) \ar[u]^{\cong} \ar[r]^{\phi_\ast} &
  Z(N_1)\hotimes Z(N_2). \ar[u]_{\cong}
} \]

Let $W$ be a cobordism with progressive time from $M$ to $N$ and 
$W'$ a cobordism with progressive time from $M'$ to $N'$. Assume without
loss of generality that $W$ and $W'$ are endowed with equal cylinder scales
$\epsilon_W = \epsilon_{W'}.$
Let $\phi: \partial W \to \partial W'$ be a diffeomorphism that preserves
incoming and outgoing boundaries, that is, for every $k$ there is an $l$
with $\phi (M(k))=M'(l)$ and for every $k$ there is an $l$ with
$\phi (N(k))=N'(l)$. In this situation, we have the following result concerning
the introduction of boundary conditions.
\begin{lemma} \label{lem.phifb}
Let $f' \in \Fa (M')$ and $g' \in \Fa (N')$ be boundary conditions.
If $\phi$ extends to a time consistent diffeomorphism 
$\Phi: W\to W'$ which agrees with $\phi$ levelwise near the boundary, i.e.
$\Phi (t,x)=(t,\phi (x))$ for all $(t,x)\in [0,\epsilon_W] \times M \sqcup
[1-\epsilon_W, 1]\times N$, and has
for every $k$ an $l$ with $\Phi (W(k))=W'(l),$ then $\Phi_{\Fa}$ restricts
to a homeomorphism $\Phi_{\Fa \Ba}$ respecting the boundary conditions:
\[ \xymatrix{
\Fa (W') \ar[r]^{\Phi_{\Fa}} & \Fa (W) \\
\Fa (W';f',g') \ar@{^{(}->}[u] \ar@{..>}[r]_<<<<<<{\Phi_{\Fa \Ba}} & 
     \Fa (W; f'\ophi|_{I\times M}, g'\ophi|_{I\times N}). \ar@{^{(}->}[u]
} \]
\end{lemma}
\begin{proof}
Given a field $G\in \Fa (W'; f', g')$, there exist for every $l=0,1,2,\ldots$
numbers $\delta (l), \delta'(l)\in (0,\epsilon_{W'})$ such that
\[ G|_{[0,\delta (l)] \times M'(l)} \approx f'(l),~
  G|_{[1-\delta' (l),1] \times N'(l)} \approx g'(l) \]
for all $l$. Hence there are diffeomorphisms
$\xi: [0,\delta (l)]\to [0,1]$ with $\xi (0)=0$ and such that
\[ G(t,y) = f'(\xi (t), y),~ (t,y)\in [0,\delta (l)]\times M'(l). \]
Let $k$ be a natural number and let $l$ be such that
$\phi (M(k))=M'(l)$. Setting $\epsilon (k)=\delta (l)$, we have
$0<\epsilon (k)<\epsilon_W$ since $\epsilon_W = \epsilon_{W'}$.
For $(t,x)\in [0,\epsilon (k)]\times M(k)$, we obtain
\[ G\Phi (t,x) = G(t,\phi (x))=f'(\xi (t),\phi (x))=
 (f'\ophi|_{I\times M})(\xi (t),x), \]
which shows that
\[ (G\Phi)|_{[0,\epsilon (k)] \times M(k)} \approx (f' \ophi|_{I\times M})(k). \]
Now taking $l$ such that $\phi (N(k))=N'(l)$ and setting 
$\epsilon' (k)=\delta' (l)$, we see similarly that
\[ (G\Phi)|_{[1-\epsilon' (k),1] \times N(k)} \approx (g' \ophi|_{I\times N})(k). \]
Hence $\Phi_{\Fa} (G)=G\Phi \in \Fa (W; f'\ophi|_{I\times M}, g'\ophi|_{I\times N})$
and we obtain the desired map $\Phi_{\Fa \Ba}$. The inverse of
$\Phi_{\Fa \Ba}$ is given by restricting the inverse $(\Phi_{\Fa})^{-1} = (\Phi^{-1})_{\Fa}$
to $\Fa (W; f'\ophi|_{I\times M}, g'\ophi|_{I\times N})$, noting that
$(f'\ophi|_{I\times M}) \overline{\phi^{-1}}|_{I\times M'} = f'$
(analogously for $g'$). 
\end{proof}

\begin{thm} \label{thm.tcdiffeoinvariance}
(Time Consistent Diffeomorphism Invariance.)
Let $W$ be a cobordism with progressive time from $M$ to $N$ and 
$W'$ a cobordism with progressive time from $M'$ to $N'$. Assume without
loss of generality that $W$ and $W'$ are endowed with equal cylinder scales
$\epsilon_W = \epsilon_{W'}.$
Let $\phi: \partial W \to \partial W'$ be a diffeomorphism that preserves
incoming and outgoing boundaries in the sense that
$\phi (M(k))=M'(k)$ and $\phi (N(k))=N'(k)$ for every $k$.
If $\phi$ extends to a time consistent diffeomorphism 
$\Phi: W\to W'$ which agrees with $\phi$ levelwise near the boundary, i.e.
$\Phi (t,x)=(t,\phi (x))$ for all $(t,x)\in [0,\epsilon_W] \times M \sqcup
[1-\epsilon_W, 1]\times N$, and has $\Phi (W(k))=W'(k)$ for every $k$, then
the state sums of $W$ and $W'$ are related by
\[ \phi_\ast (Z_W) = Z_{W'} \in Z(M')\hotimes Z(N') \]
under the isomorphism
\[ \phi_\ast: Z(M)\hotimes Z(N)\longrightarrow Z(M')\hotimes Z(N'). \]
\end{thm}
\begin{proof}
Let $\phi^1: M\to M'$ and $\phi^2: N\to N'$ denote the restrictions of $\phi$
to the incoming and outgoing boundaries, respectively. We have
$\overline{\phi^1} = \ophi|_{I\times M}$ and
$\overline{\phi^2} = \ophi|_{I\times N}$.
Subject to the boundary condition $(f',g')\in \Fa (M')\times \Fa (N'),$
the pushforward $\phi_\ast (Z_W)$ is given by
\[ \phi_\ast (Z_W)(f',g') = Z_W (\phi^1_{\closed} (f'), \phi^2_{\closed} (g')) =
 Z_W (f' \overline{\phi^1}, g' \overline{\phi^2}) =
 \sum_{F\in \Fa (W; f'\ophi|_{I\times M}, g'\ophi|_{I\times N})} Y\mbs (F), \]
while the state sum $Z_{W'}$ on $(f',g')$ is
\[ Z_{W'} (f',g') = \sum_{G\in \Fa (W'; f',g')} Y\mbs (G). \]
Let $\mathfrak{S}(W)$ be the set
\[
\mathfrak{S}(W)  = 
 \{ \mbs (F) ~|~ F\in \Fa (W;  f'\ophi|_{I\times M}, g'\ophi|_{I\times N}) \} 
\]
and let $\mathfrak{S}(W')$ be the set
\[
\mathfrak{S}(W')  = 
 \{ \mbs (G) ~|~ G\in \Fa (W';  f', g') \}. 
\]
Lemma \ref{lem.phifb} furnishes a homeomorphism
\[ \Phi_{\Fa \Ba}: \Fa (W';f',g') \stackrel{\cong}{\longrightarrow} 
     \Fa (W; f'\ophi|_{I\times M}, g'\ophi|_{I\times N})  \]
with inverse $(\Phi^{-1})_{\Fa \Ba}$.
Since $\Phi (W(k))=W'(k)$ for all $k$ and
$\Phi (W_0)=\Phi (M)=M' = W'_0,$ the Brauer invariance under $\Phi_{\Fa}$
and under $(\Phi^{-1})_{\Fa}$ provided by Lemma \ref{lem.phiresfawpfaw}
holds. Given an element $\mbs (G)\in \mathfrak{S}(W'),$ 
$\Phi_{\Fa \Ba} (G)$ is a fold field in $\Fa (W; f'\ophi|_{I\times M}, g'\ophi|_{I\times N}),$
and by Brauer invariance
\[ \mbs (G)= \mbs (\Phi_{\Fa \Ba} G) \in \mathfrak{S}(W). \]
Conversely, given an element $\mbs (F)\in \mathfrak{S}(W),$ 
$(\Phi^{-1})_{\Fa \Ba} (F)$ is a fold field in $\Fa (W'; f', g')$
with
\[ \mbs (F)= \mbs ((\Phi^{-1})_{\Fa \Ba} F) \in \mathfrak{S}(W'). \]
Consequently, $\mathfrak{S}(W) =
\mathfrak{S}(W')$ and, applying the functor $Y$,
$Y(\mathfrak{S}(W)) = Y(\mathfrak{S}(W')).$
By Proposition \ref{prop.qcontinuous}, the idempotent additive monoid $Q$ is continuous.
Thus we can conclude from Proposition \ref{prop.contidemsemiring}
that
\[  \sum_{F\in \Fa (W; f'\ophi|_{I\times M}, g'\ophi|_{I\times N})} Y\mbs (F) =
   \sum_{G \in \Fa (W'; f', g')} Y\mbs (G). \]
\end{proof}

\begin{cor} \label{cylstabpertbnd}
(Cylindrical Stability under Perturbation of Boundary Conditions.)
Let $\phi,\psi: M\to M$ be isotopic diffeomorphisms (with $\phi (M(k))=M(k)$ for all $k$) and
$f,g\in \Fa (M)$. Then the state sum of the cylinder $I\times M$
satisfies
\[ Z_{I\times M} (f\ophi, g\overline{\psi})= Z_{I\times M} (f,g). \]
In particular, if $\phi: M\to M$ is isotopic to the identity, then
$Z_{I\times M}(f\ophi, g)=Z_{I\times M} (f,g)$; analogously for $\psi$ and $g$.
\end{cor}
\begin{proof}
We note that $\phi M(k)=M(k)$ implies $\psi M(k)=M(k)$, as $\phi$ and $\psi$ are isotopic.
The cylinder $I\times M$ has progressive time. Let $\Xi: I\times M \to M$
be a smooth isotopy from $\phi$ to $\psi$, arranged so that
$\Xi$ is constant in $t\in I$ on a neighborhood of $0\times M$ and
on a neighborhood of $1\times M$. Then $\Phi: I\times M \to I\times M,$
$\Phi (t,x)=(t,\Xi (t,x))$ satisfies all the hypotheses of Theorem \ref{thm.tcdiffeoinvariance}
on time consistent diffeomorphism invariance (in fact $\Phi$ is even time preserving)
and we conclude that
$\alpha_\ast (Z_{I\times M})=Z_{I\times M},$ where $\alpha: \{ 0,1 \} \times M
\to \{ 0,1 \} \times M$ is $\phi$ on $0\times M$ and $\psi$ on $1\times M$.
Hence
\[ Z_{I\times M} (f,g) = \alpha_\ast Z_{I\times M} (f,g) =
 Z_{I\times M} (f\ophi, g\overline{\psi}). \]
If $\psi$ is the identity, then $\overline{\psi}$ is the identity as well.
\end{proof}

Since pseudo-isotopies need not be time consistent, Corollary \ref{cylstabpertbnd}
does not immediately generalize to $\phi, \psi$ being only pseudo-isotopic.
However, if $M$ is simply connected and $n-1=\dim M\geq 5,$ then
Cerf's pseudo-isotopy theorem asserts that the two relations of isotopy and
pseudo-isotopy coincide.

\begin{cor}
Let $W$ and $W'$ be cobordisms with progressive time, both from $M$ to $N$ 
(endowed with equal cylinder scales). If there exists a
time consistent diffeomorphism 
$\Phi: W\to W'$ which is the identity near $M$ and near $N$ 
and has $\Phi (W(k))=W'(k)$ for every $k$, then
\[ Z_W = Z_{W'} \in Z(M)\hotimes Z(N). \]
\end{cor}

\section{Exotic Spheres}
\label{sec.exoticspheres}

Throughout this section, we assume the dimension $n$ to satisfy $n\geq 5$.
Let $\Sigma^n$ be an $n$-dimensional exotic smooth sphere, for example the $7$-dimensional
Milnor sphere. Thus $\Sigma^n$ is a smooth manifold homeomorphic, but not diffeomorphic, to 
the standard sphere $S^n$. On $S^n$, there is a special generic function 
(see Definition \ref{def.specialgenfn})
with precisely $2$ critical points,
that is, a Morse function $S^n \to \real$ with precisely one maximum point and one minimum point.
For $n\geq 5$, every $n$-dimensional exotic sphere is diffeomorphic to a twisted sphere,
a result due to Smale.
(Twisted spheres are smooth manifolds of the form $D^n \cup_\tau D^n,$ where
$\tau: \partial D^n = S^{n-1} \stackrel{\cong}{\longrightarrow} S^{n-1}$ is an
orientation preserving diffeomorphism.) 
Every twisted sphere has Morse number $2$. 
Thus on $\Sigma^n$ there is also a special generic function with
precisely $2$ critical points, a maximum point and a minimum point. For two closed
$n$-manifolds $M,N,$ let $\operatorname{Cob}(M,N)$ denote the collection of all 
oriented (embedded as in Section \ref{sec.embandcob}) 
cobordisms $W^{n+1}$ from $M$ to $N$. (If $M$ and $N$ are connected,
we will assume that $M=M(0),$ $N=N(0)$ and
$W=W(0)$ for every $W$ in $\operatorname{Cob}(M,N)$.)
\begin{lemma}
The collection
$\operatorname{Cob}(S^n, \Sigma^n)$ is not empty.
\end{lemma}
\begin{proof}
Homotopy spheres are stably parallelizable.
Thus all Pontrjagin numbers and all Stiefel-Whitney numbers of $\Sigma^n$ are trivial.
It follows that $\Sigma^n = \partial W'$ for some oriented compact smooth $(n+1)$-manifold $W'$.
Removing a small open ball from the interior of $W'$, we obtain an oriented cobordism $W$
from $S^n$ to $\Sigma^n$.
\end{proof}
Let $f_S: S^n \to \real$ be a Morse function with precisely $2$ critical points. 
Note that $f_S$ is excellent (Definition \ref{def.excellent}). 
Let $\overline{f}_S$ be the suspension of 
$f_S$, that is,
\[ \overline{f}_S = \id_I \times f_S: I\times S^n \longrightarrow I\times \real \subset \cplx. \]
Then $\overline{f}_S$ is a fold field with
$\mbs (\overline{f}_S)=1$ and hence defines an
element $\overline{f}_S \in \Fa (S^n),$ see also Remark \ref{rem.suspexcmorse}.
For any closed, smooth manifold $M$ homeomorphic to a sphere
(exotic or not), let $C_2 (M)$ denote the space of all $\overline{f}_M \in \Fa (M)$, which are
of the form $\overline{f}_M (t,x) = (\xi (t), f_M (x)),$ $(t,x)\in [0,1]\times M,$
for some diffeomorphism $\xi: [0,1] \to [a,b]$ with $\xi (0)=a$,
where $f_M: M\to \real$ is a Morse function with precisely two critical points.
A cobordism $W \in \operatorname{Cob}(S^n, M)$ has a state sum
$Z_W \in Z(S^n)\hotimes Z(M),$ $Z_W: \Fa (S^n)\times \Fa (M)\to
Q$, so we can evaluate on the boundary condition
$(\overline{f}_S, \overline{f}_M) \in \Fa (S^n)\times \Fa (M)$, where 
we wish to concentrate on $\overline{f}_M \in C_2 (M)$.
We get
$Z_W (\overline{f}_S, \overline{f}_M) \in Q$ and shall study the \emph{aggregate invariant}
\[ \mathfrak{A}(M) := \sum_{\overline{f}_M \in C_2 (M)}
 \sum_{W\in \operatorname{Cob}(S^n, M)} 
  Z_W (\overline{f}_S, \overline{f}_M) \in Q, \]
which is well-defined by the Eilenberg-completeness of $Q$.
Here, as before, $Q$ is the semiring associated via profinite idempotent completion to
a duality structure on a real vector space of finite dimension at least $2$, such that
the associated functor $Y$ is faithful on loops.
\begin{thm} \label{thm.iexoticsphere}
If $\Sigma^n,$ $n\geq 5,$ is an exotic sphere not diffeomorphic to $S^n$, 
then the invariant $\mathfrak{A}(\Sigma^n)$ is a multiple of $q$. 
\end{thm}
\begin{proof}
We need to show that for every cobordism
$W\in \operatorname{Cob}(S^n, \Sigma^n),$ every map $\overline{f}_\Sigma \in C_2 (\Sigma^n)$ and
every fold field $F: W\to \cplx$ with 
\[ F|_{[0,\epsilon]\times S^n} \approx \overline{f}_S,~
F|_{[1-\epsilon,1]\times \Sigma^n} \approx \overline{f}_\Sigma, \]
the Brauer morphism $\mbs (F)\in 
\operatorname{Mor}(\Br),$ a morphism from $2$ points to
$2$ points in $\Br$, contains a loop $\lambda: I\to I$. 
For then every summand $Y\mbs (F)\otimes 1$ of $\mathfrak{A}(\Sigma^n)$ is a multiple of $q$,
\[ Y\mbs (F)\otimes 1 = Y(\lambda \otimes \phi)\otimes 1 =
  (\lh \otimes Y(\phi))\otimes 1 = Y(\phi)\otimes q, \]
and consequently $\mathfrak{A}(\Sigma^n)$ itself is a multiple of $q$.

To fit into the cobordism framework of \cite{saekihtpyspheres}, 
we shall construct a modification $\widetilde{F}:W\to \real^2$ of $F$ such that
\begin{enumerate}
\item[(i)] $\widetilde{F}$ is a fold map with $\mbs (\widetilde{F})=\mbs (F)$,
\item[(ii)] $\widetilde{F} (W_0)\subset \{ 0 \} \times \real,$
  $\widetilde{F} (W_1)\subset \{ 1 \} \times \real,$ $\widetilde{F} (\interi W)\subset (0,1) \times \real,$
\item[(iii)] there exist $\epsilon_0, \epsilon_1 \in (0,\epsilon)$ such that
\[ \widetilde{F}(t,x) = \begin{cases} (t,f_S (x)), &\text{ for } (t,x)\in [0,\epsilon_0]\times S^n, \\
  (t,f_\Sigma (x)) &\text{ for } (t,x)\in [1-\epsilon_1,1]\times \Sigma^n. \end{cases} \] 
\end{enumerate}
(It is not required that $\widetilde{F}$ is a fold \emph{field}.)
As $F|_{[0,\epsilon]\times S^n} \approx \overline{f}_S,$ there exists a diffeomorphism
$\xi_S:[0,\epsilon] \to [0,1],$ $\xi_S (0)=0,$ such that
$F(t,x)=\overline{f}_S (\xi_S (t),x)=(\xi_S (t), f_S (x))$ for $(t,x)\in [0,\epsilon]\times S^n$.
In particular, we have $F(0,x)=(\xi_S (0), f_S (x))=(0,f_S (x))$.
Similarly for the outgoing 
boundary $\Sigma^n$: As $F|_{[1-\epsilon,1]\times \Sigma^n} \approx \overline{f}_\Sigma$,
there exists a diffeomorphism
$\xi_\Sigma:[1-\epsilon,1] \to [0,1],$ $\xi_\Sigma (1)=1,$ such that
$F(t,x)=\overline{f}_\Sigma (\xi_\Sigma (t),x)$ for
$(t,x)\in [1-\epsilon,1]\times \Sigma^n$. 
Since $\overline{f}_\Sigma \in C_2 (\Sigma^n),$ there exists a diffeomorphism
$\overline{\xi}: [0,1]\to [a,b],$ $\overline{\xi}(0)=a,$ and a 
Morse function $f_\Sigma: \Sigma^n \to \real$ with precisely two critical points such that
$\overline{f}_\Sigma (t,x)=(\overline{\xi}(t),f_\Sigma (x))$. 
Hence, when $(t,x)\in [1-\epsilon,1]\times \Sigma^n$,
$F(t,x)=(\overline{\xi}_\Sigma (t), f_\Sigma (x)),$ where $\overline{\xi}_\Sigma$ is
the diffeomorphism
$\overline{\xi}_\Sigma = \overline{\xi} \circ \xi_\Sigma: [1-\epsilon,1]\to [a,b].$
For $t=1$, $F(1,x)=(b,f_\Sigma (x))$.
We will first construct an intermediate modification $F':W\to \real^2$.
Choose an $\epsilon_0 >0$ with $\xi_S (\epsilon_0)\in (0,\frac{1}{3})$ and
$\epsilon_0 < \min (\epsilon, \frac{1}{3})$. This is possible since
$\xi_S (t)\to 0$ as $t\to 0$. As $F(W)$ is compact and thus bounded as a subset of $\real^2$,
there exists an $R> \max \{ |a|, |b|, 1 \}$ such that $F(W)\subset (-R,R)\times (-R,R)$. Let
$\eta_0 = \xi_S (\epsilon_0)$.
As $\eta_0 < \frac{1}{3}$ and $\xi^{-1}_S (\eta_0)-R = \epsilon_0 -R < \frac{2}{3},$
there exists a diffeomorphism
$\phi_S: [0,1] \to [-R,1]$ such that
\[ \phi_S (t) = \begin{cases}
t, & t\in [\frac{2}{3}, 1], \\
\xi^{-1}_S (t)-R, & t\in [0,\eta_0].
\end{cases} \]
Similarly, choose an $\epsilon_1 >0$ with 
$\overline{\xi}_\Sigma (1-\epsilon_1)\in (a+\frac{2}{3} (b-a),b)$ and
$\epsilon_1 < \min (\epsilon, \frac{1}{3}(b-a), \frac{1}{3})$. This is possible since
$\overline{\xi}_\Sigma (t)\to b$ as $t\to 1$. Let
$\eta_1 = \overline{\xi}_\Sigma (1-\epsilon_1)$.
As $a+\frac{1}{3} (b-a)< \eta_1$ and 
\[ a+\frac{1}{3} (b-a)< a + \frac{2}{3}(b-a) = b- \frac{1}{3}(b-a) < b-\epsilon_1 < R-\epsilon_1 =
  \overline{\xi}^{-1}_\Sigma (\eta_1) -1+R, \]
there exists a diffeomorphism
$\phi_\Sigma: [a,b] \to [a,R]$ such that
\[ \phi_\Sigma (t) = \begin{cases}
t, & t\in [a,a+\frac{1}{3}(b-a)], \\
\overline{\xi}^{-1}_\Sigma (t)-1+R, & t\in [\eta_1,b].
\end{cases} \]
Setting
\[ W' = W- ([0,\xi^{-1}_S (\frac{2}{3})) \times S^n \sqcup
  (\overline{\xi}^{-1}_\Sigma (a+\frac{1}{3}(b-a)),1] \times \Sigma^n ), \]
we define $F':W\to \real^2$ on $y\in W$ by
\[ F' (y) = \begin{cases}
(\phi_S \times \id_\real)(F(t,x)),& y=(t,x)\in [0,\epsilon]\times S^n \\
F(y),& y\in W' \\
(\phi_\Sigma \times \id_\real)(F(t,x)),& y=(t,x)\in [1-\epsilon,1]\times \Sigma^n.
\end{cases} \]
Then $F'$ is indeed well-defined, since for $y=(t,x)$ in the overlap
\[ W' \cap [0,\epsilon]\times S^n = [\xi^{-1}_S (\frac{2}{3}),\epsilon]\times S^n, \]
we have
\begin{eqnarray*}
(\phi_S \times \id_\real)(F(t,x)) 
& = & (\phi_S \times \id_\real)(\xi_S (t), f_S (x)) =
  (\phi_S (\xi_S (t)), f_S (x)) \\
& = & (\xi_S (t), f_S (x)) = F(t,x)=F(y), 
\end{eqnarray*}
while for $y=(t,x)$ in the overlap
\[ W' \cap [1-\epsilon,1]\times \Sigma^n = [1-\epsilon, 
\overline{\xi}^{-1}_\Sigma (a+\frac{1}{3}(b-a))]\times \Sigma^n, \]
we have
\begin{eqnarray*}
(\phi_\Sigma \times \id_\real)(F(t,x)) 
& = & (\phi_\Sigma \times \id_\real)(\overline{\xi}_\Sigma (t), f_\Sigma (x)) =
  (\phi_\Sigma (\overline{\xi}_\Sigma (t)), f_\Sigma (x)) \\
& = & (\overline{\xi}_\Sigma (t), f_\Sigma (x)) = F(t,x)=F(y). 
\end{eqnarray*}
(Note that $\overline{\xi}_\Sigma (1-\epsilon)=a<a+\frac{1}{3}(b-a)$ and thus
$1-\epsilon < \overline{\xi}^{-1}_\Sigma (a+\frac{1}{3}(b-a))$.)
Since the restrictions of $F'$ to $[0,\epsilon]\times S^n,$ to $W'$ and to $[1-\epsilon,1]\times \Sigma^n$
are each smooth, and every point $y\in W$ has an open neighborhood in $W$ that lies entirely in
 $[0,\epsilon)\times S^n,$ or in the interior of $W',$ or in $(1-\epsilon,1]\times \Sigma^n,$ we conclude that
$F'$ is smooth. The restriction $F'|_{W'} = F|_{W'}$ is a fold map. The restriction
$F'|_{[0,\epsilon]\times S^n}$ is the composition of a fold map with a diffeomorphism, hence itself
a fold map. Similarly, $F'|_{[1-\epsilon,1]\times \Sigma^n}$ is a fold map. Therefore, $F':W\to \real$
is a fold map. Note that $S(F')=S(F)$ for the singular sets. The imaginary part of the image
$F' (0,x)$ of an incoming singular point $(0,x)\in S(F)\cap (\{ 0 \} \times S^n)$ is the same as the imaginary part
of $F(0,x),$ namely $f_S (x)$. An analogous statement holds for the outgoing singular points.
This shows that $\mbs (F') = \mbs (F)$.
On points $(t,x)\in [0,\epsilon_0]\times S^n$ we find that
\[
F'(t,x) = (\phi_S (\xi_S (t)), f_S (x)) = (\xi^{-1}_S (\xi_S (t)) -R, f_S (x)) 
 = (t-R, f_S (x)),  
\]
and on $(t,x)\in [1-\epsilon_1, 1]\times \Sigma^n,$
\[
F'(t,x) = (\phi_\Sigma (\overline{\xi}_\Sigma (t)), f_\Sigma (x)) = 
        (\overline{\xi}^{-1}_\Sigma (\overline{\xi}_\Sigma (t)) -1+R, f_\Sigma (x)) 
 =  (t-1+R, f_\Sigma (x)).  
\]
This shows in particular that the image of $W$ under $F'$ satisfies
\begin{equation} \label{equ.imageoffprime}
F'(W_0)\subset \{ -R \} \times \real,~
  F'(W_1)\subset \{ R \} \times \real,~ F' (\interi W)\subset (-R,R) \times \real. 
\end{equation}
This completes the construction of $F'$. 

There exists a diffeomorphism $\lambda: [-R,R]\to [0,1]$ such that
\[ \lambda (t) = \begin{cases}
t+R,& t\in [-R, \frac{1}{3} -R], \\
t+1-R,& t\in [R- \frac{1}{3}, R],
\end{cases} \]
since
$(\frac{1}{3} -R)+R = \frac{1}{3} < \frac{2}{3} = (R-\frac{1}{3})+1-R.$
A diffeomorphism $\Lambda: [-R,R]\times \real \to [0,1]\times \real$ is then given
by $\Lambda (t,u)=(\lambda (t),u).$ We put
\[ \widetilde{F} = \Lambda \circ F': W\longrightarrow [0,1]\times \real \subset \real^2. \]
As $\Lambda$ is a diffeomorphism and $F'$ is a fold map, the composition $\widetilde{F}$
is a fold map as well and $S(\widetilde{F})=S(F')=S(F)$. Since $\Lambda$ acts as the
identity on the second coordinate $u$ (i.e. the imaginary part), we have
$\mbs (\widetilde{F})=\mbs (F')=\mbs (F)$. This proves (i).
By (\ref{equ.imageoffprime}),
\[ \widetilde{F}(W_0)\subset \Lambda (\{ -R \} \times \real) = \{ 0 \} \times \real,~
  \widetilde{F}(W_1)\subset \Lambda (\{ R \} \times \real) = \{ 1 \} \times \real, \]
\[ \widetilde{F}(\interi W)\subset \Lambda ((-R,R) \times \real) = (0,1)\times \real,  \]
which establishes (ii).
Finally, (iii) holds for $(t,x)\in [0,\epsilon_0]\times S^n,$
\begin{eqnarray*}
\widetilde{F}(t,x)
& = & \Lambda (F'(t,x)) = \Lambda (t-R,f_S (x)) = (\lambda (t-R), f_S (x)) \\
& = & ((t-R)+R, f_S (x)) = (t, f_S (x)),
\end{eqnarray*}
and when 
$(t,x)\in [1-\epsilon_1,1]\times \Sigma^n,$
\begin{eqnarray*}
\widetilde{F}(t,x)
& = & \Lambda (F'(t,x)) = \Lambda (t-1+R,f_\Sigma (x)) = (\lambda (t-1+R), f_\Sigma (x)) \\
& = & ((t-1+R)+1-R, f_\Sigma (x)) = (t, f_\Sigma (x)).
\end{eqnarray*}

Since $\Sigma^n$ is not diffeomorphic to
$S^n$, $\widetilde{F}$ is not a special generic map, according to
\cite[Lemma 3.3, p. 4]{saekihtpyspheres}. (If $\widetilde{F}$ were special generic,
then one could, following Saeki, construct an h-cobordism between 
$\Sigma^n$ and $S^n$. By Smale's h-cobordism theorem, $\Sigma^n$
would then be diffeomorphic to $S^n$, a contradiction.)
Thus $\widetilde{F}$ must have an indefinite fold point. Recall 
(Proposition \ref{prop.foldmapnormalform}) that 
$\widetilde{F}: W^{n+1} \to [0,1]\times \real$ is given in local coordinates by
\[ (t,x_1,\ldots, x_n)\mapsto (t,y), \]
\[ y = -x_1^2 - \ldots - x^2_{i} + x^2_{i+1} + \ldots
  + x^2_n,~ 0\leq i \leq n. \]
A fold point is definite if and only if $i \in \{ 0,n \}$; it is indefinite
if and only if $i \in \{ 1,\ldots, n-1 \}.$ The
absolute index of a fold point $p$ is defined by $\tau (p) = \max
\{ i, n-i \}$. Thus $p$ is definite precisely when
$\tau (p)=n$ and indefinite precisely when $\tau (p)<n$.
Since $\widetilde{F}$ has an indefinite fold point $p$, it thus has a singular point
with $\tau (p)<n$. If $\mbs(F)=\mbs(\widetilde{F})$ had no loop, then it would have to be \\
\vspace{-1cm}
\begin{center}
$
\xygraph{ !{0;/r1pc/:}
!{\hcross[2]}
} $ \hspace{1cm} or \hspace{1cm}
$\xygraph{ !{0;/r1pc/:}
!{\xcaph[2]@(0)}  [ddl] !{\xcaph[2]@(0)} 
} $
\hspace{1cm} or \hspace{1cm}
$\xygraph{ !{0;/r1pc/:}
!{\hcap[2]} [rrr] !{\hcap[-2]}
}$ \hspace{.5cm} .
\end{center}
But the absolute index $\tau$ is constant along every connected component
of $S(\widetilde{F})$. Since all four endpoints are either a maximum or a minimum point,
their absolute indices are all equal to $n$. Thus $\tau (p)=n$ for every 
$p\in S(\widetilde{F})$ in all three cases --- a contradiction. Hence, $\mbs(F)$ has a loop.
\end{proof}
Let us compare this to the aggregate invariant $\mathfrak{A}(S^n)$ of the standard sphere,
\[ \mathfrak{A}(S^n) = \sum_{\overline{g}_S \in C_2 (S^n)}
 \sum_{W\in \operatorname{Cob}(S^n, S^n)} 
  Z_W (\overline{f}_S, \overline{g}_S) \in Q. \]
\begin{prop} \label{prop.istandardsphere}
For the standard sphere $S^n$, 
\[ \mathfrak{A}(S^n) = 1_{V\otimes V} + r(q) \]
for some power series $r(q)\in Q$.
\end{prop}
\begin{proof}
Let $W$ be the cylinder $W= [0,1]\times S^n$ and let $\overline{g}_S:W\to \real^2$
be the map $\overline{g}_S (t,x)=(2+t, f_S (x))$.
Up to a translation in the $t$-direction, $\overline{g}_S$ is the suspension of a Morse
function with precisely two critical points.
Thus $\overline{g}_S$ is a fold field with
$\mbs (\overline{g}_S)=1$ and hence defines an
element $\overline{g}_S \in \Fa (S^n).$ Taking $\xi (t)=2+t,$ this element clearly lies in
$C_2 (S^n)$.
The map $F:W\to \cplx$ defined by $F(t,x)=(3t, f_S (x))$ is a fold field.
Taking $\xi: [0,\frac{1}{3}]\to [0,1]$ to be $\xi (t)=3t,$ we have
$F(t,x)=(\xi (t),f_S (x))=\overline{f}_S (\xi (t),x)$ for $t\in [0,\frac{1}{3}]$,
which shows that
$F|_{[0,1/3]\times S^n} \approx \overline{f}_S.$
Taking $\xi: [\frac{2}{3},1]\to [0,1]$ to be $\xi (t)=3t-2,$ we have
$F(t,x)=(2+\xi (t),f_S (x))=\overline{g}_S (\xi (t),x)$ for $t\in [\frac{2}{3},1]$,
which shows that
$F|_{[2/3,1]\times S^n} \approx \overline{g}_S$.
Thus $F\in \Fa (W; \overline{f}_S, \overline{g}_S).$
The Brauer morphism of $F$ is
\vspace{-.4cm}
\[ \xygraph{ !{0;/r1pc/:}
!{\xcaph[2]@(0)}  [ddl] !{\xcaph[2]@(0)} [rr]{,} [lllllu] {\mbs(F)=}
} \]
the identity morphism on the object $[2]$ in $\Br$.
Then $Z_W (\overline{f}_S, \overline{g}_S)$ contains the summand
\[ Y\mbs (F) = Y(1_{[2]}) = 1_{Y([2])} = 1_{V\otimes V}. \]
\end{proof}
\begin{cor} \label{cor.asigmanotasn}
Let $\Sigma^n,$ $n\geq 5,$ be an exotic $n$-sphere, not diffeomorphic to $S^n$. Then
$\mathfrak{A}(\Sigma^n)\not= \mathfrak{A}(S^n)$ in the semiring $Q$.
\end{cor}
\begin{proof}
The image of $\mathfrak{A}(-)$ is contained in $Q(H_{2,2}) \subset Q$.
By Lemma \ref{lem.minshellshmn}, the minimal shell $S(H_{2,2})$
of $H_{2,2}$ is $S(H_{2,2})=Y(\OP_{2,2}).$ The set $\OP_{2,2}$
of open endomorphisms on the object $[2]$ in $\Br$ has
$3!!=3$ elements, and these are
\[ \OP_{2,2} = \{ 1_{[2]}, b_{1,1}, i_1 \circ e_1 \}. \]
Hence
\[ Y(\OP_{2,2}) = \{ y_1, y_2, y_3 \} \]
has cardinality at most $3$, with
\[ y_1 = Y(1_{[2]})=1_{V\otimes V},~ 
 y_2 = Y(b_{1,1})=b,~ y_3 = Y(i_1 e_1)=ie. \]
We claim that $Y(\OP_{2,2})$ has cardinality $3$. To see this, we note
that $y_1 \not= y_2$ as the braiding $b: V^{\otimes 2} \to V^{\otimes 2}$ 
is not equal to the identity.
Both $y_1$ and $y_2$ are isomorphisms with nonzero determinant.
By Proposition \ref{prop.detieiszero}, the determinant of 
$i\circ e$ vanishes. Thus $y_3 \not\in \{ y_1, y_2 \},$ as claimed.
Lemma \ref{lem.phiiso} then provides an isomorphism
\[ \phi: Q(H_{2,2}) \stackrel{\cong}{\longrightarrow}
  \bool [[q]] \oplus \bool [[q]] \oplus \bool [[q]] \]
such that
\[ \phi (y_1 \otimes b_1 + y_2 \otimes b_2 + y_3 \otimes b_3) =
  (b_1, b_2, b_3). \]
By Proposition \ref{prop.istandardsphere}, 
\[ \phi \mathfrak{A}(S^n) = (1+a, b_2, b_3), \]
whereas by Theorem \ref{thm.iexoticsphere},
\[ \phi \mathfrak{A}(\Sigma^n) = (qa', b'_2, b'_3). \]
In the Boolean power series semiring $\bool [[q]]$, there do not
exist elements $a,a'$ with $1+a=qa'$. Thus $\phi \mathfrak{A}(S^n)\not= \phi \mathfrak{A}(\Sigma^n)$
and in particular $\mathfrak{A}(S^n)\not= \mathfrak{A}(\Sigma^n)$.
\end{proof}

Given any homotopy sphere $\Sigma^n_1$, the invariant $\mathfrak{A}$ can
be modified in such a way that it will distinguish $\Sigma^n_1$ from any other
homotopy sphere $\Sigma^n_2$. Let $f_1:\Sigma_1 \to \real$ be a Morse function
with precisely two critical points and let $\overline{f}_1$ be its suspension.
Given any homotopy sphere $M^n$, let
\[ \mathfrak{A}_1 (M) := \sum_{\overline{f}_M \in C_2 (M)}
 \sum_{W\in \operatorname{Cob}(\Sigma_1, M)} 
  Z_W (\overline{f}_1, \overline{f}_M) \in Q. \]
If $M=\Sigma_2$ is not diffeomorphic to $\Sigma_1,$ then
$\mathfrak{A}(\Sigma_2) \not= \mathfrak{A}(\Sigma_1)$.

\providecommand{\bysame}{\leavevmode\hbox to3em{\hrulefill}\thinspace}
\providecommand{\MR}{\relax\ifhmode\unskip\space\fi MR }
\providecommand{\MRhref}[2]{%
  \href{http://www.ams.org/mathscinet-getitem?mr=#1}{#2}
}
\providecommand{\href}[2]{#2}

\end{document}